\documentclass[10pt,a4paper]{article}
\title{On $p$-adic $L$-functions for $\GL_{2n}$ in finite slope Shalika families}
\author{Daniel Barrera Salazar, Mladen Dimitrov and Chris Williams}
\date{}

\newcommand{\s}{\setlength{\itemsep}{0pt}}
\usepackage[margin=1.24in]{geometry}

\usepackage{soul}

\usepackage{multicol}
\usepackage{amssymb}

\usepackage{fancyhdr}
\usepackage{tocloft}

\usepackage{mathtools}
\usepackage{environ}
\NewEnviron{myquote}{\vspace{2ex}\par
	\hfill\llap{($\dagger$)}\hfill\parbox{\textwidth-2cm}
	{\emph{\BODY}}
	\vspace{2ex}\par}

\newcommand{\disp}[1]{$#1$}
\renewcommand{\baselinestretch}{1}

\makeatletter 
\def\input@path{{../}} 
\makeatother

\usepackage{preamble}

\newcommand{\DJR}{{\mathrm{DJR}}}
\newcommand{\sar}[2]{\ar@{}[#1]|-*[@]{#2}}
\newcommand{\UPS}{\theta}
\newcommand{\Ind}{\mathrm{Ind}}

\newcommand{\ep}{e_{\pri}(\tilde\pi,\chi,j)}
\newcommand{\epp}{e_{\pri}'(\tilde\pi,\chi,j)}
\newcommand{\einf}{e_{\infty}(\pi,\chi,j)}
\newcommand{\epy}{e_{\pri}(\tilde\pi_y,\chi,j)}
\newcommand{\einfy}{e_{\infty}(\pi_y,\chi,j)}

\makeatletter
\renewcommand\tableofcontents{%
	\subsubsection*{\contentsname
		\@mkboth{%
			\MakeUppercase\contentsname}{\MakeUppercase\contentsname}}%
	\@starttoc{toc}%
}
\makeatother

\begin{document}

	\renewcommand{\thefootnote}{\fnsymbol{footnote}} 
\footnotetext{\today. \ \ \  2020 MSC: Primary 11F33, 11F67; Secondary 11R23. }

\maketitle

\begin{abstract}
	In this paper, we propose and explore a new connection in the study of $p$-adic $L$-functions and eigenvarieties. We use it to prove results on the geometry of the cuspidal eigenvariety for $\mathrm{GL}_{2n}$ over a totally real number field $F$ at classical points admitting Shalika models. We also construct $p$-adic $L$-functions over the eigenvariety around these points. Our proofs proceed in the opposite direction to established methods: rather than using the geometry of eigenvarieties to deduce results about $p$-adic $L$-functions, we instead show that non-vanishing of a (standard) $p$-adic $L$-function implies  smoothness  of the eigenvariety  at such points. Key to our methods are a family of distribution-valued functionals on (parahoric) overconvergent cohomology groups, which we construct via $p$-adic interpolation  of classical representation-theoretic branching laws for $\mathrm{GL}_n \times \mathrm{GL}_n \subset \mathrm{GL}_{2n}$. 
	
	More precisely, we use our functionals to attach a $p$-adic $L$-function to a non-critical refinement  $\tilde\pi$ of a regular algebraic cuspidal automorphic representation $\pi$ of $\mathrm{GL}_{2n}/F$ which is spherical at $p$ and admits a Shalika model. Our new parahoric distribution coefficients allow us to obtain optimal non-critical slope and growth bounds for this construction. When $\pi$ has regular weight and the corresponding  $p$-adic Galois representation is  irreducible, we exploit non-vanishing of our functionals to show that the parabolic eigenvariety for $\mathrm{GL}_{2n}/F$ is \'etale at $\tilde\pi$ over an $([F:\mathbb{Q}]+1)$-dimensional  weight space	and contains a dense set of classical  points admitting Shalika models. 
	Under a hypothesis on the local Shalika models at bad places which is empty for $\pi$ of level 1, we construct a $p$-adic $L$-function for the family.
\end{abstract}

\setcounter{tocdepth}{2}
	\renewcommand{\baselinestretch}{1.0}
	\setcounter{tocdepth}{1}
\footnotesize
\tableofcontents
\normalsize

\section{Introduction}

The arithmetic of $L$-functions has long been a topic of intense interest in number theory. Via the Bloch--Kato Conjecture, the special values of $L$-functions are expected to carry deep algebraic data. Most recent progress towards this conjecture has come through $p$-adic methods -- more precisely, through understanding all of (a) $p$-adic $L$-functions, (b) classical $p$-adic families, and (c) $p$-adic $L$-functions over classical $p$-adic families. Where one has all three, they have been crucial in proofs of Iwasawa Main Conjectures and cases of the Bloch--Kato Conjecture. It is therefore natural to ask whether one can obtain (a), (b) and (c) for any regular algebraic cuspidal automorphic representation (RACAR) of a reductive group $G$.

For (a), at least, this is expected to be possible in great generality, thanks to conjectures of Coates--Perrin-Riou and Panchishkin \cite{coatesperrinriou89, coates89, Pan94}. However, our understanding of fundamental cases -- for example, $\GL_N$ for $N > 2$ -- remains poor, with relatively few constructions of $p$-adic $L$-functions in this case, most of which assume a $p$-ordinarity condition. 

The theory of $p$-adic families is more subtle still. Singularly, (b) can \emph{fail} when moving beyond $\GL_2$; there exist RACARs of $\GL_N$ that are `arithmetically rigid', not varying in any classical $p$-adic family (that is, a positive-dimensional subspace of an eigenvariety containing a Zariski-dense set of classical points; see e.g.\  \cite{APS07}). This contrasts  sharply with the cases of, for example, Hilbert or Siegel modular forms, where it is expected all RACARs can be classically varied. 

To approach (c), one needs not only the existence of a classical family, but also a precise description of its geometry. For example, one needs to know  whether  such a family is  smooth or \'etale over the weight space. Well-established methods for studying eigenvarieties break down for $\GL_N$ with $N > 2$, owing to RACARs contributing to multiple degrees of cohomology and the underlying locally symmetric spaces not admitting any algebraic structure. The geometry of the $\GL_N$ eigenvariety is thus largely mysterious, meaning there are few instances in this setting where (c) is approachable at present.

\medskip

In this paper, we prove new cases of (a), (b) and (c) for
regular algebraic  \emph{symplectic-type} cuspidal automorphic representations  (RASCARs) $\pi$ of $\GL_N$ over a totally real number field $F$, described below in Theorems \ref{thm:intro non-ord}, \ref{thm:intro shalika families} and \ref{thm:intro 3} respectively. 

\medskip

The technical heart of our approach is a construction of `evaluation maps', a $p$-adic integration theory on overconvergent modular symbols for $\GL_N$. The special values of these maps compute explicit multiples of classical complex $L$-values of RASCARs. Such maps are very familiar in the setting of $\GL_2$, where they have been used in many papers to study $p$-adic $L$-functions (see \S\ref{sec:gl2}), but they had not previously been constructed for any higher-dimensional $\GL_N$. The $\GL_2$ constructions do not easily generalise; the relative simplicity of the $\GL_2$ setting hides substantial representation-theoretic obstructions that arise in higher dimension (see \S\ref{sec:non-ordinary intro}). A key new input in our constructions is a $p$-adic interpolation, in both cyclotomic and weight directions, of higher-dimensional branching laws in representation theory. This occupies \S\ref{sec:classical evaluations} and \S\ref{sec:galois evaluations}.

Once constructed, evaluation maps have powerful consequences. Their utility in constructing $p$-adic $L$-functions is already well-documented in the $\GL_2$ case, and similarly we use them to construct $p$-adic $L$-functions for RASCARs of $\GL_N$. However, we also push their use further than previous works. One particularly striking consequence is the following strong version of (b) in this setting, made precise in Theorem~\ref{thm:intro shalika families} (and Theorem~\ref{thm:shalika family}):
\begin{myquote} 
	Let $\pi$ be a RASCAR with regular weight and irreducible Galois representation. Then the parabolic $\GL_N$-eigenvariety is \'etale over the pure weight space at certain non-critical $p$-refinements $\tilde\pi$ of $\pi$. Hence $\tilde\pi$ varies in a unique classical $p$-adic family.
\end{myquote}

Our proof of this result turns traditional methods upside down. There is a long and storied history of applying the geometry of eigenvarieties to construct and study $p$-adic $L$-functions; for example, this is the central tenet of Bella\"{i}che's celebrated paper on critical $p$-adic $L$-functions \cite{Bel12}.  There are also some works connecting the \'etaleness of an eigenvariety to the non-vanishing of an \emph{adjoint} $p$-adic $L$-function (e.g.\ \cite[\S VIII]{Eigenbook}, \cite{BBL}, \cite{PHL-JFW}), but similarly all of these works require prior knowledge about existence and properties of $p$-adic families to study the $p$-adic $L$-function.

Our methods add to this rich story. However, they differ considerably in that unlike all previous works, we proceed in the opposite direction.
\emph{We first construct $p$-adic $L$-functions and then we use them to construct $p$-adic families}. Indeed, we use our evaluation maps to show that non-vanishing of the $p$-adic $L$-function of $\tilde\pi$ -- guaranteed by regular weight -- implies faithfulness of a Hecke algebra as a module over weight space, thus producing dimension in the eigenvariety and implying the existence of classical families.  In addition, rather than using the adjoint $p$-adic $L$-function, to our knowledge we give the first instance where non-vanishing of a \emph{standard} $p$-adic $L$-function is used to control the geometry of an eigenvariety.

\medskip

The methods we develop in this paper have more general applications. In particular, the proof of $(\dagger)$ -- which occupies all of \S\ref{sec:shalika families no new} -- shows that evaluation maps, through interpolation of branching laws, can be a powerful tool in understanding the geometry of classical $p$-adic families. We have explored this further in sequel papers \cite{BDGJW,classical-locus}. More generally, our methods suggest the natural setting to consider evaluation maps is that of spherical varieties, giving strong connections between the geometry of eigenvarieties and automorphic period integrals in the Gan--Gross--Prasad conjectures (see e.g.\ \cite{Zha14,PWZ19,Zyd19}), as well as to $p$-adic interpolation in Sakellaridis--Venkatesh's relative Langlands program \cite{SV17}. We will use the methods of this paper to construct and study $p$-adic interpolations of such period integrals in future work.

We expect there to be further arithmetic applications of our evaluation maps. In the $\GL_2$ setting, beyond their applications to $p$-adic $L$-functions, analogues of these maps have further been used to study periods and congruences between base-change and non-base-change Bianchi modular forms \cite{Hid99,TU18}, study $\cL$-invariants and trivial zero conjectures \cite{BW17,BDJ17}, construct Stark--Heegner cycles predicted by the Bloch--Kato conjecture \cite{VW19,Ven21}, and prove generalisations of Hida duality \cite{BetWil20}. Few of these results/constructions have been carried out for higher-dimensional $\GL_N$, by any method. We anticipate similar applications of evaluation maps are possible for RASCARs, and again hope to return to this in subsequent work.

Finally, we mention applications of our results themselves, which -- as explained above -- should ultimately have applications towards the Bloch--Kato and Iwasawa Main Conjectures. They have already led to research in this direction for $\mathrm{GL}_{2n}$ \cite{Roc20, LR-GL2n}. 
There are more immediate applications to other groups such as $\mathrm{GSp}_4$. In a sequel \cite{BDGJW} to this paper, we have crucially used the methods developed here to prove a result on the variation of $p$-adic $L$-functions required in \cite{LZ20}. This result -- which was announced as Theorem 17.6.2 \emph{ibid}., where it was deferred to future work of the present authors -- was used by Loeffler and Zerbes to prove cases of the Bloch--Kato Conjecture for $\mathrm{GSp}_4$ \cite{LZ20} and for symmetric cubes of modular forms \cite{LZ-cube}.

\subsection{Set-up and previous work} \label{sec:set-up and previous work}

Fix forever an isomorphism $i_p : \C \isorightarrow \overline{\Q}_p$.

We say a regular algebraic cuspidal automorphic representation (RACAR) $\pi$ of $\GL_N(\A_F)$ is \emph{essentially self-dual} if $\pi^\vee \cong \pi \otimes \eta^{-1}$ for a Hecke character $\eta$, which we henceforth fix. 
Such a $\pi$ is either  $\eta$-\emph{orthogonal} or $\eta$-\emph{symplectic},    
depending respectively on whether the twisted symmetric  square $L$-function $L(\mathrm{Sym}^2 \pi \times \eta^{-1}, s)$ or 
the twisted exterior  square $L$-function $L(\bigwedge^2 \pi \times \eta^{-1}, s)$  has a  pole at $s = 1$ (see \cite{MW89,Sha97}, summarised in detail in \cite[Lem.\ 2.1]{HS16}). 
Our focus is on the $\eta$-symplectic case. By \cite{AS06},  $\pi$ is $\eta$-symplectic if and only if 
$N=2n$ is  even and either (so both) of the following hold:
\begin{itemize}\setlength{\itemsep}{0pt}
	\item[(i)] $\pi$ admits an $(\eta,\psi)$-Shalika model (see \S\ref{sec:shalika models});
	\item[(ii)] $\pi$ is the transfer of a globally generic cuspidal automorphic representation $\Pi$ of $\mathrm{GSpin}_{2n+1}(\A_F)$ with central character $\eta$.
\end{itemize}
Henceforth we will mainly use  (i) and call such a representation $\pi$ a \emph{RASCAR}. 

\pagebreak[2]

Let $G = \mathrm{Res}_{\cO_F/\Z}\GL_{2n}$. Let $\Sigma$ be the set of real embeddings of $F$, and let $\lambda = (\lambda_\sigma)_{\sigma \in \Sigma}$ be a Borel-dominant weight for $G$, i.e.\ $\lambda_\sigma = (\lambda_{\sigma,1},\dots,\lambda_{\sigma,2n}) \in \Z^{2n}$ with $\lambda_{\sigma,1} \geqslant \cdots \geqslant \lambda_{\sigma,2n}$. Let $\pi$ be a RASCAR of $G(\A)$ of weight $\lambda$; our convention is that $\pi$ is cohomological with respect to the coefficient system $V_\lambda^\vee$, where $V_\lambda$ is the algebraic representation of highest weight $\lambda$. Let 
\begin{equation}\label{eq:deligne-critical}
\mathrm{Crit}(\lambda) \defeq \{ j \in \Z: -\lambda_{\sigma,n} \leqslant j \leqslant -\lambda_{\sigma,n+1}\ \forall \sigma  \in \Sigma \}.
\end{equation}
Then $j \in \mathrm{Crit}(\lambda)$ if and only if $L(\pi,j+\tfrac{1}{2})$ is a Deligne-critical value; and Grobner--Raghuram showed in \cite{GR2} that these $L$-values, and their twists by finite order Hecke characters, are algebraic multiples of a finite set of complex periods. 

Let $p$ be a prime such that $\pi_{\pri}$ is spherical for each $\pri|p$ (or more generally, each $\pi_{\pri}$ satisfies (C2) of Conditions \ref{cond:running assumptions}). A $p$-adic $L$-function for $\pi$ is a $p$-adic distribution of controlled growth that interpolates the algebraic parts of Deligne-critical $L$-values.  For such $p$-adic interpolation it is essential to take a $p$-refinement $\tilde\pi$ of $\pi$, i.e.\ to work at non-maximal level at $p$ (e.g.\ for $\GL_2$, this is the process of passing from a newform of level $\Gamma_1(M)$ to an eigenform of level $\Gamma_1(M) \cap \Gamma_0(p)$). A standard approach is to refine to Iwahori level at $\pri|p$, which in our case corresponds to choosing a full triangulation of the $2n$-dimensional local Galois representation. However, the Panchishkin condition  (see \cite[\S2.1]{Loe20}, inspired by \cite{Pan94}) predicts that the $p$-adic $L$-function should not depend on a full triangulation, but only on a suitable $n$-dimensional stable submodule. This suggests that the natural level to take at $\pri|p$ is not Iwahoric, but the parahoric subgroup $J_{\pri}$ relative to the  parabolic subgroup $Q$ of $G$ with Levi 
\[
H = \mathrm{Res}_{\cO_F/\Z}(\GL_n \times \GL_n).
\]
In this paper, we indeed show that parahoric level is optimal for this construction (see \S\ref{sec:parahoric level}).

Let $\cO_{\pri}$ be the ring of integers in the completion $F_{\pri}$ of $F$ at $\pri$, and fix a uniformiser $\varpi_{\pri}$ in $\cO_{\pri}$. 
 Of central importance is the Hecke operator $U_{\pri} = \left[J_{\pri}\smallmatrd{\varpi_{\pri}I_n}{}{}{I_n} J_{\pri}\right]$ and its optimal integral normalisation $U_{\pri}^\circ$ (see \S\ref{sec:slope-decomp}). A \emph{$Q$-refinement of $\pi_{\pri}$} is a choice of (non-zero) eigenvalue $\alpha_{\pri}$ of $U_{\pri}$ acting on  $\pi_{\pri}^{J_{\pri}}$, and a \emph{$Q$-refinement of $\pi$} is a choice $\tilde\pi=(\pi, (\alpha_{\pri})_{\pri|p})$ of $Q$-refinement of $\pi_{\pri}$ for each $\pri|p$. Following \cite[Definition~3.5]{DJR18}, we say the $Q$-refinement $\tilde\pi$ is 
\emph{Shalika} if each $\alpha_{\pri}$ is a simple eigenvalue that interacts well with the Shalika model in a precise sense (see \S\ref{ss:the U_p-refined line}).

The following $p$-adic reformulation of \cite{GR2} performed in \cite{DJR18} will be crucial for us. 
For an open compact subgroup $K = \prod_v K_v \subset G(\A_f)$ that is \emph{parahoric at $p$} (that is, with $K_{\pri} = J_{\pri}$ at each $\pri|p$), let $S_K$ be the associated locally symmetric space for $G$ (see \eqref{eq:loc sym space}). Consider the  compactly supported cohomology groups $\hc{t}(S_K,\sV_\lambda^\vee(\overline{\Q}_p))$ in degree  $t = d(n^2+n-1)$, where $\sV_\lambda^\vee$ is the ($p$-adic) algebraic local system attached to $V_\lambda^\vee$. Then \cite{DJR18} proved that
\begin{enumerate}[(a)]\setlength{\itemsep}{0pt}
	\item there exists a family of $p$-adic classical evaluation maps 
	\disp{
	\mathrm{Ev}_{\chi,j} : \hc{t}(S_K,\sV_\lambda^\vee(\overline{\Q}_p)) \to \overline{\Q}_p,
}
	indexed by finite order Hecke characters $\chi$ of $p$-power conductor and $j \in \mathrm{Crit}(\lambda)$, and
	\item for a sufficiently small but inexplicit level $K = K(\tilde\pi)$ (see \eqref{eq:level group}), attached to $\tilde\pi$ and $i_p$ there exists
	 a classical eigenclass \disp{\phi_{\tilde\pi} \in \hc{t}(S_{K(\tilde\pi)},\sV_\lambda^\vee(\overline{\Q}_p))}
	  whose image $\mathrm{Ev}_{\chi,j}(\phi_{\tilde\pi})$ equals $i_p\left(L(\pi\otimes\chi, j+\tfrac{1}{2}) / \Omega_{\pi}^\epsilon\right)$ 
	 up to a non-zero factor, where $\Omega_\pi^\epsilon$ is a complex period, depending only on $\pi$ and the multi-sign $\epsilon = (-1)^j\cdot (\chi\eta)_\infty \in \{\pm1\}^\Sigma$ (recalling $\pi$ is $\eta$-symplectic).  
\end{enumerate}
In Theorem~\ref{thm:critical value}, based on \cite{JST}, we improve this by making the factor  completely explicit.

In \cite{DJR18}, when further $\tilde\pi$ is $Q$-\emph{ordinary} -- that is, when the integral normalisation
$\alpha_{\pri}^\circ$ of $\alpha_{\pri}$ is a $p$-adic unit --  the authors constructed the corresponding $p$-adic $L$-function $\cL_p(\tilde\pi)$ by proving directly the so-called Manin relations.

\begin{figure}[h!]
	\[
	\xymatrix@R=10mm@C=1mm{
		(\text{T})&&&\Phi_{\sC}\ar@{|->}@/^2pc/[rrrrrrrrrr]\ar@{|->}[dd] &\in & \hc{t}(S_{K(\tilde\pi)},\sD_\Omega)\ar[dd]^{\mathrm{sp}_\lambda} \ar[rrrrrr]^-{\mathrm{Ev}_\Omega} &&&&& & \cD(\mathrm{Gal}_p,\cO_\Omega)\ar[dd]^{\mathrm{sp}_\lambda} & \ni & \cL_p^{\sC}\ar @{|->}[dd]\\
		&&&&&&&&&&&&&\\
		(\text{M})&&& \Phi_{\tilde\pi} \ar@{|->}@/^2pc/[rrrrrrrrrr] & \in & \hc{t}(S_{K(\tilde\pi)},\sD_\lambda)\ar[dd]^{r_\lambda} \ar[rrrrrr]^-{\mathrm{Ev}_\lambda} &&&& && \cD(\mathrm{Gal}_p,\overline{\Q}_p)\ar[dd]^{\int \chi\chi_{\cyc}^j} & \ni & \cL_p(\tilde\pi)\ar @{|->}[dd]\\
		&&&&&&&&&&&&&\\
		(\text{B}) &&&\phi_{\tilde\pi} \ar@{|->}@/^2pc/[rrrrrrrrrr]\ar@{-->}[uu]^{\rotatebox{90}{\text{classicality}}}  &\in &\hc{t}(S_{K(\tilde\pi)},\sV_\lambda^\vee) \ar[rrrrrr]^{\mathrm{Ev}_{\chi,j}} &&&&& & \overline{\Q}_p & \ni &i_p\left[(*)\frac{L(\pi \otimes \chi,j+1/2)}{\Omega_{\pi}^\epsilon}\right]\\
		&&&&&&&&&& \pi\ar @{.>}[lllllllu]\ar @{.>}[rrru] &&&&&
	}
	\]
	\caption{\label{main diagram}\emph{Strategy of our constructions}}
\end{figure}
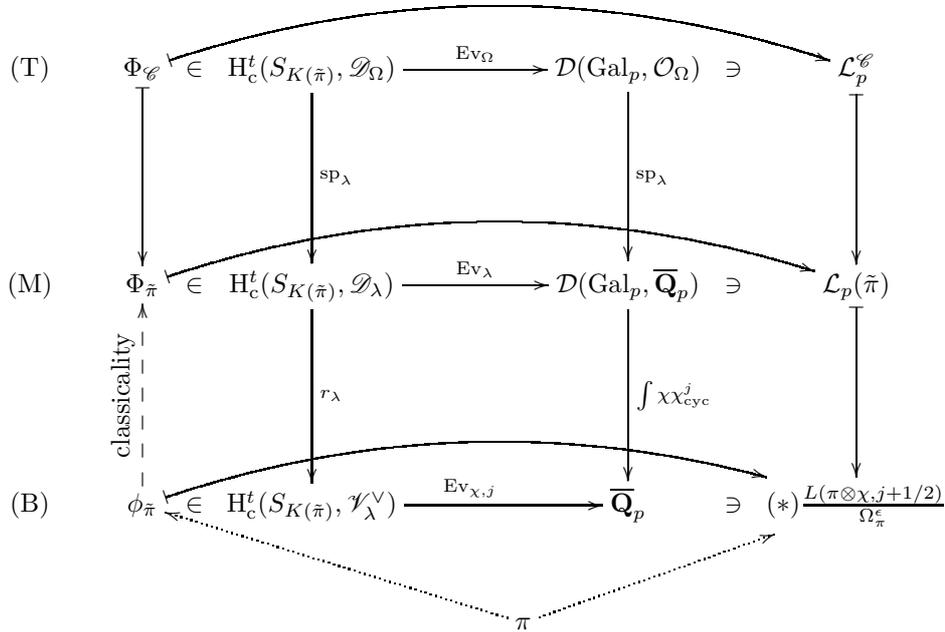

\subsection{Main results and methods} In this section, we state our three main results precisely (in Theorems \ref{thm:intro non-ord}, \ref{thm:intro shalika families} and \ref{thm:intro 3}). All three are proved using \emph{overconvergent cohomology}.

\subsubsection{Overconvergent cohomology and families beyond $\GL_2$}\label{sec:gl2}

The utility of overconvergent cohomology in constructing $p$-adic $L$-functions and $p$-adic families is very familiar. For example, 
\begin{enumerate}[(1)]
	\item it was the central method used in the papers \cite{GS93,PS11,PS12,Bel12,Bar15,Wil17,BW_CJM, BDJ17,BH17,BW18,BW-Iwasawa,Lee21}, including several by the present authors, to construct and study $p$-adic $L$-functions attached to modular forms on $\mathrm{GL}_2$ (in various settings, 
 ranging from ordinary modular forms over $\Q$ to finite slope families over general number fields);
	\item for a general quasi-split group $G$, it was used in \cite{AS08,Urb11,Han17,BW20} to construct eigenvarieties (and hence $p$-adic families of overconvergent systems of eigenvalues).
\end{enumerate}

Despite the huge generality for which overconvergent cohomology has been developed in (2), none of the program in (1) has been generalised from $\GL_2$ to higher $\GL_N$. From the papers above the \emph{strategy} for carrying out such a generalisation is clear; it is summarised in Figure \ref{main diagram} below.  The results of \cite{GR2,DJR18} comprise the bottom row (B). However,  fundamental obstacles and new features arise when trying to implement this strategy to construct the middle and top row  beyond $\GL_2$. For example, we have already commented that the theory of $p$-adic families (and hence row (T)) is not well-understood. Additionally, unlike for $\GL_2$ there is subtlety over the level at $p$ (Iwahoric vs.\ parahoric) at which one works; we describe this in detail in \S\ref{sec:parahoric level} below.

	 Before any of this, however, one must first construct (horizontal) \emph{evaluation maps}  in Figure \ref{main diagram}. In row (B), where there is no $p$-adic variation, the evaluation maps $\mathrm{Ev}_{\chi,j}$ depend on a separate choice of classical representation-theoretic branching law for each $j \in \mathrm{Crit}(\lambda)$.	For $\GL_2$, the classical coefficients $V_\lambda^\vee$ are just spaces of polynomial functions on $\cO_F\otimes_{\Z}\Zp$. This simplicity yields obvious canonical choices of branching laws, which are readily $p$-adically interpolated, making the construction of rows (M) and (T) straightforward. 
	 	
	For higher $\GL_N$, the coefficient modules are hard to describe explicitly. Choices of branching laws are no longer canonical, and must be carefully aligned for $p$-adic interpolation to even be possible. A key technical result of \cite{DJR18} was Theorem 2.4.1, where Januszewski, Raghuram and one of us carried out such an alignment, using finite-dimensional coefficient modules, for fixed $\lambda$ and $j$ varying in a finite set. However, their method does not generalise to our infinite-dimensional distribution coefficients. In this paper, we develop a new way of aligning branching laws, sketched in \S\ref{sec:non-ordinary intro} below, and use it to construct $\mathrm{Ev}_\lambda$ in our setting. We also explain how to further align these branching laws as $\lambda$ varies in a family $\Omega$, and use this to construct $\mathrm{Ev}_\Omega$.

\subsubsection{$p$-adic $L$-functions for finite slope RASCARs}\label{sec:non-ordinary intro}

Our first main result is the construction of a $p$-adic $L$-function attached to a finite slope RASCAR of $\GL_{2n}$.  More precisely, as predicted by Panchishkin \cite{Pan94} we construct a $p$-adic distribution on $\Galp$, the Galois group of the maximal abelian extension of $F$ unramified outside $p\infty$, satisfying growth and interpolation properties.

Overconvergent cohomology groups are defined by replacing the algebraic coefficients $V_\lambda^\vee$ with spaces of  $p$-adic distributions $\cD_\lambda$ and $\cD_\Omega$, where $\Omega$ is a (rigid analytic) family in which the weight $\lambda$ varies.  To work at $Q$-parahoric level, as mentioned above (see also \S\ref{sec:parahoric level} below), we use a new class of parahoric distributions from our companion paper \cite{BW20}, described here in \S\ref{sec:overconvergent cohomology}. These spaces are constructed as a `double induction': first we take an algebraic induction of $\lambda$ to $H$, and then a locally analytic induction to the parahoric $J_p \subset G(\Qp)$. These distributions are analytic along the unipotent radical of $Q$, but algebraic along all other variables (unlike Iwahoric distributions, which are analytic in all variables).

At any fixed $\lambda$, the space $\cD_\lambda$ admits $V_{\lambda}^\vee$ as a quotient, inducing a specialisation map 
\[
r_\lambda : \hc{\bullet}(S_{K(\tilde\pi)},\sD_\lambda) \to \hc{\bullet}(S_{K(\tilde\pi)},\sV_\lambda^\vee).
\]
We say the refinement $\tilde\pi = (\pi, \{\alpha_{\pri}\}_{\pri|p}\}$ is \emph{non-$Q$-critical} if $r_\lambda$ becomes an isomorphism after restricting to the generalised Hecke eigenspaces at $\tilde\pi$ (see Definition~\ref{def:non-Q-critical}); this guaranteed by having \emph{non-$Q$-critical slope} at $p$ (Theorem~\ref{thm:control}). For non-$Q$-critical $\tilde\pi$, lifting $\phi_{\tilde\pi}$ under the isomorphism $r_\lambda$, we obtain a class $\Phi_{\tilde\pi} \in \hc{t}(S_{K(\tilde\pi)},\sD_\lambda)$. 

We now come to our major new input: the construction of a family of \emph{evaluation maps} on overconvergent cohomology groups, comprising the horizontal maps in Figure \ref{main diagram} and occupying all of \S\ref{sec:abstract evaluation maps}--\S\ref{sec:galois evaluations}. More precisely, we construct a map
\[
\mathrm{Ev}_\lambda : \hc{t}(S_{K(\tilde\pi)},\cD_\lambda)[U_{\pri}^\circ - \alpha_{\pri}^\circ : \pri|p] \longrightarrow \cD(\Galp,\overline{\Q}_p)
\] 
on the $\{U_{\pri}^\circ = \alpha_{\pri}^\circ\}_{\pri|p}$ eigenspace, valued in a space $\cD(\Galp,\overline{\Q}_p)$ of locally analytic distributions, which interpolates all of the $\mathrm{Ev}_{\chi,j}$ simultaneously as $\chi$ and $j$ vary. 
We note that $\cD(\Galp,\overline{\Q}_p)$ is the space which Panchishkin predicts should contain the $p$-adic $L$-function of $\tilde\pi$.

The existence of the map $\mathrm{Ev}_\lambda$ is a `$p$-adic interpolation of branching laws' for $H \subset G$ that gives, for free, all of the classical Manin relations (as computed in \cite{DJR18}). More precisely, for $j_1,j_2 \in \Z$, let $V_{(j_1,j_2)}^H$ denote the $H$-representation $\det_1^{j_1}\det_2^{j_2}$, the algebraic representation of highest weight $(j_1,...,j_1,j_2,...,j_2)_{\sigma\in\Sigma}$. Let $\sw \in \Z$ be the purity weight of $\lambda$ (see \S\ref{sec:algebraic weights}). Then we have the following reinterpretation of the Deligne-critical $L$-values \eqref{eq:deligne-critical}: 

\begin{quote}
\textbf{Branching law:}\emph{ $j \in \mathrm{Crit}(\lambda) \iff V_{(-j,\sw+j)}^{H} \subset V_\lambda\big|_{H}$ with multiplicity one.}
\end{quote}

For each fixed $j \in \mathrm{Crit}(\lambda)$, the map $\mathrm{Ev}_{\chi,j}$ depends on a (non-canonical) choice of basis vector $v_j \in V_{(-j,\sw+j)}^H \subset V_\lambda|_H$. To construct $\mathrm{Ev}_\lambda$ these must be carefully aligned. We do this in \S\ref{sec:classical evaluations} by reinterpreting the module $V_\lambda$ as a double algebraic induction, collapsing all choices onto a single choice of branching law for $\mathrm{Res}_{F/\Q}\GL_n$ diagonally embedded inside $H$. In the process, we obtain an explicit description of the branching law for $H \subset G$.

In \S\ref{sec:galois evaluations}, we use our parahoric distributions to construct the required $p$-adic interpolation of the above branching laws in the cyclotomic direction. The parahoric setting allows us to isolate (in the first induction) the algebraic branching law for $\mathrm{Res}_{F/\Q}\GL_n$ fixed above, whilst allowing the second induction to vary $p$-adic analytically. Essential for this interpolation is a family of support conditions for evaluation maps, arising from a choice of open orbit representative $\xi$ for the spherical variety $G/H$. This is a representation-theoretic avatar of the familiar phenomenon in constructing $p$-adic $L$-functions, whereby one must modify the Euler factor at $p$.

\begin{definition}\label{intro:definition}
	Let $\pi$ be a RASCAR of $G(\A)$ spherical at $p$. Suppose $\pi$ admits a non-$Q$-critical, Shalika $Q$-refinement $\tilde\pi = (\pi,(\alpha_{\pri})_{\pri|p})$, and let $\Phi_{\tilde\pi}$ be the resulting overconvergent lift of $\phi_{\tilde\pi}$, which is a $U_{\pri}^\circ$-eigenclass for all $\pri|p$.  Define the \emph{$p$-adic $L$-function} of $\tilde\pi$ to be 
	\[
	\cL_p(\tilde\pi) \defeq \mathrm{Ev}_\lambda(\Phi_{\tilde\pi}) \in \cD(\Galp,\overline{\Q}_p).
	\]
	This depends on $\phi_{\tilde\pi}$, hence (by \eqref{eq:rational diagram}) on the restriction of $i_p$ to the number field $E$; but this is an expected indeterminacy (corresponding to \cite[(14)]{coates89}) which we largely suppress.
\end{definition}

Our first main result, proved in Theorem~\ref{thm:non-ordinary} and illustrated in the middle row (M) of Figure \ref{main diagram}, is that the distribution $\cL_p(\tilde\pi)$ satisfies suitable growth and interpolation properties, justifying the terminology `$p$-adic $L$-function'. Observe that finite order Hecke characters of $p$-power conductor, and the $p$-adic cyclotomic character $\chi_{\cyc}$, are characters of $\Galp$ (see \S\ref{sec:galois groups}).

\begin{theorem-intro}\label{thm:intro non-ord}
	Let $\pi$ and $\tilde\pi$ be as in Definition~\ref{intro:definition}. Then:
	\begin{itemize}\s
		\item[(1)] $\cL_p(\tilde\pi)$ is admissible of growth $h_p \defeq v_p\left(\prod_{\pri|p} (\alpha_{\pri}^\circ)^{e_{\pri}}\right)$ (see Definition \ref{def:admissible});
		\item[(2)] for all finite order Hecke characters $\chi$ of conductor $\prod_{\pri|p}\pri^{\beta_{\pri}}$ and all $j \in \mathrm{Crit}(\lambda)$, we have
		\begin{align*}
		i_p^{-1}(\mathcal{L}_p(\tilde\pi, \chi\chi_{\cyc}^j )) = A\cdot   \tau(\chi_f)^n   \mathrm{N}_{F/\Q}(\mathfrak{d})^{jn} \prod_{\pri|p} \ep \cdot \einf \cdot \frac{L^{(p)}\big(\pi\otimes\chi, j+\tfrac{1}{2}\big)}{\Omega_\pi^{\epsilon}},
		\end{align*}
		where $\ep$ is the Coates--Perrin-Riou factor at $\pri$ (defined in Theorem \ref{thm:non-ordinary}), $\einf$ is the modified Euler factor at infinity (Definition \ref{def:e_infty}), $A \in \Q^\times$ is a constant \eqref{eq:A}, $\mathfrak{d}$ is the different of $F/\Q$, $\epsilon = (\chi\chi_{\cyc}^j\eta)_\infty \in \{\pm 1\}^\Sigma$,  $\tau(\chi_f)$ is the Gauss sum, $L^{(p)}(-)$ is the (finite) $L$-function without factors at $p$, and $\Omega_\pi^\epsilon$ is a complex period. 
	\end{itemize}
	If $h_p < \#\mathrm{Crit}(\lambda)$, then the restriction of $\cL_p(\tilde\pi)$ to the cyclotomic line is unique with these properties; if further Leopoldt's conjecture holds for $F$ at $p$, then $\cL_p(\tilde\pi)$ itself is unique. 
\end{theorem-intro}

\begin{remark}
	Note that, exploiting work of Jiang--Sun--Tian \cite{JST}, we are able to prove the full expected period relations at infinity. As a consequence, both sides of (2) lie in an explicit number field $E(\chi)$, with $E$ defined in \S\ref{sec:hecke outside S}. Moreover $\mathcal{L}_p(\tilde\pi)$ can be taken with coefficients in a finite extension $L/\Qp$ containing $i_p(E)$. We have suppressed this here to ease notation.
	
	In the first draft of this paper, in the interpolation we restricted to characters ramified at all $\pri|p$; but in a separate paper \cite{BDGJW} with Graham and Jorza, we computed the relevant unramified zeta integrals which -- when combined with the construction here -- give the full Coates--Perrin-Riou/Panchishkin conjecture in this case. Even in the ordinary case Theorem \ref{thm:intro non-ord} upgrades \cite{DJR18}; indeed, it also corrects a small error in the interpolation formula \emph{ibid}.\ (see Appendix (2)).
\end{remark}

\subsubsection{Benefits of the parahoric approach}\label{sec:parahoric level}

Let us now precisely highlight the benefit of using parahoric (rather than Iwahori) distributions. Our primary motivation is the conjecture of Panchishkin \cite[Conj.\ 6.2]{Pan94}; using our approach we prove exactly the automorphic version of his conjecture, including the growth/unicity bounds. These would not follow from Iwahori methods. 

We illustrate via examples. Let $\pi$ be a RASCAR of $\GL_4(\A)$ spherical and regular at $p$. Then:
\begin{itemize}	
	\item The Iwahori-invariants $\pi_p^{\mathrm{Iw}}$ are 24-dimensional, a direct sum of 24 1-dimensional simultaneous eigenspaces for the Hecke operators $U_{p,1},U_{p,2},U_{p,3}$, with $U_{p,i}$ attached to $\smallmatrd{pI_i}{}{}{I_{4-i}}$. An Iwahori refinement $\tilde\pi' = (\pi,\alpha_{p,1},\alpha_{p,2},\alpha_{p,3})$ is a choice of one of these 24 eigensystems.
	
	\item The $Q$-parahoric invariants $\pi_p^{J_p}$ are 6-dimensional, giving at most 6 $Q$-refinements $\tilde\pi = (\pi_p,\alpha_p)$, where $\alpha_p$ is a choice of $U_p$-eigenvalue on $\pi_p^{J_p}$.
\end{itemize}
When there are 6 distinct $Q$-refinements, above each such $(\pi,\alpha_p)$ there are 4 Iwahori refinements $\tilde\pi_1', \tilde\pi_2', \tilde\pi_3',\tilde\pi_4'$, each with $\alpha_{p,2} = \alpha_p$. 

To work at Iwahori level, we must choose an Iwahori refinement. To lift eigenclasses to overconvergent cohomology, we must control all of the (normalised) $U_{p,i}^\circ$ operators; for example, the non-critical slope bound depends on all three of the slopes $h_i \defeq v_p(\alpha_{p,i}^\circ) \geqslant 0$. Working solely at Iwahori level, you bound the growth of the $p$-adic $L$-function only by the sum $h_1 + h_2 + h_3$. 

By contrast, working at $Q$-parahoric level, lifting requires control only of $U_p^\circ$ from \S\ref{sec:set-up and previous work}, the non-critical slope bound depends only on $h_2$, and we get growth bounded by $h_2$. 

In this paper, and its sequel \cite{BDGJW}, we show that $p$-adic $L$-functions depend only on the parahoric refinement. In particular, in \cite[\S12.4, \S14]{BDGJW} we attach $p$-adic $L$-functions $\cL_p(\tilde\pi_i')$ to the four Iwahori refinements $\tilde\pi_i'$ above $\tilde\pi$ (under stronger hypotheses, and with ostensibly weaker growth) and compare to this paper to prove that $\cL_p(\tilde\pi) = \cL_p(\tilde\pi'_1) = \cL_p(\tilde\pi_2') = \cL_p(\tilde\pi_3') = \cL_p(\tilde\pi_4')$ (up to rational scalar). Thus the parahoric $Q$-refinement is the exact amount of data required to construct a $p$-adic $L$-function; passing to deeper level requires additional but redundant hypotheses.

We give two explicit examples from the tables at \texttt{smf.compositio.nl} (cf.\ \cite[\S7]{classical-locus}).
\begin{itemize}\s
	\item There is a unique RASCAR $\pi$ of $\GL_4(\A)$ of weight $\lambda = (12,1,-1,-12)$ and level 1 (taking $j=2, k=14$ in the table). At $p=11$, $\pi_{11}$ is $Q$-ordinary, but not Iwahori-ordinary. Let $\tilde\pi$ be the unique $Q$-ordinary refinement. Using parahoric methods, we get an $11$-adic $L$-function $\cL_{11}(\tilde\pi)$ and can prove it is a bounded measure on $\Z_{11}^\times$ (i.e.\ growth bounded by 0), that is uniquely determined by growth and interpolation.
	
	There are four Iwahori refinements $\tilde\pi_i'$ above $\tilde\pi$, all non-critical and regular. Using only Iwahori methods, we obtain four $11$-adic $L$-functions $\cL_{11}(\tilde\pi_i')$, and can prove these are distributions on $\Z_{11}^\times$ with growth bounded by 22, 12, 12 and 2 respectively. Without further input all four of these might be unbounded, and three are not uniquely determined by interpolation and these growth bounds. 
	
Via this paper and \cite{BDGJW}, however, we know all of these $11$-adic $L$-functions are in fact equal.
	
	\item There is a unique RASCAR of $\GL_4(\A)$ of weight $(9,6,-6,-9)$ and level 1. The non-critical slope bounds here are $v_p(U_{p,1}^\circ), v_p(U_{p,3}^\circ) < 4$, and $v_p(U_{p,2}^\circ) < 13$.  At $p=3$, there exists a $Q$-refinement $\tilde\pi = (\pi, \alpha)$ of $\pi_3$ with $v_3(\alpha) = 5 < 13$; this is non-$Q$-critical slope, so our constructions give a $3$-adic $L$-function $\cL_3(\tilde\pi)$, uniquely determined by growth and interpolation. For each of the four Iwahori refinements $\tilde\pi' = (\pi,\alpha_1,\alpha_2,\alpha_3)$ above $\tilde\pi$, we have $v_p(\alpha_1), v_p(\alpha_3) \geqslant 5 \geqslant 4$, so each $\tilde\pi'$ has critical slope, and we have no unconditional construction of $\cL_3(\tilde\pi')$ at Iwahori level.	
\end{itemize}

From \cite{Pan94}, we expect Theorem \ref{thm:intro non-ord} is optimal, and one cannot improve the non-critical slope/growth bounds. We also expect that we construct \emph{all} examples of $p$-adic $L$-functions that are: (a) attached to RASCARs $\pi$ of $\GL_N(\A)$ that are spherical and regular at $p$, and (b) uniquely determined by their interpolation/growth. An Iwahori approach does not give this.

A final, substantial, benefit is that at parahoric level, we can study local zeta integrals for parahoric-invariant vectors at $\pri|p$. For such vectors the zeta integrals with unramified twists have been computed in \cite[\S9]{BDGJW}, so we can prove our $p$-adic $L$-function has the expected interpolation at all characters.

\subsubsection{Existence and uniqueness of Shalika families}
Our second main result is the use of the $p$-adic $L$-functions of Theorem \ref{thm:intro non-ord} to study the $\GL_{2n}$-eigenvariety.

We now wish to vary $\lambda$, so let $\lambda_\pi$ denote the weight of our (fixed)  RASCAR $\pi$. This is a point in a rigid analytic \emph{(parabolic) weight space} $\sW_{\lambda_\pi}^Q$ of dimension $d+1$ (see \S\ref{sec:weight spaces}). For any (parahoric-at-$p$) level $K$ (see \eqref{eq:general K}) and for any $h\geqslant h_p$ (defined in Theorem \ref{thm:intro non-ord}),  there exists a (slope\ $\leqslant h$)-adapted  affinoid neighbourhood $\Omega$ of $\lambda_\pi$ in $\sW_{\lambda_\pi}^Q$, a sheaf $\sD_\Omega$ on $S_{K}$ interpolating $\sD_\lambda$ as $\lambda$ varies in $\Omega$, and a local piece $\sE_{\Omega,h}(K)$ of the  global parabolic eigenvariety from  \cite{BW20}, endowed with a finite weight map $w: \sE_{\Omega,h}(K) \to \Omega.$ By construction $\sE_{\Omega,h}(K)$ parametrises slope $\leqslant h$ eigensystems appearing in $\hc{t}(S_{K},\sD_\Omega)$.
	 
	 A \emph{classical (Shalika) family} in $\sE_{\Omega,h}(K)$ is a positive-dimensional irreducible component containing a Zariski-dense set of points whose eigensystems appear in RA(S)CARs.

In our earlier works \cite{BDJ17,BW18,BW-Iwasawa}, we developed methods for studying $\hc{\bullet}(S_K,\sD_\Omega)$ as an $\cO_\Omega$-module. Cuspidal cohomology contributes to a continuous range of degrees $\{dn^2,dn^2+1,...,t\}$. As we work in top degree $t$, for appropriate $\tilde\pi$, these methods easily yield a (Shalika) point $x_{\tilde\pi}$ in $\sE_{\Omega,h}(K)$. To study the geometry around this point, it is crucial to understand the $\cO_\Omega$-torsion in $\hc{t}(S_K,\sD_\Omega)$. However, previous methods controlled this torsion only in bottom degree $dn^2$. For $n>1$, cuspidal cohomology is supported in multiple degrees; so existing methods say nothing about the local geometry around $x_{\tilde\pi}$, including the dimension of components through $x_{\tilde\pi}$. Indeed, such methods do not even rule out $x_{\tilde\pi}$ being an isolated point. It is thus a non-trivial question if there are any classical families, let alone Shalika families, containing $\tilde\pi$.

Let $K_1(\tilde\pi) \subset G(\A_f)$ be the open compact subgroup that is parahoric at $p$ and Whittaker new level (for $\pi$) away from $p$ (see \eqref{eq:K_1}).  Our second main result, proved in Theorem~\ref{thm:shalika family}, describes precisely the local geometry of $\sE_{\Omega,h}(K_1(\tilde\pi))$ at $x_{\tilde\pi}$ and, in particular, answers positively the above question for RASCARs under very mild technical assumptions.

\begin{theorem-intro}\label{thm:intro shalika families}
	Let $\pi$ be a RASCAR of $G(\A)$, and $\tilde\pi$ a Shalika $p$-refinement. Suppose that
	\begin{enumerate}\setlength{\itemsep}{0pt}
		\item[(a)] $\lambda_\pi$ is regular, 
		\item[(b)] $\tilde\pi$ has non-$Q$-critical slope,  and 
		\item[(c)] the $p$-adic Galois representation $\rho_{\pi}$ attached to $\pi$ is absolutely irreducible. 
	\end{enumerate}
	Then $\sE_{\Omega,h}(K_1(\tilde\pi))$ is \'etale over $\Omega$ at $x_{\tilde\pi}$. Up to shrinking $\Omega$,
	$w$ induces an isomorphism $\sC \isorightarrow \Omega,$ 
	where $\sC$ is the connected component of $\sE_{\Omega,h}(K_1(\tilde\pi))$ through $x_{\tilde\pi}$, and $\sC$ is a Shalika family.
\end{theorem-intro}

The same conclusions hold replacing (a) and (b) with the strictly weaker assumptions that 
\begin{itemize}\s
	\item[(a$'$)] $\lambda_\pi$ is $H$-regular (Definition~\ref{def:H-regular}) and $\cL_p(\tilde\pi)$ is non-zero, and 
	\item[(b$'$)] $\tilde\pi$ is strongly non-$Q$-critical (Definition~\ref{def:non-Q-critical}). 
\end{itemize}

 We control the torsion in $\hc{t}(S_K,\sD_\Omega)$, and deduce Theorem \ref{thm:intro shalika families}, via a novel application of our evaluation maps, suggested to us by Eric Urban.  We briefly summarise this argument.  We complete the construction of Figure \ref{main diagram}, including the map $\mathrm{Ev}_\Omega : \hc{t}(S_K,\sD_\Omega) \longrightarrow \cD(\Galp,\cO_\Omega)$, in \S\ref{sec:galois evaluations}. A standard argument provides a class $\Phi \in \hc{t}(S_{K(\tilde\pi)},\sD_\Omega)$ lifting $\Phi_{\tilde\pi}$ under the natural specialisation map. Then:

\begin{itemize}\setlength{\itemsep}{0pt}
		
	\item If $\cL_p(\tilde\pi) \neq 0$, then -- by the proof of Theorem \ref{thm:intro non-ord} -- we know $\mathrm{Ev}_\lambda(\Phi_{\tilde\pi}) \neq 0$. Hence, via the commutativity of the top square in Figure \ref{main diagram}, we deduce $\mathrm{Ev}_\Omega(\Phi) \neq 0$.
	
	\item The map $\mathrm{Ev}_\Omega$ is $\cO_\Omega$-linear, and valued in the torsion-free $\cO_\Omega$-module  $\cD(\Galp,\cO_\Omega)$. As $\mathrm{Ev}_\Omega(\Phi) \neq 0,$ we deduce $\Phi$ is non-$\cO_\Omega$-torsion.

	\item It follows that $\hc{t}(S_{K(\tilde\pi)},\sD_\Omega)$ is a faithful $\cO_\Omega$-module. We exploit this to deduce existence of a component in the eigenvariety of maximal dimension through $\tilde\pi$.
\end{itemize}

If there is a non-zero Deligne-critical $L$-value for $\pi$ (which always exists when $\lambda_\pi$ is regular), then the $p$-adic $L$-function is non-zero. The above argument then yields a classical family in the eigenvariety at the (sufficiently small) level $K(\tilde\pi)$ used in \S\ref{sec:non-ordinary intro}. To upgrade this to a Shalika family, we again exploit our evaluation maps, giving (via Proposition \ref{prop:shalika non-vanishing}) a criterion for being Shalika that is open over the eigenvariety.

It remains to prove uniqueness and \'etaleness. However, to exploit non-vanishing $L$-values, we must work at level $K(\tilde\pi)$. As this is inexplicit, it is difficult to further control the geometry of families of level $K(\tilde\pi)$. We can obtain more control by working at new tame level, that is at level $K_1(\tilde\pi)$. We perform a delicate level-switching argument -- using the local Langlands correspondence and $p$-adic Langlands functoriality (see \S\ref{sec:etaleness S}) -- to transfer the family to level $K_1(\tilde\pi)$, where we then complete the proof of Theorem \ref{thm:intro shalika families}.

We conclude \S\ref{sec:shalika families no new} with an application of Theorem \ref{thm:intro shalika families} to the global geometry of the eigenvariety. In Theorem \ref{thm:all classical shalika}, we show that if $\tilde\pi$ is as in Theorem \ref{thm:intro shalika families}, and $\sI$ is the (unique) irreducible component of the global eigenvariety through $\tilde\pi$, then every non-$Q$-critical slope classical point of $\sI$ is Shalika. We actually prove more: that every point (classical or not) of $\sI$ is symplectic, arising from $\mathrm{GSpin}_{2n+1}$. Our proof goes through $p$-adic Langlands functoriality and occupies all of \S\ref{sec:symplectic components}. 
	We thank the referee for pushing us to prove such a result.

\begin{remark}
Theorem \ref{thm:intro shalika families} describes the geometry of the $Q$-parabolic eigenvariety. This is natural in light of \S\ref{sec:parahoric level}. Eigenvarieties for non-minimal parabolics have been well-studied; for a summary of constructions and arithmetic applications,  see \cite[Intro]{BW20}. Major recent applications include Bloch--Kato for $\mathrm{GSp}_4$ \cite[\S17]{LZ20} and modularity of elliptic curves over imaginary quadratic fields \cite[\S2.2]{Caraiani-Newton}.

The most traditional flavour of eigenvariety comes attached to a minimal parabolic subgroup, the Borel subgroup $B\subset G$, corresponding to Iwahoric level at $p$. With appropriate adaptation, and stronger assumptions, our methods also apply to this case, and we can prove the analogue of Theorem \ref{thm:intro shalika families} for the Iwahori eigenvariety over the pure weight space (whose dimension grows with $n$). This -- and applications to families of $p$-adic $L$-functions -- is the subject of a follow-up paper \cite{BDGJW} with Graham and Jorza.
\end{remark}

\subsubsection{$p$-adic $L$-functions in Shalika families}

To prove Theorem~\ref{thm:intro non-ord}, we worked at a specific (inexplicit) level $K(\tilde\pi)$, at which we have a precise connection to $L$-values. In Theorem~\ref{thm:intro shalika families}, we worked at a second specific (explicit) level $K_1(\tilde\pi)$, where we obtain control over $p$-adic families.

When $\pi$ is everywhere spherical away from $p$ -- that is, when $\pi$ has tame level 1 -- these two levels coincide. In Chapter \ref{sec:families of p-adic L-functions}, we crucially exploit this to vary $p$-adic $L$-functions in families in the tame level 1 case (see Theorem \ref{thm:intro 3} below).

\medskip

 The obstruction to generalising to higher tame level arises from local representation theory: namely, given a place $v\nmid p\infty$ such that $\pi_v$ is ramified, we need to find explicit `test vectors' in the Shalika model of $\pi_v$ such that an attached Friedberg--Jacquet zeta integral computes the $L$-factor of $\pi_v$ (see \eqref{eq:jacquet-friedberg test vector}). It is known that such vectors always exist abstractly, but explicit vectors -- of the kind required for variation of $p$-adic $L$-functions -- have not yet been found.

In \S\ref{sec:hypothesis} and \S\ref{sec:shalika-new-line}, we describe, in very general terms, what kind of results would allow us to generalise Theorem \ref{thm:intro 3} to higher tame level. On a concrete level, we hypothesise a possible theory of explicit test vectors, via \emph{Shalika new vectors}, a Shalika analogue of the classical (Whittaker) newform theory of \cite{JPSS} (see Definition \ref{def:shalika new vectors}).  Ramified examples where these hypotheses are satisfied have been found in the work \cite{DJ-parahoric} of the second author and Jorza.

	Under this hypothesis, in \S\ref{sec:shalika families refined} we study a modified  eigenvariety $\sE_{\Omega,h}^\dec$, and prove  the following refinement of Theorem~\ref{thm:intro shalika families}.

\setcounter{lettprime}{1}
\begin{theorem-intro-2}\label{thm:intro shalika family 2} \emph{(Theorem~\ref{thm:section 7 main theorem})}. 
	Suppose that: (a) $\lambda_\pi$ is regular, (b) $\tilde\pi$ has non-$Q$-critical slope, and (c) for all $v$,   Hypothesis~\ref{ass:shalika} holds for $c = c(\pi_v)$.	Then:
	\begin{itemize}\s
		\item  $\sE_{\Omega,h}^\dec$ is \'etale over $\Omega$ at $\tilde\pi$, and (up to shrinking $\Omega$)  the connected component $\sC$ through $\tilde\pi$ is a Shalika family mapping isomorphically onto $\Omega$ under $w$.
		
		\item  $\sC$ contains a very Zariski-dense set $\sC_{\mathrm{nc}}$ of classical points satisfying the conditions of Definition~\ref{intro:definition} (see also Conditions \ref{cond:running assumptions}). For all $v$, every point in $\sC_{\mathrm{nc}}$ has a Shalika new vector of conductor $c(\pi_v)$.
		
		\item There exists an eigenclass $\Phi_{\sC} \in \hc{t}(S_{K(\tilde\pi)},\sD_\Omega)$, interpolating the classes $\Phi_{\tilde\pi_y}$ for $y \in {\sC}_{\mathrm{nc}}$ (upto scaling by $p$-adic periods).
	\end{itemize}
\end{theorem-intro-2}

When $\pi$ has tame level 1, condition (c) is automatically satisfied with each $c(\pi_v) = 0$, and the eigenvariety $\sE_{\Omega,h}^\dec$ is nothing but $\sE_{\Omega,h}$ from above; so there are a ready supply of RASCARs where this result is unconditional. In general, we may also weaken assumptions as in Theorem~\ref{thm:intro shalika families}.

\medskip

Given Theorem \ref{thm:intro shalika family 2}, standard methods give the analytic variation of $\cL_p(\tilde\pi)$ over $\sC$ as a formal consequence of our evaluation maps in families. The definition of the multi-variable $p$-adic $L$-function is summarised in row (T) of Figure \ref{main diagram}.  

\begin{definition}
	Under the hypotheses of Theorem~\ref{thm:intro shalika family 2}, let $\sC$ be the Shalika family through $\tilde\pi$. Define the \emph{$p$-adic $L$-function over $\sC$} to be 
	\[
	\cL_p^{\sC} \defeq \mathrm{Ev}_{\Omega}(\Phi_{\sC}) \in \cD(\Galp,\cO_\Omega).
	\]
\end{definition}

Let $\sX(\Galp)$ be the $\Q_p$-rigid space of characters on $\Galp$; then via the Amice transform (\cite[Def.~5.1.5]{BH17}, building on \cite{Ami64,ST01}), we may view $\cL_p^{\sC}$ as a rigid function 
\disp{
\cL_p^{\sC} : \sC \times \sX(\Galp) \to \C_p.
}
Our third main result (Theorem~\ref{thm:family p-adic L-functions}) is that $\cL_p^{\sC}$ interpolates $\cL_p(\tilde\pi_y)$ as $y$ varies in the set $\sC_{\mathrm{nc}}$.

\begin{theorem-intro}\label{thm:intro 3}
	Suppose the hypotheses of Theorem~\ref{thm:intro shalika family 2}. Then at every $y \in \sC_{\mathrm{nc}}$, there exists a set 
	$\{c_y^\epsilon \in L^\times: \epsilon \in \{\pm 1\}^\Sigma\}$ of $p$-adic periods such that for every $\chi \in \sX(\Galp)$, we have
	\begin{equation}\label{eq:int in C}
		\cL_p^{\sC}(y,\chi) = c_y^{(\chi\eta)_\infty} \cdot \cL_p(\tilde\pi_y, \chi).
	\end{equation}
\end{theorem-intro}

Let $\sX(\Galp^{\cyc}) \subset \sX(\Galp)$ be the cyclotomic line, i.e.\ the Zariski-closure of $\{\chi_{\cyc}^j : j \in \Z\}$. Via Theorem~\ref{thm:intro non-ord}, $\cL_p^{\sC}$ simultaneously interpolates the values $L(\pi_y\times \chi,j+\tfrac{1}{2})$ over the set of points
\begin{align*}
	\mathrm{Crit}(\sC) = \Big\{(y,\chi\chi_{\cyc}^j) \in \sC\times \sX(\Galp^{\cyc}&) : y \in \sC_{\mathrm{nc}}, j \in \mathrm{Crit}(w(y)), \chi \text{ finite order}\Big\}.
\end{align*}
The set $\mathrm{Crit}(\sC)$ is Zariski-dense in $\sC\times \sX(\Galp^{\cyc})$, so the restriction $\sL_p^{\sC}|_{\sC \times \sX(\Galp^{\cyc})}$ is uniquely determined by this interpolation. Specialising at $\tilde\pi$, we deduce: 
\begin{corollary} 
	Assume the hypotheses of Theorem~\ref{thm:intro shalika family 2}. Up to a non-zero scalar, the restriction of $\cL_p(\tilde\pi)$ to the cyclotomic line is uniquely determined by interpolation over $\mathrm{Crit}(\sC)$. 
\end{corollary}
If further Leopoldt's conjecture holds for $F$ at $p$, then we obtain a similar uniqueness statement for $\cL_p(\tilde\pi)$ itself, made precise in \S\ref{sec:uniqueness 2}. These results should be compared to Theorem~\ref{thm:intro non-ord}, where we showed $\cL_p(\tilde\pi)$ is determined by growth and interpolation, but only when $h_p < \#\mathrm{Crit}(\lambda_\pi)$. 

Finally, let us highlight some examples for which the assumptions of Theorem~\ref{thm:intro 3} are satisfied. Let $f$ be a classical cuspidal Hilbert eigenform of level $1$ and weights $\geqslant 3$. The symmetric cube $\mathrm{Sym}^3(f)$ is a RASCAR for $\GL_4$ of level 1 \cite[Prop.~8.1.1]{GR2}. When 
$\mathrm{Sym}^{3}(f)$ is non-$Q$-critical (e.g.\ if $f$ itself is $p$-ordinary), then Theorem~\ref{thm:intro 3} shows that its $p$-adic $L$-function, as constructed in Theorem~\ref{thm:intro non-ord}, can be interpolated over the Hilbert cuspidal eigenvariety from \cite{AIP-Asterisque}. More generally, Newton--Thorne recently showed that arbitrary symmetric powers of $f$ are RACARs in \cite{NT19,NT22}, and the odd symmetric powers are RASCARs.

\subsubsection*{Acknowledgements}
We would particularly like to thank Andrew Graham and Andrei Jorza, with whom we have shared many discussions on the subject of this paper. We also thank Mahdi Asgari, Jo\"el Bella\"{i}che, Harald Grobner, Rob Kurinczuk, Dongwen Liu, David Loeffler, James Newton, A.\ Raghuram and Dinakar Ramakrishnan for helpful comments and  many stimulating conversations, and Eric Urban for suggesting the elegant proof of Proposition~\ref{prop:ann = zero}. Finally we thank the anonymous referee, whose suggestions did much to improve the results and exposition of this paper. This project received a focused project grant from the Heilbronn Institute for Mathematical Research.  The first author   has received funding from the European Research Council grant $n^\circ$682152 and Fondecyt grants $n^\circ$77180007 and $n^\circ$11201025;   
the second author was partially supported  by the Agence Nationale de la Recherche grants ANR-18-CE40-0029 and 
ANR-16-IDEX-0004;  the  third author was supported by EPSRC Postdoctoral Fellowship EP/T001615/1 \& 2.

%%=======================================================
%%=======================================================
%%
%% 						NOTATION
%%
%%=======================================================
%%=======================================================
\section{Automorphic preliminaries} \label{sec:basic notation}

The following fixes notation and recalls how to attach a compactly supported cohomology class (with $p$-adic coefficients) to a suitable automorphic representation. Everything here is standard.

\subsection{Notation}\label{sec:notation}
Let $F$ be a totally real number field of degree $d$ over $\Q$, let $\cO_F$ be its ring of integers and $\Sigma$ the set of its real embeddings. Let $\A = \A_f\times \R$ denote the ring of adeles of $\Q$. For  $v$  a non-archimedean place of $F$, we let $F_v$ be the completion of $F$ at $v$, denote by $\cO_v$ the ring of integers in $F_v$, and fix a uniformiser $\varpi_v$. 

Let $n \geqslant 1$ and let $G$ be the algebraic group $\mathrm{Res}_{\cO_F/\Z}\mathrm{GL}_{2n}$, $B= \mathrm{Res}_{\cO_F/\Z}B_{2n}$ be the Borel subgroup of upper triangular matrices, with opposite $B^-$,
$N$ and $N^-$ be the unipotent radicals of $B$ and $B^-$ respectively, and $T= \mathrm{Res}_{\cO_F/\Z}T_{2n}$ be the maximal split torus of diagonal matrices. We have decompositions $B= TN$ and $B^- = N^-T$.  We let $K_\infty=C_\infty Z_G(\R),$ 
where $C_\infty = O_{2n}(\R)^d$ is the standard maximal compact subgroup of $G(\R)$ and $Z_G$ is the centre of $G$. 
For any reductive real Lie group $A$ we let $A^\circ$ denote the connected component of the identity.

Let $H$ denote the algebraic group $\mathrm{Res}_{\cO_F/\Z}(\GL_n \times \GL_n)$, which we frequently identify with its image under the natural embedding $\iota : H \hookrightarrow G$ given by $(h,h') \mapsto \smallmatrd{h}{0}{0}{h'}.$

Let $Z_H$ be the centre of $H$. We write $Q = \mathrm{Res}_{\cO_F/\Z} \smallmatrd{\GL_n}{\mathrm{M}_n}{0}{\GL_n}$ for the maximal standard parabolic subgroup of $G$ (containing $B$) whose Levi subgroup is $H$, and we denote by $N_Q$ its  unipotent radical. 

Fix a rational prime $p$ and an embedding $i_p : \overline{\Q} \hookrightarrow \overline{\Q}_p$. We fix an extension of $i_p$ to an isomorphism $i_p : \C \isorightarrow \overline{\Q}_p$. For each embedding $\sigma : F \hookrightarrow \R$ in $\Sigma$, there exists a unique prime $\pri|p$ in $F$ such that $\sigma$ extends to an embedding $F_{\pri} \hookrightarrow \overline{\Q}_p$; we write $\pri(\sigma)$ for this prime, and let $\Sigma(\pri) \defeq \{\sigma \in \Sigma: \pri(\sigma) = \pri\}.$
We let $\OFp \defeq \cO_F \otimes \Zp$.

Let $F^{p\infty}$ be the maximal abelian extension of $F$ unramified outside $p\infty$, and let $\Galp \defeq  \mathrm{Gal}(F^{p\infty}/ F)$, which has the structure of a $p$-adic Lie group. Let $\Galp^{\cyc} \defeq \mathrm{Gal}(\Q^{p\infty}F/F).$

Given an ideal $I\subset \cO_F$ we let $\sU(I) \defeq \{x \in \widehat{\cO}_F^\times: x \equiv 1 \newmod{I}\}$, and consider the  narrow ray class group $\Cl(I) = F^\times\backslash \A_F^\times/\sU(I)F_\infty^{\times\circ}.$

All our group actions will be on the left. If $M$ is a $R$-module, with a left action of a group $\Gamma$, then we write $M^\vee = \mathrm{Hom}_R(M,R)$, with associated left dual action 	$(\gamma \cdot\mu)(m) = \mu(\gamma^{-1} \cdot m).$

 For an affinoid rigid space $X$, we write $\cO_X$ (or, for clarity of notation, occasionally $\cO(X)$) for the ring of rigid functions on $X$, so $X = \mathrm{Sp}(\cO_X)$.

\subsection{The weights}\label{sec:algebraic weights}
Let $X^{\ast}(T)$ be the set of algebraic characters of $T$. Each element of $X^{\ast}(T)$ corresponds to an integral weight   $\lambda= (\lambda_{\sigma})_{\sigma \in \Sigma}$, where $\lambda_{\sigma}= (\lambda_{\sigma, 1}, \dots , \lambda_{\sigma, 2n}) \in \Z^{2n}.$
	
	We let $X^{\ast}_+(T) \subset X^{\ast}(T)$ be the subset of $B$-dominant weights. We say that $\lambda$ is \emph{pure} if there exists $\sw  \in \Z$, the \emph{purity weight} of $\lambda$, such that 
	\[
	\lambda_{\sigma, i}+ \lambda_{\sigma, 2n- i+ 1}= \sw \qquad\text{ for each $\sigma \in \Sigma$ and $i \in \{1, \dots , 2n\}$}.
	\] 
	Let $X^{\ast}_0(T) \subset X^{\ast}_+(T)$ be the subset of pure $B$-dominant  integral weights, i.e.\ those supporting  cuspidal cohomology \cite[Lem.~4.9]{Clo90}.   We say $\lambda$ is \emph{regular} if $\lambda_{\sigma,i} > \lambda_{\sigma,i+1}$ for all $\sigma$ and $i$.

For $\lambda \in X^{\ast}_{+}(T)$, we let $V_{\lambda}$ be the algebraic irreducible representation of $G$ of highest weight $\lambda$; for a sufficiently large field $L/\Qp$, the $L$-points $V_{\lambda}(L)$ can be explicitly realised as 
\begin{align}\label{eq:alg induction}
	V_\lambda(L) = \{ f : G(\Qp) \to L& \text{ algebraic}:\\
	& f(n^-tg) = \lambda(t)f(g) \text{ for all }n^- \in N^-(\Qp), t \in T(\Qp), g \in G(\Qp)\}.\notag
\end{align}
The (left) action of $G(\Qp)$ is by right translation, i.e. $(h\cdot f)(g) = f(gh)$ for $g,h \in G(\Qp)$ and $f \in V_\lambda$. Let $V_{\lambda}^\vee$ denote the linear dual, with its (left) dual action; we have an isomorphism $V_{{\lambda}}^{\vee} \cong V_{\lambda^\vee}$ where $\lambda^\vee = -w_{2n}(\lambda)$ is the contragredient of $\lambda$, for $w_{2n}$ the longest Weyl element for $G$. Note the central characters of $V_{\lambda}^\vee$ and $V_{\lambda}$ are inverse to each other, and if $\lambda$ is pure, then as $G$-modules we have (e.g. \cite[\S2.3]{GR2}) 
\[
V_{\lambda}^\vee \cong V_{\lambda} \otimes [\mathrm{N}_{F/\Q}\circ\det]^{-\sw}.
\]
	By Zariski-density any $f \in V_\lambda$ is uniquely determined by $f|_{G(\Zp)}$. We have a natural integral subspace $V_\lambda(\cO_L)$ of $f \in V_\lambda(L)$ such that $f(G(\Zp)) \subset \cO_L$; we let $V_\lambda^\vee(\cO_L) = \mathrm{Hom}_{\cO_L}(V_\lambda(\cO_L),\cO_L).$

Let $X^{\ast}(H)$ be the set of algebraic characters of $H$. Each element of $X^{\ast}(H)$ is identified with an integral weight   
\[
(j,j')= (j_{\sigma}, j'_{\sigma})_{\sigma \in \Sigma}, \qquad j_\sigma, j_\sigma' \in \Z.
\] 
We say $(j,j') \in X^{\ast}(H)$ is \emph{$Q$-dominant} if $j_{\sigma}\geqslant  j'_{\sigma}$ for each $\sigma \in \Sigma$, and let  $X^{\ast}_+(H) \subset X^{\ast}(H)$ be the subset of $Q$-dominant weights. 
We say that $(j,j') \in X^{\ast}(H)$  is \emph{pure} if there exists $\sw  \in \Z$
such that $j_{\sigma}+  j'_{\sigma}= \sw $ for all $\sigma \in \Sigma$, 
and let 	$X^{\ast}_0(H) \subset X^{\ast}_+(H)$ 
be the subset of pure $Q$-dominant weights.
Since $B\subset Q$, we naturally have
\[
X^{\ast}(H) \subset X^{\ast}(T), \qquad X^{\ast}_+(H) \subset X^{\ast}_+(T), \qquad X^{\ast}_0(H) \subset X^{\ast}_0(T).
\]
Given a pure $B$-dominant integral weight $\lambda = (\lambda_{\sigma})_{\sigma \in \Sigma}$, we define a set
\begin{equation}\label{eqn:crit lambda}
	\mathrm{Crit}(\lambda) \defeq \{j \in \Z: -\lambda_{\sigma,n} \leqslant j \leqslant -\lambda_{\sigma,n+1} \ \forall \sigma \in \Sigma\}.
\end{equation}
If $\pi$ is a RACAR for $G(\A)$ of weight $\lambda$ (which we take to mean cohomological with respect to $V_\lambda^\vee$, in the sense of \S\ref{sec:cuspidal cohomology}), then \cite[\S6.1]{GR2} proves
\begin{align*}
	j \in \mathrm{Crit}(\lambda) \iff \text{for all  finite } &\text{order Hecke characters  $\chi$ of $F$, the  $L$-value}\\ &\text{   $L(\pi\otimes\chi,j+\tfrac{1}{2})$ is  critical in the sense of Deligne.}
\end{align*}

\subsection{Local systems and Betti cohomology}\label{sec:classical cohomology}
Let $K \subset G(\A_f)$ be an open compact subgroup.
The \emph{locally symmetric space of level $K$} is the  $d(2n-1)(n+1)$-dimensional real orbifold
\begin{equation}\label{eq:loc sym space}
	S_K \defeq G(\Q)\backslash G(\A)/KK_{\infty}^{\circ}.		
\end{equation}

\subsubsection{Archimedean local systems}\label{sec:arch ls}
Let $M$ be a left $G(\Q)$-module such that $Z_G(\Q) \cap KK_\infty^\circ$ acts trivially (else, the local systems we define are zero). To $M$ we attach a local system $\cM = \cM_K$ on $S_K$, defined as the locally constant sections of 
\[
G(\Q)\backslash[G(\A) \times M]/KK_\infty^\circ \longrightarrow S_K,
\] with action $\gamma(g,m)kz = (\gamma gkz, \gamma \cdot m)$.  
We use calligraphic letters for such local systems.

	Applying to $M = V_{\lambda}^\vee(E)$ for a characteristic zero field $E$, we can consider compactly supported Betti cohomology groups $\rH^\bullet_c(S_K, \cV_{\lambda,K}^\vee(E))$. We let	
	$\rH^\bullet_*(S^G, \cV_{\lambda}^\vee(E)) \defeq \varinjlim_K \rH^\bullet_*(S_K, \cV_{\lambda,K}^\vee(E)).$ 	This admits a natural $G(\A_f)$-action, whose $K$-invariants are $\h^\bullet_*(S_K,\cV_{\lambda,K}^\vee(E))$. For ease of notation we henceforth drop the subscript $K$ and write only $\cV_{\lambda}^\vee(E)$.

%%===============================================
\subsubsection{Non-archimedean local systems}\label{sec:non-arch ls}
Let $M$ be a left $K$-module on which the centre $Z_G(\Q) \cap KK_\infty^\circ$ acts trivially. To $M$, we attach a local system $\sM = \sM_K$ on $S_K$ as the locally constant sections of 
\[
G(\Q)\backslash [G(\A) \times M]/K K_\infty^\circ \longrightarrow S_K,
\]
with action $\gamma(g,m)kz = (\gamma gkz, k^{-1}\cdot m).$ For these we use script letters, e.g.\ $\sV, \sD$.

Suppose now $M$ has a left action of $G(\A_f)$. This gives left actions of $G(\Q)$ and $K$ on $M$, and we get associated (archimedean and non-archimedean) local systems $\cM$ and $\sM$ attached to $M$. One may check (see \cite[\S1.2.2]{Urb11})   there is an isomorphism 
\[
\cM \cong \sM,  \qquad \text{given on sections by }(g,m) \mapsto (g, g_f^{-1} \cdot m).
\]

The following is the example of most importance to us. If $L/\Qp$ contains the field of definition of $\lambda$, then $M = V_{\lambda}(L)$ can be realised as a space of functions $f : G(\Qp) \to L$. If $\mu \in V_{\lambda}^\vee(L)$,  $f \in V_{\lambda}(L)$, and $g \in G(\Qp)$, then $V_\lambda^\vee(L)$ carries an action of $h \in G(\A_f)$ by 
\[
(h\cdot \mu)[f(g)] \defeq \mu[f(gh_p^{-1})],
\] 
where $h_p$ is the image of $h$ under the projection $G(\A_f) \to G(\Qp)$. We get two local systems $\cV_\lambda^\vee(L)$ and $\sV_\lambda^\vee(L)$, and as above, we get an isomorphism $\cV_{\lambda}^\vee \cong \sV_{\lambda}^\vee$.

\subsubsection{Hecke operators}\label{sec:hecke operators 1}
Let $M$ be a left module for $G(\Q)$ (resp.\ $K$), and let $\gamma \in G(\A_f)$, which we suppose acts on $M$. As in \cite[\S1.4]{DJR18} define a Hecke operator on $\hc{\bullet}(S_K, \cM)$ by
\[
[K\gamma K] \defeq \mathrm{Tr}(p_{\gamma  K \gamma^{-1}\cap K, K}) \circ [\gamma] \circ p_{K\cap \gamma^{-1}K \gamma,K}^* : \hc{\bullet}(S_K, \cM) \to \hc{\bullet}(S_K, \cM),
\]
where $\mathrm{Tr}$ is the trace map attached to the finite cover $S_{\gamma K\gamma^{-1}\cap K} \to S_K$,  $p_{K',K}:S_{K'}\to S_K$ is the natural projection, and
\[
[\gamma] : \hc{\bullet}(S_{K\cap \gamma^{-1}K\gamma}, \cM) \to \hc{\bullet}(S_{\gamma K\gamma^{-1}\cap K}, \cM)
\]
is given on local systems by $(g,m) \mapsto (g\gamma^{-1}, \gamma \cdot m)$ (and similarly for $\sM$-coefficients).  

One can check that if $M$ is a $G(\A_f)$-module as in \S\ref{sec:non-arch ls}, then the isomorphism 
\begin{equation}\label{eq:sigma}
	\hc{\bullet}(S_K,\cM) \isorightarrow \hc{\bullet}(S_K,\sM)
\end{equation}
induced from the isomorphism $\cM \cong \sM$ is Hecke-equivariant \cite[\S1.2.5]{Urb11}.

\subsubsection{Operators at infinity}\label{sec:decomp at infinity}

If $\sigma \in \Sigma$, then $K_\sigma/K_\sigma^\circ = \{\pm1\}$, and thus $K_\infty/K_\infty^\circ = \{\pm 1\}^{\Sigma}$.  Any character $\epsilon : K_\infty/K_\infty^\circ \to \{\pm1\}$ can also be identified with an element of $\{\pm 1\}^{\Sigma}$. If $M$ is a module upon which $K_\infty/K_\infty^\circ$ acts, let $M^\epsilon$ be the submodule upon which the action is by $\epsilon$. If $M$ is a vector space over a field of characteristic $\neq 2$, then $M = \oplus_\epsilon M^\epsilon$. Since the group acts naturally on $S_K$ and its cohomology, and this action commutes with the $G(\A_f)$-action, we thus obtain decompositions of its cohomology into Hecke-stable submodules (see e.g. \cite[p.15]{GR2}).

\subsection{The spherical Hecke algebra}\label{sec:unramified H}
Let $\pi$ be a RACAR of $G(\A)$ of weight $\lambda$, and let 	$S=\{v\nmid p\infty: \pi_v\text{ not spherical}\}$ 
be the set of bad places for $\pi$. Let 
\[
K = \textstyle\prod_{v\nmid \infty} K_v \subset G(\A_f)
\]
be an open compact subgroup such that $\pi_f^K \neq 0$; for $v\notin S\cup\{\pri|p\}$, we take 
\disp{
K_v = K_v^\circ \defeq \GL_{2n}(\cO_v).
}
Let $X_*^+(T_{2n})$ denote the set of algebraic $B$-dominant cocharacters of $T_{2n} \subset \GL_{2n}$, identified with tuples $\nu = (\nu_1,\dots,\nu_{2n}) \in \Z^{2n}$ with
\[
\nu_{1} \geqslant \nu_{2} \geqslant \cdots \geqslant \nu_{2n}, \qquad \text{via} \qquad x \mapsto \mathrm{diag}(x^{\nu_1}, \dots, x^{\nu_{2n}}).
\]

\begin{definition} \label{def:spherical hecke algebra}
	For $v \notin S\cup\{\pri|p\}$, and any $\nu \in X_*^+(T_{2n})$, let 
	$T_{\nu,v} \defeq [K_v^\circ \nu(\varpi_v) K_v^\circ]$. The \emph{unramified Hecke algebra of level $K$} is the commutative algebra $\cH'$ generated by all such operators:
	\[
	\cH' \defeq \Z[T_{\nu,v} : \nu \in X_*^+(T_{2n}), v\notin S\cup\{\pri|p\}].
	\]
\end{definition}

For any choice of $K$ such that $K_v = K_v^\circ$ for $v \notin S\cup\{\pri|p\}$, the algebra $\cH'$ acts on $\pi^{K}$ via right translation, and on $\mathrm{H}^\bullet_*(S_K, -)$ as described in \S\ref{sec:hecke operators 1}.

\begin{definition}\label{def:gen eigenspace ur}
	Let $E$ be a number field containing the Hecke field of $\pi_f$. Attached to $\pi$ we have a homomorphism 
	\[
	\psi_{\pi} : \cH'\otimes E \rightarrow E
	\]
	which for $\nu \in X_*^+(T_{2n})$ and $v \notin S\cup\{\pri|p\}$ sends $T_{\nu,v}$ to its  eigenvalue acting on the line  $\pi_v^{K_v^\circ}$. Let $\m_{\pi} \defeq \ker(\psi_\pi)$, a maximal ideal in $\cH'\otimes E$. If $L$ is any field containing $E$, we get an induced maximal ideal in $\cH'\otimes L$, which in an abuse of notation we also denote $\m_\pi$. 
\end{definition}
Note in the set-up above, if $M$ is a finite-dimensional $L$-vector space with an action of $\cH'$, then the localisation $M_{\m_\pi}$ is the generalised eigenspace $M\lsem \m_{\pi}\rsem $ attached to $\psi_\pi$.

\subsection{Cohomology classes attached to RACARs}\label{sec:cuspidal cohomology}

We recall the standard attachment of compactly supported cohomology classes to RACARs (e.g.\ \cite{BC83, Clo90}). For a weight $\lambda \in X_0^*(T)$, we have an inclusion of $\cH'$-modules
\begin{equation}\label{eq:cuspidal cohomology}
	\bigoplus_{\pi}\rH^\bullet\big(\fg_\infty,K_\infty^\circ; \pi_\infty\otimes V_{\lambda}^{\vee}(\C) \big)\otimes \pi_f^K \subset \rH^\bullet_{\mathrm{cusp}}(S_K,\cV_{{\lambda}}^\vee(\C)) \subset \h^\bullet(S_K,\cV_{\lambda}^\vee(\C))
\end{equation}
where $\fg_\infty = \mathrm{Lie}(G(\R))$ and the sum is over all RACARs $\pi$ of $G(\A)$. If $\pi$ contributes non-trivially to \eqref{eq:cuspidal cohomology}, we say it has \emph{weight $\lambda$}, and it then (e.g.\ \cite[(3.4.2)]{GR2}) contributes to all degrees $i$ with
\begin{equation}\label{eq:t}
	dn^2 \leqslant i \leqslant d(n^2 + n - 1) \defeqrev t.
\end{equation}
 If we localise the resulting $\cH'$-module at $\m_{\pi} \subset \cH'$, then by Strong Multiplicity One 
\[
\h^\bullet_{\mathrm{cusp}}(S_K,\cV_{{\lambda}}^\vee(\C))_{\m_{\pi}} =\rH^\bullet\big(\fg_\infty,K_\infty^\circ; \pi_\infty\otimes V_{\lambda}^{\vee}(\C) \big)\otimes \pi_f^K.
\]
There is a natural action of $K_\infty/K_\infty^\circ$ on the factor at infinity, hence on cuspidal cohomology, and taking $\epsilon$-parts for $\epsilon \in \{\pm 1\}^\Sigma$ (as in \S\ref{sec:decomp at infinity}), we then obtain
\begin{equation}\label{eq:cuspidal cohomology 2}
	\h^\bullet_{\mathrm{cusp}}(S_K,\cV_{{\lambda}}^\vee(\C))_{\m_{\pi}}^\epsilon =\rH^\bullet\big(\fg_\infty,K_\infty^\circ; \pi_\infty\otimes V_{\lambda}^{\vee}(\C) \big)^\epsilon\otimes \pi_f^K.
\end{equation}
As in \cite[\S4.1]{GR2}, for degree $t$ (that is, at the top of the range \eqref{eq:t}) we have
\begin{equation}\label{eq:gK line}
	\dim_{\C} \ \rH^t(\fg_\infty,K_\infty^\circ; \pi_\infty\otimes V_{\lambda}^\vee(\C))^\epsilon = 1
\end{equation}
for all $\epsilon \in \{\pm1\}^\Sigma$. Fixing a basis $\Xi_\infty^\epsilon$ of \eqref{eq:gK line} fixes an $\cH'$-equivariant isomorphism
\begin{equation}\label{eq:Theta 2}
	\pi_f^K \isorightarrow \h^t_{\mathrm{cusp}}(S_K,\cV_{\lambda}^\vee(\C))^\epsilon_{\m_{\pi}} \isorightarrow \h^t_{\mathrm{c}}(S_K,\cV_{\lambda}^\vee(\C))^\epsilon_{\m_{\pi}},
\end{equation}
where the first map is $\varphi_f \mapsto \Xi_\infty^\epsilon \otimes \varphi_f$ and the second isomorphism follows as $\pi$ is cuspidal. Finally, via our fixed isomorphism $i_p : \C \cong \overline{\Q}_p$ and the isomorphism \eqref{eq:sigma}, we have isomorphisms
\begin{equation}\label{eq:Theta 3}
	\hc{\bullet}(S_K,\cV_{\lambda}^\vee(\C)) \labelisorightarrow{\ i_p\ } \hc{\bullet}(S_K,\cV_{\lambda}^\vee(\overline{\Q}_p))
	\labelisorightarrow{\ \text{\eqref{eq:sigma}}\ }\hc{\bullet}(S_K,\sV_{\lambda}^\vee(\overline{\Q}_p)). 
\end{equation}
As all the maps above are Hecke-equivariant, we finally deduce:

\begin{proposition}\label{prop:non-canonical}
	There is a Hecke-equivariant isomorphism
	\begin{equation}\label{eq:cohomology non-canonical}
		\pi_f^K \isorightarrow \hc{t}(S_{K}, \sV_{\lambda}^\vee(\overline{\Q}_p))_{\m_\pi}^\epsilon.
	\end{equation}
	This isomorphism is non-canonical, depending on the choice of basis $\Xi_\infty^\eps$ of \eqref{eq:gK line} and on $i_p$.
\end{proposition}

\subsection{Shalika models and Friedberg--Jacquet integrals}\label{sec:shalika models}
We recall some relevant facts about Shalika models (see e.g.\ \cite[\S1,\S3.1]{GR2}). Let 
\[
\cS_{/F} = \{s = \smallmatrd{h}{}{}{h}\cdot \smallmatrd{I_n}{X}{}{I_n}: h \in \GL_n, X \in \mathrm{M}_n\}
\]
be the Shalika subgroup of $\GL_{2n/F}$, and $\cS= \mathrm{Res}_{F/\Q}\cS_{/F}$. Let $\psi$ be the standard non-trivial additive character of $F\backslash \A_F$ from \cite[\S4.1]{DJR18}, and let $\eta$ be a Hecke character of $F^\times\backslash\A_F^\times$. For $s \in \cS$, write 
\[
(\eta\otimes\psi)(s) = \eta(\det(h))\psi(\mathrm{Tr}(X)).
\]
A cuspidal automorphic representation $\pi$ of $G(\A)$ (of weight $\lambda$) is said to have an $(\eta,\psi)$-\emph{Shalika model} if there exist $\varphi \in \pi$ and $g \in G(\A)$ such that
\begin{equation}\label{eq:shalika integral}
	\cS_{\psi}^\eta(\varphi)(g) \defeq \int_{Z_G(\A)\cS(\Q)\backslash\cS(\A)} \varphi(sg) \ (\eta \otimes \psi)^{-1}(s)ds \neq 0.
\end{equation}
This forces $\eta^n$ to be equal to the central character of $\pi$, and hence $\eta = \eta_0|\cdot|^{\sw}$, where $\eta_0$ has finite order and $\sw$ is the purity weight of $\lambda$. If \eqref{eq:shalika integral} holds, then $\cS_\psi^\eta$ defines an intertwining $\pi \hookrightarrow \mathrm{Ind}_{\cS(\A)}^{G(\A)}(\eta \otimes \psi)$, realising $\pi$ inside the space of functions $W : G(\A) \to \C$ satisfying
\begin{equation}\label{eq:shalika transform}
	W\left(\left(\begin{smallmatrix} h & 0 \\ 0 & h \end{smallmatrix}\right)  \left(\begin{smallmatrix} 1 & X \\ 0 & 1 \end{smallmatrix} \right) \bullet  \right)= \eta(\det(h)) \psi(\mathrm{tr}(X)) W(\bullet)\ \ \forall h\in \GL_n(F), X\in M_n(F).
\end{equation}
If $\pi$ has an $(\eta,\psi)$-Shalika model, then for each place $v$ of $F$ the local component $\pi_v$ has a local $(\eta_v,\psi_v)$-Shalika model \cite[\S3.2]{GR2}, that is, we have an intertwining 
\begin{equation}\label{eq:shalika integral local}
	\cS_{\psi_v}^{\eta_v} : \pi_v \hookrightarrow \mathrm{Ind}_{\cS(F_v)}^{\GL_{2n}(F_v)}(\eta_v\otimes\psi_v).
\end{equation}

\begin{remark}\label{rem:choice shalika model}
	Note \eqref{eq:shalika integral} defines a canonical global intertwining. We emphasise that the local intertwinings are \emph{not} canonical. However the local Shalika model is unique in the sense that 
	\[
	\dim_{\C}\ \mathrm{Hom}_{\GL_{2n}(F_v)}\left[\pi_v,\  \mathrm{Ind}_{\mathcal{S}(F_v)}^{\GL_{2n}(F_v)}(\eta_v\otimes\psi_v)\right] = 1
	\]
	(see \cite{Nia09,CS20}),  so the image $\cS_{\psi_v}^{\eta_v}(\pi_v)$ of $\cS_{\psi_v}^{\eta_v}$ is canonical. We henceforth fix a (non-canonical) choice of intertwining $\cS_{\psi_f}^{\eta_f}$ of $\pi_f$ (or equivalently, via \eqref{eq:shalika integral}, an intertwining $\cS_{\psi_\infty}^{\eta_\infty}$ of $\pi_\infty$).
\end{remark}

When $\pi_v$ is spherical it is shown in \cite[Prop.~1.3]{AG94} that it admits a $(\eta_v,\psi_v)$-Shalika model if and only if $\pi_v^\vee=\pi_v\otimes \eta_v^{-1}$. In this case we deduce $\eta_v$ is unramified.

Let $\pi$ be a cuspidal automorphic representation of $G(\A)$, and $\chi$ a finite order Hecke character for $F$. For $W \in \cS_{\psi}^{\eta}(\pi)$ (the image of $\pi$ under $\cS_\psi^\eta$) consider the \emph{Friedberg--Jacquet} zeta integral
\[
\zeta(s,W,\chi) \defeq \int_{\GL_n(\A_F)} W\left[\matrd{h}{}{}{I_n}\right] \ \chi(\det(h)) \ |\det(h)|^{s-\tfrac{1}{2}} \ dh,
\]
which converges absolutely in a right-half plane and extends to a meromorphic function in $s \in \C$.  When $W = \otimes_v W_v$ for $W_v \in \cS_{\psi_v}^{\eta_v}(\pi_v)$, this integral is a product of local zeta integrals $\zeta(s,W_v,\chi_v)$. 

\medskip

A \emph{Friedberg--Jacquet test vector} $W^{\mathrm{FJ}}_v \in \cS_{\psi_v}^{\eta_v}(\pi_v)$ is a vector such that for all unramified quasi-characters $\chi_v : F_v^\times \to \C^\times$, we have
\begin{equation}\label{eq:jacquet-friedberg test vector}
	\zeta_v\left(s+\tfrac{1}{2}, W^{\mathrm{FJ}}_v, \chi_v\right) = [\mathrm{N}_{F/\Q}(v)^s\chi_v(\varpi_v)]^{n\delta_v}\cdot L\left(\pi_v \otimes \chi_v,s+\tfrac{1}{2}\right),
\end{equation}
where $\delta_v$ is the valuation of the different of $F_v$ and $L(\pi_v \otimes \chi_v,s+\tfrac{1}{2})$ is the Langlands $L$-function. of $\pi_v \otimes \chi_v$. By \cite[Prop.~3.1]{FJ93}, if $\pi$ is a RACAR admitting a $(\eta,\psi)$-Shalika model, then for every finite place $v$ there exists there exists such a Friedberg--Jacquet test vector in $\cS_{\psi_v}^{\eta_v}(\pi_v)$. If $\pi_v$ is spherical, then one can take $W_v^{\mathrm{FJ}}$ to be a spherical vector, i.e.\ a vector fixed by $\GL_{2n}(\cO_v)$, normalised so that $W_{v}^{\mathrm{FJ}}(t_v^{-\delta_v}) = 1$ \cite[Prop.~3.2]{FJ93}, \cite[Prop.\ 3.3]{DJR18}.

\subsection{Parahoric $p$-refinements}\label{ss:the U_p-refined line}
Let 
\begin{equation}\label{eq:parahoric}
	J_{\pri} = \{g \in \GL_{2n}(\cO_{\pri}) : g \newmod{\pri} \in Q(\cO_{\pri}/\pri)\} \subset \GL_{2n}(F_{\pri})
\end{equation}
be the parahoric subgroup of type $Q$. We will always assume $\pi_{\pri}$ is $Q$-parahoric-spherical, that is, admits $J_{\pri}$-fixed vectors. Recall $\iota(h,h') = \smallmatrd{h}{0}{0}{h'}$, and let $t_{\pri} = \iota(\varpi_{\fp}I_n, I_n)$, recalling $\varpi_{\fp}$ is a uniformiser of $F_{\fp}$. On $\pi_{\pri}^{J_{\pri}}$, we have the Hecke operator $U_{\fp} := [J_\fp t_{\pri}  J_\fp].$

\begin{definition}	\label{def:regular Q-refinement}
	A \emph{$Q$-refinement} $\tilde\pi_{\pri} = (\pi_{\pri}, \alpha_{\pri})$ of $\pi_{\pri}$ is a choice of $U_{\pri}$-eigenvalue $\alpha_{\pri}$ on $\pi_{\pri}^{J_{\pri}}$. We say a $Q$-refinement $\tilde\pi_{\pri}$ is \emph{regular} if $\alpha_{\pri}$ is a simple $U_{\pri}$-eigenvalue on $\pi_{\pri}^{J_{\pri}}$; that is, 
		$\pi_{\pri}^{J_{\pri}}\lsem  U_{\pri} - \alpha_{\pri}\rsem$ is a line.
	We say $\tilde\pi_{\pri}$ is \emph{Shalika} if it is regular and
	if for any generator $W_{\pri}$ of 
	$\cS_{\psi_{\pri}}^{\eta_{\pri}}(\pi_{\pri}^{J_{\pri}})\lsem U_\pri - \alpha_{\pri}\rsem$,
	we have
	\begin{equation}\label{eq:shalika refinement}
		W_{\pri}(t_{\pri}^{-\delta_{\pri}}) \neq 0.
	\end{equation} 
		\end{definition}

\begin{remark}\label{rem:eta unramified}
	If $\tilde\pi_{\pri}$ is a Shalika $Q$-refinement, then for $h\in \GL_n(\cO_{\pri})$ 
		a simple check shows $\eta(\mathrm{det}(h))\cdot W_{\pri}(t_{\pri}^{-\delta_{\pri}}) = W_{\pri}(t_{\pri}^{-\delta_{\pri}}).$ 
	By non-vanishing, we have $\eta_{\pri}(\cO_{\pri}^\times) = 1$, so $\eta_{\pri}$ is unramified.
\end{remark}

Condition \eqref{eq:shalika refinement} is motivated by non-vanishing of a local zeta integral; see Proposition \ref{lem:zeta at p}. 
					The following stronger assumptions give a ready source of $\tilde\pi_{\pri}$ as above. Suppose $\pi_{\pri} = \Ind_B^G\theta_{\pri}$ is spherical, hence unramified principal series, for $\theta_{\pri} = (\theta_{\pri,1},...,\theta_{\pri,2n})$ an unramified character of $T(F_{\pri})$. Then \cite[Prop.\ 1.3]{AG94} shows that such an $\pi_{\pri}$ has an $(\eta_{\pri}, \psi_{\pri})$-Shalika model if and only if the $\theta_{\pri,i}$'s can be ordered so that $\theta_{\pri,i}\theta_{\pri,n+i} = \eta_{\pri}$ for $1 \leqslant i \leqslant n$. Then \cite[Lem.\ 3.6]{DJR18} shows:

\begin{proposition}\label{lem:Q-regular criterion} \cite[Lem.\ 3.6]{DJR18}.
	Suppose $\pi_{\pri} = \Ind_B^G\theta_{\pri}$ is spherical, that $\theta_{\pri,i}\theta_{\pri,n+i} = \eta_{\pri}$ for $1\leqslant i \leqslant n$. Let $\alpha_{\pri} = q_{\pri}^{n^2/2}\theta_{\pri,n+1}(\varpi_{\pri})\cdots \theta_{\pri,2n}(\varpi_{\pri})$, where $q_{\pri} = \mathrm{N}_{F/\Q}(\pri)$. Suppose $(\pi_{\pri}, \alpha_{\pri})$ is a regular $Q$-refinement. Then it is a Shalika $Q$-refinement. 
\end{proposition}

For spherical $\pi_{\pri}$, the $Q$-refinements that can be described as in Proposition \ref{lem:Q-regular criterion} are exactly the \emph{$Q$-spin} refinements from \cite{classical-locus}. If all the ${2n \choose n}$ possible $Q$-refinements of $\pi_{\pri}$ are different, then $2^n$ of them are $Q$-spin, so this condition covers a wide range of $\tilde\pi_{\pri}$.

Globally, a \emph{(Shalika) $Q$-refined RA(S)CAR} is a tuple $\tilde\pi = (\pi, \{\alpha_{\pri}\}_{\pri|p})$, for a RA(S)CAR $\pi$ where $\pi_{\pri}$ is $Q$-parahoric-spherical and  $(\pi_{\pri},\alpha_{\pri})$ is a (Shalika)  $Q$-refinement for each $\pri|p$.

\subsection{Running conditions on $\tilde\pi$}\label{sec:level group} 
We finally collect our running assumptions. Fix for the rest of the paper a finite order Hecke character $\eta_0$ of $F$.  We work with two levels of generality; our results apply under (C2), but are more precise under the stronger assumption (C2$'$).

\setcounter{thmprime}{\value{thm}}

\begin{conditions} Let $\pi$ be a RACAR of $G(\A)$ of weight $\lambda$ such that
	\label{cond:running assumptions} 
	\begin{itemize}\setlength{\itemsep}{0pt}
		\item[(C1)] $\pi$ admits a global $(\eta_0|\cdot|^{\sw},\psi)$-Shalika model, for $\sw$ the purity weight of $\pi$;
		\item[(C2)] for each $\pri|p$, $\pi_{\pri}$ is parahoric-spherical admitting  a Shalika $Q$-refinement $\tilde\pi_{\pri} = (\pi_{\pri}, \alpha_{\pri})$, i.e.\
		\begin{equation}\label{eq:Up refined line 2}
			\dim_{\C} \ \cS_{\psi_{\pri}}^{\eta_{\pri}}(\pi_{\pri}^{J_{\pri}})\lsem U_\pri - \alpha_{\pri}\rsem  = 1
		\end{equation} 
		(for $\eta_{\pri} = \eta_{0,\pri}|\cdot|^{\sw}_{\pri}$) and this line admits a generator $W_{\pri}$ such that $W_{\pri}(t_{\pri}^{-\delta_{\pri}}) =1$. 
	\end{itemize}
\end{conditions}

\begin{conditions2} Let $\pi$ be a RACAR of $G(\A)$ of weight $\lambda$ such that (C1) holds and
	\label{cond:running assumptions 2} 
	\begin{itemize}\setlength{\itemsep}{0pt}
		\item[(C2$'$)] for each $\pri|p$, $\pi_{\pri} = \Ind_B^G\theta_{\pri}$ is spherical, satisfies the hypotheses of Proposition \ref{lem:Q-regular criterion}, and $\tilde\pi_{\pri} = (\pi_{\pri},\alpha_{\pri})$ is the Shalika $Q$-refinement from that result.
	\end{itemize}
\end{conditions2}

By Proposition~\ref{lem:Q-regular criterion}, (C2) is automatic from (C2$'$). The $Q$-refined RACARs $\tilde\pi$ described in Theorems~\ref{thm:intro non-ord}, \ref{thm:intro shalika families} and \ref{thm:intro 3} of the introduction satisfy (C1-2$'$), hence (C1-2).

In general $\alpha_{\pri}$ is not $\pri$-integral. We define weight $\lambda$ integral normalisations
\begin{equation}
	U_{\pri}^\circ = \lambda(t_{\pri}) U_{\pri}, \hspace{12pt} \alpha_{\pri}^\circ = \lambda(t_{\pri})\alpha_{\pri}.
\end{equation}
We justify this in \S\ref{sec:slope-decomp}.  A $Q$-refinement $\tilde\pi_{\pri}$ is equivalent to a choice of $U_{\pri}^\circ$-eigenvalue $\alpha_{\pri}^\circ$ on $\pi_{\pri}^{J_{\pri}}$, and $\pi_{\pri}^{J_{\pri}}\lsem U_{\pri} - \alpha_{\pri}\rsem  = \pi_{\pri}^{J_{\pri}}\lsem U_{\pri}^\circ - \alpha_{\pri}^\circ\rsem.$
Occasionally we abuse notation and write $\tilde\pi_{\pri} = (\pi_{\pri},\alpha_{\pri}^\circ)$.

\subsection{The $p$-refined Hecke algebra}\label{sec:hecke outside S}
Let $\tilde\pi$ satisfy (C1-2), and let $K \subset G(\A_f)$ be an open compact subgroup with
\begin{equation}\label{eq:general K}
	K = \textstyle\prod_v K_v\ \text{ s.t.\ }K_v = \GL_{2n}(\cO_v)\text{ for }v \notin S\cup\{\pri|p\}, K_{\pri} = J_{\pri}\text{ for }\pri|p,\text{ and }\pi_f^K \neq 0.
\end{equation}
Recall $\cH'$ and $\psi_{\pi}$ (which implicitly are defined at level $K$) from \S\ref{sec:unramified H}.

\begin{definition} \label{def:hecke algebra} \label{def:gen eigenspace}
	Define $\cH = \cH'[U_{\pri}^\circ: \pri|p]$. Let $E = \Q(\tilde\pi,\eta)$ be the number field generated by the Hecke field of $\pi_f$, the rationality field of $\eta$, and $\alpha_{\pri}^\circ$ for $\pri|p$. The character $\psi_{\pi}$ extends to
	\[
	\psi_{\tilde\pi} : \cH\otimes  E \longrightarrow E
	\]
	sending $U_{\pri}^\circ$ to $\alpha_{\pri}^\circ$. Let $\mathfrak{m}_{\tilde\pi} \defeq \ker(\psi_{\tilde\pi})$. If $M$ is a finite dimensional vector space with an $\cH$-action, the localisation $M_{\m_{\tilde\pi}}$ is the generalised eigenspace 	at $\psi_{\tilde\pi}$, i.e.\ $M_{\mathfrak{m}_{\pi}}\lsem U_{\pri}^\circ - \alpha_{\pri}^\circ : \pri|p \rsem  \subset M_{\mathfrak{m}_{\pi}}.$
\end{definition}

\subsection{Automorphic cohomology classes and periods}
\label{sec:periods}
Recall in Remark~\ref{rem:choice shalika model} we fixed an intertwining $\cS_{\psi_f}^{\eta_f}$ of $\pi_f$. For $\epsilon \in \{\pm1\}^\Sigma$, composing $(\cS_{\psi_f}^{\eta_f})^{-1}$ and \eqref{eq:Theta 2} we obtain a $\cH$-equivariant isomorphism
\[
\Theta^{K,\epsilon} : \cS_{\psi_f}^{\eta_f}(\pi_f^K) \isorightarrow \hc{t}(S_K,\cV_{\lambda}^\vee(\C))^\epsilon_{\m_\pi};
\]
further composing with \eqref{eq:Theta 3}, we obtain a $p$-adic analogue
\begin{equation}\label{eq:cohomology class p-adic}
	\Theta_{i_p}^{K,\epsilon} : \cS_{\psi_f}^{\eta_f}(\pi_f^K) \isorightarrow \hc{t}(S_K,\sV_{\lambda}^\vee(\overline{\Q}_p))^\epsilon_{\m_\pi},
\end{equation}
which is again $\cH$-equivariant.

Finally we descend to rational coefficients. Recall the number field $E$ from Definition~\ref{def:gen eigenspace}. 
We have a natural action of $\mathrm{Aut}(\C)$ on $\cS_{\psi_f}^{\eta_f}(\pi_f)$ (see \cite[\S3.7]{GR2}), endowing it with an $E$-structure $\cS_{\psi_f}^{\eta_f}(\pi_f,E)$ by \cite[Lem.~3.8.1]{GR2}. We may (and do) take $W^{\mathrm{FJ}}_f$ to be an element of $\cS_{\psi_f}^{\eta_f}(\pi_f,E)$ (see \cite[Lem.~3.9.1]{GR2}). 
By \cite[Prop.\ 3.1]{Clo90}, \cite[Prop.~4.2.1]{GR2} and \cite[\S4.4]{JST}, there exist complex periods $\Omega_\pi^\epsilon$ 
 such that $\Theta^{K,\epsilon}\big/\Omega_\pi^\epsilon$
is $\mathrm{Aut}(\C)$-equivariant. In particular, if $L/\Qp$ is a finite extension containing $i_p(E)$, then 
\begin{equation}\label{eq:rational diagram}
	\xymatrix@C=6mm{ 
		\cS_{\psi_f}^{\eta_f}(\pi_f^K, E)  \ar[rr]^-{\Theta^{K,\epsilon}\big/\Omega_\pi^\epsilon}\ar@{^{(}->}[d] &&
		\hc{t}(S_K,\cV_{\lambda}^\vee(E))_{\m_\pi}^\epsilon \ar@{^{(}->}[rr]^-{\eqref{eq:Theta 3}}
		&& \hc{t}(S_K,\sV_\lambda^\vee(L))^\epsilon_{\m_\pi}\ar@{^{(}->}[d]\\
		\cS_{\psi_f}^{\eta_f}(\pi_f^K) \ar[rrrr]^-{\Theta^{K,\epsilon}_{i_p}\big/i_p(\Omega_\pi^\epsilon)}_-{\sim} &&&& \hc{t}(S_K,\sV_\lambda^\vee(\overline{\Q}_p))^\epsilon_{\m_\pi}
	}
\end{equation}
commutes, where the vertical arrows are the natural inclusions.

\medskip

Assume that $\tilde\pi$ satisfies Conditions \ref{cond:running assumptions}; we now produce specific cohomology classes attached to $\tilde\pi$. At each finite place $v$ of $F$, in \cite[\S6.5]{GR2} the authors define a (sufficiently small) open compact subgroup $K_v \subset \GL_{2n}(F_v)$ such that there exists a Friedberg--Jacquet test vector $W^{\mathrm{FJ}}_v \in \cS_{\psi_v}^{\eta_v}(\pi_v)^{K_v}$ as in \eqref{eq:jacquet-friedberg test vector}. 
As in \cite{DJR18}, we can (and do) take $K_v = \GL_{2n}(\cO_v)$ whenever $\pi_v$ is spherical, and define
\begin{equation}\label{eq:level group}
	K(\tilde\pi) \defeq \prod\limits_{\pri|p} \ J_{\pri} \cdot \prod\limits_{v\nmid p}\ K_v \subset G(\A_f).
\end{equation}
Note $K(\tilde\pi)$ satisfies \eqref{eq:general K}. For $\pri|p$, let $W_{\pri}$ be a generator of the line in \eqref{eq:Up refined line 2}, normalised so that $W_{\pri}(t_{\pri}^{-\delta_{\pri}}) = 1$. Write 
\[
W^{\mathrm{FJ}}_f = \otimes_{\pri|p}W_{\pri}\otimes_{v\nmid p\infty} W^{\mathrm{FJ}}_v \in \cS_{\psi_f}^{\eta_f}(\pi_f)^{K(\tilde\pi)}.
\]

\begin{definition}\label{def:phi}
	Let  
	\[
	\phi_{\tilde\pi}^{\epsilon} = \Theta^{K(\tilde\pi),\epsilon}_{i_p}(W_f^{\mathrm{FJ}})\big/i_p(\Omega_\pi^\epsilon) \in \hc{t}(S_{K(\tilde\pi)},\sV_\lambda^\vee(L))^\epsilon_{\m_\pi}.
	\]
\end{definition}
This is precisely the class defined in \cite[\S4.3.1]{DJR18}, where the scaling by $\Omega_\pi^\epsilon$ is implicit. Note that by construction, the class $\phi_{\tilde\pi}^\epsilon$ is a $U_{\pri}$-eigenclass with eigenvalue $\alpha_{\pri}$ for all $\pri|p$ (see also \cite[Lem.~3.6]{DJR18}), thus lies in the $p$-refined generalised eigenspace $\hc{t}(S_K,\sV_\lambda^\vee(L))^\epsilon\locpi$.

%%===================================================
%%===================================================
%%===================================================
%%			OVERCONVERGENT COHOMOLOGY
%%===================================================
%%===================================================
%%===================================================
\section{Overconvergent cohomology and classicality}\label{sec:overconvergent cohomology}

We recall the $Q$-parahoric overconvergent cohomology and non-$Q$-critical slope conditions of \cite{BW20}, while 
making the theory  explicit in our setting.

%%===================================================
\subsection{Weight spaces}\label{sec:weight spaces}

Recall $X^*(T),$ $X^*_0(T),$ $X^*(H)$ and $X^*_0(H)$ from \S\ref{sec:algebraic weights}.
\begin{definition}[Weights for $T$]
	The \emph{weight space} $\sW^{G}$ for $G$ is the rigid analytic space whose $L$-points, for $L \subset \C_p$ any sufficiently large extension of $\Q_p$, are given by
	\[ 
	\sW^{G}(L) = \Hom_{\mathrm{cont}}(T(\Z_p),L^\times).
	\]
	This space contains the set $X_+^*(T)$ of dominant integral weights in a natural way. We call any element of this subspace an \emph{algebraic weight}.		A weight $\lambda \in \sW^G$ decomposes as $\lambda = (\lambda_1,...,\lambda_{2n})$, where each $\lambda_i$ is a character of $(\OFp)^\times$. We see that $\sW^G$ has dimension $2dn$.
\end{definition}

\begin{definition}\label{def:weights for G}
	Let $\sW_0^{G}$ be the $(dn+1)$-dimensional \emph{pure weight space}, that is the Zariski closure of the pure, dominant, integral weights $X_0^\ast(T)$ in $\sW^{G}$. We have
	\begin{align*}		
		\sW_{0}^{G}(L):= \{ \lambda \in  \sW^{G}(L)\,  \ | \  \, & \exists\,  \sw_\lambda  \in  \Hom_{\mathrm{cont}}(\Z_p^{\times}, L^{\times}) \text{ s.t.\ }\\
		&\lambda_{i} \cdot \lambda_{2n+1-i} = \sw_\lambda \circ \mathrm{N}_{F/\Q} \ \forall \  1\leqslant i \leqslant n\}.
	\end{align*}
\end{definition}

\begin{definition}[Weights for $Q$] \label{def:weights for H}	 
	Define $\sW^Q \subset \sW^G$ to be the rigid subspace whose $L$-points are continuous characters that factor through a character $H(\Zp) \to L^\times$.
	Let $\sW^Q_0 \defeq \sW^Q\cap\sW^G_0$ be the pure subspace. These are the Zariski closures of $X^*(H)$ and $X^*_0(H)$ in $\sW^G$.
\end{definition}
The space $\sW^Q$ is the subspace of $\sW^G$ where 
	$\lambda_1 = \cdots = \lambda_n (= \nu_1, \text{ say})$ and $\lambda_{n+1} = \cdots =\lambda_{2n} (= \nu_2, \text{ say}).$ 
The association $\lambda \mapsto (\nu_1,\nu_2)$ identifies $\sW^Q$ isomorphically with the $2d$-dimensional (Hilbert) weight space of $\mathrm{Res}_{F/\Q}\GL_2$; the $(d+1)$-dimensional pure subspace $\sW^Q_0$ is  canonically identified with the pure Hilbert weights.

\begin{definition}\label{def:Wlam}
	For $\lambda_\pi \in X_0^\ast(T)$ a pure, dominant, algebraic `base' weight (implicitly, the weight of an automorphic representation $\pi$) let 
	\[
	\sW^Q_{\lambda_\pi} \defeq \{\lambda \in \sW_0^G \ | \ \lambda\lambda_\pi^{-1} \in \sW_0^Q\} = \lambda_\pi \sW_0^Q \subset \sW_0^G.
	\]
\end{definition}

\begin{remark} \label{rem:weights level}
	To get non-trivial weight $\lambda$ local systems on $S_K$ we need $\lambda(Z(\Q)\cap K)=1$. If $\pi$ is a RACAR of weight $\lambda_\pi$,  and $K$ satisfies \eqref{eq:general K} for $\pi$, this condition is satisfied by existence of an automorphic form fixed by $K$. As $\lambda(Z(\Q)\cap K) \subset \{\pm1\}$ is discrete for any weight $\lambda$, this must hold in any sufficiently small affinoid neighbourhood $\Omega \subset \sW_{\lambda_\pi}^Q$  of $\lambda_\pi$; and we will asume this for all $\Omega$ we consider henceforth.
\end{remark}

\subsection{Parahoric distribution modules}\label{sec:parabolic distributions} 
Recall for $L/\Qp$ sufficiently large, $V_\lambda(L)$ is the algebraic induction $\mathrm{Ind}_{B(\Zp)}^{G(\Zp)}\lambda$. Typically overconvergent cohomology coefficients are  dual to the locally analytic induction $\cA_\lambda^B$ of $\lambda$ to the Iwahori subgroup. We define $Q$-parahoric analogues. 

\label{sec:A_s}
If $X \subset \Qp^r$ is compact and $R$ is a $\Qp$-Banach algebra, let $\cA(X,R)$ be the  space of locally analytic functions $X \to R$, and $\cD(X,R)$ be its topological $R$-dual. If $W$ is a finite Banach $R$-module, then we say a function $f : X \to W$ is locally analytic if it is an element of $\cA(X,R)\otimes_R W$, and write $\cA(X,W)$ for the  space of such functions. (These definitions are explained in detail in \cite[\S3.2.2]{BW20}).

\subsubsection{Parahoric algebraic induction modules}
As motivation, we first give a parahoric description of $V_\lambda$. Let $G_n = \mathrm{Res}_{\cO_F/\Z}\GL_n$ and recall $H = G_n \times G_n$. Considering $\lambda \in X_0^*(T)$ as a weight for $H$, the algebraic representation of $H$ of  highest weight $\lambda$ is $V_{\lambda}^H(L) = V_{\lambda'}^{G_n}(L) \otimes V_{\lambda''}^{G_n}(L),$ 
where $\lambda' = (\lambda_1,...,\lambda_n)$ and $\lambda'' = (\lambda_{n+1},...,\lambda_{2n})$. 

The action of $H(\Zp)$ on $V_\lambda^H(L)$ yields a homomorphism
\begin{equation}\label{eq:langle-rangle}
	\langle \cdot \rangle_\lambda : H(\Zp) \to \mathrm{Aut}(V_{\lambda}^H(L)).
\end{equation} 

We say a function $\cF : G(\Zp) \to V_\lambda^H(L)$ is \emph{algebraic} if it is an element of $L[G] \otimes_L V_\lambda^H(L)$. Let
\begin{align}\label{eq:parahoric alg transform}
	\mathrm{Ind}_{Q^{-}(\Zp)}^{G(\Zp)} V_\lambda^H(L) \defeq \{\cF :\ &G(\Zp) \to V_\lambda^H(L)\ |\  \cF \text{ algebraic}, \cF(n_Q^- h g) = \langle h \rangle_\lambda \cF(g) \\  
	&\hspace{70pt}  \forall n_Q^- \in N_Q^-(\Zp), h \in H(\Zp), g \in G(\Zp)\}.\notag
\end{align}
This has a $G(\Zp)$ action by $(\gamma \cdot \cF)(g) = \cF(g\gamma)$.

\begin{lemma}\label{lem:induction transitive}
	\cite[\S I.3.5]{Jantzen}. There is a canonical isomorphism of $G(\Zp)$-representations
	\[
 \mathrm{Ind}_{Q^-(\Zp)}^{G(\Zp)} V_\lambda^H(L) \isorightarrow V_\lambda(L), \qquad \cF \mapsto [g \mapsto \cF(g)(\mathrm{id}_H)].
	\]
\end{lemma}

\subsubsection{Parahoric analytic induction modules}\label{sec:parabolic functions}
To define $Q$-parahoric analogues of $\cA_\lambda^B$, in Lemma \ref{lem:induction transitive} we replace the algebraic induction from $Q^-(\Zp)$ with locally analytic induction. Let $J_p \defeq \prod_{\pri|p} J_{\pri}$ denote the parahoric subgroup for $Q$, as defined in \eqref{eq:parahoric}. Let $\cA_{\lambda}^Q(L)$ denote the space of functions $f \in \cA(J_p,V_\lambda^H(L))$ such that
\[
f(n^-h g) = \langle h \rangle_\lambda f(g)\text{ for all }n^- \in N_Q^-(\Zp)\cap J_p, h \in H(\Zp),\text{ and }g \in J_p.
\]
Again restriction identifies $\cA_\lambda^Q(L)$ with $\cA(N_Q(\Zp),V_\lambda^H(L))$. Let $\cD_\lambda^Q(L)$ be the topological dual; this is a compact Fr\'{e}chet space \cite[\S3.2.3]{BW20}.

\begin{remark}
	Note that any $n \in N(\Zp)$ can be uniquely written as a product
	\[\def\arraystretch{0.4}
	n = h\cdot n_Q =		\footnotesize	\left(\begin{array}{cc|cc}
		1 & x_{ij}& &\\[-4pt]
		&\ddots &  & \\\hline
		&  & \ddots & 
		\text{\raisebox{6pt}{$x_{k\ell}$}}\\
		&  &  & 1
	\end{array}\right)	\left(\begin{array}{cc|cc}
		1 & & y_{1,n+1} & \cdots\\[-4pt]
		&\ddots & \cdots & \text{\raisebox{3pt}{$y_{n,2n}$}} \\\hline
		&  &  \ddots& 
		\\
		&  &  & 1
	\end{array}\right)\normalsize,
	\]
	of $h \in H(\Zp)$ and $n_Q \in N_Q(\Zp)$, where $x_{ij}, y_{ij} \in \OFp$. Then for any  $f \in \cA_\lambda^B$ the restriction $f|_{N(\Zp)}$ is locally analytic   in the $x_{ij,\sigma}$ and $y_{ij,\sigma}$.  On the other hand,  for any  $f \in \cA_{\lambda}^Q$ the 
	restriction $f|_{N_Q(\Zp)}$ is a  locally analytic function in the $y_{ij,\sigma}$
	with coefficients in $V_\lambda^H$. As $V_\lambda^H$ can be realised as a space of polynomials in the $x_{ij,\sigma}$, one sees that   $\cA_{\lambda}^Q$ is an intermediate space between  $V_{\lambda}$ and $\cA_{\lambda}^B$. 
	A precise description of the natural inclusion $\cA_{\lambda}^Q \subset \cA_{\lambda}^B$ is given in \cite[Props.\ 4.9, 4.11]{BW20}. 
\end{remark}

\begin{notation}
	Since throughout we will only be interested in $Q$-parahoric distributions, we will henceforth suppress superscript $Q$'s and write $\cA_{\lambda} \defeq \cA_{\lambda}^Q$ and $\cD_{\lambda} \defeq \cD_{\lambda}^Q.$
\end{notation}

%%==============================================
\subsubsection{Distributions in families}\label{sec:distributions in families}
Let $\Omega \subset \Wlam$ be an affinoid, for a fixed $\lambda_\pi \in X_0^\ast(T)$. If $\lambda \in \Omega$ is algebraic, then by definition $\lambda\lambda_\pi^{-1} \in \sW^Q_0$ and there is an isomorphism
\begin{equation}\label{eq:changing lambda}
	V^H_{\lambda} = V^H_{\lambda_\pi} \otimes \lambda\lambda_\pi^{-1}
\end{equation}
of $H(\Zp)$-modules \cite[Lem.~3.8]{BW20}. In particular, the underlying spaces of $V^H_{\lambda}$ and $V^H_{\lambda_\pi}$ are the same, allowing analytic variation of the representation $V^H_\lambda$ as $\lambda$ varies in an affinoid of $\Wlam$.  

The space $\Omega_0 \defeq \{\lambda\lambda_\pi^{-1} : \lambda \in \Omega\}$ is an affinoid in $\sW^Q_0 \subset \sW^G$.

\begin{lemma}
	The character $\chi_{\Omega_0} : H(\Zp) \longrightarrow \cO_{\Omega_0}^\times$ given by $h \mapsto [\lambda_0 \mapsto \lambda_0(h)]$ is locally analytic. 
\end{lemma} 
\begin{proof}
	This is proved in \cite[\S3.2.6]{BW20} using \cite[Prop.~8.3]{Buz07}.
\end{proof}

As $\chi_{\Omega_0}$ is a character of $H(\Zp)$, it factors through its abelianisation $H(\Zp) \to (\OFp^\times)^2$,
so there exists a character $(\chi_{\Omega_0}^1, \chi_{\Omega_0}^2)$ of $(\OFp^\times)^2$ such that 
\[
\chi_{\Omega_0}(h_1,h_2) = (\chi_{\Omega_0}^1 \circ \det(h_1)) \cdot (\chi_{\Omega_0}^2 \circ \det(h_2)).
\]  
As $\Omega_0$ is a subspace of the pure weights, there exists 
\disp{
\sw_{\Omega_0} : \Zp^\times \to \cO_{\Omega_0}^\times
} such that 
\disp{
\chi_{\Omega_0}^1(x) \cdot \chi_{\Omega_0}^2(x) = \sw_{\Omega_0} \circ \mathrm{N}_{F/\Q}(x)
}
for all $x \in \OFp^\times$, and hence
\begin{equation}\label{eq:pure family}
	\chi_{\Omega_0}(h,h) = 	\sw_{\Omega_0} \circ \mathrm{N}_{F/\Q}(\det(h)).
\end{equation}
If $\lambda_0 \in \Omega_0$, then evaluation at $\lambda_0$ sends $\sw_{\Omega_0}$ to $\sw_{\lambda_0}$, as defined in Definition \ref{def:weights for G}, so $\sw_{\Omega_0}$ interpolates purity weights over $\Omega_0$.

Now define $V^H_{\Omega} \defeq V^H_{\lambda_\pi}(L) \otimes_L \cO_{\Omega_0}$, a free $\cO_{\Omega_0}$-module of finite rank, and a homomorphism
\begin{align}\label{eq:action V_Omega}
	\langle \cdot \rangle_{\Omega} : H(\Zp) &\longrightarrow \mathrm{Aut}\big(V^H_{\lambda_\pi}(L)\big) \otimes \cO_{\Omega_0}^\times \subset \mathrm{Aut}\big(V^H_{\Omega}\big), \\
	h & \longmapsto \langle h \rangle_{\lambda_\pi} \otimes \chi_{\Omega_0}(h).\notag
\end{align}
This makes $V^H_{\Omega}$ into an $H(\Zp)$-representation. 

\begin{definition}\label{def:sp lambda}
	Let $\lambda \in \Omega(L)$, and let $\lambda_0 = \lambda\lambda_\pi^{-1} \in \Omega_0(L)$. Define a map $\mathrm{sp}_{\lambda_0} : \cO_{\Omega_0} \rightarrow L$ by evaluating functions at $\lambda_0$. This induces a map
	\begin{equation}\label{eq:parahoric transform}
		\mathrm{sp}_\lambda : V^H_{\Omega} \labelrightarrow{\mathrm{id}\otimes \mathrm{sp}_{\lambda_0}} V^H_{\lambda_\pi}(L) \otimes \lambda\lambda_\pi^{-1} \labelisorightarrow{\eqref{eq:changing lambda}} V^H_{\lambda}(L).
	\end{equation}
	Since $\mathrm{sp}_{\lambda_0}\circ \chi_{\Omega_0} = \lambda_0$ by \eqref{eq:changing lambda}, this map is $H(\Zp)$-equivariant. In particular, $V^H_{\Omega}$ interpolates the representations $V^H_{\lambda}$ as $\lambda$ varies in $\Omega$ (where if $\lambda$ is non-algebraic, $V_{\lambda}^H \defeq V_{\{\lambda\}}^H$).
\end{definition}

Choosing $\lambda_\pi$ fixes an isomorphism $\Omega \isorightarrow \Omega_0$, $\lambda \mapsto \lambda_\pi^{-1}\lambda$. This induces $\cO_\Omega \isorightarrow \cO_{\Omega_0}$, compatible with specialisation maps. Under this we may define characters
\begin{align}
	\chi_\Omega &\defeq \lambda_\pi \cdot \chi_{\Omega_0} : H(\Zp) \to \cO_\Omega^\times, \notag\\
	\sw_\Omega &\defeq \sw_{\lambda_\pi}\cdot \sw_{\Omega_0} : \Zp^\times \to \cO_\Omega^\times \label{eq:sw_Omega}
\end{align}
such that evaluation at $\lambda \in \Omega$ sends $\chi_\Omega$ to $\lambda$ and $\sw_\Omega$ to $\sw_\lambda$.   Henceforth we work only with $\Omega$, suppressing $\Omega_0$, and implicitly any transfer of structure is with respect to this identification.

\begin{definition}
	Define $\cA_{\Omega}$ to be the space of functions $f \in \cA(J_p,V_\Omega^H)$
	such that
	\begin{equation}\label{eq:parahoric transform A}
		f(n^-h g) = \langle h \rangle_\Omega f(g)\text{ for all }n^- \in N_Q^-(\Zp)\cap J_p, h \in H(\Zp),\text{ and }g \in J_p.
	\end{equation}
	Define $\cD_{\Omega} \defeq \Hom_{\mathrm{cont}}(\cA_{\Omega}, \cO_\Omega)$.  This is a compact Fr\'echet $\cO_\Omega$-module (see \cite[Lem.~3.16]{BW20}).
	
\end{definition}

\begin{remark}\label{rem:alternative description}
	As in \cite[Rem.~3.18]{BW20}, if $\Omega' \subset \Omega$ is a closed affinoid, then 
	\disp{
		\cD_{\Omega} \otimes_{\cO_{\Omega}}\cO_{\Omega'} \cong \cD_{\Omega'}.
	}
	As a special case, suppose $\Omega' = \{\lambda\}$ is a single point, whence $\cO_{\Omega'} = L$ is a field, of the form $\cO_\Omega/\m_\lambda$ for $\m_\lambda \subset \cO_\Omega$ the maximal ideal attached to $\lambda$. The map $\mathrm{sp}_\lambda : \cO_\Omega \to L$ is reduction modulo $\m_\lambda$. Then we see
	\[
	\mathrm{sp}_\lambda(\cD_\Omega) \defeq \cD_\Omega \otimes_{\cO_\Omega} \cO_\Omega/\m_\lambda \cong \cD_\lambda(L).
	\]
	In particular, $\cD_\Omega$ interpolates $\cD_\lambda$ as $\lambda$ varies in $\Omega$.
\end{remark}

%%==============================================
\subsection{The action of $U_p^\circ$ and slope decompositions}\label{sec:slope-decomp}\label{sec:slope-decomp 2}
Fix any open compact subgroup $K \subset G(\A_f)$ such that $K_p\subset J_p$ (e.g.\  $K=K(\tilde\pi)$). 
Let $\Omega$ be an affinoid in $\sW^Q_{\lambda_\pi}$; we allow $\Omega = \{\lambda\}$ a single weight, in which case $\cO_\Omega=L$. 
We have a natural left action of $J_p$ on $\cA_{\Omega}$ by 
\[
(k * f)(g) = f(gk), \qquad k \in J_p, f \in \cA_\Omega, g \in J_p,
\] inducing a dual left action of $J_p$ on $\cD_\Omega$ by $(k * \mu)(f) \defeq \mu(k^{-1}* f)$. Thus $\cD_{\Omega}$ is a left $K$-module (via projection to $K_p$), giving a local system  $\sD_\Omega$ on the space $S_K$ as in \S\ref{sec:non-arch ls}.

Recall $t_{\pri} = \iota(\varpi_{\pri}I_n,I_n)$. Note 
\disp{
t_{\pri}N_Q(\Zp)t_{\pri}^{-1} \subset N_Q(\Zp).
}
For any $f \in \cA_\Omega$, define a function 
\[
t_{\pri}^{-1} * f : N_Q(\Zp) \to V_{\Omega}^H
\]
sending $n \in N_Q(\Zp)$ to $f(t_{\pri}n t_{\pri}^{-1})$. As $H(\Zp)$ commutes with $t_{\pri}$, using  \eqref{eq:parahoric transform A} and parahoric decomposition,  $t_{\pri}^{-1} * f$ extends to a unique function in $\cA_\Omega$. 

Let $\Delta_p \subset G(\Q_p)$ be the semigroup generated by $J_p$ and $t_{\pri}$ for $\pri| p$. 
One checks  (e.g. as in \cite[\S3.1.3]{Urb11}) that  the actions of $J_p$ and $t_{\pri}^{- 1}$ on $\cA_\Omega$ extend to a 
left action of $\Delta_p^{-1}$ on $\cA_\Omega$. 
We get a dual left action of $\Delta_p$ on  $\cD_\Omega$ by $(\delta * \mu)(f) = \mu(\delta^{-1} * f)$. We then equip the cohomology groups $\hc{i}(S_K,\sD_\Omega)$ with an $\cO_\Omega$-linear action of the Hecke operators $U_{\pri}^\circ \defeq [K_{\pri} t_{\pri} K_{\pri}]$ in the usual way.

\begin{remark}\label{rem:star vs dot}
	To justify the notation, note the $*$-action on $\cD_\lambda(L)$ preserves its natural integral subspace (see \cite[\S3.2.5]{BW20}). Also, the $*$-action on $\cA_{\lambda}$ preserves the subspace $V_{\lambda}$, so dualising we see $\cD_\lambda(L)$ admits $V_\lambda^\vee(L)$ as a $*$-stable quotient. We thus obtain an induced $*$-action on $V_{\lambda}^\vee(L)$. 
	
	On $V_{\lambda}^\vee(L)$, we also have the natural algebraic $\cdot$-action of $G(\Qp)$ from \S\ref{sec:non-arch ls}. The $*$- and $\cdot$-actions of $\Delta_p$ on $V_{\lambda}^{\vee}(L)$ coincide for $J_p \subset \Delta_p$, so give the same $p$-adic local system $\sV_{\lambda}^\vee(L)$. However, the actions of $t_{\pri}$ are different; analogously to \cite[Rem.\ 3.23]{BW20} one computes that
	\begin{equation}\label{eq:* vs dot}
		t_{\pri} * \mu = \lambda(t_{\pri}) \cdot \big(t_{\pri}\cdot \mu\big),\ \ \ \ \mu \in V_{\lambda}^\vee(L).
	\end{equation} 
	The two actions induce Hecke operators $U_{\pri}^*$ and $U_{\pri}^\cdot$ on the classical cohomology, related by $U_{\pri}^* = \lambda(t_{\pri}) U_{\pri}^\cdot$.	The isomorphism 
	\eqref{eq:cohomology class p-adic} is equivariant for the natural $U_{\pri}$-operator on $\pi_f^K$ and the operator $U_{\pri}^\cdot$ on cohomology, so the class $\phi_{\tilde\pi}^\epsilon$ from Definition~\ref{def:phi} is an eigenclass with $U_{\pri}^\cdot$-eigenvalue $\alpha_{\pri}$. However $U_{\pri}^\cdot$ does not preserve integrality, whilst $U_{\pri}^*$ does. Because of this it is standard to write		$U_{\pri} \defeq U_{\pri}^\cdot$ for the usual automorphic Hecke operator, and $U_{\pri}^\circ \defeq U_{\pri}^*$ for its integral normalisation. 
	Thus $\phi_{\tilde\pi}^\epsilon$ is a $U_{\pri}^\circ$-eigenclass with eigenvalue $\alpha_{\pri}^\circ \defeq \lambda(t_{\pri})\alpha_{\pri}$. As $U_{\pri}^\circ$ preserves integrality, $v_p(\alpha_{\pri}^\circ) \geqslant 0$.
\end{remark}

Let $t_p = \iota(pI_n, I_n)$. We have $t_p \in T_Q^{++}$ in the notation\footnote{As distributions in \cite{BW20} are right-modules, all conventions are opposite to those here: see \cite[Rem.~4.20]{BW20}.} of \cite[\S2.5]{BW20}, and (via the $*$-action) we get a $Q$-controlling operator $U_p^\circ \defeq [K_p t_p K_p]$ on the cohomology.   By \S3.5 \emph{ibid.}, for any $h \in \Q_{\geqslant 0}$, up to shrinking $\Omega$ the $\cO_{\Omega}$-module $\hc{\bullet}(S_K,\sD_{\Omega})$ admits a slope $\leqslant h$ decomposition with respect to $U_p^\circ$  (see \cite[Def.~2.3.1]{Han17}). 
We let  $\hc{\bullet}(S_K,\sD_{\Omega})^{\leqslant h}$ denote the subspace of elements of slope at most $h$, and note that it is an $\cO_{\Omega}$-module of finite type.

%%==============================================
\subsection{Non-critical slope conditions for $Q$}\label{sec:non-critical}
Let $\lambda \in X_0^*(T)$ be a pure dominant integral weight and $K$ as in \S\ref{sec:slope-decomp 2}. The natural inclusion of $V_\lambda(L) \subset \cA_\lambda(L)$ induces dually a surjection $\cD_\lambda(L) \longrightarrow V_\lambda^\vee(L)$, which is equivariant for the $*$-actions of $\Delta_p$. This induces a map
\begin{equation}\label{eqn:specialisation}
	r_\lambda : \hc{\bullet}(S_K, \sD_\lambda(L)) \longrightarrow \hc{\bullet}(S_K,\sV_\lambda^\vee(L)),
\end{equation}
equivariant for the $*$-actions of $\Delta_p$ on both sides; hence by Remark~\ref{rem:star vs dot}, it is equivariant for the actions of $U_{\pri}^\circ$ on both sides.

Let $\tilde\pi$ be a $Q$-refined RACAR of $G(\A)$ of weight $\lambda$ and $h \gg 0$. As $\hc{\bullet}(S_K,\sD_\lambda(L))^{\leqslant h}$ is a finite dimensional vector space, the localisation $\hc{\bullet}(S_K,\sD_\lambda(L))^{\leqslant h}\locpi$ is the generalised eigenspace in $\hc{\bullet}(S_K,\sD_\lambda(L))$ where the Hecke operators act with the same eigenvalues as on $\tilde\pi$ (see \S\ref{sec:hecke outside S}). Abusing notation, we drop the $\leqslant h$ and just write $\hc{\bullet}(S_K,\sD_\lambda(L))\locpi$ for this generalised eigenspace.

\begin{definition}\label{def:non-Q-critical}
	Let $\tilde\pi$ be a $Q$-refined RACAR of $G(\A)$ of weight $\lambda$. 
	We say $\tilde\pi$ is \emph{non-$Q$-critical (at level $K$)} if the restriction of $r_\lambda$  to the generalised eigenspaces 
	\begin{align*}
		r_\lambda : \hc{\bullet}(S_{K}, \sD_\lambda(L))\locpi  \isorightarrow&\  \hc{\bullet}(S_{K},\sV_\lambda^\vee(L))\locpi
	\end{align*}
	is an isomorphism.  If $K$ is clear from the context we will not specify it.
	
	We say $\tilde\pi$ is \emph{strongly non-$Q$-critical} if this is true for all $K$ satisfying \eqref{eq:general K}, and also with $\hc{\bullet}$ replaced with $\h^\bullet$ (i.e., if $\tilde\pi$ is non-$Q$-critical for $\h^\bullet$ \emph{and} for $\hc{\bullet}$ as in \cite[Rem.~4.6]{BW20}). 
\end{definition}

Recall $\Sigma=\coprod_{\pri\mid p} \Sigma(\pri)$ from \S\ref{sec:notation}. For $\pri|p$, let  $e_{\pri}$ be the ramification degree of $\pri|p$.

\begin{definition}\label{def:non-critical slope}
	For $\pri|p$, we say that $\tilde\pi_\pri= (\pi_\pri,\alpha_{\pri})$ has \emph{non-$Q$-critical slope}  if
	\[	e_\pri\cdot  v_p\big(\alpha_{\pri}^\circ\big) < \mathrm{min}_{\sigma \in \Sigma(\pri)}(1 + \lambda_{\sigma,n} - \lambda_{\sigma, n+1}),
	\]
	where $\alpha_{\pri}^\circ = \lambda(t_{\pri})\alpha_{\pri}$. If this holds at all $\pri|p$, we say that $\tilde\pi$ has \emph{non-$Q$-critical slope}. 
\end{definition}

\begin{theorem}[Classicality] \label{thm:control}
	If $\tilde\pi$ has non-$Q$-critical slope, then it is strongly non-$Q$-critical.
\end{theorem}
\begin{proof}
	This is a special case of \cite[Thm.~4.4, Rem.~4.6]{BW20}, explained in Examples 4.5 \emph{ibid.}.
\end{proof}

%%================================================================
%%
%%				EVALUATION MAPS
%%
%%================================================================

\section{Abstract evaluation maps}\label{sec:abstract evaluation maps}

	We now describe abstract evaluation maps, linear functionals on the degree $t$ cohomology of $S_K$ with arbitrary coefficient systems, generalising those from \cite{GR2,DJR18} (where the coefficients were taken in $V_\lambda^\vee$). The underlying idea is to integrate classes over \emph{automorphic cycles} coming from $H \subset G$. Crucially, these cycles have real dimension equal to $t$, the `magical numerology' observed in \cite{GR2}. We show our constructions are functorial in the coefficient system.

\subsection{Automorphic cycles}\label{sec:auto cycles}

\begin{definition} 
	Let $L_H \subset H(\A_f)$ be an open compact subgroup. Define the \emph{automorphic cycle} of level $L_H$ to be the space
	\[
	X_{L_H} \defeq H(\Q)\backslash H(\A)/L_HL_\infty^\circ,
	\]
	where $L_\infty = H_\infty\cap K_\infty$ for $H_\infty \defeq H(\R)$ (note all intersections are taken with respect to $\iota$). Note $Z_G(\R)\cap H_\infty \subsetneq Z_H(\R)$, so this is not the  locally symmetric space for $H$. This is denoted $\tilde{S}_{L_H}^H$ in \cite{DJR18}, and is a real orbifold of dimension $t$ \cite[(23)]{DJR18}.
\end{definition}

We choose a specific $L_H$, as in \cite[\S2.1]{DJR18}. Let $K \subset G(\A_f)$ be an open compact subgroup.
\begin{definition}\label{def:xi}
	\begin{enumerate}[(i)]
		\item Define a matrix $\xi \in \GL_{2n}(\A_F)$ by setting $\xi_v = 1$ for all $v\nmid p$ and
		\[
		\xi_{\pri} = \smallmatrd{I_n}{w_n}{0}{w_n} \in \GL_{2n}(\cO_{F,\pri})
		\]
		for $\pri|p$, where $w_n$ is the antidiagonal $n\times n$ matrix whose $(i,j)$-th entry is $\delta_{i,n-j+1}$.
		
		\item For a multi-exponent $\beta = (\beta_{\pri})_{\pri|p}$, with $\beta_{\pri} \in \Z_{\geqslant 0}$, we write $p^{\beta} = \prod \varpi_{\pri}^{\beta_{\pri}}$ and $t_p^\beta = \prod t_{\pri}^{\beta_{\pri}}.$ Fix an ideal $\m \subset \cO_F$ prime to $p$. Then let $L_{\beta} = L^{(p)}\prod_{\pri|p}L_\pri^{\beta_{\pri}},$
			where
			\begin{enumerate}[(L1)]
				\item $L^{(p)} = \{h \in H(\widehat{\Z}^{(p)}) : h \equiv 1 \newmod{\m}\}$,
				\item at $\pri|p$, $L_\pri^{\beta_{\pri}} \defeq H(\Zp) \cap K_{\pri} \cap \xi_{\pri} t_{\pri}^{\beta_{\pri}}K_{\pri} t_{\pri}^{-\beta_{\pri}} \xi_{\pri}^{-1}.$
			\end{enumerate}

		Let $X_\beta \defeq X_{L_\beta}$  be the \emph{automorphic cycle of level $p^\beta$}.
	\end{enumerate}
\end{definition}

The ideal $\m$ will always be fixed large enough so that
\begin{equation}\label{eq:auxiliary m 1}
	L^{(p)}\text{ is contained in }K^{(p)} \cap H(\A_f),\ \text{ and}
\end{equation}
\begin{equation}\label{eq:auxiliary m 2}
	H(\Q) \cap hL_\beta L^\circ_\infty h^{-1} = Z_G(\Q) \cap L_\beta L_\infty^\circ\text{ for all }h \in H(\A).
\end{equation}
By \eqref{eq:auxiliary m 2}, $X_\beta$ is a real manifold \cite[(21)]{DJR18}. Changing $\m$ will scale all our constructions of $p$-adic $L$-functions by a fixed non-zero rational scalar (captured in the volume constant $\gamma_{p\m}$ of Theorem~\ref{thm:critical value} below); but each construction is only well-defined up to scaling the choice of periods $\Omega_\pi^\epsilon$, so changing $\m$ yields no loss in generality. We fix $\m$ to be the minimal such choice, dropping it from all notation.

By definition of $L_{\pri}^{\beta_{\pri}}$ and \eqref{eq:auxiliary m 1}, there is a proper map (see \cite[Lemma 2.7]{Ash80})
\begin{equation}\label{eq:iota beta}
	\iota_\beta : X_\beta \longrightarrow S_K,\qquad 	[h] \longmapsto [\iota(h)\xi t_p^\beta].
\end{equation}
The cycle $X_\beta$ decomposes into connected components \cite[(22)]{DJR18} indexed by 
\begin{equation}\label{eq:component group}
	\pi_0(X_\beta) \defeq \Cl(p^\beta\m) \times \Cl(\m),
\end{equation}
where the component of $[(h_1,h_2)] \in S_K$ is given by the class
\disp{
[(\det(h_1)/\det(h_2), \det(h_2))] \in \pi_0(X_\beta).
}
For $\delta \in H(\A_f)$, we write $[\delta]$ for its associated class in $\pi_0(X_\beta)$ and denote the corresponding connected component
\disp{
X_\beta[\delta] \defeq H(\Q)\backslash H(\Q)\delta L_\beta H_\infty^\circ/L_\beta L_\infty^\circ.
}

\begin{remark}
		The map $\iota : H \to G$ does \emph{not} induce a well-defined map on the true locally symmetric spaces, which is why we quotient by $(Z_G\cap H)^\circ(\R)$, not $Z_H^\circ(\R)$.
		
	 Whilst we do not make this explicit, underlying our later proof of $p$-adic interpolation is the following fact: $\xi$ is an open orbit representative for the spherical pair $H \subset G$, in the sense that $B^-(\Zp) \xi H(\Zp)$ is Zariski-dense in $G(\Zp)$.
\end{remark}

%%==========================================================
%%
%%			Abstract evaluation maps
%%
%%==========================================================
\subsection{Abstract evaluation maps} \label{sec:abstract evaluations}
 Let $K \subset G(\A_f)$ be open compact such that $N_Q(\cO_{\pri}) \subset K_{\pri} \subset J_{\pri}$ for $\pri|p$, and recall $\Delta_p$ from \S\ref{sec:slope-decomp}. Let $M$ be a left $\Delta_p$-module, with action denoted $*$. Then $K$ acts on $M$ via its projection to $K_p \subset \Delta_p$, giving a local system $\sM$ on $S_K$ via \S\ref{sec:non-arch ls}.

\subsubsection{Pulling back to cycles}\label{sec:pulling back to cycles}
We first pull back under $\iota_\beta$. As in \cite[\S2.2.2]{DJR18}, there is a twisting map of local systems 
	$\tau_\beta^\circ : \iota_\beta^*\sM \rightarrow \iota^*\sM$ given by $(h,m) \mapsto (h,\xi t_p^\beta * m)$, where $\iota^*\sM$ is the local system given by locally constant sections of 
\begin{align}\label{eq:projection}
	H(\Q)\backslash(H(\A) \times M)/L_\beta L_\infty^\circ \rightarrow X_\beta, \qquad 
	\zeta(h,m)\ell z = (\zeta h \ell z, \ell^{-1} * m).
\end{align}
On cohomology we get a map
\disp{
\tau_\beta^\circ \circ \iota_\beta^* : \htc(S_K,\sM) \longrightarrow\htc(X_\beta, \iota_\beta^*\sM) \longrightarrow  \htc(X_\beta,\iota^*\sM).
}

\subsubsection{Passing to components}\label{sec:passing to components}
We trivialise $\iota^*\sM$ by passing to connected components. Let $\delta \in H(\A_f)$ represent $[\delta] \in \pi_0(X_\beta)$. Define $\cX_H \defeq  H_\infty^\circ/L_\infty^\circ$. The congruence subgroup
\begin{align}\label{eq:gamma beta delta}
	\Gamma_{\beta,\delta}& \defeq H(\Q) \cap \delta L_\beta H_\infty^\circ \delta^{-1}\\
	& \subset H(\Q)^+ \defeq \{(h_1,h_2) \in H(\Q): \det(h_i) \in \cO_{F}^\times\cap F_\infty^{\times\circ}\}\notag
\end{align} 
acts on $\cX_H$ (by left translation by its embedding into $H_\infty^\circ$). Further if $\gamma \in \Gamma_{\beta,\delta}$ then $(\delta^{-1}\gamma\delta)_f \in L_\beta$, so $\Gamma_{\beta,\delta}$ acts on $M$ via 
\begin{equation}\label{eq:gamma action}
	\gamma *_{\Gamma_{\beta,\delta}}m \defeq (\delta^{-1}\gamma\delta)_f* m.
\end{equation}

If $[h_\infty] \in \cX_H$, write $[h_\infty]_\delta$ for its image in $\Gamma_{\beta,\delta}\backslash \cX_H$. 	There is a map $c_\delta : \Gamma_{\beta,\delta}\backslash \cX_H \isorightarrow X_\beta[\delta] \subset X_\beta$ given by
	$[h_\infty]_\delta \mapsto [\delta h_\infty]$. Pulling back gives a map of local systems $c_\delta^* : \iota^*\sM\to c_\delta^*\iota^*\sM$.

\begin{lemma} \label{lem:c_delta}
	The local system $c_\delta^*\iota^*\sM$ on $\Gamma_{\beta,\delta}\backslash\cX_H$ is given by locally constant sections of 
	\begin{align}\label{eq:component system}
		\Gamma_{\beta,\delta}\backslash [\cX_H \times M] \longrightarrow \Gamma_{\beta,\delta}\backslash \cX_H,
	\end{align}
	with action $\gamma([h_\infty],m) = ([\gamma_\infty h_\infty], \gamma *_{\Gamma_{\beta,\delta}} m)$. The map $c_\delta^*$ of local systems is induced by the map
	\begin{align*}
		c_{\delta}^* :  H(\Q)\backslash \big[H(\Q)\delta L_\beta H_\infty^\circ \times M\big]/ L_\beta L_\infty^\circ &\longrightarrow \Gamma_{\beta,\delta}\backslash[\cX_H \times M],\\
		(\zeta \delta \ell h_\infty , m) &\longmapsto ([h_\infty],  \ell * m).
	\end{align*}
\end{lemma}
\begin{proof}
	In 	$H(\Q)\backslash [H(\Q)\delta L_\beta H_\infty^\circ \times M]/ L_\beta L_\infty^\circ$, we have 
	\[
	(\zeta\delta \ell h_\infty, m) = \zeta(\delta h_\infty, \ell * m)\ell = (c_\delta(h_\infty), \ell*m),
	\]
	so the map is as claimed. To see  $c_\delta^*\iota^*\sM$ is given by \eqref{eq:component system}, let $\gamma \in \Gamma_{\beta,\delta}$; then $\gamma_f = \delta\ell\delta^{-1}$ for some $\ell \in L_\beta$, and $\gamma *_{\Gamma_{\beta,\delta}} m = \ell * m$ by definition. In $H(\Q)\backslash [H(\Q)\delta L_\beta H_\infty^\circ \times M]/ L_\beta L_\infty^\circ$ we have
	\begin{align*}
		\gamma(\delta h_\infty, m) &= (\gamma_f\gamma_\infty\delta h_\infty, m) = (\delta\ell \cdot \gamma_\infty h_\infty, m) = (\delta \gamma_\infty h_\infty, \ell * m).
	\end{align*}
	Thus
	\[
	\xymatrix@R=3mm@C=8mm{
		\gamma([h_\infty], m)\sar{d}{=} & \ar@{|->}[l]_-{c_\delta^*} \gamma(\delta h_\infty, m) \sar{d}{=} \\
		([\gamma_\infty h_\infty], \gamma *_{\Gamma_{\beta,\delta}} m) & \ar@{|->}[l]_-{c_\delta^*}(\delta \gamma_\infty h_\infty, \gamma *_{\Gamma_{\beta,\delta}} m)
	}
	\]
	commutes, from which we deduce the action must be by \eqref{eq:component system}.
\end{proof}

\subsubsection{Trivialising and integration over a fundamental class}\label{sec:coinvariants} 
Let $M_{\Gamma_{\beta,\delta}}$ be the coinvariants of $M$ by $\Gamma_{\beta,\delta}$. Since $\Gamma_{\beta,\delta}$ acts trivially on $M_{\Gamma_{\beta,\delta}}$, the quotient $M \rightarrow M_{\Gamma_{\beta,\delta}}$, $m \mapsto (m)_\delta$ induces a trivialisation map (over $\Gamma_{\beta,\delta}\backslash \cX_H$)	$\coinv_{\beta,\delta} : \Gamma_{\beta,\delta}\backslash [\cX_H \times M] \to [\Gamma_{\beta,\delta}\backslash\cX_H]
	\times M_{\Gamma_{\beta,\delta}}$ given by $([h_\infty], m) \mapsto ([h_\infty]_\delta, (m)_\delta).$	 We get
\[
\coinv_{\beta,\delta} : \htc(\Gamma_{\beta,\delta}\backslash\cX_H, c_\delta^*\iota^*\sM) \rightarrow \htc(\Gamma_{\beta,\delta}\backslash\cX_H, \Z)\otimes M_{\Gamma_{\beta,\delta}}.
\]
 In \cite[\S2.2.5]{DJR18}, a class 
\disp{
\theta_{[\delta]} \in \h^{\mathrm{BM}}_t(X_\beta[\delta],\Z) \cong \Z
}
is chosen for each class $[\delta]$, and we take 
\disp{
\theta_{\delta} \defeq c_\delta^*(\theta_{[\delta]}).
}
Cap product induces an isomorphism
\begin{align*}
	(- \cap \theta_\delta) : \htc(\Gamma_{\beta,\delta}\backslash\cX_H, \Z)\otimes M_{\Gamma_{\beta,\delta}} &\isorightarrow M_{\Gamma_{\beta,\delta}}, \qquad (\phi, m) \longmapsto (\phi \cap \theta_\delta, m).
\end{align*}

\begin{definition}\label{def:abstract evaluation}
	Define the \emph{evaluation map of level $p^\beta$} to be the composition
	\begin{align}\label{eq:evaluation definition}
		\mathrm{Ev}_{\beta,\delta}^{M}   : \htc(S_K,\sM)  \xrightarrow{\ \tau_\beta^\circ \circ \iota_\beta^*\ } \hc{t}(X_\beta,&\iota^*\sM) \xrightarrow{c_\delta^*} \hc{t}(\Gamma_{\beta,\delta}\backslash \cX_H, c_\delta^*\iota^*\sM)\\
		& \xrightarrow{\coinv_{\beta,\delta}}  \hc{t}(\Gamma_{\beta,\delta}\backslash\cX_H, \Z) \otimes M_{\Gamma_{\beta,\delta}} \labelisorightarrow{ - \cap \theta_\delta} M_{\Gamma_{\beta,\delta}}. \notag
	\end{align}
\end{definition}

\subsection{Variation of $M$, $\delta$ and $\beta$}\label{sec:variation}

\subsubsection{Variation in $M$} The functoriality in $M$ is the object of the following statement. 
\begin{lemma}\label{prop:pushforward}
	Let $\kappa : M \rightarrow N$ be a map of $\Delta_p$-modules. There is a commutative diagram
	\[
	\xymatrix@C=18mm@R=6mm{
		\hc{t}(S_K, \sM) \ar[r]^-{\mathrm{Ev}_{\beta,\delta}^M} \ar[d]^-{\kappa} & M_{\Gamma_{\beta,\delta}}\ar[d]^-{\kappa}\\
		\hc{t}(S_K,\sN) \ar[r]^-{\mathrm{Ev}_{\beta,\delta}^N} & N_{\Gamma_{\beta,\delta}}.
	}
	\]
\end{lemma} 
\begin{proof}
	Writing out the definitions, it is immediate that $\kappa$ induces a map on cohomology and a map on coinvariants, and $\kappa$ commutes with each of the maps in \eqref{eq:evaluation definition} (compare \cite[Lem.~3.2]{BDJ17}).
\end{proof}

\subsubsection{Variation in $\delta$}\label{sec:changing delta}

For fixed $\delta$, the action of $\ell \in L_\beta$ need not preserve $M_{\Gamma_{\beta,\delta}}$. Nevertheless:

\begin{lemma}\label{lem:action on coinvariants} Let $\delta \in H(\A)$ and $\ell \in L_\beta$. If $\delta' = \zeta\delta\ell h_\infty \in H(\Q)\delta \ell H_\infty^\circ \cap H(\A_f)$ is another representative of $\delta$, then:

	\begin{itemize}\s
		\item[(i)]The action of $\ell$ on $M$ induces a map $M_{\Gamma_{\beta,\delta'}} \to M_{\Gamma_{\beta,\delta}}$ given by $(m)_{\delta'} \mapsto \ell * (m)_{\delta'} \defeq (\ell * m)_{\delta}.$

		\item[(ii)] There is a well-defined map 				$[\zeta_\infty^{-1}-] : \Gamma_{\beta,\delta'}\backslash \cX_H \to \Gamma_{\beta,\delta}\backslash \cX_H$ induced by 	$[h_\infty]_{\delta'} \mapsto [\zeta^{-1}_\infty h_\infty]_{\delta}.$
	\end{itemize}	 

\end{lemma}
\begin{proof} 
This is a simple explicit check.
\end{proof}

Combining Lemma \ref{lem:c_delta} with coinvariants, the composed map of local systems is induced by
\begin{align}\label{eq:triv beta delta}
	\mathrm{coinv}_{\beta,\delta} \circ c_{\delta}^* : H(\Q)\backslash \big[H(\Q)\delta L_\beta H_\infty^\circ \times M\big]/L_\beta L_\infty^\circ &\longrightarrow \Gamma_{\beta,\delta}\backslash\cX_H \times M_{\Gamma_{\beta,\delta}},\\
	(\zeta \delta \ell h_\infty , m) &\longmapsto ([h_\infty]_\delta,\  (\ell * m)_\delta).\notag
\end{align}

\begin{lemma}\label{lem:change of delta}
	Let $\delta'  \in H(\Q)\delta \ell H_\infty^\circ \cap H(\A_f)$ be another representative of $[\delta]$, with $\ell \in L_\beta$. Then for any class $\phi$, we have
	\[
	\ell *  \mathrm{Ev}_{\beta,\delta'}^M(\phi) = \mathrm{Ev}_{\beta,\delta}^M(\phi) \ \in M_{\Gamma_{\beta,\delta}}.
	\]
\end{lemma}
\begin{proof}
	Write $\delta' = \zeta\delta \ell h_\infty$ with $\zeta \in H(\Q), h_\infty \in H_\infty^\circ$. By Lemma \ref{lem:action on coinvariants}, we may define a map
	\begin{align*}
		([\zeta_\infty^{-1}-]\times [\ell*-]) : \Gamma_{\beta,\delta'}\backslash \cX_H \times M_{\Gamma_{\beta,\delta'}} &\longrightarrow \Gamma_{\beta,\delta}\backslash \cX_H \times M_{\Gamma_{\beta,\delta}}
	\end{align*}
	given by $([h_\infty]_{\delta'},(m)_{\delta'}) \mapsto ([\zeta^{-1}_\infty h_\infty]_{\delta}, (\ell * m)_\delta)$. We claim there is an equality of maps
	\begin{equation}\label{eq:zeta_infty}
		([\zeta_\infty^{-1} -]\times [\ell * -]) \circ[\mathrm{coinv}_{\beta,\delta'} \circ c_{\delta'}^*] = [\mathrm{coinv}_{\beta,\delta} \circ c_{\delta}^*].
	\end{equation}
	To see this, note that as $\delta$ and $\delta'$ are both trivial at infinity, $h_\infty = \zeta_\infty^{-1}$, so $[h_\infty h'_\infty]_\delta = [\zeta_\infty^{-1}-]([h_\infty']_{\delta'})$ for all $h_\infty' \in H_\infty^\circ$. Then \eqref{eq:zeta_infty} follows from commutativity of  
	\[
	\xymatrix@C=4mm{
		(\gamma'\delta'\ell' h_\infty',m)\ar@{|->}[rrrr]^-{\mathrm{coinv}_{\beta,\delta'} \circ c_{\delta'}^*}\ar@{|->}[d]^-{\mathrm{id}} &&&&
		([h_\infty']_{\delta'},(\ell' * m)_{\delta'})\ar@{|->}[d]^-{[\zeta^{-1}_\infty -]\times [\ell * -] } \sar{r}{\in} & 
		\Gamma_{\beta,\delta'}\backslash \cX_H \times M_{\Gamma_{\beta,\delta'}} \ar[d]^-{[\zeta^{-1}_\infty -]\times [\ell * -] }\\
		(\gamma'\zeta\delta\ell\ell' h_\infty h_\infty' ,m)\ar@{|->}[rrrr]^-{\mathrm{coinv}_{\beta,\delta} \circ c_{\delta}^*}  &&&&
		([h_\infty h_\infty']_\delta,(\ell\ell'*m)_\delta ) \sar{r}{\in} &
		\Gamma_{\beta,\delta}\backslash \cX_H \times M_{\Gamma_{\beta,\delta}}\\
	}	
	\]
	Note pullback by $[\zeta_\infty^{-1}-]$ induces an isomorphism 
	\disp{
	\h^{\mathrm{BM}}_t(\Gamma_{\beta,\delta}\backslash \cX_H, \Z) \isorightarrow \h^{\mathrm{BM}}_t(\Gamma_{\beta,\delta'}\backslash \cX_H, \Z)
}
	that by definition sends $\theta_{\delta}$ to $\theta_{\delta'}$. In particular, on cohomology we get a commutative diagram
	\begin{equation}\label{eq:fund class compat}
		\xymatrix@R=5mm{
			\hc{t}(\Gamma_{\beta,\delta'}\backslash \cX_H, \Z) \ar[rr]_-{\sim}^-{[\zeta^{-1}_\infty-]_*} \ar[rd]_-{-\cap\theta_{\delta'}}^-{\sim} && \hc{t}(\Gamma_{\beta,\delta}\backslash \cX_H, \Z)\ar[ld]^-{-\cap \theta_{\delta}}_-{\sim}\\
			& \Z &
		}.
	\end{equation}
	Then we compute that 
	\begin{align*}
		\ell * [\mathrm{Ev}_{\beta,\delta'}^M(\phi)] &=
		\ell * [(-\cap \theta_{\delta'}) \circ \mathrm{coinv}_{\beta,\delta'}\circ c_{\delta'}^* \circ \tau_\beta^\circ \circ \iota_\beta^*(\phi)]\\ &=
		(-\cap \theta_{\delta})\circ ([\zeta^{-1}_\infty-]_* \times [\ell * -]) \circ \mathrm{coinv}_{\beta,\delta'}\circ c_{\delta'}^* \circ \tau_\beta^\circ \circ \iota_\beta^*(\phi)]\\
		&= (-\cap \theta_{\delta})\circ \mathrm{coinv}_{\beta,\delta}\circ c_{\delta}^* \circ \tau_\beta^\circ \circ \iota_\beta^*(\phi) =  \mathrm{Ev}_{\beta,\delta}^M(\phi),
	\end{align*}
	where the second equality is \eqref{eq:fund class compat} and the third is \eqref{eq:zeta_infty} combined with \eqref{eq:triv beta delta}.
\end{proof}

\begin{proposition}\label{prop:ind of delta}
	Let $N$ be a left $H(\A)$-module, with action denoted $*$, such that $H(\Q)$ and $H_\infty^\circ$ act trivially. Let $\kappa : M\to N$ be a map of $L_\beta$-modules (with $N$ an $L_\beta$-module by restriction). Then
	\[
	\mathrm{Ev}_{\beta,[\delta]}^{M,\kappa} \defeq \delta * \left[\kappa \circ \mathrm{Ev}_{\beta,\delta}^M\right] : \hc{t}(S_K,\sM) \longrightarrow N
	\]
	is well-defined and independent of the representative $\delta$ of $[\delta]$.
\end{proposition}

\begin{proof}
	As $\Gamma_{\beta,\delta} \subset H(\Q)$ acts trivially on $N$, $\kappa$ factors through $M \twoheadrightarrow M_{\Gamma_{\beta,\delta}} \to N$, so $\kappa \circ \mathrm{Ev}_{\beta,\delta}^{M}$ (hence $\mathrm{Ev}_{\beta,[\delta]}^{M,\kappa}$) is well-defined.	If
	\disp{
	\delta' = \zeta\delta\ell h_\infty \in H(\Q)\delta\ell H_\infty^\circ \cap H(\A_f),
}
	then
	\begin{align*}
		\mathrm{Ev}_{\beta,[\delta']}^{M,\kappa} = \delta' * \left(\kappa \circ \mathrm{Ev}_{\beta,\delta'}^M\right)
		&= \delta\ell * \left(\kappa \circ \mathrm{Ev}_{\beta,\delta'}^M\right) \\
		&= \delta\ell * \left(\kappa \circ\left[\ell^{-1} * \mathrm{Ev}_{\beta,\delta}^M\right]\right) =  \delta * \left(\kappa \circ\mathrm{Ev}_{\beta,\delta}^M\right) = \mathrm{Ev}_{\beta,[\delta]}^{M,\kappa}.
	\end{align*}
	In the second equality we use that $\zeta$ and $h_\infty$ act trivially on $N$, and the third is Lemma~\ref{lem:change of delta}.
\end{proof}

\subsubsection{Variation in $\beta$}
We now investigate how evaluation maps behave as $\beta = (\beta_\mathfrak{q})_{\mathfrak{q}|p}$ varies.
Fix $\pri|p$, and define $\beta' = (\beta_\mathfrak{q}')_{\mathfrak{q}|p}$, where $\beta_{\pri}' = \beta_{\pri} + 1$ and $\beta_\mathfrak{q}' = \beta_\mathfrak{q}$ for $\mathfrak{q} \neq \pri$. We have a natural projection  
\disp{
\mathrm{pr}_{\beta,\pri} : X_{\beta'} \longrightarrow X_\beta,
}
inducing a projection 
\[
\mathrm{pr}_{\beta,\pri} : \pi_0(X_{\beta'}) \rightarrow \pi_0(X_\beta).
\]
Fix $\delta \in H(\A_f)$ and a set of representatives $D \subset H(\A_f)$ of the set $\mathrm{pr}_{\beta,\pri}^{-1}([\delta])\subset \pi_0(X_{\beta'})$. For each $\eta \in D$ there exists $\ell_\eta \in L_\beta$ such that $\eta \in H(\Q)\delta \ell_\eta H_\infty^\circ$. 
Via calculations directly analogous to those of Lemma~\ref{lem:action on coinvariants}, there is a map
\[
M_{\Gamma_{\beta', \eta}}\rightarrow M_{\Gamma_{\beta, \delta}}, \ \ (m)_\eta \mapsto \ell_{\eta}\ast (m)_\eta \defeq (\ell_\eta * m)_\delta.
\]
The action of $t_{\pri} \in \Delta_p$ yields an action of $U_{\pri}^\circ$ on $\hc{t}(S_K,\sM)$. 
Then we have the following direct generalisation of \cite[Thm.\ 2.2]{DJR18} to general coefficients (cf.\ \cite[Prop.~3.9]{BDJ17}):

\begin{proposition}\label{prop:evaluations changing beta} In the notation of the previous paragraph:
	\begin{itemize}
		\item[(i)] For each class $\Phi\in \hc{t}(S_K,\sM)$, we have 
		\disp{ \left[\mathrm{Ev}_{\beta,\delta}^{M} \circ U_\pri^\circ\right](\Phi)=	\sum_{\eta \in D} \ell_\eta\ast\mathrm{Ev}_{\beta',\eta}^{M}(\Phi).	
		}
		\item[(ii)] Let $N$ and $\kappa$ be as in Proposition~\ref{prop:ind of delta}. If $\beta_{\pri} \geqslant 1$, then as maps $\hc{t}(S_K,\sM) \to N$ we have
		\[
		\sum_{[\eta] \in \mathrm{pr}_{\beta,\pri}^{-1}([\delta])} \mathrm{Ev}_{\beta',[\eta]}^{M,\kappa}   =
		\mathrm{Ev}_{\beta,[\delta]}^{M,\kappa} \circ U_\pri^\circ.
		\]
	\end{itemize}
\end{proposition}

\begin{proof}
			We follow closely the proof of \cite[Proposition 3.4]{BDJ17}.  Recall $U_\pri^\circ$ from \S\ref{sec:hecke operators 1}, via $t_{\pri} = \iota(\varpi_{\pri}I_n,I_n) \in G(\A_f),$ and let $K_{0}(\pri)= K \cap t_{\pri}^{-1}Kt_{\pri}$ and $K^{0}(\pri)= t_{\pri}Kt_{\pri}^{-1}\cap K$. Denote the corresponding projections $\mathrm{pr}_1: S_{K_{0}(\pri)}\rightarrow S_K$ and $\mathrm{pr}_2: S_{K^{0}(\pri)}\rightarrow S_K.$ The action of $t_{\pri}$ on $M$ induces a morphism $[t_{\pri}]: \hc{t}(S_{K_0(\pri)},\sM) \rightarrow \hc{t}(S_{K^0(\pri)},\sM),$ 
		induced by the map $(g,m) \mapsto (gt_{\pri}^{-1}, t_{\pri} * m)$ on local systems, and then by definition we have
		\[
		U_\pri^\circ= \mathrm{Tr}(\mathrm{pr}_2)\circ [t_{\pri}] \circ \mathrm{pr}_1^{\ast}: \hc{t}(S_K,\sM) \rightarrow \hc{t}(S_K,\sM).
		\] 
		We give an analogue over automorphic cycles. Following the definition of $U_{\pri}^{\circ}$ we introduce maps
		\begin{align*}
			\iota_{\beta}^0 : X_{\beta'} &\longrightarrow S_{K^0(\pri)},\ \ \ \ \
			[h] \longmapsto [\iota(h)\xi t_p^{\beta}],\\
			\iota_{\beta', 0} : X_{\beta'} &\longrightarrow S_{K_0(\pri)},\ \ \ \ \
			[h] \longmapsto [\iota(h)\xi t_p^{\beta'}],
		\end{align*}
		which by definition fit into a commutative diagram
		\begin{equation}\label{eq:X Hecke 1}
			\xymatrix@C=15mm@R=7mm{
				S_K & S_{K^0(\pri)} \ar[l]^-{}&
				S_{K_0(\pri)}\ar[r]^{}\ar[l]_{\cdot \ t_{\pri}^{-1}}& S_K\\
				X_\beta \ar[u]^{\iota_{\beta}} &&
				X_{\beta'}\ar[u]^{\iota_{\beta', 0}}   \ar[ll]^-{\mathrm{pr}_{\beta,\pri}}\ar[ul]^{\iota_{\beta}^0}\ar[ru]_{\iota_{\beta'}}&.
			}
		\end{equation}
		Note that the left-hand quadrilateral is Cartesian.

		The action of $t_{\pri}$ on $M$ induces a morphism $\iota_{\beta', 0}^{\ast} \sM \rightarrow (\iota_{\beta}^0)^{\ast} \sM$  of sheaves over $X_{\beta'}$, giving
		\[
		[t_{\pri}]: \hc{t}(X_{\beta'}, \iota_{\beta', 0}^{\ast} \sM) \rightarrow \hc{t}(X_{\beta'}, (\iota_{\beta}^0)^{\ast} \sM). 
		\]
		Now define the analogue of $U_\pri^\circ$ on the cohomology of the automorphic cycles by
		\[ 
		U_\pri^\circ=  \mathrm{Tr}(\mathrm{pr}_{\beta,\pri})\circ [t_{\pri}]:  \hc{t}(X_{\beta'}, \iota_{\beta'}^{\ast} \sM)= \hc{t}(X_{\beta'}, \iota_{\beta', 0}^{\ast} \sM)  \rightarrow \hc{t}(X_{\beta}, \iota_{\beta}^{\ast} \sM) 
		\]
		From \eqref{eq:X Hecke 1}, the definition of $U_{\pri}^{\circ}$, and the fact that $\beta_{\pri}> 0$, we get another commutative diagram
		\[
		\xymatrix@C=15mm@R=7mm{
			\hc{t}(S_K,\sM) \ar[r]^{U_{\pri}^{\circ}}\ar[d]^{\iota_{\beta'}^\ast} & \hc{t}(S_K,\sM)  \ar[d]^{\iota_{\beta}^\ast}\\
			\hc{t}(X_{\beta'}, \iota_{\beta'}^{\ast} \sM) \ar[r]^{U_{\pri}^{\circ}}& \hc{t}(X_{\beta}, \iota_{\beta}^{\ast} \sM).
		}
		\]
		Tracing back each step of the construction of the evaluation maps, and using Lemma \ref{lem:change of delta} in the bottom square, we obtain the following commutative diagram, completing the proof of (i): 
		\[
		\xymatrix@C=17mm@R=7mm{
			\hc{t}(X_{\beta'}, \iota_{\beta'}^{\ast} \sM)\ar[d]^-{\tau^\circ_{\beta'}}\ar[r]^{U_{\pri}^{\circ}}& 	\hc{t}(X_{\beta}, \iota_{\beta}^{\ast} \sM)\ar[d]^-{\tau^\circ_{\beta}}\\
			\hc{t}(X_{\beta'}, \iota^{\ast} \sM)\ar[d]^-{\oplus_\eta c_{\eta}^\ast}\ar[r]^{\mathrm{Tr}(\mathrm{pr}_{\beta,\pri})} &
			\hc{t}(X_{\beta}, \iota^{\ast} \ar[d]^-{c_\delta^\ast}\sM)\\
			\oplus_{\eta\in D}\hc{t}(X_{\beta'}[\eta], c_{\eta}^\ast\iota^{\ast} \sM) \ar[r]^{\mathrm{Tr}(\mathrm{pr}_{\beta,\pri})}\ar[d]_-{\oplus_\eta (-\cap\theta_\eta) \circ \mathrm{coinv}_{\beta', \eta}}&
			\hc{t}(X_{\beta}[\delta], c_{\delta}^\ast\iota^{\ast} \sM) \ar[d]^-{(-\cap\theta_\delta) \circ \mathrm{coinv}_{\beta, \delta}}\\
			\oplus_{\eta \in D} M_{\Gamma_{\beta, \eta}}\ar[r]^{\sum_{\eta \in D}(\ell_{\eta}\ast -)} &
			M_{\Gamma_{\beta, \delta}}.
		}
		\]
		Finally (ii) follows from (i) directly following the proof of Proposition~\ref{prop:ind of delta}.
\end{proof}

%%==========================================================
%%
%%			Evaluation maps from DJR
%%
%%==========================================================
\section{Classical evaluation maps and $L$-values}\label{sec:classical evaluations}

Let $K \subset G(\A_f)$ be an open compact  subgroup as in \S\ref{sec:abstract evaluations}. 
Now we take $M = V_\lambda^\vee(L)$, with $\Delta_p$ acting via the $*$-action defined in \S\ref{sec:slope-decomp}. We now rephrase the classical evaluation maps 
\[
\cE_{\beta,\delta}^{j,\sw} : \htc(S_K,\sV_{\lambda}^\vee(L)) \rightarrow L \subset \overline{\Q}_p
\]
of \cite[\S2.2]{DJR18} in the language of \S\ref{sec:abstract evaluation maps}. We give two main applications of these classical evaluation maps: firstly, they provide a criterion for the existence of a Shalika model (Proposition~\ref{prop:shalika non-vanishing}); and when such a model exists and $K  = K(\tilde\pi)$, they compute classical $L$-values (Theorem~\ref{thm:critical value}).

We use an opposite convention to \cite{DJR18}. They take $\pi$ to have weight $\lambda^\vee$ and use coefficients in $V_\lambda$. Our choices mean we replace $\sw$ from \emph{ibid}.\ with $-\sw$, and $\mu^\vee$ \emph{ibid}.\ with $\lambda$.

\subsection{Classical evaluation maps}\label{sec:classical evaluations 5.1}
We recap \cite[\S2.2]{DJR18}. The \emph{$p$-adic cyclotomic character} is
\begin{equation}\label{eq:cyc}
	\chi_{\cyc} : F^\times\backslash \A_F^\times/F_\infty^+ \longrightarrow \Zp^\times,  \qquad
	y \longmapsto \textstyle\prod_{\sigma \in \Sigma}\mathrm{sgn}(y_\sigma) \cdot |y_f| \textstyle\cdot \prod_{\pri|p} N_{F_{\pri}/\Q_p}(y_{\pri}).
\end{equation}	

\begin{definition}\label{def:action on Vjw}
	For $(j_1,j_2) \in \Z^2$, let $V^H_{(j_1,j_2)}$ be a 1-dimensional $L$-vector space with $H(\A)$-action
	\[
	(h_1,h_2) * v \defeq \chi_{\cyc}\big[\det(h_1)^{j_1}\det(h_2)^{j_2}\big]v, 
	\]
	for $h_1,h_2 \in \GL_n(\A_F)$ and $v \in V^H_{(j_1,j_2)}$. This is the set of $L$-points of the algebraic representation of $H$ of highest weight $(j_1,...,j_1,j_2,...,j_2)$. 
	Note that $H(\Q)$ and $H_\infty^\circ$ act trivially on $V^H_{(j_1,j_2)}$. For $(\ell_1,\ell_2) \in L_\beta$, as $|\det(\ell_i)_f| = 1$, we have
	\disp{
	(\ell_1,\ell_2) * v \defeq \mathrm{N}_{F/\Q}\big[\det(\ell_{1,p})^{j_1}\det(\ell_{2,p})^{j_2}\big]v.
	}
\end{definition}

The following branching law for $H \subset G$ gives a representation-theoretic description of $\mathrm{Crit}(\lambda)$.

\begin{lemma}\label{lem:hom line}
	Let $j \in \Z$. We have $j \in \mathrm{Crit}(\lambda)$ if and only if 
	\disp{
	\dim_L(\Hom_{H(\Zp)}(V_\lambda^\vee,V^H_{(j,-\sw-j)})) = 1.
}
\end{lemma}
\begin{proof}
	By \cite[Prop.~6.3]{GR2}, we know $0 \in \mathrm{Crit}(\lambda)$ if and only if 
	\disp{
	\dim_L(\Hom_{H(\Zp)}(V_\lambda^\vee, V^H_{(0,-\sw)})) = 1.
}
	Note $L(\pi,j+\tfrac{1}{2})$ is critical if and only if $L(\pi\otimes |\cdot|^j,\tfrac{1}{2})$ is critical. Let 
	\disp{
	\tilde\lambda = \lambda + j(1,...,1),
} of purity weight $\sw + 2j$; then $j \in \mathrm{Crit}(\lambda)$ if and only if $0 \in \mathrm{Crit}(\tilde\lambda)$, and in this case 
		\[
			1 = \dim_L(\Hom_{H(\Zp)}(V_{\tilde\lambda}^\vee, V^H_{(0,-\sw-2j)}) = \dim_L(\Hom_{H(\Zp)}(V_\lambda^\vee,V^H_{(j,-\sw-j)})).\qedhere
	\]
\end{proof}	
Recall  the map
\disp{
\tau_\beta^\circ\circ \iota_\beta^* : \hc{t}(S_K,\sV_{\lambda}^\vee) \to \hc{t}(X_\beta, \iota^*\sV_{\lambda}^\vee)
}
from \S\ref{sec:abstract evaluations}. For $j \in \mathrm{Crit}(\lambda)$, fix a basis $\kappaj$ of $\Hom_{H(\Zp)}(V_\lambda^\vee,V^H_{(j,-\sw-j)})$. This induces a homomorphism
\[
\kappaj : \htc(X_\beta, \iota^*\sV_\lambda^\vee) \longrightarrow \htc(X_\beta,\sV^H_{(j,-\sw-j)}),
\]
where $\sV^H_{(j,-\sw-j)}$ is the local system defined as in \S\ref{sec:non-arch ls}. Let $\delta \in H(\A_f)$. As in \S\ref{sec:abstract evaluations}, applying $(-\cap \theta_{\delta})\circ \mathrm{coinv}_{\beta,\delta} \circ c_\delta^*$ and choosing a basis $u_j$ of $V^H_{(j,-\sw-j)}$ gives a map
	\[
		\htc(X_\beta, \sV^H_{(j,-\sw-j)}) \xrightarrow{\ \mathrm{coinv}_{\beta,\delta} \circ c_\delta^* \ } \htc(\Gamma_{\beta,\delta}\backslash \cX_H,\Z)\otimes V^H_{(j,-\sw-j)} \labelisorightarrow{\ (- \cap \theta_\delta) \otimes \mathrm{id} \ } V^H_{(j,-\sw-j)} \cong L.
	\]
Then in \cite[(33)]{DJR18}, the authors define 
\disp{
\cE_{\beta,\delta}^{j,\sw} \defeq  (-\cap \theta_{\delta})\circ \mathrm{coinv}_{\beta,\delta} \circ c_\delta^* \circ (\kappaj)_*\circ \tau_\beta^\circ \circ \iota_\beta^*.
}
The choice of basis $u_j$ of $V^H_{(j,-\sw-j)}$ identifies $V^H_{(j,-\sw-j)}$ with $L$, and we get a map $\kappaj^\circ$ of $H(\zp)$-modules defined via
\begin{align}\label{eq:kappa circ}
	\kappaj^\circ : V_\lambda^\vee(L) &\longrightarrow L,\qquad \kappaj(\mu) = \kappaj^\circ(\mu)\cdot u_j \ \ \text{for all} \ \mu \in V_\lambda^\vee(L). 
\end{align}
As $\Gamma_{\beta,\delta}$ acts trivially on $V^H_{(j,-\sw-j)}$, $\kappaj$ and $\kappaj^\circ$ factor through $(V_\lambda^\vee(L))_{\Gamma_{\beta,\delta}}$. It is easy to see that $\kappaj$ commutes with restricting to components, passing to coinvariants, and integrating against the fundamental class. We deduce the following description of $\cE_{\beta,\delta}^{j,\sw}$ via \S\ref{sec:abstract evaluations}:

\begin{lemma}\label{lem:abstract reformulation}
	We have $\cE_{\beta,\delta}^{j,\sw} = \kappaj^\circ \circ \mathrm{Ev}_{\beta,\delta}^{V_\lambda^\vee}$.
\end{lemma}

	Let $\cE_{\beta,[\delta]}^{j,\sw} \defeq \delta * \cE_{\beta,\delta}^{j,\sw}$; by Propositions~\ref{prop:ind of delta} and \ref{prop:evaluations changing beta} this is independent of $\delta$ (cf.\ \cite[(33)]{DJR18}).

%%=====================================
%%		 L-VALUE FORMULA
%%=====================================

Recall 
\disp{
\pi_0(X_\beta) = \Cl(p^\beta\m)\times \Cl(\m)
}
from \eqref{eq:component group}. Write $\mathrm{pr}_1, \mathrm{pr}_2$ for the projections of $\pi_0(X_\beta)$ onto the first and second factors respectively, and let $\mathrm{pr}_{\beta}$ denote the natural composition
\begin{equation}\label{eq:pr_beta}
	\mathrm{pr}_{\beta} : \Cl(p^\beta\m) \times \Cl(\m) \xrightarrow{ \ \mathrm{pr}_1 \ } \Cl(p^\beta\m) \longrightarrow \Cl(p^\beta).
\end{equation}

\begin{definition}\label{def:ev chi}
	Let $\eta_0$ be any finite order character of $\Cl(\m)$, and $\mathbf{x}\in \Cl(p^\beta)$. Define an $\eta_0$-averaged evaluation map
			\[
		\cE_{\beta,\mathbf{x}}^{j,\eta_0} : \hc{t}(S_K,\sV_\lambda^\vee(L)) \to L, \qquad 	\cE_{\beta,\mathbf{x}}^{j,\eta_0} \defeq \sum_{[\delta] \in \mathrm{pr}_\beta^{-1}(\mathbf{x})} \eta_0^{-1}\big(\mathrm{pr}_2([\delta])\big)\ \cE_{\beta,[\delta]}^{j,\sw}.
		\]
	In \cite{DJR18} this is denoted $\cE_{\beta,\mathbf{x}}^{j,\eta}$, where $\eta = \eta_0|\cdot|^{\sw}$; as later $\sw$ will vary whilst $\eta_0$ will not, we continue to use a superscript $\eta_0$ instead of $\eta$ throughout, with $\sw$ implicit in the source.
	
	Let $\chi$ be a finite order Hecke character of conductor (exactly) $p^{\beta'}$, for $\beta' = (\beta_{\pri}')_{\pri|p}$. Let $\beta_{\pri} = \mathrm{max}(\beta_{\pri}', 1)$ and $\beta = (\beta_{\pri})_{\pri|p}$. Then $\chi$ induces a character on $\Cl(p^\beta)$. Let $L(\chi)$ be the smallest extension of $L$ containing $\chi(\Cl(p^\beta))$. For $j\in \mathrm{Crit}(\lambda)$, define 
	\begin{align}\label{eq:classical evaluation}
		\cE_{\chi}^{j,\eta_0} = \sum_{\mathbf{x} \in \Cl(p^\beta)}& \chi(\mathbf{x})\ \cE_{\beta,\mathbf{x}}^{j,\eta_0}\ :\ \hc{t}(S_K, \sV_\lambda^\vee(L)) \longrightarrow L(\chi),\\
		\phi &\longmapsto \sum_{[\delta] \in \pi_0(X_\beta)} \chi\big(\mathrm{pr}_\beta([\delta])\big)\cdot  \eta_0^{-1}\big(\mathrm{pr}_2([\delta])\big)\cdot\left( \delta*\left[\kappaj^\circ \circ \mathrm{Ev}_{\beta,\delta}^{V_{\lambda}^\vee}(\phi)\right]\right).\notag
	\end{align}
\end{definition}

\begin{remark}\label{rem:classical diagram}
	Summarising,  $\cE_{\chi}^{j,\eta_0}$ is the composition
	\begin{equation}\label{eq:explicit classical}
		\xymatrix@R=10mm@C=5mm{
			\hc{t}(S_K,\sV_\lambda^\vee) \ar@/^3pc/[rrrrr]_-{\oplus \cE_{\beta,[\delta]}^{j,\sw}} \ar@/_3pc/[rrrrrrr]^-{\oplus \cE_{\beta,\mathbf{x}}^{j,\eta_0}} \ar[rr]^-{\oplus\mathrm{Ev}_{\beta,\delta}^{V_\lambda^\vee}} && \displaystyle\bigoplus_{[\delta]}(V_\lambda^\vee)_{\Gamma_{\beta,\delta}} \ar[rrr]^-{v \mapsto \delta * \kappaj^\circ(v)} &&&
			\displaystyle\bigoplus_{[\delta]} L \ar[rr]^-{\oplus\Xi_{\mathbf{x}}^{\eta_0}} && 
			\displaystyle\bigoplus_{\mathbf{x}} L \ar[rrr]^-{\ell \mapsto \Sigma \chi(\mathbf{x})\ell_\mathbf{x}} &&&
			L,
		}
	\end{equation}
	where the sums are over $[\delta] \in \pi_0(X_\beta)$ or $\mathbf{x}\in \Cl(p^\beta)$, and $\Xi_{\mathbf{x}}^{\eta_0}$ is the $\eta_0$-averaging map 
	\[
	\Xi_{\mathbf{x}}^{\eta_0} : (m_{[\delta]})_{[\delta]} \longmapsto \sum_{[\delta] \in \mathrm{pr}_{\beta}^{-1}(\mathbf{x})} \eta_0^{-1}(\mathrm{pr}_2([\delta])) \cdot m_{[\delta]}.
	\]
\end{remark}

\subsection{Compatible choices of bases: branching laws for $H \subset G$}\label{sec:choice of basis}

Let $j \in \mathrm{Crit}(\lambda)$. The map $\cE_{\chi}^{j,\eta_0}$ depends on choices of bases 
\[
u_j \text{ of }V^H_{(j,-\sw-j)} \cong L \qquad \text{and} \qquad\kappaj \text{ of }\mathrm{Hom}_{H(\Zp)}(V_\lambda^\vee(L), V^H_{(j,-\sw-j)}),
\]
which we combined into a single choice of non-zero $\kappaj^\circ$ in \eqref{eq:kappa circ}. At present, we have made a separate, independent choice for each $j$. For $p$-adic interpolation it is essential to make all these choices compatibly. We now do this via branching laws.

\subsubsection{Idea: critical integers via branching laws}
Dualising Lemma \ref{lem:hom line} gives a reinterpretation of the set $\mathrm{Crit}(\lambda)$ in terms of \emph{branching laws} for $H \subset G$, describing characters of $H$ that appear in $V_{\lambda}|_H$ with multiplicity 1. For each $j \in \mathrm{Crit}(\lambda)$, we obtain a line $V^H_{(-j,\sw+j)} \subset V_\lambda|_H$. Our key idea for $p$-adic interpolation is to reinterpret this again in terms of smaller groups; instead of considering branching laws for $H \subset G$, one can consider branching laws for $G_n = \mathrm{Res}_{\cO_F/\Z}\GL_n \subset H$, embedded diagonally. Indeed, recall $\lambda$ is pure with purity weight $\sw$, and $V_\lambda^H$ is the irreducible representation of $H$ of highest weight $\lambda$; then as $G_n$-representations we have
	\begin{align}\label{eq:V_lambda^H}
		V_\lambda^H|_{G_n} \cong V_{\lambda'}^{G_n}\otimes (V_{\lambda'}^{G_n})^\vee \otimes (\mathrm{N}_{F/\Q}\circ\det)^{\sw},
	\end{align}
where $\lambda' = (\lambda_1,\dots,\lambda_n)$. As $V_{\lambda'}^{G_n} \otimes (V_{\lambda'}^{G_n})^\vee$ contains the trivial representation with multiplicity 1, $V_\lambda^H|_{G_n}$ contains $(\mathrm{N}_{F/\Q}\circ\det)^{\sw}$ with multiplicity 1. In Notation \ref{not:v_lambda 1} and Lemmas \ref{lem:v_lambda non-vanishing} and \ref{lem:diagonal action}, we show that the $\#\mathrm{Crit}(\lambda)$ \emph{different} lines $V^H_{(-j,\sw+j)}$ in $V_\lambda|_{H}$ (given by Lemma \ref{lem:hom line}) can all be collapsed onto this \emph{single} line in $V_\lambda^H|_{G_n}$. Choosing a generator of this single line thus allows us to align generators of the distinct lines $V^H_{(-j,\sw+j)}$ for $j \in \mathrm{Crit}(\lambda)$.

\subsubsection{Passing from $H \subset G$ to $G_n \subset H$}\label{sec:G to H}

Let $j \in \mathrm{Crit}(\lambda)$, and 
\disp{
\kappaj \in \mathrm{Hom}_{H(\Zp)}(V_\lambda^\vee(L), V^H_{(j,-\sw-j)})
}
and 
\disp{
u_j \in V^H_{(j,-\sw-j)}
}
be auxiliary bases.
We have a dual basis 
\disp{
u_j^\vee \text{ of }V^H_{(-j,\sw+j)} \cong (V^H_{(j,-\sw-j)})^\vee.
}
Dualising $\kappaj$ gives a map
\[
\kappaj^\vee : V^H_{(-j,\sw+j)} \longrightarrow (V_\lambda^\vee(L))^\vee \cong V_\lambda(L)
\]
of $H(\Zp)$-modules. Then $\kappaj^\vee(u_j^\vee) \in V_\lambda(L)$ generates the unique $H(\zp)$-submodule isomorphic to $V^H_{(-j,\sw+j)}$ inside $V_\lambda(L)|_{H(\Zp)}$.

\begin{notation}\label{not:v_lambda 1}
	Viewing $\kappaj^\vee(u_j^\vee) \in V_\lambda(L)$ as an element of $\mathrm{Ind}_{Q^-(\Zp)}^{G(\Zp)} V_\lambda^H(L)$ by Lemma \ref{lem:induction transitive}, let 
	\[
	\sv_{\lambda,j} \defeq \kappaj^\vee(u_j^\vee)\left[\smallmatrd{I_n}{I_n}{0}{I_n}\right] \in V_\lambda^H(L).
	\]
\end{notation}

Let 
\disp{
N_Q^\times(\Zp) \defeq \left\{ \smallmatrd{1}{X}{0}{1} \in N_Q(\Zp) : X \in G_n(\Zp)\right\} \subset N_Q(\Zp).
}
\begin{lemma}\label{lem:v_lambda non-vanishing}\begin{itemize}\setlength{\itemsep}{0pt}
		\item[(i)] For each $\smallmatrd{1}{X}{0}{1} \in N_Q^\times(\Zp)$, we have
		\[
		\kappaj^\vee(u_j^\vee)\left[\smallmatrd{1}{X}{0}{1}\right] = \big[\mathrm{N}_{F/\Q}\circ \det(X)]^j  \bigg(\left\langle \smallmatrd{X}{}{}{1}\right\rangle_\lambda \cdot\sv_{\lambda,j}\bigg).
		\] 
		\item[(ii)] The vector $\sv_{\lambda,j} \in V_\lambda^H(L)$ is non-zero.
	\end{itemize}
\end{lemma}
\begin{proof} 
	(i) For $\smallmatrd{1}{X}{0}{1} \in N_Q^\times(\Zp)$, we have
		\begin{align*}
			\kappaj^\vee(u_j^\vee)\left[\smallmatrd{1}{X}{0}{1}\right] &= 	\kappaj^\vee(u_j^\vee)\left[\smallmatrd{X}{}{}{1}\smallmatrd{1}{1}{0}{1}\smallmatrd{X^{-1}}{}{}{1}\right] = 	\left\langle\smallmatrd{X}{}{}{1}\right\rangle_\lambda \cdot \bigg(\kappaj^\vee(u_j^\vee)\big[\smallmatrd{1}{1}{0}{1}\smallmatrd{X^{-1}}{}{}{1}\big]\bigg),
		\end{align*}
	where the last equality follows by \eqref{eq:parahoric alg transform}. Moreover, $\smallmatrd{X^{-1}}{}{}{1} \in H(\Zp)\subset G(\Zp)$ acts on $\kappaj^\vee(u_j^\vee)$ by right translation, and $\kappaj^\vee$ is $H(\Zp)$-equivariant, whence we see
			\begin{align*}
			\kappaj^\vee(u_j^\vee)\left[\smallmatrd{1}{1}{0}{1}\smallmatrd{X^{-1}}{}{}{1}\right]	&= \bigg(\smallmatrd{X^{-1}}{}{}{1} \cdot \kappaj^\vee(u_j^\vee)\bigg)\left[\smallmatrd{1}{1}{0}{1}\right] = \kappaj^\vee\big(\smallmatrd{X^{-1}}{}{}{1}\cdot u_j^\vee\big)\left[\smallmatrd{1}{1}{0}{1}\right]\\
			& = (\mathrm{N}_{F/\Q}\circ\det(X))^{j} \cdot \kappaj^\vee(u_j^\vee)\left[\smallmatrd{1}{1}{0}{1}\right] = (\mathrm{N}_{F/\Q}\circ\det(X))^{j} \cdot \sv_{\lambda,j},
		\end{align*}
	using that $u_j^\vee \in V^H_{(-j,\sw+j)}$. Combining these equalities proves (i).
	
	(ii) Suppose $\sv_{\lambda,j} = 0$. By (i), we see 
	\disp{
	\kappaj^\vee(u_j^\vee)|_{N_Q^\times(\Zp)} = 0.
}
	Since $N_Q^\times(\Zp)$ is Zariski-dense in $N_Q(\Zp)$, we deduce $\kappaj^\vee(u_j^\vee)$ vanishes on $N_Q(\Zp)$, hence on $J_p$ by the parahoric decomposition; but by Zariski-density of $J_p \subset G(\Zp)$ this forces $\kappaj^\vee(u_j^\vee) = 0$. This is absurd by its definition.
\end{proof}

Alternative choices of $\kappaj$ or $u_j$ scale $\sv_{\lambda,j}$ by $L^\times$-multiple. As $\sv_{\lambda,j}$ is non-zero, we see choosing $\kappaj$ and $u_j$ is equivalent to fixing a basis of the line $L \cdot \sv_{\lambda,j}$. This line is independent of $j$:

\begin{lemma}\label{lem:diagonal action}
	\begin{itemize}\setlength{\itemsep}{0pt}
		\item[(i)] Let $h \in G_n(\Zp)$. Then 
		\disp{
		\left\langle\smallmatrd{h}{}{}{h}\right\rangle_\lambda \cdot  \sv_{\lambda,j} = (\mathrm{N}_{F/\Q}\circ\det(h))^{\sw}\ \sv_{\lambda,j}.
	}
		\item[(ii)] The line $L \cdot \sv_{\lambda,j} \subset V_\lambda^H(L)$ is independent of $j$.
	\end{itemize}
\end{lemma}
\begin{proof}
	(i) By definition: $\langle \cdot \rangle_\lambda$ acts on $\kappaj^\vee(u_j^\vee)$ by left translation; the $\cdot$-action of $H(\Zp)$ on $\kappaj^\vee(u_j^\vee)$ is by right translation; and $(h_1,h_2) \in H(\Zp)$ acts on $u_j^\vee$ by $\mathrm{N}_{F/\Q}(\det(h_1)^{-j}\det(h_2)^{\sw+j})$. Then
	\begin{align*}	
		\left\langle\smallmatrd{h}{}{}{h}\right\rangle_\lambda \cdot \sv_{\lambda,j} &= \kappaj^\vee(u_j^\vee)\left[\smallmatrd{h}{}{}{h}\smallmatrd{1}{1}{0}{1}\right]=
		\kappaj^\vee(u_j^\vee)\left[\smallmatrd{1}{1}{0}{1}\smallmatrd{h}{}{}{h}\right] \\
		&=
		\big(\smallmatrd{h}{}{}{h} \cdot \kappaj^\vee(u_j^\vee)\big)\left[\smallmatrd{1}{1}{0}{1}\right] =	\kappaj^\vee\big(\smallmatrd{h}{}{}{h} \cdot u_j^\vee\big)\left[\smallmatrd{1}{1}{0}{1}\right] \\
		&=
		(\mathrm{N}_{F/\Q}\circ \det(h))^{\sw}  \kappaj^\vee(u_j^\vee)\left[\smallmatrd{1}{1}{0}{1}\right] =
		(\mathrm{N}_{F/\Q}\circ \det(h))^{\sw} \sv_{\lambda,j}. 
	\end{align*}
	In the first equality, we use 	 \eqref{eq:parahoric alg transform} for $\kappa_{\lambda,j}^\vee(u_j^\vee) \in V_\lambda$ (via Lemma \ref{lem:induction transitive}).
	
	(ii) As after \eqref{eq:V_lambda^H}, the restriction $V_\lambda^H|_{G_n}$ contains $(\mathrm{N}_{F/\Q}\circ\det)^{\sw}$ as a unique summand. This summand visibly has no dependence on $j$, but by (i), for each $j$ it coincides with $L \cdot \sv_{\lambda,j}$.
\end{proof}

Thus evaluation at $\smallmatrd{1}{1}{0}{1}$ collapses all the lines $V^H_{(-j,\sw+j)} \subset V_\lambda|_H$ onto the \emph{same} line in $V_\lambda^H|_{G_n}$.

\subsubsection{From $G_n\subset H$ back to $H\subset G$}\label{sec:from G_n to H}
We now use \S\ref{sec:G to H} to align our initial choices of $\kappa_{\lambda,j}^\circ$.

\begin{notation}\label{not:v_lambda}
	Fix a generator $\sv_\lambda$ of $(\mathrm{N}_{F/\Q}\circ\det)^{\sw} \subset V_\lambda^H|_{G_n}$. We take 
	\[
	\sv_\lambda \in V_\lambda^H(\cO_L)
	\]
	optimally integrally normalised (in the sense that $\varpi_L^{-1}\sv_\lambda \notin V_\lambda^H(\cO_L)$).
\end{notation}

\begin{definition}\label{def:kappa_j} 
	Using Lemma \ref{lem:diagonal action}(ii), rescale $\kappaj$ and $u_j$ so that 
	\[
	\sv_{\lambda,j} = (-1)^{dnj} \sv_\lambda.
	\]
	Then let $\kappaj^\circ: V_\lambda^\vee(L) \longrightarrow L$ be the map determined by the property \eqref{eq:kappa circ}.
\end{definition}
From the definitions, and using duality, we can describe $\kappaj^\circ$ as the map
\begin{align}\label{eq:explicit kappa}
	\kappaj^\circ : V_\lambda^\vee(L) &\longrightarrow L,\qquad \mu \longmapsto \mu[\kappaj^\vee(u_j^\vee)].
\end{align}	
It is easy to see $\kappaj^\circ$ is uniquely determined by these properties and the (single) choice of $\sv_\lambda$
We now give an alternative description of $\kappaj^\circ$ better suited to $p$-adic interpolation.

\begin{lemma}\label{lem:nu_j}\begin{itemize}\setlength{\itemsep}{0pt}
		\item[(i)] For each $j$, there exists a unique 
		\disp{
		[\nuj : G(\Zp) \to V_\lambda^H(L)] \in V_\lambda(L)
	}
		with 
		\begin{equation}\label{eq:xi_j}
			\nuj\left[\smallmatrd{1}{X}{0}{1}\right] = (-1)^{dnj} \big[\mathrm{N}_{F/\Q}\circ \det(X)]^j\bigg( \left\langle \smallmatrd{X}{}{}{1}\right\rangle_\lambda \cdot   \sv_\lambda\bigg)
		\end{equation}
		for $\smallmatrd{1}{X}{0}{1} \in N_Q^\times(\Zp)$. 
		\item[(ii)] For $(h_1,h_2) \in H(\Zp)$, we have 
		\disp{
		\smallmatrd{h_1}{}{}{h_2} \cdot \nuj = \mathrm{N}_{F/\Q}[\det(h_1)^{-j}\det(h_2)^{\sw+j}]\nuj.
	}
		\item[(iii)] The map 
		\disp{
		\kappaj^\circ : V_\lambda^\vee(L) \longrightarrow L
	}
		from Definition \ref{def:kappa_j} is 
		given by $\mu \mapsto \mu(\nuj).$
	\end{itemize}
\end{lemma}
\begin{proof}
	(i) We take 
	\disp{
	\nuj \defeq \kappaj^\vee(u_j^\vee).
}
	Then \eqref{eq:xi_j} is exactly Lemma \ref{lem:v_lambda non-vanishing}(i). Note the values of $\nuj$ on $N_Q^-(\Zp)H(\Zp)N_Q^\times(\Zp)$ are determined by \eqref{eq:xi_j} and the transformation property of $\mathrm{Ind}_{Q^-(\Zp)}^{G(\Zp)} V_\lambda^H(L)$; and this is Zariski-dense in $G(\Zp)$. Hence $\nuj$ is unique with this property.
	
	(ii) Since $\kappaj^\vee$ is $H(\Zp)$-equivariant and $u_j^\vee \in V^H_{(-j,\sw+j)}$, we compute that 
			\begin{align*}
			\smallmatrd{h_1}{}{}{h_2} \cdot \nuj &= \kappaj^\vee\left(\smallmatrd{h_1}{}{}{h_2}\cdot u_j^\vee\right) = \mathrm{N}_{F/\Q}[\det(h_1)^{-j}\det(h_2)^{\sw+j}]\nuj.
		\end{align*}	
	
	(iii) This follows directly from \eqref{eq:explicit kappa}. 
\end{proof}

\subsubsection{Comparison with previous work}

In \cite[(40)]{DJR18}, the authors choose a lowest weight vector $v_0 \in V_\lambda^\vee(L)$, and use this choice and Lie theory to define an integral lattice 
\disp{
V_\lambda^\vee(\cO_L)^{\mathrm{DJR}} \subset V_\lambda^\vee(L)
}
(which may be different from the lattice $V_\lambda(\cO_L)$ defined in \S\ref{sec:algebraic weights}). For $j \in \mathrm{Crit}(\lambda)$, they construct a map 
\[
\kappa_j^{\mathrm{DJR}} : V_\lambda^\vee(\cO_L)^{\mathrm{DJR}} \to V^H_{(j,-\sw-j)}(\cO_L) \cong \cO_L,
\]
normalised so that $\kappa_j^{\mathrm{DJR}}(\xi \cdot v_0) = 1$ (which they prove is possible in results analogous to \S\ref{sec:G to H}). This map is denoted $\kappa_j^\circ$ \emph{ibid}. 

We freely identify $\kappa_j^{\mathrm{DJR}}$ with its scalar extension $V_\lambda^\vee(L) \to L$. By Lemma \ref{lem:hom line}, for each $j \in \mathrm{Crit}(\lambda)$ the maps $\kappaj^\circ$ and $\kappa_j^{\mathrm{DJR}}$ agree up to $L^\times$-multiple. Fix $j_0 \in \mathrm{Crit}(\lambda)$; we can align the choice of $v_0$ (and hence the integral structure $V_\lambda^\vee(\cO_L)^{\mathrm{DJR}}$) in \cite{DJR18} so that $\kappajn^\circ = \kappa_{j_0}^\DJR$. Then:

\begin{proposition}\label{prop:comp to DJR}
	For each $j \in \mathrm{Crit}(\lambda)$, we have $\kappaj^\circ = \kappa_j^\DJR$. 
\end{proposition}

\begin{proof}
	Dualising the map $\kappa_j^{\mathrm{DJR}}$, and evaluating at $1 \in \cO_L$, one obtains an element 	$v_j^{\mathrm{DJR}} \in V_\lambda(\cO_L)^\DJR$ such that $\kappa_j^\DJR(\mu) = \mu(v_j^{\mathrm{DJR}}).$ 	Moreover 
	\[
	v_j^{\mathrm{DJR}} \in V^H_{(-j,\sw+j)} \subset V_{\lambda}|_H,
	\]
	so $v_j^{\mathrm{DJR}}$ is an $L^\times$-multiple of $v_{\lambda,j}$ from Lemma \ref{lem:nu_j}. In particular, there exists $c_j \in \cO_L$ such that either $v_{\lambda,j} = c_j v_{j}^{\mathrm{DJR}}$ or $c_j v_{\lambda,j} = v_j^{\mathrm{DJR}}$. We assume the latter; the proof is identical for the former. By the above and Lemma \ref{lem:nu_j}(iii), it suffices to prove that $c_j = 1$ for each $j \in \mathrm{Crit}(\lambda)$. By assumption $c_{j_0} = 1$.
	
	By \cite[Prop.\ 2.6]{DJR18}, for all $j \in \mathrm{Crit}(\lambda)$, $\mu \in V_\lambda^\vee(\cO_L)$, and $\beta \in \Z_{\geqslant 1}$, we have
	\[
	\mu[(\xi^{-1} t_p^\beta) * v_j^{\mathrm{DJR}}] \equiv \mu[(\xi^{-1} t_p^\beta) * v_{j_0}^{\mathrm{DJR}}] \newmod{p^\beta \cO_L}.
	\]
	As this holds for all $\mu$, by considering $\cO_L$-bases we deduce 
	\[
	(\xi^{-1} t_p^\beta) * \big[v_j^{\mathrm{DJR}} - v_{j_0}^{\mathrm{DJR}}\big] \in p^\beta V_\lambda(\cO_L)^\DJR.
	\]
	Any two integral lattices in $V_\lambda(L)$ are commensurable, so there exists $\beta_0 \in \Z_{\geqslant 0}$ such that
	\[
	(\xi^{-1} t_p^\beta) * \big[v_j^{\mathrm{DJR}} - v_{j_0}^{\mathrm{DJR}}\big] \in p^{\beta-\beta_0} V_\lambda(\cO_L),
	\]
	for all $\beta \geqslant \beta_0$, and in particular, our normalisations ensure we have
	\begin{equation}\label{eq:c_j 1}
		(\xi^{-1} t_p^\beta) * \big[c_jv_{\lambda,j} - v_{\lambda,j_0}\big] \in p^{\beta-\beta_0} V_\lambda(\cO_L).
	\end{equation}
	Thus, considering this element in $\mathrm{Ind}_{Q^-(\Zp)}^{G(\Zp)}V_\lambda^H(L)$ via Lemma \ref{lem:induction transitive}, for all $g \in G(\Zp)$ we have
	\begin{equation}\label{eq:c_j 2}
		(\xi^{-1} t_p^\beta) * \big[c_jv_{\lambda,j} - v_{\lambda,j_0}\big](g) \in p^{\beta-\beta_0} V_\lambda^H(\cO_L).
	\end{equation}

	Recall $\sv_{\lambda} \in L[H]$ (from Notation \ref{not:v_lambda}) is polynomial in the coordinates of $H$; after possibly enlarging $\beta_0$, we may assume that $\varpi_L^{\beta_0}\sv_\lambda \in \cO_L[H]$, that is, the coefficients are all integral. As the action $\langle\cdot\rangle_\lambda$ on $\sv_\lambda$ is by right-translation, one deduces easily that if $h,h' \in H(\Zp)$ with $h \equiv h' \newmod{p^\beta}$, then 
	\disp{
	\langle h\rangle_\lambda \cdot \varpi_L^{\beta_0}\sv_\lambda \equiv \langle h'\rangle_\lambda\cdot \varpi_L^{\beta_0}\sv_\lambda \newmod{p^{\beta}V_\lambda^H(\cO_L)},
}
	so
	\begin{equation}\label{eq:beta_0}
		\langle h\rangle_\lambda \cdot \sv_\lambda \equiv \langle h'\rangle_\lambda\cdot \sv_\lambda \newmod{p^{\beta-\beta_0}V_\lambda^H(\cO_L)}.
	\end{equation}

	Now, by \S\ref{sec:slope-decomp} note the action of $\xi^{-1} t_p^\beta$ on $ V_\lambda$ is induced by the action
	\begin{equation}\label{eq:inducing action}
		\smallmatrd{1}{X}{0}{1} \longmapsto \left[t_p^\beta \smallmatrd{1}{X}{0}{1}t_p^{-\beta}\right]\xi^{-1} = \smallmatrd{1}{p^\beta X}{0}{1}\smallmatrd{1}{-1}{0}{w_n} = \smallmatrd{1}{}{}{w_n}\smallmatrd{1}{-1 + p^\beta Xw_n}{0}{1}
	\end{equation}
	on $\smallmatrd{1}{X}{0}{1} \in N_Q^\times(\Z_p)$. In particular, we see
	\begin{align*}
		(\xi^{-1} t_p^{\beta} * \nuj)&[\smallmatrd{1}{X}{0}{1}] = \langle\smallmatrd{1}{}{}{w_n}\rangle_\lambda \nuj\left(\smallmatrd{1}{-1 + p^\beta Xw_n}{0}{1}\right) \ \ \ \   \text{(by defn.\ and \eqref{eq:inducing action})} \\
		&= (-1)^{dnj} (\mathrm{N}_{F/\Q}\circ \det(-1 + p^\beta Xw_n))^j \left\langle\smallmatrd{-1 + p^\beta Xw_n}{}{}{w_n}\right\rangle_\lambda \cdot   \sv_\lambda\\
		&\equiv \left\langle\smallmatrd{-1}{}{}{w_n}\right\rangle_\lambda \cdot \sv_\lambda \newmod{p^{\beta-\beta_0}V_\lambda^H(\cO_L)}   \ \ \ \ \ \ \ \ \ \ \ \  \text{(by \eqref{eq:beta_0}},
	\end{align*}
	which is visibly independent of $j$. Substituting this into \eqref{eq:c_j 2}, we obtain
	\begin{equation}
		(c_j-1) \cdot \big[\left\langle\smallmatrd{-1}{}{}{w_n}\right\rangle_\lambda \cdot \sv_\lambda\big] \in p^{\beta-\beta_0} V_\lambda^H(\cO_L).
	\end{equation}
	As 
	\disp{
	\langle\smallmatrd{-1}{}{}{w_n}\rangle_\lambda \cdot v_\lambda^H \neq 0,
}
	and this holds for all $\beta \geqslant \beta_0,$ we deduce $c_j = 1$, completing the proof.
\end{proof}

In particular, all of our choices, and hence the maps $\cE_{\chi}^{j,\eta_0}$, coincide with those in \cite{DJR18}, so we may freely use the later results \emph{ibid}.\ on the specific values of $\cE_{\chi}^{j,\eta_0}$.

\begin{remark}\label{rem:different xi}
	Proposition \ref{prop:comp to DJR} would fail without the scalar $(-1)^{dnj}$ in Definition \ref{def:kappa_j}. If we had defined $\xi = \smallmatrd{1}{-w_n}{0}{w_n}$ when defining $\mathrm{Ev}_{\beta,\delta}^{M}$, we would not need this scalar. However we choose $\xi = \smallmatrd{1}{w_n}{0}{w_n}$, as chosen in \cite{DJR18}, for compatibility with their results.
\end{remark}

We now compare with the alignment of Jiang--Sun--Tian, who in \cite{JST} proved period relations at infinity for RASCARs. They fix a highest weight vector $v_\infty \in V_\lambda^\vee$, let $u = \smallmatrd{1}{-w_n}{w_n}{1}$, and normalise the branching law\footnote{To translate between this statement and ours here: observe that the torus defined in \cite[(3.14)]{JST} is $uTu^{-1}$, where $T$ is the usual torus; so the space they denote $(F_{\mathbb{K}}^\vee)^{\mathfrak{u}_{\mathbb{K}}}$ is $u\cdot (V_{\lambda}^\vee)^N$ here. But the space $(V_{\lambda}^\vee)^N$ of $N$-invariants is the highest weight space, so their $v_0$ is $u\cdot v_\infty$ here.} $\kappa_j^{\mathrm{JST}} : V_\lambda^\vee \to V_{j,-\sw-j}^H$ so that $\kappa_j^{\mathrm{JST}}(u\cdot v_\infty) = 1$. Again, note all the $\kappa_j^{\mathrm{JST}}$ depend only on the choice of $v_\infty$, which is well-defined up to scalar. Then we have:

\begin{proposition}\label{prop:comp to JST}
We may choose $v_\infty$ such that
\disp{
\kappa_{\lambda,j}^\circ = (\det w_n)^{jd} \cdot \kappa_j^{\mathrm{JST}}
}
for each $j \in \mathrm{Crit}(\lambda)$.
\end{proposition}
\begin{proof}
By Proposition \ref{prop:comp to DJR} it suffices to show $\kappa_j^{\mathrm{JST}} = (\det w_n)^{jd} \cdot \kappa_j^{\DJR}.$ As both lie in the same line, we know at least there exists $C_j \neq 0$ such that $\kappa_j^{\DJR} = C_j \kappa_j^{\mathrm{JST}}$. We want $C_j = (\det w_n)^{jd}$. 

Note that 
\begin{equation}\label{eq:xi vs u}
	\xi = \smallmatrd{1}{}{}{w_n}u\smallmatrd{1}{-w_n}{0}{1}\smallmatrd{1_n}{}{}{-2^{-1}\cdot 1_n} w_{2n}
\end{equation}
Note $w_{2n}\cdot v_0$ is a highest weight vector. Thus any $t \in T$ acts on $w_{2n}\cdot v_0$ as $\lambda^\vee(t)$, and any $n \in N$ acts trivially. Letting both sides of \eqref{eq:xi vs u} act on $v_0$ thus gives
\[
	\xi\cdot v_0 = \lambda^\vee\left[\smallmatrd{1_n}{}{}{-2^{-1}\cdot 1_n}\right] \smallmatrd{1}{}{}{w_n}u \cdot w_{2n} \cdot v_0  = (\det w_n)^{\sw d}\smallmatrd{1}{}{}{w_n}u \cdot v_\infty,
\]
where we define $v_\infty \defeq (\det w_n)^{-\sw d}\lambda^\vee\left[\smallmatrd{1_n}{}{}{-2^{-1}\cdot 1_n}\right] \cdot w_{2n}\cdot v_0$. Then
\begin{align*}
1 = \kappa_j^{\mathrm{DJR}}&\left[\xi\cdot v_0\right] = \det(w_n)^{\sw d}\kappa_j^{\mathrm{DJR}}\left[\smallmatrd{1}{}{}{w_n}u \cdot v_\infty\right] = (\det w_n)^{-jd} \kappa_j^{\DJR}\left[u\cdot v_\infty\right]\\
& = (\det w_n)^{-jd} C_j\kappa_j^{\mathrm{JST}}[u \cdot v_\infty] = (\det w_n)^{-jd}C_j.
\end{align*}
For this choice of $v_\infty$ we have $C_j = (\det w_n)^{jd}$, as required.
\end{proof}

\subsection{Non-vanishing of evaluation maps and Shalika models}
We now show how classical evaluation maps can detect existence of Shalika models. Let $\pi$ be any RACAR with attached maximal ideal $\m_\pi \subset \cH'$ as in \S\ref{sec:unramified H}. Let $\lambda$ be the weight of $\pi$, with purity weight $\sw$. 
\begin{proposition}\label{prop:shalika non-vanishing}
	Suppose there exists $\phi \in \hc{t}(S_K,\sV_{\lambda}^\vee(\overline{\Q}_p))^\epsilon_{\m_\pi}$ such that 
	\begin{equation}\label{eq:non-vanishing 1}
		\cE_{\chi}^{j,\eta_0}(\phi) \neq 0
	\end{equation}
	for some $\chi, j$ and $\eta_0$. Then $\pi$ admits a global $(\eta,\psi)$-Shalika model, where $\eta = \eta_0|\cdot|^{\sw}$.
\end{proposition}

\begin{proof}
	By Proposition~\ref{prop:non-canonical}, there exists a unique $\varphi_{f} \in \pi_f^K$ mapping to $\phi$ under \eqref{eq:cohomology non-canonical}. This isomorphism depended on a choice $\Xi_\infty^\epsilon$ of generator of 
	\[
	\h^t(\fg_\infty, K_\infty^\circ; \pi_\infty \otimes V_{\lambda}^\vee(\C))^\epsilon \subset \left[\wedge{}^t (\fg_\infty/\ft_\infty)^\vee \otimes \pi_\infty \otimes V_{\lambda}^\vee(\C)\right]^{K_\infty^\circ},
	\]
	where $\ft_\infty = \mathrm{Lie}(T_\infty)$ and the inclusion is \cite[II.3.4]{BW00} (see also \cite[\S4.1]{GR2}). Fixing bases $\{\omega_i\}$ of $(\fg_\infty/\ft_\infty)^\vee$ and $\{e_\alpha\}$ of $V_\lambda^\vee(\C)$, there then exist unique vectors 
	\[
	\varphi_{\infty,\mathbf{i},\alpha}^\epsilon \in \pi_\infty \qquad\text{ such that } \qquad
	\Xi_{\infty}^\epsilon = \sum_{\mathbf{i}} \sum_\alpha \omega_{\mathbf{i}} \otimes \varphi_{\infty,\mathbf{i},\alpha}^\epsilon \otimes e_\alpha,
	\]
	where $\mathbf{i}$ ranges over tuples $(i_1,...,i_t)$ and $\omega_{\mathbf{i}} = \omega_{i_1} \wedge \cdots \wedge \omega_{i_t}$. Define 
	\disp{
	\varphi_{\mathbf{i},\alpha}^\epsilon\defeq \varphi_{\infty,\mathbf{i},\alpha}^\epsilon \otimes \varphi_f.
}
	By \cite[Prop.~4.6]{DJR18}, we see there exists an automorphic form
	\[
	\varphi_{\phi,j}^\epsilon = \sum_{\mathbf{i}}\sum_{\alpha} a_{\mathbf{i},\alpha,j}^\epsilon \cdot \varphi^\epsilon_{\mathbf{i},\alpha} \ \ \in \pi,
	\]
	where the scalars $a_{\mathbf{i},\alpha,j}^\epsilon \in \C$ depend on $\kappa_{\lambda,j}^\circ$, and with $\mathbf{i}$ and $\alpha$ ranging over the same sets as above, such that
	\[
	i_p^{-1}\big[\cE_{\beta,[\delta]}^{j,\sw}(\phi)\big] = \lambda(t_p^\beta)\int_{X_\beta[\delta]} \varphi_{\phi,j}^\epsilon(h\xi t_p^\beta)\ |\det(h_1^j h_2^{-\sw-j})|_F \ dh.
	\]
	Now arguing exactly as in the proof of \cite[Thm.~4.7]{DJR18}, we have an equality
	\begin{equation}\label{eq:non-vanishing 2}
		i_p^{-1}\big[\cE_{\chi}^{j,\eta_0}(\phi)\big] = \bigg[\gamma_{p\m}\cdot\lambda(t_p^\beta) \  \prod_{\pri|p}\mathrm{N}_{F/\Q}(\pri)^{n^2\beta_{\pri}} \bigg]\cdot \Psi\left(j+\tfrac{1}{2}, \ \varphi', \chi, \eta\right),
	\end{equation}
	where	$\gamma_{p\m}$ is a non-zero volume constant defined in \cite[(77)]{DJR18}, $\varphi' \defeq (\xi t_p^\beta)\cdot \varphi_{\phi,j}^\epsilon$, and $\Psi$ is the period integral defined in \cite[Prop.~2.3]{FJ93}. 	Now, as in the proof of \cite[Prop.~2.3]{FJ93}, we may write
	\begin{align}\label{eq:non-vanishing 3}
		&\Psi\left(j+\tfrac{1}{2}, \ \varphi', \chi, \eta\right)\\
		&= \int_{Z_n(\Q)\backslash Z_n(\A)}\left[\int_{H(\Q)\backslash H^0} \varphi'\left[\smallmatrd{h_1 x}{}{}{h_2}\right] \chi\left(\tfrac{\det h_1}{\det h_2}\right) \eta^{-1}(\det h_2) dh\right] \chi(x)|\det(x)|^{j} dx, \notag
	\end{align}
	where $Z_n$ is the centre of $\mathrm{Res}_{F/\Q}\GL_n$ and 
	\disp{
	H^0 = \{(h_1,h_2) \in H(\A) : |\det(h_1)| = |\det(h_2)| = 1\}.
	}
	By  \eqref{eq:non-vanishing 1}, both \eqref{eq:non-vanishing 2} and \eqref{eq:non-vanishing 3} do not vanish; hence the inner integral of \eqref{eq:non-vanishing 3} also does not vanish. But existence of such a $\varphi', \chi$ and $\eta$ implies $\pi$ admits an $(\eta,\psi)$-Shalika model by \cite[Prop.~2.2]{FJ93}.
\end{proof}

\subsection{Local zeta integrals}

In this and the next section, we state and prove Theorem \ref{thm:critical value}, relating evaluation maps to $L$-values for our $\tilde\pi$. This is a compilation of results from \cite{FJ93,GR2,DJR18, JST, BDGJW}. First we relate to local zeta integrals in a general setting.

Let $\pi$ be a RASCAR of $G(\A)$, and $\chi = \prod\chi_v$ a Hecke character of $F$ of conductor $p^\beta$. Recall $\Theta_{i_p}^{K,\epsilon} : \cS_{\psi_f}^{\eta_f}(\pi_f^K) \to \hc{t}(S_K,\sV_\lambda^\vee(\overline{\Q}_p))^\epsilon_{\m_\pi}$ from \eqref{eq:cohomology class p-adic}, depending on a choice $\Xi_{\infty}^\epsilon$ at infinity. Attached to $\Xi_\infty^\epsilon$ and $j \in \mathrm{Crit}(\lambda)$ is a `local zeta integral' $\zetainfty$, the quantity $\cP_{\infty,j}(\Xi_\infty^\epsilon)$ from \cite[(4.15)]{JST}. Recall the finite analogues $\zeta_v(-)$ from \S\ref{sec:shalika models}. Let 
\[
(\chi_{\cyc}^j\chi\eta)_\infty = [(-1)^j\chi_\sigma(-1)\eta_\sigma(-1)]_{\sigma \in \Sigma} \in \{\pm1\}^\Sigma.
\]

\begin{lemma}\label{lem:local zeta integrals}
	Let $W_f = \otimes_v W_v \in \cS_{\psi_f}^{\eta_f}(\pi_f)$. If $\epsilon \neq (\chi_{\cyc}^j\chi\eta)_\infty$, then $\cE_{\chi}^{j,\eta_0}(\Theta_{i_p}^{K,\epsilon}(W_f)) =0$. 
	
	If $\epsilon = (\chi_{\cyc}^j\chi\eta)_\infty$, then
	\begin{multline*}
	i_p^{-1}\left(\cE_{\chi}^{j,\eta_0}(\Theta_{i_p}^{K,\epsilon}(W_f))\right) = \bigg[\gamma_{p\m}\cdot\lambda(t_p^\beta) \  \prod_{\pri|p}\mathrm{N}_{F/\Q}(\pri)^{n^2\beta_{\pri}} \bigg]\\
	\times \zetainfty\cdot \prod_{v\nmid p\infty} \zeta_v\Big(j+1/2;W_v, \chi_v\Big) \cdot \prod_{\pri|p} \zeta_{\pri}\Big(j+1/2;W_{\pri}\big(-\cdot \xi t_{\pri}^{\beta_{\pri}}\big),\chi_{\pri}\Big).
	\end{multline*}
\end{lemma}
\begin{proof}
	When $\epsilon \neq (\chi_{\cyc}^j\chi\eta)_\infty$, we deduce $\cE_{\chi}^{j,\eta_0}(\phi_{\tilde{\pi}}^\epsilon) = 0$ as in the proof of \cite[Thm.~4.7]{DJR18}.
	
	Suppose the sign condition is satisfied. We start from \eqref{eq:non-vanishing 2} above, where $\varphi' = \xi t_p^\beta\cdot \varphi_{\phi,j}^\epsilon$ in the notation \emph{op.\ cit}, with $\phi = \Theta_{i_p}^{K}(W_f)$. Note $\cS_{\psi}^\eta(\varphi') = W_{\infty,j}^\epsilon \otimes [\xi t_p^\beta \cdot W_f]$ for some $W_{\infty, j}^\epsilon \in \cS_{\psi_\infty}^{\eta_\infty}(\pi_\infty)$.  Now \cite[\S4.1.2]{DJR18} shows that $\Psi(j+1/2,\varphi',\chi,\eta)$ equals the product of local zeta integrals, as required.
\end{proof}

\subsubsection{Local zeta integrals at infinity} 
At infinity, the following is a combination of Sun \cite[Thm.~5.5]{Sun19}, Jiang--Sun--Tian \cite[Thm.\ 3.12]{JST}, and Geng \cite[Thm.\ 8.6]{Geng}. 
\begin{theorem}\label{thm:JST}
	Up to rescaling the basis elements $\Xi_\infty^\epsilon \in \h^t(\fg_\infty,K_\infty^\circ; \pi_\infty \otimes V_\lambda^\vee(\C))^\epsilon$, if $\epsilon = (\chi_{\cyc}^j\chi\eta)_\infty$, we have
	\[
	\zetainfty =  i^{-jnd} \cdot L(\pi_\infty \otimes \chi_\infty, j+1/2).
	\]	
\end{theorem}

\begin{proof}
	 For each $j \in \mathrm{Crit}(\lambda)$, Jiang--Sun--Tian construct a zeta integral $\zeta_{\infty,j}^{\mathrm{JST}}(\Xi_\infty^\epsilon)$ at infinity, and show it arises from an evaluation map/modular symbol process as above. Their main result is existence of $\varepsilon(\pi_\infty) = \prod_{\sigma \in \Sigma} \varepsilon(\pi_\sigma) \in \{\pm1\}$ such that the quantity 
	\begin{equation}\label{eq:JST ratio}
	\frac{\zeta_{\infty,j}^{\mathrm{JST}}(\Xi_\infty^\epsilon)}{i^{-jnd} \cdot L(\pi_\infty\times\chi_\infty, j+1/2) \cdot \varepsilon(\pi_\infty)^j}
	\end{equation}
	is non-zero and independent of $j$ when $\epsilon = (\chi_{\cyc}^j\chi\eta)_\infty$.	Further, in \cite{Geng}, Geng shows that $\varepsilon(\pi_\sigma) = \det(w_n)$ for all $\sigma$, so $\varepsilon(\pi_\infty)^j = \det(w_n)^{jd}$.
	
	The map $\zeta_{\infty,j}^{\mathrm{JST}}$ differs from $\zeta_{\infty,j}$ only in the choice of branching law, so by Proposition \ref{prop:comp to JST}
		\begin{equation}\label{eq:JST vs BDW}
	  \zeta_{\infty,j}(\Xi_\infty^\epsilon) = \det(w_n)^{jd} \cdot	\zeta_{\infty,j}^{\mathrm{JST}}(\Xi_\infty^\epsilon).
		\end{equation}
		Combining \eqref{eq:JST ratio} and \eqref{eq:JST vs BDW}, we see 
			\begin{equation}\label{eq:BDW ratio}
			\frac{\zeta_{\infty,j}(\Xi_\infty^\epsilon)}{i^{-jnd} \cdot L(\pi_\infty\times\chi_\infty, j+1/2)}
		\end{equation}
		is non-zero and independent of $j$. Now note $\zetainfty$ scales linearly with $\Xi_\infty^\epsilon$; so by rescaling the latter, we may assume \eqref{eq:BDW ratio} equals 1 for some $j_0$, hence for all $j$, as required.
\end{proof}

\begin{definition}\label{def:e_infty}
	We let	$\einf \defeq i^{-jnd}  \cdot L(\pi_\infty \otimes \chi_\infty, j+1/2)$.
\end{definition}

\subsubsection{Local zeta integrals at $p$}

Recall  from \S\ref{sec:level group} that we work in two local settings at $p$:
\begin{itemize}\s
	\item[(C2)$_{\pri}$]  $\pi_{\pri}$ is parahoric spherical admitting a Shalika model, $\tilde\pi_{\pri} = (\pi_{\pri}, \alpha_{\pri})$ is a Shalika $Q$-refinement, and $W_{\pri} \in \cS_{\psi_{\pri}}^{\eta_{\pri}}(\pi_{\pri}^{J_{\pri}})\lsem U_{\pri} - \alpha_{\pri}\rsem$ a generator. 

	\item[(C2$'$)$_{\pri}$] $\pi_{\pri} = \Ind_B^G\theta_{\pri}$ is spherical, satisfies the hypotheses of Proposition \ref{lem:Q-regular criterion}, and $\tilde\pi_{\pri} = (\pi_{\pri},\alpha_{\pri})$ is the Shalika $Q$-refinement from that result.
\end{itemize}
We will assume (C2)$_{\pri}$ throughout, and (C2$')_{\pri}$ when considering unramified characters. For a quasi-character $\chi_{\pri}$ of $F_{\pri}^\times$, let $\tau(\chi_{\pri})$ be the local Gauss sum, normalised as in \cite[\S9.2]{BDGJW}.

\begin{definition}\label{def:e'_p} 
	Let $s \in \C$. If $\chi_{\pri}$ is ramified of conductor $\pri^{\beta_{\pri}}$, let $T(\chi_{\pri}) = \tau(\chi_{\pri})^n$ and
	\[
	e'_{\pri}(\tilde\pi,\chi,s) \defeq q_{\pri}^{\beta_{\pri} n\big(s+\tfrac{n}{2}-\tfrac12\big)+\delta_{\pri} n\big(s - \tfrac{n}{2}-\tfrac12\big)} \cdot \frac{q_{\pri}^n}{(q_{\pri}-1)^n}.
	\]
	This depends only on $\beta_{\pri}$ and $s$, but we denote it this way for later consistency.
	
		If $\chi_{\pri}$ is unramified and (C2$')_{\pri}$ holds, let $T(\chi_{\pri}) = \chi(\varpi_{\pri})^{-n\delta_{\pri}}$ and
	\[
	e_{\pri}'(\tilde\pi, \chi, s) \defeq q_{\pri}^{\delta_{\pri} n\big(s - \tfrac{n}{2}-\tfrac12\big)} \cdot 	\displaystyle{\frac{q_{\pri}^n}{(q_{\pri}-1)^n}\cdot \alpha_{\pri} \cdot \prod_{i=n+1}^{2n}
		\frac{1-\UPS_{\pri,i}^{-1}\chi_{\pri}^{-1}(\varpi_{\pri})q_{\pri}^{s-\tfrac12}}{1-\UPS_{\pri,i}\chi_{\pri}(\varpi_{\pri})q_{\pri}^{-s-\tfrac12}}}.
	\]
\end{definition}

\begin{proposition}\emph{(D.--Januszewski--Raghuram; B.--D.--Graham--Jorza--W.).} \label{lem:zeta at p}

	Let $W_{\pri}$ be a generator of $\cS_{\psi_{\pri}}^{\eta_{\pri}}\big(\pi_{\pri}^{J_{\pri}}\big)\lsem U_{\pri} - \alpha_{\pri}\rsem.$ 
\begin{itemize}
	\item[(i)] If (C2)$_{\pri}$ holds, then for all ramified quasi-characters $\chi_{\pri}$, we have
	\begin{equation}\label{eq:zeta_p spherical}
	\zeta_{\pri}\big(s;W_{\pri}(-\cdot \xi t_{\pri}^{\beta_{\pri}}),\chi_{\pri}\big) = T(\chi_{\pri}) \cdot e'_{\pri}(\tilde\pi,\chi,s-\tfrac12) \cdot W_{\pri}(t_{\pri}^{-\delta_{\pri}}).
	\end{equation}
\item[(i)] If (C2$')_{\pri}$ holds, \eqref{eq:zeta_p spherical} also holds for unramified $\chi_{\pri}$.
\end{itemize}
\end{proposition}
\begin{proof}
Given (C2)$_{\pri}$, (i) is \cite[Prop.\ 3.4]{DJR18} (with a corrected power of $q_{\pri}$; see Appendix (2)).
	
	If (C2$')_{\pri}$ holds, (ii) was proved by the present authors with Graham and Jorza in \cite[Prop.\ 9.3]{BDGJW}. The only differences are that instead of $\xi = \smallmatrd{1}{w_n}{0}{w_n}$ here, there is used $u^{-1} = \smallmatrd{1}{-w_n}{0}{1}$; but we can compare the two integrals by noting that the integrand in \cite{BDGJW} contains 
	\[
	\smallmatrd{h}{}{}{1}\smallmatrd{1}{-w_n}{0}{1}\smallmatrd{t_p^\beta}{}{}{1} = \smallmatrd{-hw_n}{}{}{1}\smallmatrd{1}{w_n}{0}{w_n}\smallmatrd{t_p^\beta}{}{}{1}\smallmatrd{-w_n}{}{}{w_n}.
	\]
	The change of variables $h \leftrightarrow -hw_n$ removes the factor of $\chi_{\pri}(\det(-w_n))$ appearing in \cite{BDGJW}, and $\smallmatrd{-w_n}{}{}{w_n}$ disappears by parahoric invariance. In \cite{BDGJW} the term $W_{\pri}(t_{\pri}^{-\delta_{\pri}})$ is denoted $F_0(w_{2n})$ and taken to be 1 (see \S9.1 \emph{ibid}.), so does not appear there. We also rearrange using  $\alpha_{\pri} = q_{\pri}^{n^2/2}\UPS_{\pri,n+1}\cdots\UPS_{\pri,2n}(\varpi_{\pri})$.
\end{proof}

\begin{remark}\label{rem:Q-regular zeta}
	Proposition \ref{lem:zeta at p}(i) holds assuming only $\tilde\pi_{\pri}$ is regular, rather than Shalika (i.e.\ without demanding that $W(t_{\pri}^{-\delta}) \neq 0$). In particular if $\tilde\pi_{\pri}$ is regular and $\zeta_{\pri}(s, W_{\pri}(-\cdot \xi t_{\pri}^{\beta_{\pri}}), \chi_{\pri}) \neq 0$ for some ramified $\chi_{\pri}$, then this result implies $\tilde\pi_{\pri}$ is Shalika.
\end{remark}

\subsection{Cohomological interpretation of $L$-values}
Now suppose $\tilde\pi$ satisfies Conditions  \ref{cond:running assumptions}  or \ref{cond:running assumptions 2}, and recall  
\[
\phi_{\tilde\pi}^{\epsilon} = \Theta_{i_p}^{K(\tilde\pi),\epsilon}(W_f^{\mathrm{FJ}})\big/i_p(\Omega_\pi^\epsilon)
\]
from Definition~\ref{def:phi}. It is important that we now work at level $K = K(\tilde\pi)$. The results from  \cite{FJ93,GR2,DJR18, JST, BDGJW} combine to show:

\begin{theorem}\label{thm:critical value}
	Suppose $\tilde\pi$ satisfies Conditions \ref{cond:running assumptions 2}. Fix  $\epsilon \in \{\pm 1\}^\Sigma$. Let $\chi$ be a finite order Hecke character of conductor $p^{\beta'}$, with $\beta' = (\beta_{\pri}')_{\pri|p}$ with each $\beta_{\pri}'\geqslant 0$. Let $\beta_{\pri} \defeq \mathrm{max}(\beta_\pri', 1).$ Let $j \in \mathrm{Crit}(\lambda)$. Then if $\epsilon \neq (\chi_{\cyc}^j\chi\eta)_\infty$, then $\cE_{\chi}^{j,\eta_0}(\phi_{\tilde{\pi}}^\epsilon) = 0$.  If $\epsilon = (\chi_{\cyc}^j\chi\eta)_\infty$, we have
		\begin{multline*}
			i_p^{-1}\left(\cE_{\chi}^{j,\eta_0}(\phi_{\tilde{\pi}}^\epsilon)\right) = \gamma_{p\m}\cdot\lambda(t_p^\beta)\cdot \mathrm{N}_{F/\Q}(\mathfrak{d}^{(p)})^{jn} \cdot\tau(\chi_f)^n  \\ 
			\times \big[\textstyle\prod_{\pri|p}\epp\big] \cdot \einf \cdot\displaystyle\frac{L^{(p)}(\pi\otimes\chi,j+\tfrac{1}{2})}{\Omega_\pi^\epsilon}.
		\end{multline*}
	If $\tilde\pi$ satisfies the (more general) Conditions \ref{cond:running assumptions} then the same is true when $\beta_{\pri}' \geqslant 1$ for all $\pri|p$.
\end{theorem} 
Here $\tau(\chi_f)$ is the Gauss sum, $\fd^{(p)}$ is the prime-to-$p$ part of the different, and the $e(-)$ terms are as in Definitions \ref{def:e_infty} and \ref{def:e'_p} above.

\begin{proof}
	By Lemma \ref{lem:local zeta integrals}, we have vanishing unless the sign condition is satisfied, whence the left-hand side is a product of local zeta integrals. The integral at infinity was computed in Theorem \ref{thm:JST}. At $v\nmid p\infty$, the integral is $\zeta_v(j+1/2,W_v^{\mathrm{FJ}},\chi_v)$, which was evaluated in \eqref{eq:jacquet-friedberg test vector}. In the product, we get the claimed $L$-value and $\mathrm{N}_{F/\Q}(\fd^{(p)})^{jn},$ and a product of $\chi_v(\varpi_v)$'s. 
		
		At $\pri|p$, if Conditions \ref{cond:running assumptions 2} hold, then we are in case (C2$')_{\pri}$ of Proposition \ref{lem:zeta at p}, and this computes the integral for all $\chi_{\pri}$. If only Conditions \ref{cond:running assumptions} hold, then we are in case (C2)$_{\pri}$ and Proposition \ref{lem:zeta at p} computes it whenever $\beta_{\pri} \geqslant 1$. 
		
		The $T(\chi_{\pri})$'s combine with the products of $\chi_v(\varpi_v)$'s at $v\nmid p\infty$ to give $\tau(\chi_f)^n$, as in \cite[Thm.\ 4.7]{DJR18}. The $L$-factors combine into $L^{(p)}(-)$. The other terms are as claimed.
	\end{proof}

%%=========================================================
%%
%%			NON-ORDINARY p-ADIC L-FUNCTIONS
%%
%%=========================================================

\section{Finite slope $p$-adic $L$-functions} \label{sec:galois evaluations}

	We now use the formalism of evaluation maps to prove Theorem \ref{thm:intro non-ord} of the introduction.

\subsection{Distributions over Galois groups} 
\label{sec:galp}

\subsubsection{Definition of Galois distributions}\label{sec:galois groups} 	Throughout this section, fix $\lambda_\pi \in X_0^*(T)$ a pure classical `base' weight, and let 
\disp{
\Omega = \mathrm{Sp}(\cO_\Omega) \subset \Wlam
}
be an affinoid. We allow $\Omega = \{\lambda\}$ for $\lambda$ classical, in which case $\cO_\Omega = L$. Let $\chi_\Omega : T(\Zp) \to \cO_\Omega^\times$ be the tautological character attached to $\Omega$, and recall the purity weight $\sw_\Omega : \Zp^\times \to \cO_\Omega^\times$, all defined in \S\ref{sec:distributions in families}.

Recall $\cO_{F,p} = \cO_F\otimes\Zp$. For $\beta= (\beta_{\fp})_{\fp | p}$ with $\beta_{\fp}> 0$ for each $\fp \mid p$, let 
\[
\sU_\beta \defeq  [1+ p^{\beta}\OFp]/\overline{E(p^{\beta})},
\]
where $\overline{E(p^{\beta})}$ is the $p$-adic closure of 
\disp{
	E(p^{\beta})\defeq \{u \in \cO_{F}^{\times}\cap F_\infty^{\times\circ}  \arrowvert u\equiv 1 \newmod{p^{\beta}}\}.
}
Then by Class Field Theory we have an exact sequence
\begin{equation}\label{eq:CFT}
	1 \rightarrow \sU_\beta \xrightarrow{\iota} \Galp \xrightarrow{\jmath} \Cl(p^{\beta})\rightarrow 1.
\end{equation}
Recall the distribution modules $\cD(X,R)$ from \S\ref{sec:A_s}. The sum of the natural restriction maps induces a decomposition
	\begin{equation}\label{e:decomposition galois distributions}
		\cD(\Galp, \cO_\Omega)\cong \bigoplus_{\mathbf{x} \in \Cl(p^{\beta})}\cD(\Galp[\mathbf{x}], \cO_\Omega), \qquad \Galp[\mathbf{x}] \defeq \jmath^{-1}(\mathbf{x}).
	\end{equation}
The map $\iota$ induces a map 
\disp{
\iota_*: \cD(\sU_\beta,\cO_\Omega) \hookrightarrow \cD(\Galp,\cO_\Omega),
}
whose image can be identified with $\cD(\Galp[\mathbf{1}_\beta],L)$, where $\mathbf{1}_\beta$ is the identity element in $\Cl(p^\beta)$. 

In the limit, the Artin reciprocity map $\mathrm{rec} : \A_F^\times \to \Galp$ induces an isomorphism 
\begin{equation}\label{eq:galp ideles}
	\Galp \labelisoleftarrow{\ \mathrm{rec}\ } \Cl(p^\infty) \defeq F^\times\backslash \A_F^\times/\overline{\sU(p^\infty) F_\infty^{\times\circ}},
\end{equation}
where $\sU(p^\infty)= \prod_{v\nmid p}\cO_v^{\times}$. Note that the cyclotomic character $\chi_{\cyc}$ from \eqref{eq:cyc} is naturally a character on $\Galp$; it is the character attached to the adelic norm via \cite[\S2.2.2]{BW_CJM}.

\subsubsection{Group actions on Galois distributions}
If $c \in \A_F^\times$ and $x \in \Galp$, to simplify notation we write $cx \defeq \mathrm{rec}(c)x$. We define a left action of $(\delta_1,\delta_2) \in \A_F^\times \times \A_F^\times$ on $\cA(\Galp,\cO_\Omega)$ by 
\begin{equation}\label{eq:action ideles}
	(\delta_1,\delta_2) * f(x) = \chi_{\cyc}(\delta_2)^{\sw_\Omega}f(\delta_{1}^{-1}\delta_{2} x),
\end{equation}
and dually a left action on $\cD(\Galp,\cO_\Omega)$. Recall  $\mathrm{pr}_\beta: \pi_0(X_\beta) \to \Cl(p^\beta)$ from \eqref{eq:pr_beta}.

\begin{lemma}\label{lem:map to Galp[z]}
	Let $\delta = (\delta_1,\delta_2) \in \A_F^\times \times \A_F^\times$, representing an element $[\delta] \in \pi_0(X_\beta)$, and let 
	\disp{
	\mathbf{x} = \mathrm{pr}_\beta([\delta]) \in \Cl(p^\beta).
}
	The action of $\delta$ induces an isomorphism
	\[
	\cD(\sU_\beta,\cO_\Omega) \labelisorightarrow{\ \iota_*\ } \cD(\Galp[\mathbf{1}_\beta],\cO_\Omega) \xrightarrow{\ \mu \ \mapsto\  \delta * \mu \ } \cD(\Galp[\mathbf{x}],\cO_\Omega).
	\]
\end{lemma}
\begin{proof}
	The action of $\delta$ on $\mu \in \cD(\Galp,\cO_\Omega)$ is induced by the action of $\delta^{-1}$ on $\cA(\Galp,\cO_\Omega)$ by 
	\[
	(\delta^{-1} * f)(x) = \chi_{\mathrm{cyc}}(\delta_2)^{-\sw_\Omega}f(\delta_1\delta_2^{-1}x).
	\]
	By  \eqref{eq:component group} $\delta_1 \delta_2^{-1}$ is a representative of $\mathbf{x}$, so multiplication by $\delta_1\delta_2^{-1}$ on $\Galp$ sends $\Galp[\mathbf{1}_\beta]$ isomorphically to $\Galp[\mathbf{x}]$. Hence this action induces a map 
	\disp{
	\delta^{-1} * - : \cA(\Galp[\mathbf{x}],\cO_\Omega) \to \cA(\Galp[\mathbf{1}_\beta],\cO_\Omega)
}
	which dualises to the claimed map. 
\end{proof}

Via \eqref{eq:action ideles}, we have an action of $H(\A)$ on $f \in \cA(\Galp,\cO_\Omega)$ by
\begin{equation}\label{eq:action H}
	(h_1,h_2) * f \defeq  (\det(h_1), \det(h_2)) * f,
\end{equation}
and hence a dual action on $\cD(\Galp,\cO_\Omega)$. Note that both $H(\Q)$ and $H_\infty^\circ$ act trivially.

\subsection{$p$-adic interpolation of branching laws for $H \subset G$}\label{sec:p-adic branching}
	We now interpolate the branching law of Lemma \ref{lem:hom line}, and hence evaluation maps, in the following sense: for classical $\lambda \in \Omega$, by Lemma \ref{prop:pushforward} we have a commutative diagram
\begin{equation}\label{eq:motivation nu}
	\xymatrix@C=20mm{
		\hc{t}(S_K,\sD_\Omega) \ar[r]^-{\mathrm{Ev}_{\beta,\delta}^{\cD_\Omega}} \ar[d]^{r_\lambda\circ \mathrm{sp}_\lambda} & (\cD_\Omega)_{\Gamma_{\beta,\delta}} \ar[d]^{r_\lambda \circ \mathrm{sp}_\lambda} & \cD(\Galp,\cO_\Omega)\ar[d]^{\mu \mapsto \mathrm{sp}_\lambda(\mu)(\chi_{\cyc}^j)}\\
		\hc{t}(S_K,\sV_\lambda^\vee) \ar[r]^-{\mathrm{Ev}_{\beta,\delta}^{V_\lambda^\vee}} & (V_\lambda^\vee)_{\Gamma_{\beta,\delta}} \ar[r]^{\kappaj^\circ} & L,
	}
\end{equation}
and we now define the `missing' horizontal map in the top row so that the horizontal compositions commute with the outer vertical maps (see Proposition \ref{prop:kappa interpolate}).

The map $\kappaj^\circ$ was defined by an element $v_{\lambda,j} \in V_\lambda$, which we described explicitly in Lemma \ref{lem:v_lambda non-vanishing} (noting $v_{\lambda,j}$ was defined to be $\kappaj^\vee(u_j^\vee)$). Our definition of $v_\Omega^\beta$, given in  \eqref{eq:times}, is really an interpolation of this last description of $v_{\lambda,j}$.

\begin{remark*}
	Note restricting elements of $\cD_\Omega$ to the subspace $\cA_\Omega^{\Gamma_{\beta,\delta}}\subset \cA_\Omega$ induces a well-defined map $(\cD_\Omega)_{\Gamma_{\beta,\delta}} \to (\cA_\Omega^{\Gamma_{\beta,\delta}})^\vee$. Slightly abusing notation/terminology, we will identify elements of $(\cD_\Omega)_{\Gamma_{\beta,\delta}}$ with their image under this map, and continue to call them distributions (on $\cA^{\Gamma_{\beta,\delta}}$).
\end{remark*}

\subsubsection{Support conditions on distributions}

We want to define $v_\Omega^\beta$ to interpolate $v_{\lambda,j} \in V_\lambda$ from \S\ref{sec:choice of basis}. However, we have explicitly described the function $v_{\lambda,j} :G(\Zp) \to L$ only  on the subset $N_Q^\times(\Zp) \subset G(\Zp)$ (see Lemma \ref{lem:v_lambda non-vanishing}). The following support condition shows that for the outer vertical maps of \eqref{eq:motivation nu} to commute with the horizontal compositions, it is sufficient to specify $v_\Omega^\beta$ on subsets $N_Q^\beta(\Zp)$ of $N_Q^\times(\Zp)$.

For $\beta = (\beta_{\pri})_{\pri|p}$ with each $\beta_{\pri} \geqslant 1$, let
\begin{equation}\label{eq:Jp minus}
	N_Q^\beta(\Zp) \defeq \big\{\smallmatrd{1}{X}{0}{1} \in N_Q(\Zp) : X \equiv - I_n \newmod{p^\beta}\big\} \subset N_Q^\times(\Zp),
\end{equation}
and define 
\[
\Jpbeta \defeq (N_Q^-(\Zp)\cap J_p)\cdot H(\Zp) \cdot N_Q^\beta(\Zp) \subset J_p.
\]

\begin{lemma}\label{lem:support}
	Let $\Phi \in \hc{t}(S_K,\sD_\Omega)$, and let $\delta \in H(\A)$. The distribution $\mathrm{Ev}_{\beta,\delta}^{\cD_\Omega}(\Phi) \in (\cD_{\Omega})_{\Gamma_{\beta,\delta}}$ has support in $\Jpbeta$, in the sense that if $f \in \cA_{\Omega}^{\Gamma_{\beta,\delta}},$ then
	\[
	\mathrm{Ev}_{\beta,\delta}^{\cD_\Omega}(\Phi)(f) = \mathrm{Ev}_{\beta,\delta}^{\cD_\Omega}(\Phi)\left(f|_{\Jpbeta}\right)
	\]
	depends only on the restriction of $f$ to $\Jpbeta$.
\end{lemma}
\begin{proof}
	Via the map $\tau_\beta^\circ$, we see that 
	\[
	\mathrm{Ev}_{\beta,\delta}^{\cD_\Omega}(\Phi) \in (\xi t_p^\beta * \cD_\Omega)_{\Gamma_{\beta,\delta}}.
	\]
	It thus suffices to prove that for any $\mu \in \cD_\Omega$ and $f\in \cA_\Omega$, we have 
	\[
	(\xi t_p^\beta * \mu)(f) =  (\xi t_p^\beta * \mu)(f|_{J_p^\beta}),
	\]
	or equivalently that 
	\[
	(\xi t_p^{\beta})^{-1} * f = (\xi t_p^\beta)^{-1} * f|_{J_p^\beta}.
	\]
	By definition (see \S\ref{sec:slope-decomp}), the action of $(\xi t_p^\beta)^{-1}$ on $f \in \cA_\Omega$ is induced by the action
	\begin{align}\label{eq:action of t_p on X}
		\smallmatrd{1}{X}{0}{1} \longmapsto &\left[t_p^\beta \smallmatrd{1}{X}{0}{1}t_p^{-\beta}\right]\xi^{-1}\\ 
		&= \smallmatrd{1}{p^\beta X}{0}{1}\smallmatrd{I_n}{-I_n}{0}{w_n} = \smallmatrd{1}{0}{0}{w_n}\smallmatrd{1}{-I_n + p^\beta Xw_n}{0}{1} \in \Jpbeta\notag
	\end{align}
	on $\smallmatrd{1}{X}{0}{1} \in N_Q(\Z_p)$. Thus $((\xi t_p^\beta)^{-1} * f)|_{N_Q(\Zp)}$ depends only on $f|_{J_p^\beta}$. By parahoric decomposition \eqref{eq:parahoric transform A}, we deduce that $(\xi t_p^\beta)^{-1} * f$ depends only on $f|_{J_p^\beta}$, as claimed.
\end{proof}

\subsubsection{Interpolation of $v_\lambda^H$ in families}

Recall that the description of $v_{\lambda,j}$ in Lemma \ref{lem:v_lambda non-vanishing} was given in terms of a specific vector $v_\lambda^H \in V_\lambda^H$. We now interpolate $v_{\lambda,j}$ as $\lambda$ varies in $\Omega$.

\medskip

In Notation \ref{not:v_lambda}, we fixed $v_{\lambda_\pi}^H \in V_{\lambda_\pi}^H(\cO_L)$ to be an (optimally integral) generator of the unique line in $V_{\lambda_\pi}^H(L)$ on which the action of $\left\langle \smallmatrd{h}{}{}{h}\right\rangle_{\lambda_\pi}$ is multiplication by $(\mathrm{N}_{F/\Q}\circ\det)^{\sw_{\lambda_\pi}}$. 
\begin{notation}\label{not:v_omega}
	Let $\sv_\Omega \defeq v_{\lambda_\pi}^H \otimes 1 \in V_\Omega^H$.
\end{notation}

The following statement is an analogue of Lemma \ref{lem:diagonal action} for families. 

\begin{lemma}\label{lem:v_omega}
	Let $h \in G_n(\Zp)$. Then 
	\[
	\left\langle \smallmatrd{h}{}{}{h}\right\rangle_\Omega \cdot \sv_\Omega = \sw_\Omega(\mathrm{N}_{F/\Q}\circ \det(h))\  \sv_\Omega.
	\]
\end{lemma}

\begin{proof}
	By the definition of the action of $H(\Zp)$ on $V_{\Omega}^H$ (see \eqref{eq:action V_Omega}), we have
	\[
	\left\langle\smallmatrd{h}{}{}{h}\right\rangle_\Omega \cdot (\sv_{\lambda_\pi} \otimes 1) = \sw_{\lambda_\pi}(\mathrm{N}_{F/\Q}\circ \det(h))v_{\lambda_\pi}^H \otimes \sw_{\Omega_0}(\mathrm{N}_{F/\Q}\circ \det(h)),
	\]
	recalling 
	\disp{
	\chi_{\Omega_0}(h,h) = \sw_{\Omega_0}(\mathrm{N}_{F/\Q}\circ \det(h))
}
	from \eqref{eq:pure family}. We conclude as $\sw_\Omega = \sw_{\lambda_{\pi}}\sw_{\Omega_0}$.
\end{proof}

\begin{lemma}\label{lem:renormalise}
	If $\lambda \in \Omega$ is a classical weight, then $\mathrm{sp}_\lambda(\sv_\Omega) \in V_\lambda^H(\cO_L)$ is optimally integral, non-zero, and 
	\[
	\left\langle\smallmatrd{h}{}{}{h}\right\rangle_\lambda \cdot \mathrm{sp}_\lambda(\sv_\Omega) = (\mathrm{N}_{F/\Q}\circ\det(h))^{\sw_\lambda}\mathrm{sp}_\lambda(\sv_\Omega).
	\]
\end{lemma}
\begin{proof}
	Non-vanishing is immediate from the definition, and the action property follows from specialising Lemma \ref{lem:v_omega}.  To see $\mathrm{sp}_\lambda(\sv_\Omega)$ is integral, recall $\sv_{\lambda_\pi} \in V_{\lambda_\pi}^H(\cO_L)$ is integral, so $\sv_{\lambda_\pi}(H(\Zp)) \subset \cO_L$. Since $\lambda$ is algebraic, we have $\lambda\lambda_\pi^{-1}(H(\Zp)) \subset \cO_L^\times$. By definition $\mathrm{sp}_\lambda(\sv_\Omega) = \sv_{\lambda_\pi}\otimes \lambda\lambda_\pi^{-1}$, so we deduce 
	\[
	\mathrm{sp}_\lambda(\sv_\Omega)(H(\Zp)) \subset \cO_L.
	\]
	As $v_{\lambda_\pi}^H$ is optimally integral, it also follows that 
	\disp{
	\varpi_L^{-1}\mathrm{sp}_\lambda(\sv_\Omega)(H(\Zp)) \not \subset \cO_L,
}
	so $\mathrm{sp}_\lambda(v_\Omega^H)$ is optimally integral as claimed. 
\end{proof}

Lemma \ref{lem:renormalise} allows us to make the following renormalisation of the vectors from Notation \ref{not:v_lambda}, which aligns them in the family $\Omega$. 

\begin{definition}\label{def:v_lambda family}
	If $\lambda \in \Omega$ is a classical weight, let 
	\[
	\sv_\lambda \defeq \mathrm{sp}_\lambda(\sv_\Omega) \in V_\lambda^H(\cO_L).
	\] 
\end{definition}

\begin{remark}
	This does not change our earlier choice of $v_{\lambda_\pi}^H$, since $\mathrm{sp}_{\lambda_\pi}^H(v_\Omega^H) = v_{\lambda_\pi}^H$. 
	
	From $\sv_\lambda$, as in \S\ref{sec:choice of basis} we obtain compatible choices of $\kappaj^\circ$ as $j$ varies in $\mathrm{Crit}(\lambda)$. The definition of $\sv_\lambda$ depends only on $\sv_{\Omega}$, which depends only on the choice of $v_{\lambda_\pi}^H$. In particular, the (single) choice of $v_{\lambda_\pi}^H$ determines compatible choices of $\kappaj^\circ$ for all classical $\lambda \in \Omega$ and all $j \in \mathrm{Crit}(\lambda)$.
\end{remark}

\subsubsection{Construction of $v_\Omega^\beta$ and $\kappa_\Omega^\beta$}
Note 
\disp{
\cA(\sU_\beta,\cO_\Omega) \subset \cA(1+p^\beta\OFp, \cO_\Omega)
}
is the subset of functions invariant under $\overline{E(p^\beta)}$. Recall $N_Q^\beta(\Zp)$ from \eqref{eq:Jp minus}. We have a map 
		\[N_Q^\beta(\Zp) \longrightarrow 1+p^\beta\OFp, \qquad 	\smallmatrd{1}{X}{0}{1} \longmapsto (-1)^n\det(X).\]
Define a map
\[
v_\Omega^\beta: \cA(1+p^\beta\OFp, \cO_\Omega) \longrightarrow \cA_\Omega 
\]
as follows. For $f \in \cA(1+p^\beta\OFp, \cO_\Omega)$, define 
\[
\nubeta(f): N_Q(\Zp) \to V_{\Omega}^H
\]
by setting, for $X \in \mathrm{M}_{n}(\OFp)$, 
\begin{align}\label{eq:times}
	\nubeta(f) \smallmatrd{I_n}{X}{0}{I_n}= 
	\begin{cases} 	f\big((-1)^n\det(X)\big) \bigg(\left\langle\smallmatrd{X}{0}{0}{I_n} \right\rangle_{\Omega}\cdot\sv_\Omega\bigg)  &  : \smallmatrd{1}{X}{0}{1} \in N_Q^\beta(\Zp), \\ 
		\hfil 0 &  :\mbox{else}.
	\end{cases}
\end{align}
Extending under the parahoric decomposition using \eqref{eq:parahoric transform A} determines $\nubetaf$ as an element of $\cA_\Omega$.

\begin{definition}\label{def:kappa_beta}
Restricting $v_\Omega^\beta$ yields a map $v_\Omega^\beta : \cA(\sU_\beta, \cO_\Omega) \to \cA_\Omega$. Dualising gives a map 
	\begin{align} \label{e:distributions lambda to galois}
		\kappaObeta : \cD_\Omega &\longrightarrow \cD(\sU_\beta, \cO_\Omega).
	\end{align}
\end{definition}

\begin{remark}
	The map $\kappaObeta$, combined with Lemma \ref{lem:map to Galp[z]}, will induce the `missing' map in \eqref{eq:motivation nu}. To motivate \eqref{eq:times} and Definition \ref{def:kappa_beta}, compare to the description of $\kappaj^\circ$ in Lemma \ref{lem:nu_j}. For the support condition in  \eqref{eq:times}, note that for the outer maps of \eqref{eq:motivation nu} to commute, by Lemma \ref{lem:support} it suffices to  consider $v_\Omega^\beta(f)$ supported on $J_p^\beta$, and hence (by parahoric decomposition) on $N_Q^\beta(\Zp)$.
\end{remark}

Restricting under \eqref{eq:CFT}, we may see $\chi_{\cyc}$ as an element of $\cA(\sU_\beta,L)$, and thus make sense of $v_\lambda^\beta(\chi_{\cyc}^j)$, supported on $J_p^\beta$. Let $\lambda \in \Omega$ be classical. Recall  
\disp{
\nuj : N_Q^\times(\Zp) \to V_\lambda^H(L)
}
from Lemma \ref{lem:nu_j}. In \eqref{eq:xi_j} of this lemma, we normalise $v_\lambda^H$ as in Definition \ref{def:v_lambda family}. The following shows that $v_\lambda^\beta$ interpolates $v_{\lambda,j}$ as $j$ varies in $\mathrm{Crit}(\lambda)$, and hence interpolates branching laws in the `cyclotomic direction'.

\begin{lemma}\label{lem:kappa compatible}
	Let $\lambda \in \Omega$ be classical. For all $j \in \mathrm{Crit}(\lambda)$, we have 
	\[
	v_\lambda^\beta(\chi_{\cyc}^j)\big|_{\Jpbeta} = \nuj\big|_{\Jpbeta}.
	\]
\end{lemma}

\begin{proof}  
	If $g = n_Q^- h \smallmatrd{1}{X}{0}{1} \in \Jpbeta$, then 
	\[
	\det(X) \in (-1)^n + p^\beta\OFp\subset (\OFp)^\times,
	\]
	where the last inclusion follows as $\beta_{\pri} \geqslant 1$ for all $\pri$. Hence $	\mathrm{N}_{F/\Q}\circ\det(X) \in \Zp^\times,$ so $\smallmatrd{1}{X}{0}{1} \in N_Q^\times(\Zp).$	On such $X$ we have 
	\[
	\chi_{\cyc}^j((-1)^n\det(X)) = (-1)^{dnj}\mathrm{N}_{F/\Q}\circ\det(X)^j.
	\]
	Combining this and the definition of $v_\lambda^\beta$ with Lemma \ref{lem:nu_j}, we see that
	\[
	v_\lambda^\beta(\chi_{\cyc}^j)\left[\smallmatrd{1}{X}{0}{1}\right] = \nuj\left[\smallmatrd{1}{X}{0}{1}\right]
	\]
	for all $\smallmatrd{1}{X}{0}{1} \in N_Q^\beta(\Zp)$. We conclude $v_\lambda^\beta(\chi_{\cyc}^j)$ and $\nuj$ agree on all of $J_p^\beta$, as they satisfy the same transformation law under parahoric decomposition.
\end{proof}

We now combine $v_\Omega^\beta$ with the formalism of evaluation maps developed in \S\ref{sec:variation}.

\begin{proposition}\label{prop:Lbeta action}
	\begin{enumerate}[(i)]\setlength{\itemsep}{0pt}
		\item The action of $\ell = (\ell_1, \ell_2) \in L_\beta$ on $\cA(\Galp,\cO_\Omega)$ under \eqref{eq:action H} is by
		\begin{align}\label{eq:action of ell}
			\left[\ell * f\right](x) &\defeq [(\det(\ell_1),\det(\ell_2)) * f](x)\\ 
			&= \mathrm{N}_{F/\Q}(\det(\ell_{2,p}))^{\sw_{\Omega}} f( \det(\ell_{1,p}^{-1}\ell_{2,p}) x).\notag
		\end{align}
		It preserves $\cA(\sU_\beta,\cO_\Omega)$, giving it the structure of an $L_\beta$-module.
		\item The map $v_\Omega^\beta : \cA(\sU_\beta,\cO_\Omega) \to \cA_\Omega$	is a map of $L_\beta$-modules.
		
		\item The image of $v_\Omega^\beta$ is a subspace of the $\Gamma_{\beta,\delta}$-invariants $\cA_\Omega^{\Gamma_{\beta,\delta}}$.
		
		\item The map $\kappaObeta$ from Definition \ref{def:kappa_beta} is a map of left $L_\beta$-modules, and factors through 
		\[
		\kappaObeta : (\cD_{\Omega})_{\Gamma_{\beta,\delta}} \to \cD(\sU_\beta,\cO_\Omega).
		\]
	\end{enumerate}
\end{proposition}

\begin{proof}
	(i)  Since $\det(\ell_i) \in (\cO_F\otimes \widehat{\Z})^\times$, we have 
	\disp{
	\chi_{\cyc}(\det(\ell_2)) = \mathrm{N}_{F/\Q}(\det(\ell_{2,p}))
}
	and $\det(\ell_{i,v}) \in \sU(p^\infty)$ for all $v\nmid p\infty$. Hence 
	\[
	[\det(\ell_1^{-1}\ell_2)x] = [\det(\ell_{1,p}^{-1}\ell_{2,p}^{-1})x]
	\]
	in $\Galp$, and \eqref{eq:action ideles} induces the stated action. It preserves $\cA(\sU_\beta,\cO_\Omega)$ since $\det(\ell_{1,p}^{-1}\ell_{2,p}) \equiv 1 \newmod{p^\beta}$ by \cite[Lem.~2.1]{DJR18}.
	
	\medskip
	
	(ii) For $f \in \cA(\sU_\beta,L)$, we must show that 
	\[
	\ell* \nubetaf= \nubeta(\ell * f).
	\]
	Let $X \in \mathrm{M}_n(\OFp)$. If $\det(X) \neq (-1)^n \newmod{p^\beta}$, both sides are zero at $\smallmatrd{1}{X}{0}{1}$. If $\det(X) \equiv (-1)^n  \newmod{p^\beta}$, then
	\begin{align*}
		(\ell * \nubetaf)&\smallmatrd{I_n}{X}{0}{I_n} = \nubetaf\smallmatrd{\ell_{1,p}}{X\ell_{2,p}}{0}{\ell_{2,p}}\\
		&= 
		\left\langle\smallmatrd{\ell_{1,p}}{0}{0}{\ell_{2,p}} \right\rangle_{\Omega}\cdot \nubetaf\smallmatrd{I_n}{\ell_{1,p}^{-1}X\ell_{2,p}}{0}{I_n}\\
		&= \left\langle\smallmatrd{\ell_{1,p}}{0}{0}{\ell_{2,p}} \right\rangle_{\Omega}\cdot\left\langle\smallmatrd{\ell_{1,p}^{-1}X\ell_{2,p}}{0}{0}{I_n} \right\rangle_{\Omega}\cdot \bigg(f\big[(-1)^n\det(\ell_{1,p}^{-1}X\ell_{2,p})\big] \sv_\Omega\bigg)\\
		&= \left\langle\smallmatrd{X}{0}{0}{I_n}\right\rangle_{\Omega} \cdot \left\langle\smallmatrd{\ell_{2,p}}{0}{0}{\ell_{2,p}} \right\rangle_{\Omega} \cdot \bigg(f\big[(-1)^n\det(\ell_{1,p}^{-1}X\ell_{2,p})\big] \sv_\Omega\bigg)\\
		&= \left\langle\smallmatrd{X}{0}{0}{I_n}\right\rangle_\Omega \cdot  \bigg(\mathrm{N}_{F/\Q}(\det(\ell_{2,p}))^{\sw_\Omega} f\big[(-1)^n\det(\ell_{1,p}^{-1}\ell_{2,p})\det(X)\big] \sv_\Omega\bigg)\\
		&= \left\langle\smallmatrd{X}{0}{0}{I_n}\right\rangle_\Omega \cdot  \bigg((\ell* f)\big[(-1)^n\det(X)\big] \sv_\Omega\bigg)\\
		&= \nubeta(\ell * f)\smallmatrd{I_n}{X}{0}{I_n},
	\end{align*}
	proving (ii); the first equality is the $*$-action, the second is \eqref{eq:parahoric transform A}, the third is \eqref{eq:times}, the fifth is Lemma \ref{lem:v_omega}, the sixth by (i), and the last is \eqref{eq:times}. 
	
	\medskip
	
	(iii) Note $\Gamma_{\beta,\delta} \subset H(\Q)$ acts trivially on $\cA(\Galp,\cO_\Omega)$ (see \eqref{eq:action H}). Hence $\delta^{-1}\Gamma_{\beta,\delta}\delta \subset L_\beta$ acts trivially on $\cA(\Galp,\cO_\Omega)$, hence trivially on $\cA(\sU_\beta,\cO_\Omega)$. From (ii), it follows that $\delta^{-1}\Gamma_{\beta,\delta}\delta$ -- acting as a subgroup of $L_\beta$ -- acts trivially on the image of $v_\Omega^\beta$. But by definition of the $\Gamma_{\beta,\delta}$-action (see \eqref{eq:gamma action}), this means $\Gamma_{\beta,\delta}$ acts trivially on this image.
	
	\medskip
	
	(iv) That $\kappaObeta$ is a map of $L_\beta$-modules follows from (ii), and thus it factors through $(\cD_\Omega)_{\Gamma_{\beta,\delta}}$ since the target is $\Gamma_{\beta,\delta}$-invariant by (iii).
\end{proof}

\subsubsection{Proof that $\kappa_\Omega^\beta$ interpolates $\kappaj^\circ$}
The following is the main result of \S\ref{sec:p-adic branching}.

\begin{proposition}\label{prop:kappa interpolate}
	Let $\lambda \in \Omega$ classical and $j \in \mathrm{Crit}(\lambda)$. The following diagram commutes:
	\begin{equation}\label{eq:kappa interpolate}
		\xymatrix@C=20mm@R=7mm{
			\hc{t}(S_K,\sD_\Omega) \ar[r]^-{\mathrm{Ev}_{\beta,\delta}^{\cD_\Omega}}\ar[d]^{\mathrm{sp}_\lambda} & (\cD_\Omega)_{\Gamma_{\beta,\delta}} \ar[r]^{\kappaObeta}\ar[d]^{\mathrm{sp}_\lambda} & \cD(\sU_\beta,\cO_\Omega)\ar[d]^{\mathrm{sp}_\lambda}\\
			\hc{t}(S_K,\sD_\lambda) \ar[r]^-{\mathrm{Ev}_{\beta,\delta}^{\cD_\lambda}} \ar[d]^{r_\lambda} & (\cD_\lambda)_{\Gamma_{\beta,\delta}} \ar[r]^{\kappalbeta} & \cD(\sU_\beta,L)\ar[d]^{\int_{\sU_\beta} \chi_{\cyc}^j }\\
			\hc{t}(S_K,\sV_\lambda^\vee) \ar[r]^-{\mathrm{Ev}_{\beta,\delta}^{V_\lambda^\vee}} & (V_\lambda^\vee)_{\Gamma_{\beta,\delta}} \ar[r]^{\kappaj^\circ} & L.
		}
	\end{equation}
\end{proposition}
\begin{proof}
	The top left-hand square commutes by Lemma \ref{prop:pushforward}. We next consider the top-right square. In the definition of $\kappaObeta$, note by definition that $\mathrm{sp}_\lambda(\sv_\Omega) = \sv_\lambda$ and the action $\langle \cdot \rangle_\Omega$ specialises to $\langle \cdot \rangle_\lambda$ under $\mathrm{sp}_\lambda$. In particular, if $f_\lambda \in \cA(\sU_\beta,L)$ and $f_\Omega \in \cA(\sU_\beta,\cO_\Omega)$ is any lift under $\mathrm{sp}_\lambda$, then 
	\[
	\mathrm{sp}_\lambda[\nubeta(f_\Omega)]= v_\lambda^\beta(f_\lambda) \in \cA_\lambda.
	\] 
	
	We describe the map $\mathrm{sp}_\lambda : \cD_\Omega \to \cD_\lambda$ directly. Let $\mu_\Omega \in \cD_\Omega$, and $g_{\lambda} \in \cA_{\lambda}$. Choose any $g_\Omega \in \cA_\Omega$ with $\mathrm{sp}_\lambda(g_\Omega) = g_{\lambda}$. Then 
	\[
	[\mathrm{sp}_{\lambda}(\mu_\Omega)](g_{\lambda})= \mathrm{sp}_\lambda[\mu_\Omega(g_{\Omega})] \in \cO_{\Omega}/\m_{\lambda},
	\]
	which is easily seen to be independent of lift. Here $\m_{\lambda} \subset \cO_\Omega$ is the maximal ideal attached to $\lambda$. In particular, for $f_\lambda,f_\Omega$ as above, 
	\[
	\mathrm{sp}_\lambda(\mu_\Omega)[v_\lambda^\beta(f_\lambda)] = \mathrm{sp}_\lambda[\mu_\Omega(\nubeta(f_\Omega))].
	\]
	
	Let $\mu_\Omega \in (\cD_\Omega)_{\Gamma_{\beta,\delta}}$; then the top right square commutes as
	\begin{align*}
		[\kappalbeta \circ \mathrm{sp}_\lambda(\mu_\Omega)](f_\lambda) &= \mathrm{sp}_\lambda(\mu_\Omega)\big[v_\lambda^\beta(f_\lambda)\big] = \mathrm{sp}_\lambda\big[\mu_\Omega\big(\nubeta(f_\Omega)\big)\big]\\
		&= \mathrm{sp}_\lambda\big[(\kappa_\Omega^\beta(\mu_\Omega))(f_\Omega) \big] = [\mathrm{sp}_\lambda\circ \kappaObeta(\mu_\Omega)\big](f_\lambda).
	\end{align*}
	We have used the previous paragraph in the second equality.
	
	Now consider the bottom rectangle. If we `complete' \eqref{eq:kappa interpolate} with the natural map $r_\lambda : (\cD_\lambda)_{\Gamma_{\beta,\delta}} \to (V_\lambda^\vee)_{\Gamma_{\beta,\delta}}$, then the bottom-left square would commute by Lemma \ref{prop:pushforward}, but the bottom-right square would \emph{not} commute, due to support conditions. However if $\mu \in \mathrm{Im}(\mathrm{Ev}_{\beta,\delta}^{\cD_\lambda})$, then $\mu$ is supported on $\Jpbeta$ by Lemma~\ref{lem:support}. For such $\mu$, compute
	\begin{align*}
		\int_{\sU_\beta}\chi_{\cyc}^j \cdot \kappalbeta(\mu) &= \int_{J_p}v_\lambda^\beta(\chi_{\cyc}^j) \cdot \mu \\
		&= \int_{\Jpbeta}v_\lambda^\beta(\chi_{\cyc}^j) \cdot \mu = \int_{\Jpbeta}\nuj \cdot \mu = \int_{G(\zp)}\nuj \cdot r_\lambda(\mu)
	\end{align*}
	(recalling from Lemma \ref{lem:nu_j} that $\kappaj^\circ$ is evaluation at $\nuj$). In the second equality, we use that $\mu$ has support on $\Jpbeta$, whence the third equality follows from Lemma~\ref{lem:kappa compatible}. In the last, because $\nuj \in V_\lambda$ we have $\mu(\nuj) = r_\lambda(\mu)(\nuj)$, and then we expand from $J_p^\beta$ to $G(\Zp)$ using that $\mu$ (hence $r_\lambda(\mu))$ is supported on $\Jpbeta$ again. Thus the bottom-right square is commutative on the image of $\mathrm{Ev}_{\beta,\delta}^{\cD_\lambda}$, and the bottom rectangle is commutative.
\end{proof}

\subsection{Distribution-valued evaluation maps}\label{sec:distribution evaluations}
We now define overconvergent analogues of $\cE_{\chi}^{j,\eta_0}$.  Let
	 $\delta = (\delta_1, \delta_2) \in H(\A)$, let $[\delta]$ be its class in $\pi_0(X_\beta)$, and  let $\mathbf{x} = \mathrm{pr}_\beta([\delta]) \in \Cl(p^{\beta})$, for $\mathrm{pr}_\beta$ as in \eqref{eq:pr_beta}. As above, $\det(\delta_{1}\delta_{2}^{-1}) \in \A_F^\times$ is a representative of $\mathbf{x}$. Recall the evaluation map $\mathrm{Ev}_{\beta,\delta}^{\cD_{\Omega}}$ from \S\ref{sec:coinvariants}, and define a `Galois evaluation' $\mathrm{Ev}_{\beta, [\delta]}$
as the composition 
	\begin{align}
		\mathrm{Ev}_{\beta, [\delta]} : \hc{t}(S_K,\sD_\Omega) \xrightarrow{\ \mathrm{Ev}_{\beta,\delta}^{\cD_{\Omega}}\ } (\cD_\Omega)_{\Gamma_{\beta,\delta}} \xrightarrow{\  \kappaObeta\  } \cD(\sU_\beta,\cO_\Omega)\xrightarrow{\ \mu\  \mapsto\  \delta * \mu \ }
		\cD(\Galp[\mathbf{x}],\cO_\Omega).\notag
	\end{align}
Here the action of $\delta$ on $\mu$ is by \eqref{eq:action H}, the map $\kappa_\beta^\Omega$ was defined in Definition \ref{def:kappa_beta}, and the target is $\cD(\Galp[\mathbf{x}],\cO_\Omega)$ by Lemma~\ref{lem:map to Galp[z]}.

\begin{lemma} 
	\label{l:evaluations independent of delta} 
	$\mathrm{Ev}_{\beta, [\delta]}$ is independent of the choice of the representative $\delta$ of $[\delta] \in  \pi_0(X_\beta)$.
\end{lemma}
\begin{proof}
	Recall $\cD(\Galp,\cO_\Omega)$ is a $H(\A)$-module via \eqref{eq:action H}, with $H(\Q)$ and $H_\infty^\circ$ acting trivially. Let
	\[
	\kappa : \cD_\Omega\xrightarrow{\ \kappaObeta\ } \cD(\sU_\beta,\cO_\Omega) \labelisorightarrow{ \ \iota_*\ } \cD(\Gal[\mathbf{1}_{\beta}], \cO_\Omega) \subset \cD(\Galp,\cO_\Omega)
	\]
	denote the composition. From Proposition \ref{prop:ind of delta}, we have a map 
	\disp{
	\mathrm{Ev}_{\beta,[\delta]}^{\cD_{\Omega},\kappa}: \hc{t}(S_K,\cD_\Omega) \to \cD(\Galp,\cO_\Omega).
}
	If $\Phi \in \hc{t}(S_K,\sD_\Omega)$ then by definition we have 
	\[
	\mathrm{Ev}_{\beta,[\delta]}^{\cD_{\Omega},\kappa}(\Phi) = \delta * [\kappa \circ \mathrm{Ev}_{\beta,\delta}^{\cD_\Omega}(\Phi)] = \mathrm{Ev}_{\beta,[\delta]}(\Phi).
	\]
	Then independence of $\delta$ follows from Proposition~\ref{prop:ind of delta}.
\end{proof}

As in Definition~\ref{def:ev chi}, let $\eta_0$ be any finite order character of $\Cl(\m)$. Then define 
\begin{align*}
	\mathrm{Ev}_{\beta, \mathbf{x}}^{\eta_0} : \htc(S_K, \sD_\Omega) &\longrightarrow \cD(\Galp[\mathbf{x}], \cO_\Omega)\\
	\Phi &\longmapsto \sum_{[\delta] \in \mathrm{pr}_{\beta}^{-1}(\mathbf{x})}\eta_0^{-1}\big(\mathrm{pr}_2([\delta])\big) \  \mathrm{Ev}_{\beta, [\delta]}(\Phi).
\end{align*}
Using (\ref{e:decomposition galois distributions}), we finally obtain an evaluation map
\begin{align}\label{eq:final evaluation}
	\mathrm{Ev}_{\beta}^{\eta_0} \defeq \bigoplus_{\mathbf{x} \in \Cl(p^{\beta})}& \mathrm{Ev}_{\beta, \mathbf{x}}^{\eta_0} : \htc(S_K, \sD_\Omega) \longrightarrow \cD(\Galp, \cO_\Omega)\\
	\Phi\   &\longmapsto \sum_{[\delta] \in \pi_0(X_\beta)} \eta_0^{-1}\big(\mathrm{pr}_2([\delta])\big) \  \times \left(\delta * \left[\kappaObeta \circ \mathrm{Ev}_{\beta,\delta}^{\cD_\Omega}(\Phi)\right]\right).\notag
\end{align}

\begin{remark}\label{rem:overconvergent diagram}
	In the notation of Remark~\ref{rem:classical diagram}, $\mathrm{Ev}_{\beta}^{\eta_0}$ is the composition
	\begin{equation}\label{eq:explicit overconvergent}
		\xymatrix@R=10mm@C=1mm{
			\hc{t}(S_K,\sD_\Omega) \ar@/^3pc/[rrrrrrrrr]_-{\oplus \mathrm{Ev}_{\beta,[\delta]}} \ar@/_1.5pc/[rrrrrrrrrd]_-{ \mathrm{Ev}_{\beta}^{\eta_0}\ \ \ \ \ }\ar[rrrr]^-{\oplus\mathrm{Ev}_{\beta,\delta}^{\cD_\Omega}}  &&&& \displaystyle\bigoplus_{[\delta]}(\cD_\Omega)_{\Gamma_{\beta,\delta}} \ar[rrrrr]^-{\delta *\kappaObeta} &&&&&
			\displaystyle\bigoplus_{[\delta]}\cD(\Galp[\mathrm{pr}_\beta([\delta])],\cO_\Omega) \ar[d]^-{\Sigma_{\mathbf{x}}\Xi_{\mathbf{x}}^{\eta_0}} \\
			&&&&&&&&&\cD(\Galp, \cO_\Omega),
		}
	\end{equation}
	where again $\Xi_{\mathbf{x}}^{\eta_0}$ sends a tuple $(m_{[\delta]})_{[\delta]}$ to $\sum_{[\delta]\in\mathrm{pr}_\beta^{-1}(\mathbf{x})} \eta_0^{-1}(\mathrm{pr}_2([\delta])) \times m_{[\delta]}$.
\end{remark}

The maps $\mathrm{Ev}_\beta^{\eta_0}$ are functorial in $\Omega$. Let $\lambda \in \Omega$, and let $\mathrm{sp}_{\lambda} : \cO_\Omega \to L$ denote evaluation at $\lambda$. 

\begin{proposition}\label{prop:evaluations in families}
	Let $\beta = (\beta_{\pri})_{\pri|p}$ with $\beta_{\pri} > 0$ for each $\pri|p$. We have a commutative diagram
	\[
	\xymatrix@C=10mm{
		\hc{t}(S_K,\sD_\Omega) \ar[r]^-{\mathrm{Ev}_\beta^{\eta_0}}\ar[d]^{\mathrm{sp}_\lambda} & \cD(\Galp,\cO_\Omega)\ar[d]^{\mathrm{sp}_\lambda}\\
		\hc{t}(S_K,\sD_\lambda) \ar[r]^-{\mathrm{Ev}_\beta^{\eta_0}} & \cD(\Galp,L).
	}
	\]
\end{proposition}

\begin{proof}
	We check that every square in the following diagram is commutative, where the horizontal maps are as in Remark~\ref{rem:overconvergent diagram} (with the middle horizontal maps a composition of two of the maps in that remark) and every vertical map is induced from $\mathrm{sp}_\lambda$:
	\[
	\xymatrix@C=1mm@R=5mm{
		\hc{t}(S_K,\sD_\Omega) \ar[rrr] \ar[d] &&&
		\bigoplus_{[\delta]}\cD(\sU_\beta,\cO_\Omega) \ar[rr]\ar[d] &&
		\bigoplus_{\mathbf{x}} \cD(\Galp[\mathbf{x}],\cO_\Omega) \ar[rr]\ar[d]&&
		\cD(\Galp,\cO_\Omega)\ar[d]\\
		\hc{t}(S_K,\sD_\lambda) \ar[rrr] &&&
		\bigoplus_{[\delta]}\cD(\sU_\beta,L) \ar[rr] &&
		\bigoplus_{\mathbf{x}} \cD(\Galp[\mathbf{x}],L) \ar[rr]&&
		\cD(\Galp,L).
	}
	\]
	The first square commutes by Proposition \ref{prop:kappa interpolate}. The second horizontal arrows are induced by $\delta*-$, and $\mathrm{sp}_\lambda$ is $H(\A)$-equivariant, so the second square commutes. The remaining horizontal maps are given by taking linear combinations, which commutes with $\mathrm{sp}_\lambda$. 
\end{proof}

\begin{proposition}
	\label{p:compatibility galois evaluations} 
	Let $\beta \in (\Z_{\geqslant 0})_{\pri|p}$ and fix $\fp | p$ in $F$. Suppose that $\beta_{\mathfrak{q}}> 0$ for each $\mathfrak{q} | p$ and let $\beta'$ be the tuple defined by $\beta'_{\fp}= \beta_{\fp}+ 1$ and $\beta'_{\mathfrak{q}}= \beta_{\fq}$ for each prime $\mathfrak{q}|p$ other than $\fp$. Then 
	\[
	\mathrm{Ev}_{\beta'}^{\eta_0} = \mathrm{Ev}_{\beta}^{\eta_0}\circ U_{\fp}^\circ \ : \ \hc{t}(S_K,\sD_\Omega) \longrightarrow \cD(\Galp,\cO_\Omega).
	\]
\end{proposition}
\begin{proof} 
	For each $[\delta] \in \pi_0(X_\beta)$, from Proposition~\ref{prop:evaluations changing beta} we deduce 
	\disp{
	\sum_{[\delta'] \in \mathrm{pr}_{\beta,\pri}^{-1}([\delta])} \mathrm{Ev}_{\beta',[\delta']} = \mathrm{Ev}_{\beta,[\delta]} \circ U_{\pri}^\circ.
}
	Scaling the left-hand side by $\eta_0^{-1}(\mathrm{pr}_2([\delta]))$ and summing over $[\delta] \in \pi_0(X_\beta)$ gives $\mathrm{Ev}_{\beta'}^{\eta_0}$ (see \eqref{eq:final evaluation}), and doing the same on the right-hand side gives $\mathrm{Ev}_{\beta}^{\eta_0}\circ U_{\pri}^\circ$, from which we conclude.
\end{proof}

\begin{definition}\label{def:mu_Phi}
	Let $\Phi \in \htc(S_K, \sD_\Omega)$, and suppose that for every $\fp|p$, $\Phi$ is an eigenclass for $U_{\fp}^\circ$ with eigenvalue $\alpha_{\fp}^\circ \neq 0$. We define 
	\begin{equation}\label{eq:mu eta}
		\mu^{\eta_0}(\Phi) := \mathrm{Ev}_{\beta}^{\eta_0}(\Phi)/(\alpha_{p}^\circ)^\beta \in \cD(\Galp, \cO_\Omega),
	\end{equation}
	where $\beta$ is any tuple such that $\beta_{\fp}> 0$ for each $\fp | p$ and $(\alpha_{p}^\circ)^\beta:= \prod_{\fp | p}(\alpha_{\fp}^\circ)^{\beta_{\fp}}$. By Proposition~\ref{p:compatibility galois evaluations}, the distribution $\mu^{\eta_0}(\Phi)$ is independent of the choice of $\beta$.
\end{definition}

\subsection{Interpolation of classical evaluations}\label{sec:classical interpolation}

Fix $\lambda \in X_0^*(T)$. Via specialisation $\hc{t}(S_K,\sD_\lambda) \xrightarrow{r_\lambda} \hc{t}(S_K,\sV_\lambda^\vee)$, we now relate $\mathrm{Ev}_{\beta}^{\eta_0}$ from \eqref{eq:explicit overconvergent} to the evaluations $\cE_{\chi}^{j,\eta_0}$ of \eqref{eq:explicit classical} as $\lambda$ varies over $\Omega$, $j$ varies in $\mathrm{Crit}(\lambda)$ and $\chi$ varies over finite order characters of conductor $p^\beta$.

\begin{lemma}\label{lem:compatibility}
	Let $\Phi \in \hc{t}(S_K,\sD_\lambda)$. Let $\chi$ be a finite order Hecke character of $F$ of conductor $p^\beta$, with $\beta_{\pri} > 0$ for all $\pri|p$. For all $j \in \mathrm{Crit}(\lambda)$, we have 
	\[
	\int_{\Galp}\chi\ \chi_{\cyc}^j \cdot \mathrm{Ev}_{\beta}^{\eta_0}(\Phi) = \cE_{\chi}^{j,\eta_0}\circ r_\lambda(\Phi).
	\]
\end{lemma}
\begin{proof}
	In view of Remarks~\ref{rem:overconvergent diagram} and \ref{rem:classical diagram}, the lemma follows directly from commutativity of the following diagram, since the maps $\mathrm{Ev}_{\beta}^{\eta_0}$ and $\cE_{\chi}^{j,\eta_0}$ are respectively the left and right columns:
	\[
	\xymatrix@C=40mm{
		\hc{t}(S_K,\sD_\lambda) \ar[d]^-{\oplus \big(\kappalbeta \circ\mathrm{Ev}_{\beta,\delta}^{\cD_\lambda}\big)} \ar[r]^{r_\lambda} & \hc{t}(S_K,\sV_\lambda^\vee) \ar[d]^-{\oplus\big(\kappaj^\circ \circ\mathrm{Ev}_{\beta,\delta}^{V_\lambda^\vee}\big)}\\
		\bigoplus\limits_{[\delta]}\cD(\sU_\beta,L) \ar[d]^-{\delta*}\ar[r]^{\oplus \int_{\sU_\beta}\chi_{\cyc}^j} & \bigoplus\limits_{[\delta]}L  \ar[d]^-{\delta*}\\
		\bigoplus\limits_{[\delta]}\cD(\Galp[\mathbf{x}],L) \ar[d]^-{\oplus\Xi_{\mathbf{x}}^{\eta_0}}\ar[r]^{\oplus \int_{\Galp[\mathbf{x}]}\chi_{\cyc}^j} & 	\bigoplus\limits_{[\delta]}L \ar[d]^-{\oplus\Xi_{\mathbf{x}}^{\eta_0}} \\
		\bigoplus\limits_{\mathbf{x}} \cD(\Galp[\mathbf{x}],L) \ar[d]^-{\Sigma}\ar[r]^{\oplus \int_{\Galp[\mathbf{x}]}\chi_{\cyc}^j}& \bigoplus\limits_{\mathbf{x}} L \ar[d]^-{\ell \mapsto \Sigma \chi(\mathbf{x})\ell_{\mathbf{x}}}\\
		\cD(\Galp ,L)\ar[r]^{\int_{\Galp}\chi\chi_{\cyc}^j} & L.
	}
	\]
	Recall $\Xi_{\mathbf{x}}^{\eta_0}$ was defined in \eqref{eq:explicit classical} and all the direct sums are over $[\delta] \in \pi_0(X_\beta)$ or $\mathbf{x}\in \Cl(p^\beta)$, related by $\mathbf{x} = \mathrm{pr}_\beta(\delta)$.
	The first square commutes by Proposition \ref{prop:kappa interpolate}. The second square commutes since for $\mu \in \cD(\sU_\beta ,L)$ we have
		\begin{align*}
			\int_{\Galp[\z]}\chi_{\cyc}^j \cdot \delta* \mu
			&= \chi_{\cyc}\bigg(\det(\delta_{1})^{j} \det(\delta_{2})^{-\sw-j}\bigg) \int_{\Galp[\z]}\chi_{\cyc}^j \cdot \mu = \delta * \int_{\Galp[\z]}\chi_{\cyc}^j \cdot \mu,
		\end{align*}
	where the action of $\delta$ in the left-hand term is \eqref{eq:action ideles}, and in the right-hand term is Definition \ref{def:action on Vjw}. The third square commutes by definition of $\Xi_{\mathbf{x}}^{\eta_0}$, and the fourth square commutes since
	\[
	\int_{\Galp}\chi \chi_{\cyc}^j \cdot \mu = \sum_{\mathbf{x}\in \Cl(p^\beta)} \chi(\mathbf{x}) \int_{\Galp[\mathbf{x}]}\chi_{\cyc}^j \cdot \mu. \qedhere 
	\]
\end{proof}

\subsection{Admissibility of $\mu^{\eta_0}(\Phi)$}\label{sec:admissible}
We now assume  $\Omega = \{\lambda\}$ is a single algebraic weight, in which case $\cO_\Omega = L$ is a finite extension of $\Qp$.  Let $\Phi$ and  $\{\alpha_{\pri}^\circ : \pri|p\}$ be as in Definition \ref{def:mu_Phi}, and let 
\[
\alpha_p^\circ \defeq \prod_{\pri|p}(\alpha_{\pri}^\circ)^{e_{\pri}}, \qquad h_p = v_p(\alpha_{p}^\circ).
\]
We show $\mu^{\eta_0}(\Phi)$ satisfies a growth condition depending on $h_p$ that importantly renders it unique for the very small slope case $h_p < \#\mathrm{Crit}(\lambda)$. 

As in \cite[\S3.4]{BDJ17}, the space $\cA(\Galp,L)$ of $L$-valued locally analytic functions on $\Galp$ is the direct limit $\varinjlim_m \cA_m(\Galp,L)$ of the spaces which are analytic on all balls of radius $|p|^{-m}$, and each of these is a Banach $L$-space with respect to a discretely valued norm $||\cdot||_m$. Dualising, we get a family of norms 
	\begin{align}
		||\mu||_m &\defeq \mathrm{sup}_{f\in\cA_m(\Galp,L)}\tfrac{|\mu(f)|}{||f||_m} = \mathrm{sup}_{||f||_m \leqslant 1}|\mu(f)|\label{eq:sup norm}
	\end{align}
on $\cD(\Galp,L)$, which thus obtains the structure of a Fr\'{e}chet module.

\begin{definition}(See \cite[Def.~3.10]{BDJ17}). \label{def:admissible}
	Let $h \in \Q_{\geqslant 0}$. We say $\mu \in \cD(\Galp,L)$ is \emph{admissible of growth $h$} if there exists $C \geqslant 0$ such that for each $m \in \Z_{\geqslant 1},$ we have $||\mu||_m \leqslant p^{mh}C.$
\end{definition}

\begin{proposition}\label{prop:admissible}
	Let $\Phi$ be as in Definition~\ref{def:mu_Phi}, and $h_p = v_p(\alpha_p^\circ)$. Then $\mu^{\eta_0}(\Phi)$ is admissible of growth $h_p$.
\end{proposition}

\begin{proof} We follow the proof of \cite[Prop.~3.11]{BDJ17}, where this is proved for $\GLt$. For $m\in  \Z_{\geqslant 1}$, put 
	\disp{
	\beta_m= (me_{\pri})_{\pri\mid p},
}
	so that 
	\[
	|(\alpha_p^\circ)^{-\beta_m}| = p^{mh_p} \qquad \text{ and } \qquad p^{\beta_m}\OFp = p^m\OFp.
	\]
	By definition of $\mu^{\eta_0}$, for $f \in \cA_m(\Galp,L)$ we have
			\begin{align}
			|\mu^{\eta_0}&(\Phi)(f)| = p^{mh_p}\bigg|  \sum_{[\delta] \in \pi_0(X_{\beta_m})} \eta_0^{-1}\big(\mathrm{pr}_2([\delta])\big)  \mathrm{Ev}_{\beta_m,\delta}^{\cD_\lambda}(\Phi)\left[v_\lambda^{\beta_m}\left(\delta^{-1}\ast f\big|_{\Galp[\mathrm{pr}_{\beta_m}([\delta])]}\right)\right]\bigg|.\label{eq:eta admissible}
		\end{align}
	By \eqref{eq:eta admissible} it suffices to find $C$ such that for all $\delta$, $m \in \Z_{\geqslant 1}$ and $f \in \cA_m(\Galp,L)$ with $||f||_m \leqslant 1$, we have 
	\begin{equation}\label{eq:bound C}
		\left|\mathrm{Ev}_{\beta_m,\delta}^{\cD_\lambda}(\Phi)\left[v_\lambda^{\beta_m}\left(\delta^{-1}\ast f\big|_{\Galp[\mathrm{pr}_{\beta_m}([\delta])]}\right)\right]\right| \leqslant C.
	\end{equation}
	
	We also have descriptions $\cA_\lambda = \varinjlim_m \cA_{\lambda,m}$ and $\cD_\lambda = \varprojlim_m \cD_{\lambda,m}$ as limits of Banach spaces (see \cite[\S3.2.2]{BW20}), and each of the $\cD_{\lambda,m}$ are preserved by the action of $\Delta_p$ (\S3.4 \emph{ibid}.). For every $m \geqslant 1$, restriction from $\cA_{\lambda}$ to $\cA_{\lambda,m}$ induces a map $\cD_{\lambda} \rightarrow \cD_{\lambda,m}$. We let $\cD_{\lambda,m}^{\circ}$ denote the $\cO_L$-module of distributions $\mu \in \cD_{\lambda, m}$ with $||\mu||_m \leqslant 1$, which is a lattice preserved by the action of $\Delta_p$. Note that rescaling $\Phi$ does not affect admissibility (it rescales $C$); so without loss of generality, we can suppose that the image $\Phi_1$ of $\Phi$ in $\htc(S_K, \sD_{\lambda, 1})$ is contained in the image of $\htc(S_K, \sD^{\circ}_{\lambda, 1})$, that is, there exists $\Phi_1^\circ$ such that we have
	\[
	\xymatrix{
		& \htc(S_K,\sD_{\lambda,1}^\circ) \ar[d]  				&& \Phi_1^\circ \ar@{|->}[d] \\
		\htc(S_K, \sD_{\lambda}) \ar[r] & \htc(S_K, \sD_{\lambda, 1})& \Phi \ar@{|->}[r] &\Phi_1
	}
	\]
	
	Fix $\delta$ and $m \in \Z_{\geqslant 1}$. For ease of notation, let 
	\[
	t_m \defeq \xi t_p^{\beta_m}, \qquad \Gamma_m \defeq \Gamma_{\beta_m,\delta}.
	\] As in the proof of  Lemma~\ref{lem:support} there exist $\mu \in \cD_{\lambda}$ and $\mu_1\in \cD_{\lambda, 1}^{\circ}$ such that 
	\begin{align*}
		\mathrm{Ev}_{\beta_m,\delta}^{\cD_\lambda}(\Phi) &= (t_m\ast \mu)_\delta \\ 
		\text{and} \qquad \mathrm{Ev}_{\beta_m,\delta}^{\cD_{\lambda, 1}}(\Phi_{1}) = \mathrm{Ev}_{\beta_m,\delta}^{\cD_{\lambda, 1}^\circ}(\Phi_{1}^\circ) & = (t_m\ast \mu_1)_\delta,
	\end{align*}
	where in the second equation, we have applied Lemma~\ref{prop:pushforward} with $\kappa$ the inclusion $\cD_{\lambda,1}^\circ \hookrightarrow \cD_{\lambda,1}$.  By Lemma~\ref{prop:pushforward} applied again, now with $\kappa$ the map $\cD_\lambda \to \cD_{\lambda,1}$, we deduce 
	\begin{equation}\label{eq:mu = mu1}
		\mu\Big|_{ t_m^{-1}\ast \cA_{\lambda, 1}^{\Gamma_m}}= \mu_1\Big|_{ t_m^{-1}\ast \cA_{\lambda, 1}^{\Gamma_m}}.
	\end{equation}
	Note if $g \in \cA_{\lambda,m}$, then by definition $g$ is analytic on 
	\[
	\{\smallmatrd{1}{X}{0}{1} : X \in -I_n + p^mM_n(\OFp)\} \subset N_Q(\Zp).
	\]
	Since the action of $t_m$ sends $N_Q(\zp)$ onto this subset (see \eqref{eq:action of t_p on X}), we have 
	\[
	t_m^{-1} * g \in t_m * \cA_{\lambda,m} \subset \cA_{\lambda,0} \subset \cA_{\lambda,1}
	\]
	(i.e.\ $t_m$ sends $m$-analytic functions to analytic functions). As $v_{\lambda}^{\beta_m}$ preserves $m$-analyticity, we thus have 
	\[
	t_m^{-1} * v_\lambda^{\beta_m}[\cA_m(1+p^m\OFp,L)] \subset \cA_{\lambda,1},
	\]
	and we can evaluate $\mu_1$ on this set. Then:
	
	\begin{claim}\label{claim:admissible}
		We have
		\begin{equation}\label{eq:mu = mu1 claim}
			\mu\Big|_{ t_m^{-1}\ast v_\lambda^{\beta_m}\big[\cA_m(1+p^m\OFp,L)\big]}= \mu_1\Big|_{t_m^{-1}\ast  v_\lambda^{\beta_m}\big[\cA_m(1+p^m\OFp,L)\big]}.
		\end{equation}
	\end{claim}
	
	We explain how Proposition \ref{prop:admissible} follows from the claim. For $f$ as above, let $f_\delta \defeq \delta^{-1}\ast f|_{\Galp[\mathrm{pr}_{\beta_m}([\delta])]}$. As in Lemma \ref{lem:map to Galp[z]}, we have 
	\[
	f_\delta \in \cA_m(\sU_{\beta_m},L)\subset \cA_m(1+p^m\OFp,L).
	\]
	Moreover $||t_m^{-1} * v_\lambda^{\beta_m}(f_\delta)||_1 \leqslant 1$: indeed	$||f||_m \leqslant 1$ by assumption; the action of $\delta^{-1}$ preserves integrality (as $\chi_{\cyc}$ is valued in $\Zp^\times$); $v_{\lambda}^{\beta_m}$ preserves integrality (as $v_{\lambda}^H$ was chosen integral); and $t_m^{-1}$ preserves integrality (as it acts only on the argument). 	Thus
	\begin{align*}
		\left|\mathrm{Ev}_{\beta_m,\delta}^{\cD_\lambda}(\Phi)\left[v_\lambda^{\beta_m}\left(\delta^{-1}\ast f\big|_{\Galp[\mathrm{pr}_{\beta_m}([\delta])]}\right)\right]\right| &=  \big| \mu\big(t_m^{-1} * v_{\lambda}^{\beta_m}(f_\delta)\big)\big|\\
		&= \big| \mu_1\big(t_m^{-1} * v_{\lambda}^{\beta_m}(f_\delta)\big)\big| \leqslant ||\mu_1||_1 \leqslant 1,
	\end{align*}
	where the first equality is by definition, the second is Claim \ref{claim:admissible}, the third inequality is by definition of $||\cdot||_1$ on $\cD_{\lambda,1}$ (the direct analogue of \eqref{eq:sup norm}) using $||t_m^{-1} * v_{\lambda}^{\beta_m}(f_\delta)||_1 \leqslant 1$, and the last inequality follows as $\mu_1 \in \cD_{\lambda,1}^\circ$. Since $\delta, m$ and $f$ were arbitrary, this shows \eqref{eq:bound C} and thus Proposition \ref{prop:admissible}. \medskip
	
	It remains to prove Claim \ref{claim:admissible}. We first motivate the statement, in line with the proof of \cite[Prop.\ 3.11]{BDJ17}. We might aim to prove the stronger statement that $\mu$ and $\mu_1$ agree on the set $t_m^{-1} * \cA_{\lambda,m}^{\Gamma_m}$ (which contains $t_m^{-1}*v_\lambda^{\beta_m}[\cA_m(1+p^m\OFp,L)]$); and to do this, it would suffice to show 
	\[
	t_m^{-1}*\cA_{\lambda,1}^{\Gamma_m} \subset t_m^{-1}*\cA_{\lambda,m}^{\Gamma_m}
	\]
	is dense, whence equality would follow from \eqref{eq:mu = mu1}. However it is not clear how to write down explicit bases of $\cA_{\lambda,m}^{\Gamma_m}$. Instead we essentially prove an analogous density for the smaller, but still sufficient, subset in the claim, using explicit bases for $\cA_m(1+p^m\OFp,L)$.
	
	We have coordinates $z = (z_\sigma)_{\sigma\in\Sigma}$ on $\OFp$. Note $m$-analytic functions on $1+p^m\OFp$ are analytic, and an orthonormal basis for $\cA_m(1+p^m\OFp,L)$ is given by the monomials 
	\[
	y_m^i \defeq \left.\left(\frac{z-1}{p^{\beta_m}}\right)^i\right|_{1+p^m\OFp} = \prod_{\sigma\in\Sigma}\left.\left(\frac{z_\sigma - 1}{\pi_{\pri(\sigma)}^{e_{\pri(\sigma)}m}}\right)^{i_{\sigma}}\right|_{1+p^m\OFp}
	\]
	for $i = (i_\sigma) \in \N[\Sigma]$. First we show that for any $i$, we have
	\begin{equation}\label{eq:mu = mu1 2}
		\mu\left(t_m^{-1} * v_{\lambda}^{\beta_m}(y_1^i)\right) = \mu_1\left(t_m^{-1} * v_{\lambda}^{\beta_m}(y_1^i)\right).
	\end{equation}
	To see this, note that $v_{\lambda}^{\beta_1}(y_1^i) \in \cA_{\lambda,1}^{\Gamma_1}$ exactly as in Proposition \ref{prop:Lbeta action}(iii), and we also have
	\begin{align*}
		t_m^{-1} * v_\lambda^{\beta_1}(y_1^i) &= t_m^{-1} * \left[\left.\Big(v_\lambda^{\beta_1}(y_1^i)\Big)\right|_{N_Q^{\beta_m}(\Zp)}\right]\\ 
		&= t_m^{-1} * v_\lambda^{\beta_m}(y_1^i),
	\end{align*}
	where the first equality follows as the action of $t_m$ sends $N_Q(\Zp)$ to $N_Q^{\beta_m}(\Zp)$, and the second from the definition of $v_{\lambda}^{\beta_m}$. Combining, we have 
	\[
	t_m^{-1} * v_{\lambda}^{\beta_m}(y_1^i) \in t_m^{-1} * \cA_{\lambda,1}^{\Gamma_m}
	\]
	and \eqref{eq:mu = mu1 2} follows by \eqref{eq:mu = mu1}.
	
	Now directly from the definitions we have
	\[
	p^{i(\beta_m - 1)}\big[t_m^{-1} * v_{\lambda}^{\beta_m}(y_m^i)\big] = \big[t_m^{-1} * v_{\lambda}^{\beta_m}(y_1^i)\big],
	\]
	and combining with \eqref{eq:mu = mu1 2} we deduce
	\[
	\mu\big(t_m^{-1} * v_{\lambda}^{\beta_m}(y_m^i)\big) = \mu_1\big(t_m^{-1} * v_{\lambda}^{\beta_m}(y_m^i)\big).
	\]
	Claim \ref{claim:admissible} and Proposition \ref{prop:admissible} follow as the $y_m^i$ are an orthonormal basis of $\cA_m(1+p^m\OFp)$.
\end{proof}

\subsection{Non-$Q$-critical $p$-adic $L$-functions} \label{sec:non-Q-critical p-adic L-functions}

We prove Theorem \ref{thm:intro non-ord} from the introduction. Let $\tilde\pi = (\pi, \{\alpha_{\fp}\}_{\fp| p})$ be a $Q$-refined RACAR of weight $\lambda$ satisfying Conditions~\ref{cond:running assumptions 2}. In particular, it admits an $(\eta,\psi)$-Shalika model, with $\eta = \eta_0|\cdot|^{\sw}$ and $\sw$ the purity weight of $\lambda$. Suppose that $\tilde{\pi}$ is non-$Q$-critical (Definition~\ref{def:non-Q-critical}). Fix $K = K(\tilde\pi)$ and $\epsilon \in \{\pm1\}^\Sigma$, and let $\phi_{\tilde{\pi}}^\epsilon \in \hc{t}(S_{K},\sV_\lambda^\vee)\locpi^\epsilon$ as in Definition~\ref{def:phi}. By definition of non-$Q$-criticality, $\phi_{\tilde{\pi}}^\epsilon$ lifts uniquely to an eigenclass $\Phi_{\tilde{\pi}}^\epsilon \in \hc{t}(S_K,\sD_\lambda)^\epsilon\locpi$ with $U_{\pri}^\circ$-eigenvalue $\alpha_{\pri}^\circ$, recalling $\alpha_{\pri}^\circ = \lambda(t_{\pri})\alpha_{\pri}$. As above, write $\alpha_p^\circ = \prod_{\pri|p}(\alpha_{\pri}^\circ)^{e_{\pri}}$.

\begin{definition}\label{def:non-critical slope Lp}
	Let $\cL_p^\epsilon(\tilde\pi) \defeq \mu^{\eta_0}(\Phi_{\tilde\pi}^\epsilon)$ be the distribution on $\Galp$ attached to $\Phi_{\tilde\pi}^\epsilon$. Let $\Phi_{\tilde\pi} = \sum_{\epsilon \in \{\pm 1\}^\Sigma} \Phi_{\tilde\pi}^\epsilon$, and define the \emph{$p$-adic $L$-function} attached to $\tilde{\pi}$ as
		\begin{align*}
			\cL_p(\tilde\pi) &\defeq \mu^{\eta_0}(\Phi_{\tilde\pi}) = \sum_{\epsilon \in \{\pm1\}^{\Sigma}} \cL_p^\epsilon(\tilde\pi) \in \cD(\Galp,L).
		\end{align*}
	For shorthand, for any $\psi \in \cA(\Galp,L)$ we write
	\[
	\cL_p(\tilde\pi, \psi) \defeq \int_{\Galp} \psi \cdot \cL_p(\tilde{\pi}).
	\]
	Let 
	\disp{
	\sX(\Galp) \defeq (\mathrm{Spf}\ \Zp\lsem \Galp\rsem )^{\mathrm{rig}}
}
	denote the rigid analytic space of $p$-adic characters on $\Galp$. Via the Amice transform we may view $\cL_p(\tilde\pi, -) : \sX(\Galp) \to \overline{\Q}_p$ as an element of $\cO(\sX(\Galp))$.
\end{definition}

\begin{theorem} \label{thm:non-ordinary}
	The distribution $\cL_p({\tilde{\pi}})$ is admissible of growth $h_p = v_p(\alpha_p^\circ)$, and satisfies the following interpolation property: for every finite order Hecke character $\chi$ of $F$ of conductor $p^{\beta}$, and all $j \in \mathrm{Crit}(\lambda)$, we have
	\begin{align}\label{eq:interpolation}
		i_p^{-1}(\cL_p(\tilde\pi, \chi\chi_{\cyc}^j)) = A \tau(\chi_f)^n  \mathrm{N}_{F/\Q}(\mathfrak{d})^{jn} \prod_{\pri|p} \ep \cdot\einf \cdot \frac{L^{(p)}\big(\pi\otimes\chi, j+\tfrac{1}{2}\big)}{\Omega_\pi^{\epsilon}},
	\end{align}
	where $\epsilon = (\chi\chi_{\cyc}^j\eta)_\infty$ and $\einf$ is as in Definition \ref{def:e_infty}. At $p$ we have
			\[
	\ep \defeq \left\{\begin{array}{cl}
	q_{\pri}^{\left(nj + \tfrac{n^2-n}{2}\right)\beta_{\pri}} \alpha_{\pri}^{-\beta_{\pri}} &: \chi_{\pri} \text{ ramified},\\
			\prod_{i=n+1}^{2n}
			\frac{1-\UPS_{\pri,i}^{-1}\chi_{\pri}^{-1}(\varpi_{\pri})q_{\pri}^{j-1/2}}{1-\UPS_{\pri,i}\chi_{\pri}(\varpi_{\pri})q_{\pri}^{-j-1/2}}. &: \chi_{\pri} \text{ unramified.}\end{array}\right.
		\]
	Finally $A$ is the global constant
	\begin{equation}\label{eq:A}
		A = \gamma_{p\m} \cdot  \prod_{\pri|p}\frac{q_{\pri}^n}{(q_{\pri}-1)^n} q_{\pri}^{-\delta_{\pri} \left(\tfrac{n^2+n}{2}\right)} \in \Q^\times.
	\end{equation}

\end{theorem}
\begin{proof}
	Admissibility is Proposition~\ref{prop:admissible}. For the interpolation, from Lemma~\ref{lem:compatibility} we know 
	\[
	\int_{\Galp}\chi\ \chi_{\cyc}^j \cdot \mu^{\eta_0}(\Phi_{\tilde\pi}) = (\alpha_p^\circ)^{-\beta} \times \cE_{\chi,j}^{\eta_0}(\phi_{\tilde\pi}),
	\]
	where we must replace $\beta_{\pri}$ with $\mathrm{max}(\beta_{\pri},1)$).  This is equal to the statement by Theorem~\ref{thm:critical value}, noting $\lambda(t_p^\beta)(\alpha_p^\circ)^{-\beta} = \alpha_p^{-\beta}$	and $\mathrm{N}_{F/\Q}(\fd)^{jn} = \mathrm{N}_{F/\Q}(\fd^{(p)})^{jn}\prod_{\pri|p}q_{\pri}^{\delta_{\pri}nj}$. Note that  $i_p^{-1}(\cL_p^\epsilon(\tilde\pi, \chi\chi_{\cyc}^j)) = 0$  if $\epsilon \ne (\chi\chi_{\cyc}^j\eta)_\infty$.
\end{proof}
\begin{remarks}\label{rem:signed interpolation}
	The same theorem holds under the weaker hypothesis of Conditions \ref{cond:running assumptions 2}, but with the additional assumption that $\beta_{\pri} \geqslant 1$ for all $\pri|p$, i.e.\ each $\chi_{\pri}$ is ramified. To include $\beta_{\pri} = 0$ requires a careful analysis of the local zeta integral at $\pri$ for ramified $\pi_{\pri}$ and unramified $\chi_{\pri}$, which was carried out by the second author with Jorza \cite{DJ-parahoric}.
\end{remarks}

Finally we consider uniqueness properties of $\cL_p(\tilde\pi)$. 
\begin{proposition}\label{prop:non-critical slope unique}
	Suppose Leopoldt's conjecture holds for $F$ at $p$, and that $h_p < \#\mathrm{Crit}(\lambda)$. Then $\cL_p(\tilde\pi)$ is uniquely determined by its interpolation and admissibility properties.
\end{proposition}
\begin{proof}
	Leopoldt's conjecture implies that $\Galp$ is 1-dimensional as a $p$-adic Lie group. Uniqueness is then a result of Vishik \cite[Thm.~2.3, Lem.~2.10]{Vis76}, shown independently by Amice--Velu \cite{AV75}.
\end{proof}

When $h_p < \#\mathrm{Crit}(\lambda)$, the restriction of $\cL_p(\tilde\pi)$ to $\Galp^{\cyc}$ is unique even without Leopoldt's conjecture. This can be seen by arguments analogous to \cite[(78)]{BDJ17}.

When $h_p \geqslant \#\mathrm{Crit}(\lambda)$, we will prove analogous uniqueness results in \S\ref{sec:uniqueness 2}, as an application of our construction of $p$-adic $L$-functions in families.

%%===================================================
%%===================================================
%%===================================================
%%
%%			LOCAL GEOMETRY
%%
%%===================================================
%%===================================================
%%===================================================

\section{Shalika families}\label{sec:shalika families no new}

For the rest of the paper, we will be concerned with variation in families. In this section, we prove Theorem~\ref{thm:intro shalika families} of the introduction; namely, we show that: (1) the eigenvariety is \'etale at a non-$Q$-critical $Q$-refined RASCAR $\tilde\pi$, and (2) that the unique component through such a $\tilde\pi$ is a Shalika family. Since we believe these results to be of independent interest beyond our precise results on $p$-adic $L$-functions, we first present them in wide generality here, always working with Hecke operators away from the set $S$ of ramified primes from \S\ref{sec:unramified H}. In the process, we develop methods that will be crucially used in the next section, where we make an automorphic hypothesis and add further Hecke operators at each $v \in S$, and refine these results to better suit the study of $p$-adic $L$-functions.

Throughout, let $\tilde\pi$ be a $Q$-refined RACAR of weight $\lambda_\pi$ satisfying (C1-2) from Conditions~\ref{cond:running assumptions}. An undecorated $K$ will always mean an arbitrary subgroup satisfying the conditions of \eqref{eq:general K}. We will also consider more specific choices $K(\tilde\pi),K_1(\tilde\pi)$. Unless otherwise specified, we take all coefficients to be in a sufficiently large extension $L/\Qp$ as in \S\ref{sec:periods} and drop it from notation.

\subsection{Set-up, statement of Thm.\ \ref{thm:shalika family} and summary of proof}\label{sec:shalika no new statement}

\subsubsection{The eigenvarieties}
We introduce local charts around $\tilde\pi$ on a parabolic eigenvariety. Fix $h \in \Q_{\geqslant 0}$. Via \S\ref{sec:slope-decomp}, let $\Omega$ be an $L$-affinoid neigbourhood of $\lambda_\pi$ in  $\Wlam$ such that $\htc(S_K,\sD_{\Omega})$ admits a slope $\leqslant h$ decomposition with respect to $U_p^\circ$. Recall $\cH$ from \S\ref{sec:hecke outside S}.

\begin{definition}\label{def:local piece S}
	\begin{itemize} \setlength{\itemsep}{0pt}
		\item	Define $\T\Uha(K)$  to be the image of
		\disp{
		\cH\otimes\cO_\Omega \longrightarrow \End_{\cO_\Omega}\big(\htc(S_K,\sD_{\Omega})^{\leqslant h}\big).
	}
		
		\item Define 
		\disp{
		\sE\Uha(K) \defeq \mathrm{Sp}(\T\Uha(K)),
	}
		a rigid analytic space.\medskip
	\end{itemize}
	Let $w : \sE\Uha(K) \rightarrow \Omega$ be the \emph{weight map} induced by the structure map  $\cO_\Omega \rightarrow \T\Uha(K)$. For any $\epsilon \in \{\pm 1\}^\Sigma$, write $\bT^{\epsilon}_{\Omega,h}(K)$ and $\sE^{\epsilon}_{\Omega,h}(K)$ for the analogues using $\epsilon$-parts of the cohomology. As $\bT^{\epsilon}_{\Omega,h}$ is a quotient of $\bT_{\Omega,h}$, each $\sE_{\Omega,h}^{\epsilon}(K)$ embeds as a closed subvariety of $\sE_{\Omega,h}(K)$. Moreover 
	\[
	\sE_{\Omega,h}(K) =\textstyle\bigcup_\epsilon \sE_{\Omega,h}^{\epsilon}(K).
	\] 
\end{definition}

The local piece $\sE\Uha(K)$ is the space denoted $\cE_{\Omega,h}^{Q,t}$ in \cite[\S5]{BW20}. By definition, $\sE\Uha(K)$ is a rigid space whose $L$-points $y$ are in bijection with non-trivial algebra homomorphisms $\T\Uha(K) \rightarrow L$, or equivalently, with systems of eigenvalues  $\psi_y : \cH \to L$ appearing  in $\htc(S_K,\sD_{\Omega})\ssh$. 

We use the convention that $\sC$ (resp. $\sI$) denotes a connected (resp.\ irreducible) component of $\sE$ (with appropriate decorations).

\begin{definition}\label{def:shalika family} \begin{itemize}\setlength{\itemsep}{0pt}
		\item[(i)]
		We say a point $y \in \sE_{\Omega,h}(K)$ is \emph{classical} if there exists a cohomological automorphic representation $\pi_y$ of $G(\A)$ having weight $\lambda_y \defeq w(y)$ such that $\psi_y$ appears in $\pi_y^{K}$, whence $\tilde\pi_y = (\pi_y, \{\psi_y(U_{\pri}^\circ)\}_{\pri|p})$ is a $Q$-refined automorphic representation. Throughout we use the notation $\m_y = \m_{\tilde\pi_y}$ for the associated maximal ideal of $\T_{\Omega,h}(K)$.
		
		\item[(ii)] A classical point $y$ is \emph{cuspidal} (resp.\ \emph{essentially self-dual}) if $\pi_y$ is.
		\item[(iii)] For a finite order Hecke character $\eta_0$, an \emph{$(\eta_0,\psi)$-Shalika point} is a classical cuspidal point $y$ such that $\pi_y$ admits an $(\eta_0|\cdot|^{\sw_y},\psi)$-Shalika model, where $\sw_y$ is the purity weight of $\lambda_y$.
		\item[(iv)] A \emph{classical (resp.\ $(\eta_0,\psi)$-Shalika) family} in $\sE_{\Omega,h}(K)$ is an irreducible component $\sI$ in $\sE_{\Omega,h}(K)$ containing a Zariski-dense set of classical (resp.\  $(\eta_0,\psi)$-Shalika) points. 
	\end{itemize}
\end{definition}

To describe the geometry of $\sE_{\Omega,h}(K)$, we must be precise about the level $K$. In Theorem~\ref{thm:shalika family}, there will be two particularly important level groups: the group $K(\tilde\pi)$ from \eqref{eq:level group}, at which Friedberg--Jacquet test vectors exist; and a more explicit group $K_1(\tilde\pi)$, which we now define. For any place $v$ and $m \in \Z_{\geqslant 0}$ let 
\[
K_{1,v}(m) \subset \GL_{2n}(\cO_v)
\]
be the open compact subgroup of matrices whose bottom row is congruent to $(0,...,0,1) \ \mathrm{mod}\ \varpi_v^m$. The \emph{Whittaker conductor} $m(\pi_v)$ of $\pi_v$ is the minimal integer $m$ such that $\pi_v^{K_{1,v}(m)} \neq 0$, and by \cite[\S5]{JPSS} 
\begin{equation}\label{eq:new vector}
	\dim_{\C} \ \pi_v^{K_{1,v}(m(\pi_v))} = 1.
\end{equation}
Note $K_{1,v}(0) = \GL_{2n}(\cO_v)$, so $\pi_v$ is spherical if and only if $m(\pi_v) = 0$. We define
\begin{equation}\label{eq:K_1}
	K_1(\tilde\pi) \defeq \textstyle\prod_{\pri|p}J_{\pri} \textstyle\prod_{v\nmid p} K_{1,v}(m(\pi_v)) \subset G(\A_f).
\end{equation}

\subsubsection{Hypotheses on $\pi$}\label{sec:hypotheses}
Our results require hypotheses on $\pi$ that we now make precise.

\begin{definition}\label{def:non-vanishing}
	We say $\pi$ \emph{admits a non-zero Deligne-critical $L$-value at $p$} if there exists a pair $(\chi,j)$ such that 
	\disp{
	L(\pi\otimes\chi, j+\tfrac{1}{2}) \neq 0,
}
	where $j \in \mathrm{Crit}(\lambda_\pi)$ 
	and $\chi$ is a finite order Hecke character of conductor $p^\beta$ with $\beta_{\pri} \geqslant 1$ for all $\pri$. This $L$-value has \emph{sign $\epsilon$} if $\epsilon = (\chi\chi_{\cyc}^j\eta)_\infty \in \{\pm1\}^{\Sigma}$.
	
	Note that $L(\pi\otimes\chi,s) \neq 0 \iff L^{(p)}(\pi\otimes\chi,s) \neq 0$ (as the local factors at $p$ are non-vanishing).
\end{definition}
Conjecturally, this non-vanishing is true for all but finitely many such pairs $(\chi,j)$, so every $\pi$ should satisfy this hypothesis. In practice, this is guaranteed by the following  simple criterion.
\begin{lemma}\label{lem:regular weight implies non-vanishing}
	$\pi$ has a non-zero Deligne-critical $L$-value at $p$ if $(\lambda_\pi)_{\sigma,n} > (\lambda_\pi)_{\sigma,n+1}\ \forall \sigma \in \Sigma$.
\end{lemma}
\begin{proof}
	Let $j$ be the largest integer in $\mathrm{Crit}(\lambda_\pi)$, and $\chi$ any Hecke character satisfying the conditions of Definition \ref{def:non-vanishing}.  The hypothesis ensures that $\#\mathrm{Crit}(\lambda_\pi) > 1$, so that $j + \tfrac{1}{2} \geqslant \tfrac{\sw}{2} + 1$ (recalling $\sw$ is the purity weight), and hence $L(\pi\otimes\chi,j+\tfrac{1}{2}) \neq 0$ by the main result of \cite{JS76}.
\end{proof}
\begin{definition}\label{def:H-regular}
	Recall $\lambda$ is regular if $\lambda_{\sigma,i} > \lambda_{\sigma,i+1}$ for all $\sigma$ and $i$. Say it is \emph{$H$-regular} if 
	\begin{equation}\label{eq:H-regular}
		\lambda_{\sigma,1} > \cdots > \lambda_{\sigma,n}\ \ \text{ and } \ \ \lambda_{\sigma,n+1} > \cdots > \lambda_{\sigma,2n}
	\end{equation}
	for all $\sigma$ (allowing $\lambda_{\sigma,n} = \lambda_{\sigma,n+1}$). Such weights are regular as weights for $H$.
\end{definition}

For a field $E$, let $\mathrm{G}_E \defeq \mathrm{Gal}(\overline{E}/E)$. Attached to any RACAR $\pi'$ of $G(\A)$ we have a Galois representation $\rho_{\pi'} : \mathrm{G}_F \to \GL_{2n}(\overline{\Q}_p)$, depending on our fixed isomorphism $\iota_p: \C \cong \overline{\Q}_p$ (see \cite{HLTT}). For a finite prime $v$ of $F$, we say \emph{Local-Global Compatibility holds for $\pi'$ at $v$} if 
\[
\mathrm{WD}(\rho_{\pi'}|_{\mathrm{G}_{F_v}})^{\mathrm{F-ss}} = \iota_p\mathrm{rec}_{F_v}(\pi_v'\otimes |\cdot|^{(1-n)/2}),
\]
where $\mathrm{rec}_{F_v}$ denotes the local Langlands correspondence for $\GL_{2n}/F_v$. This is conjecturally always true; it is known in general up to semi-simplification \cite{VarLGC}, and is known when $\pi'$ is essentially self-dual (for self-dual RACARs this is shown in \cite{Shi11,Car12}; it is explained in \cite[\S4.3]{CHT08} why the essentially self-dual case follows). Hence it is known if $\pi'$ is a RASCAR.

\subsubsection{Statement}	Let $\tilde\pi$ be as in Conditions~\ref{cond:running assumptions}, of weight $\lambda_\pi$, and let $\alpha_p^\circ = \prod_{\pri|p}(\alpha_{\pri}^\circ)^{e_{\pri}}$. Recall $K(\tilde\pi)$ from  \eqref{eq:level group} and $\eta = \eta_0|\cdot|^{\sw}$ from \S\ref{sec:shalika models}. Fix $h \geqslant v_p(\alpha_p^\circ)$.	In the rest of \S\ref{sec:shalika families no new}, we will prove:

\begin{theorem}\label{thm:shalika family}\begin{itemize}\setlength{\itemsep}{0pt}
		\item[(a)] If $\tilde\pi$ is strongly non-$Q$-critical at $p$ (see Definition \ref{def:non-Q-critical}), then for any $K$ as in \eqref{eq:general K} there is a point $x_{\tilde\pi}(K) \in \sE\Uha(K)$ attached to $\tilde\pi$.
		
		\item[(b)] Suppose further that $\pi$ admits a non-zero Deligne-critical $L$-value at $p$. At level $K(\tilde\pi)$, there exists an irreducible component in $\sE\Uha(K(\tilde\pi))$ through $x_{\tilde\pi}(K(\tilde\pi))$ of dimension $\dim(\Omega)$.
		
		\item[(c)] Suppose further that $\lambda_\pi$ is $H$-regular. There exists an $(\eta_0,\psi)$-Shalika family $\sI(K(\tilde\pi))$ in $\sE\Uha(K(\tilde\pi))$ of dimension $\dim(\Omega)$.
		
		\item[(d)] Suppose further that $\rho_{\pi} : \mathrm{G}_F \to \GL_{2n}(\overline{\Q}_p)$ is irreducible. Then:
		\begin{itemize}\setlength{\itemsep}{0pt}
			\item[(d1)] at level $K_1(\tilde\pi)$, $\sE_{\Omega,h}(K_1(\tilde\pi))$ is \'etale over $\Omega$ at $x_{\tilde\pi}(K_1(\tilde\pi))$, and the (irreducible) connected component $\sC(K_1(\tilde\pi))$ through $x_{\tilde\pi}(K_1(\tilde\pi))$ is an $(\eta_0,\psi)$-Shalika family;
			\item[(d2)] at level $K(\tilde\pi)$, $\sI(K(\tilde\pi))$ is the unique Shalika family of $\sE_{\Omega, h}(K(\tilde\pi))$ through $x_{\tilde\pi}(K(\tilde\pi))$. Moreover the nilreduction of $\sI(K(\tilde\pi))$ is \'etale over $\Omega$ at $x_{\tilde\pi}(K(\tilde\pi))$.
		\end{itemize} 
		
		\item[(e)] Suppose further that Local-Global Compatibility holds at all $v\nmid p$ for all RACARs of $G$. Then in (d2), $\sI(K(\tilde\pi))$ is also the unique classical family of $\sE_{\Omega, h}(K(\tilde\pi))$ through $x_{\tilde\pi}(K(\tilde\pi))$.
		
	\end{itemize}
\end{theorem}

 If $K$ is completely unambiguous we will drop it from notation.

\begin{remark}
	Theorem~\ref{thm:intro shalika families} of the introduction is a special case of (d1). Indeed, non-$Q$-critical slope implies strongly non-$Q$-critical (Theorem~\ref{thm:control}), and if $\lambda_\pi$ is regular then it is $H$-regular (by definition) and $\pi$ admits a non-zero Deligne-critical $L$-value at $p$ (Lemma~\ref{lem:regular weight implies non-vanishing}), hence $\cL_p(\tilde\pi) \neq 0$. Conjecturally, if $\pi$ is cuspidal then $\rho_\pi$ is always irreducible.
\end{remark}

\subsection{Proof of Thm.\ \ref{thm:shalika family}(a): Existence of $x_{\tilde\pi}(K)$}\label{sec:existence of x S}
Recall $t$ from \eqref{eq:t}. If $\tilde\pi$ is non-$Q$-critical, then by Theorem~\ref{thm:control}  $\tilde\pi$ contributes to $\hc{t}(S_K,\sD_{\lambda_\pi})$. The character $\psi_{\tilde\pi} : \cH\otimes E \to E$ from Definition~\ref{def:gen eigenspace} induces a character $\cH \otimes \cO_\Omega \to \cO_\Omega$, and thus a map
\begin{equation} \label{eq:m pi families}
	\cH\otimes \cO_\Omega \longrightarrow \cO_\Omega \longrightarrow \cO_\Omega/\m_{\lambda_\pi} = L,
\end{equation}
where $\m_{\lambda_\pi}$ is the maximal ideal corresponding to $\lambda_\pi$. We also write $\m_{\tilde\pi}$ for the kernel of this composition. This is a maximal ideal of $\cH\otimes \cO_\Omega$, whose contraction to $\cO_\Omega$ is $\m_{\lambda_\pi}$.

For any sufficiently large $h$, the localisation $\hc{t}(S_K,\sD_\Omega)\ssh\locpi$ is independent of $h$, and in a slight abuse of notation, we denote this $\hc{t}(S_K,\sD_\Omega)\locpi$.

Let 
\disp{
\T_{\Omega,\tilde\pi}(K) = [\T_{\Omega, h}(K)]_{\m_{\tilde\pi}}
}
be the localisation of $\T\Uha(K)$ at $\m_{\tilde\pi}$, which acts on $\htc(S_K,\sD_\Omega)\locpi$.  Let $\Lambda$ denote the localisation of $\cO_\Omega$ at $\m_{\lambda_\pi}$.	Theorem~\ref{thm:shalika family}(a) follows from:

\begin{proposition}\label{prop:surjection}
	The map $\mathrm{sp}_{\lambda_\pi}: \cD_\Omega \to \cD_{\lambda_\pi}$ induces an isomorphism
	\begin{equation}\label{eq:surjection}
		\htc(S_K,\sD_{\Omega})\locpi \otimes_{\Lambda} \Lambda/\m_{\lambda_\pi} \isorightarrow \htc(S_K,\sD_{\lambda_\pi})\locpi.
	\end{equation}
	In particular, $\m_{\tilde\pi}$ is a maximal ideal of $\bT\Uha(K)$ and hence there exists a point $x_{\tilde\pi} \in \sE\Uha(K)$.
\end{proposition}

\begin{proof}
		This exactly follows \cite[Lem.\ 2.9(i)]{BDJ17}, first proving vanishing for degrees $i > t$ (rather than $i \neq d$) and concluding as \emph{ibid}.\ using vanishing for degree $t+1$.
\end{proof}

We would also like an analogue of \cite[Lem.\ 2.9(ii)]{BDJ17}, to show that $\hc{t}(S_K,\sD_\Omega)_{\m_{\tilde\pi}}$ is $\cO_\Omega$-torsion-free. However the proof of that result does \emph{not} work here, as the cohomology is not concentrated in one degree. A key novelty of this paper is the use of evaluation maps to overcome this.

\begin{remark}\label{rem:non-Q-critical implies point}
	The same argument shows for \emph{any} non-$Q$-critical classical cuspidal eigensystem $\tilde\pi'$ in weight $\lambda$, $\mathrm{sp}_{\lambda}$ induces an isomorphism 
	\disp{
	\hc{t}(S_K,\sD_\Omega)_{\m_{\tilde\pi'}} \otimes_{\cO_{\Omega,\lambda}}\cO_{\Omega,\lambda}/\m_{\lambda} \isorightarrow  \hc{t}(S_K,\sD_{\lambda})_{\m_{\tilde\pi'}}.
}
\end{remark}

\subsection{Proof of Thm.\ \ref{thm:shalika family}(b): Components of maximal dimension}\label{sec:maximal dimension}
Let $\sC(K) \subset \sE\Uha(K)$ be the connected component containing $x_{\tilde\pi}(K)$. There exists an idempotent $e$ such that $\bT_{\Omega,\sC}(K) = e\bT\Uha(K)$ is a direct summand, with $\sC = \mathrm{Sp}(\bT_{\Omega,\sC}(K))$; then
\begin{equation}\label{eq:summand}
		\hc{t}(S_{K}, \sD_\Omega)\ssh\otimes_{\bT\Uha(K)} \bT_{\Omega,\sC}(K) = e	\hc{t}(S_{K}, \sD_\Omega)\ssh \subset 	\hc{t}(S_{K}, \sD_\Omega)\ssh.
\end{equation}
Now fix $K = K(\tilde\pi)$ from \eqref{eq:level group} and for convenience drop $K$ from notation. Let $\Phi_{\tilde\pi} \in \hc{t}(S_{K(\tilde\pi)}, \sD_{\lambda_\pi})\locpi$ be the class from Definition~\ref{def:non-critical slope Lp}. By Proposition~\ref{prop:surjection}, we can lift this to a class $\Phi_{\sC}' \in \hc{t}(S_{K(\tilde\pi)},\sD_\Omega)\ssh\locpi$ under the natural surjection $\mathrm{sp}_{\lambda_\pi}$. Possibly shrinking $\Omega$ and $\sC$, we may avoid denominators in $\bT_{\Omega,\sC}$, and assume $\Phi_{\sC}' \in \hc{t}(S_{K(\tilde\pi)},\sD_\Omega)\ssh$; then applying the idempotent $e$ attached $\sC$, we define
\begin{equation}\label{eq:phi con}
	\Phi_{\sC} \defeq e\Phi_{\sC}' \in 	\hc{t}(S_{K(\tilde\pi)}, \sD_\Omega)\ssh\otimes_{\bT\Uha} \bT_{\Omega,\sC}.
\end{equation}
As shrinking $\Omega$ and applying $e$ doesn't change local behaviour at $\tilde\pi$, we still have $\mathrm{sp}_{\lambda_\pi}(\Phi_{\sC}) = \Phi_{\tilde\pi}$. 

The following is the key step in all our constructions; we are very grateful to Eric Urban, who suggested the elegant proof we present here. Recall $\mathrm{Ev}_{\beta}^{\eta_0}$ from \eqref{eq:final evaluation}.

\begin{proposition}\label{prop:ann = zero}
	Suppose there exists $\beta$ such that $\mathrm{Ev}_\beta^{\eta_0}(\Phi_{\tilde\pi}) \neq 0$. Then $\mathrm{Ann}_{\cO_\Omega}(\Phi_{\sC}) = 0$, and in particular,
	\disp{
	\hc{t}(S_{K(\tilde\pi)}, \sD_\Omega)\ssh\otimes_{\bT\Uha} \bT_{\Omega,\sC} \text{ is a faithful $\cO_\Omega$-module.}
}
\end{proposition}
\begin{proof}
	By restricting the evaluation map of \eqref{eq:final evaluation} to the summand \eqref{eq:summand}, we get a map 
	\begin{equation}\label{eq:ev on T}
		\mathrm{Ev}_\beta^{\eta_0} : \hc{t}(S_{K(\tilde\pi)},\sD_\Omega)\ssh\otimes_{\bT\Uha} \bT_{\Omega,\sC} \to \cD(\Galp,\cO_\Omega)
	\end{equation}
	of $\cO_\Omega$-modules. From Proposition~\ref{prop:evaluations in families}, we have 
	\[
	\mathrm{sp}_{\lambda_\pi}(\mathrm{Ev}_{\beta}^{\eta_0}(\Phi_{\sC}))= \mathrm{Ev}_{\beta}^{\eta_0}(\mathrm{sp}_{\lambda_\pi}(\Phi_{\sC}))= \mathrm{Ev}_{\beta}^{\eta_0}(\Phi_{\tilde{\pi}}).
	\]
	The right-hand side is non-zero by hypothesis, so we deduce that $\mathrm{Ev}_{\beta}^{\eta_0}(\Phi_{\sC}) \neq 0$.
	
	Now let $u \in \cO_\Omega$ such that $u\Phi_{\sC} = 0$. Since $\mathrm{Ev}_{\beta}^{\eta_0}$ is an $\cO_\Omega$-module map, we see 
		\[0 = \mathrm{Ev}_{\beta}^{\eta_0}(u\Phi_{\sC}) 
		= u \mathrm{Ev}_{\beta}^{\eta_0}(\Phi_{\sC}) \in \cD(\Galp,\cO_\Omega).
		\]
	As $\mathrm{Ev}_\beta^{\eta_0}(\Phi_{\sC}) \neq 0$ and $\cD(\Galp,\cO_\Omega)$ is $\cO_\Omega$-torsion-free, we see $u = 0$, from which we conclude.
\end{proof}

\begin{corollary}\label{cor:T is faithful}\label{cor:exists max dim}
	Suppose there exists $\beta$ such that $\mathrm{Ev}_\beta^{\eta_0}(\Phi_{\tilde\pi}) \neq 0$.  Then, at level $K(\tilde\pi)$, 
	\begin{itemize}\setlength{\itemsep}{0pt}
		\item[(i)] the $\cO_\Omega$-algebra $\bT_{\Omega,\sC}$ is faithful as an $\cO_\Omega$-module, and 
		\item[(ii)] there exists an irreducible component $\sI \subset \sC$ in $\sE\Uha$ through $x_{\tilde\pi}$ with $\dim(\sI) = \dim(\Omega)$.
	\end{itemize}
\end{corollary}
\begin{proof}
	The $\cO_\Omega$-action on $\hc{t}(S_{K(\tilde\pi)},\sD_\Omega)^{\leqslant h}$ is faithful (by Proposition \ref{prop:ann = zero}) and factors through the action of $\cH\otimes \cO_\Omega$, hence (by definition) the action of $\bT_{\Omega,\sC}$. Part (i) follows.
	
	For (ii), as $\hc{t}(S_{K(\tilde\pi)},\sD_\Omega)\ssh$ and $\bT\Uha$ are finite $\cO_\Omega$-modules, we deduce there are finitely many irreducible components of $\cE\Uha$, each of dimension at most $\dim(\Omega)$. Suppose every component $\sI$ through $x_{\tilde\pi}$ has dimension $\dim(\sI) <\dim(\Omega)$. Then $\mathrm{Supp}_{\cO_\Omega}(\bT_{\Omega,\sC})$ (by definition, the image of $\sC$ in $\Omega$ under the weight map) is a closed subspace of $\Omega$ of dimension strictly less than $\dim(\Omega)$. In particular it is a \emph{proper} closed subspace. But by \cite[Prop.~4.4.2]{Han17}, since $\bT_{\Omega,\sC}$ is a faithful $\cO_\Omega$-module, we have $\mathrm{Supp}_{\cO_\Omega}(\bT_{\Omega,\sC}) = \Omega$, so we conclude by contradiction.
\end{proof}

\begin{corollary}\label{cor:ann = zero}
	Suppose $\pi$ admits a non-zero Deligne-critical $L$-value at $p$. Then (at level $K(\tilde\pi)$) there exists an irreducible component $\sI$ in $\sE\Uha$ through $x_{\tilde\pi}$ such that $\dim(\sI) = \dim(\Omega)$.
\end{corollary}

\begin{proof}
	By hypothesis (Definition~\ref{def:non-vanishing}) there exists $\beta$ with $\beta_{\pri} \geqslant 1$ for all $\pri|p$, a character $\chi$ of conductor $p^\beta$, and $j \in \mathrm{Crit}(\lambda_\pi)$ such that $L(\pi \times \chi, j+\tfrac{1}{2}) \neq 0$. By Theorem~\ref{thm:non-ordinary}, for an explicit (non-zero) constant $(*)$ we have 	
		\[		(\alpha_p^\circ)^{-\beta}\int_{\Galp} \chi\ \chi_{\mathrm{cyc}}^j \cdot \mathrm{Ev}_{\beta}^{\eta_0}(\Phi_{\tilde\pi}) \defeqrev  \cL_p(\tilde\pi,\chi\chi_{\cyc}^j) = i_p[(*)L^{(p)}(\pi\otimes\chi,j+\tfrac{1}{2})] \neq 0.\]
	Thus $\mathrm{Ev}_\beta^{\eta_0}(\Phi_{\tilde\pi}) \neq 0$. We conclude by Corollary~\ref{cor:exists max dim}(ii).
\end{proof}

\subsection{Proof of Thm.\ \ref{thm:shalika family}(c): (very) Zariski-density of Shalika points} \label{sec:zariski dense classical}
We still take $K = K(\tilde\pi)$.

\begin{lemma}\label{lem:zariski regular}
	If $\lambda_{\pi}$ is $H$-regular in the sense of \eqref{eq:H-regular}, then any neighbourhood $\Omega$ of $\lambda_\pi$ in $\Wlam$ contains a very Zariski-dense set of regular algebraic dominant weights.
\end{lemma}
\begin{proof}
	Write $\lambda_\pi = (\lambda_{\pi}',\lambda_{\pi}'')$ as a weight for $H$. If a weight $\lambda \in \Wlam$ is of the form 
	\[
	\lambda = (\lambda_{\pi}'+ (a,...a), \lambda_{\pi}''+ (b,...,b)),
	\]
	where $a, b \in \Z^{\Sigma}$ are weights for $\mathrm{Res}_{\cO_F/\Z}(\GL_1)$ with $a_{\sigma} \geqslant b_{\sigma}$ for all $\sigma \in \Sigma$, then $\lambda$ is algebraic dominant and $H$-regular. The set of such $\lambda$ is very Zariski-dense in $\Wlam$. 
	
	Moreover, such a weight is regular if $a_\sigma > b_\sigma$ for all $\sigma$, as then $\lambda_{\pi,n,\sigma} + a_\sigma > \lambda_{\pi,n+1,\sigma} + b_\sigma$.  We conclude since $a_\sigma = b_\sigma$ is a closed condition.
\end{proof}

Recall we fixed $h \in \Q_{\geqslant 0}$. Now let $\Omega_{\mathrm{ncs}}$ be the subset of weights $\lambda \in \Omega$ such that
\begin{itemize}\setlength{\itemsep}{0pt}
	\item[(i)] $\lambda$ is algebraic, dominant and regular, and 
	\item[(ii)] $e_{\pri(\sigma)} h < 1 + \lambda_{\sigma,n} - \lambda_{\sigma,n+1}$ for all $\sigma \in \Sigma$ (in particular, $h$ is a non-$Q$-critical slope for $\lambda$).
\end{itemize}
Since failure of (ii) is a closed condition, $\Omega_{\mathrm{ncs}}$ is very Zariski-dense in $\Omega$ by Lemma~\ref{lem:zariski regular}.

\begin{proposition}\label{prop:zariski dense cuspidal}
	Suppose $\tilde\pi$ is strongly non-$Q$-critical and $\lambda_\pi$ is $H$-regular. Let $\sI = \mathrm{Sp}(\bT_{\Omega,\sI})$ be any irreducible component in $\sE\Uha(K(\tilde\pi))$ such that	$\cI$ contains $x_{\tilde\pi}(K(\tilde\pi))$ and $\dim(\sI) = \dim(\Omega)$.	Then the classical cuspidal non-$Q$-critical points are very Zariski-dense in $\sI$.
\end{proposition}
\begin{proof}
	Let $\sI_{\mathrm{ncs}} \defeq \sI\cap w^{-1}(\Omega_{\mathrm{ncs}})$. By \cite[Prop.~5.15]{BW20} (and its proof), 
	$\sI_{\mathrm{ncs}}$ is very Zariski-dense in $\sI$ and every $y \in \sI_{\mathrm{ncs}}$ is classical cuspidal non-$Q$-critical, from which the result follows. 
\end{proof}

 Note $w^{-1}(\lambda) \cap \sC$ is a finite set for all $\lambda \in \Omega.$

\begin{lemma}\label{lem:reduction direct sum}
	Let $\sC$ be as in \S\ref{sec:maximal dimension} and $\lambda \in \Omega_{\mathrm{ncs}}$. Reduction modulo $\m_{\lambda}$ induces an isomorphism
	\[
	\hc{t}(S_{K(\tilde\pi)},\sD_\Omega)\ssh\otimes_{\bT\Uha(K(\tilde\pi))} \bT_{\Omega,\sC}/\m_\lambda \cong \bigoplus_{y \in w^{-1}(\lambda)\cap \sC} 	\hc{t}(S_{K(\tilde\pi)},\sD_\lambda)_{\m_{y}}.
	\]
\end{lemma}
\begin{proof}
	This is local at $\lambda \in \Wlam$, so we are free to shrink $\Omega$ to a neighbourhood of $\lambda$ with
	\[
	\sC = \textstyle\bigsqcup_{y \in w^{-1}(\lambda)\cap \sC} \sC_y,
	\]
	with each $\sC_y = \mathrm{Sp}(\bT_{y})$ connected affinoid, and with $w^{-1}(\lambda)\cap \sC_{y} = \{y\}$. Note that $\sC$ itself can be disconnected over this smaller $\Omega$. As
	\disp{
	\bT_{\Omega,\sC} = \oplus_{y\in w^{-1}(\lambda)\cap\sC} \bT_y,
}
	we have
	\begin{align*}
		\hc{t}(S_{K(\tilde\pi)},\sD_\Omega)\ssh&\otimes_{\bT\Uha(K(\tilde\pi))} \bT_{\Omega,\sC}/\m_\lambda\\
		&\cong \bigoplus_{y \in w^{-1}(\lambda)\cap \sC} \hc{t}(S_{K(\tilde\pi)},\sD_\Omega)\ssh\otimes_{\bT\Uha(K(\tilde\pi))} \bT_y/\m_\lambda.
	\end{align*}
	As $y$ is the unique point of $\sC_y$ above $\lambda$, in each summand of the right-hand side, reduction mod $\m_\lambda$ factors through localisation at $\m_y$; and since $\sC_y$ is the connected component through $y$, each 
	\disp{
	\hc{t}(S_{K(\tilde\pi)},\sD_\Omega)\ssh\otimes_{\bT\Uha} \bT_y \subset \hc{t}(S_{K(\tilde\pi)},\sD_\Omega)
}
	is a summand, and 
	\disp{
	[\hc{t}(S_{K(\tilde\pi)},\sD_\Omega)\ssh\otimes_{\bT\Uha} \bT_y]_{\m_{y}} = \hc{t}(S_{K(\tilde\pi)}, \sD_\Omega)_{\m_y}. 
}
	Thus
	\begin{align*}
		\hc{t}(S_{K(\tilde\pi)},\sD_\Omega)\ssh\otimes_{\bT\Uha(K(\tilde\pi))} \bT_y/\m_\lambda &\cong \hc{t}(S_{K(\tilde\pi)},\sD_\Omega)_{\m_y} \otimes_{\Lambda}\Lambda/\m_\lambda\\
		&\cong \hc{t}(S_{K(\tilde\pi)},\sD_\lambda)_{\m_{y}},
	\end{align*}
	where the last isomorphism is Proposition~\ref{prop:surjection}, as each such $y$ has non-$Q$-critical slope (since $\lambda \in \Omega_{\mathrm{ncs}}$). Combining the last two displayed equations gives the Lemma.
\end{proof}

We now complete the proof of Theorem~\ref{thm:shalika family}(c). Recall $\sC$ is the connected component of $\sE\Uha(K(\tilde\pi))$ containing $x_{\tilde\pi}$, and let $\sC_{\mathrm{nc}}$ denote the set of classical cuspidal non-$Q$-critical points $y \in \sC$. Let $\sC_{\mathrm{nc}}^{\mathrm{Sha}}$ be the subset of points $y \in \sC_{\mathrm{nc}}$ such that $\pi_y$ admits a global $(\eta_0|\cdot|^{\sw_y},\psi)$-Shalika model, where $\sw_y$ is the purity weight of $\lambda_y = w(y)$. Theorem~\ref{thm:shalika family}(c) then follows from:

\begin{proposition}\label{prop:zariski dense shalika no new}
	Suppose the hypotheses of Theorem~\ref{thm:shalika family}(c). Up to shrinking $\Omega$, there is an irreducible component $\sI \subset \sC \subset \sE\Uha(K(\tilde\pi))$ such that $\sI$ contains $x_{\tilde\pi}(K(\tilde\pi))$, $\dim(\sI) = \dim(\Omega)$, and $\sI \cap \sC_{\mathrm{nc}}^{\mathrm{Sha}}$ is very Zariski-dense in $\sI$.
\end{proposition}

\begin{proof} 
	Let $\sC^{\sha}$ be the Zariski-closure of $\sC_{\mathrm{nc}}^{\mathrm{Sha}}$ in $\sC$. We claim:
	
	\begin{claim}
		$w(\sC_{\mathrm{nc}}^{\mathrm{Sha}})$ is a (very) Zariski-dense subset of $\Omega$.
	\end{claim}
	The claim implies  that $w(\sC^{\sha}) = \Omega$, that is $\sC^{\sha}$ has full support in $\Omega$. Given this, we conclude that $\sC^{\sha}$ has an irreducible component $\sI$ of dimension $\dim(\Omega)$, as $\sE\Uha(K(\tilde\pi))$ has finitely many irreducible components (see Corollary~\ref{cor:exists max dim}). Then $\sI$ satisfies the conditions we require. Thus the proposition follows from the claim.
	
	\medskip	
	
	\emph{Proof of claim}.	Fix a character $\chi$ of conductor $p^\beta$ and $j \in \mathrm{Crit}(\lambda_\pi)$ such that $L^{(p)}(\pi\otimes\chi, j+\tfrac{1}{2}) \neq 0$ (by hypothesis). Considering $\chi\chi_{\cyc}^j \in \cA(\Galp,\cO_\Omega)$ via the structure map $L \to \cO_\Omega$, define
		\[\mathrm{Ev}_{\chi,j}^\Omega : \htc(S_{K(\tilde\pi)},\sD_\Omega) \longrightarrow \cO_\Omega, \qquad
		\Phi \longmapsto \int_{\Galp} \chi \chi_{\cyc}^j \cdot \mathrm{Ev}_\beta^{\eta_0}(\Phi),
		\]
	for $\mathrm{Ev}_\beta^{\eta_0}$ as in \eqref{eq:final evaluation}. Restricting under \eqref{eq:summand}, $\mathrm{Ev}_{\chi,j}^\Omega$ defines a map $\hc{t}(S_{K(\tilde\pi)}, \sD_\Omega)\ssh\otimes_{\bT\Uha} \bT_{\Omega,\sC} \to \cO_\Omega$, which we can evaluate at the class $\Phi_{\sC}$ from \eqref{eq:phi con}. By construction $\mathrm{sp}_{\lambda_\pi}(\Phi_{\sC}) = \Phi_{\tilde\pi}$, so 
	\[
	\mathrm{sp}_{\lambda_\pi} \circ \mathrm{Ev}_\beta^{\eta_0}(\Phi_{\sC}) = \mathrm{Ev}_\beta^{\eta_0}(\Phi_{\tilde\pi})
	\]
	by Proposition~\ref{prop:evaluations in families}. Thus
			\begin{align*}
			\big[\mathrm{Ev}_{\chi,j}^\Omega(\Phi_{\sC})\big](\lambda_\pi) = \int_{\Galp} \chi \chi_{\cyc}^j \cdot \mathrm{sp}_\lambda \left(\mathrm{Ev}_\beta^{\eta_0}(\Phi_{\sC})\right)= \int_{\Galp} \chi \chi_{\cyc}^j \cdot \mathrm{Ev}_\beta^{\eta_0}(\Phi_{\tilde\pi}) \neq 0,
		\end{align*}
	where non-vanishing follows as in the proof of Corollary~\ref{cor:ann = zero}.  As the non-vanishing locus is open, up to shrinking $\Omega$ we may assume that $\mathrm{Ev}_{\chi,j}^\Omega(\Phi_{\sC}^\epsilon) \in \cO_\Omega$ is everywhere non-vanishing.
	
	Now let $\lambda$ be any weight in $\Omega_{\mathrm{ncs}}$, the set from \S\ref{sec:zariski dense classical}. Since $\sE\Uha(K(\tilde\pi))$ is finite over $\Omega$, the preimage $w^{-1}(\lambda) \cap \sC$ is a finite set. From Lemma~\ref{lem:reduction direct sum}, for $\lambda \in \Omega_{\mathrm{ncs}}$ we may write
	\begin{equation}\label{eq:spec sum}
		\mathrm{sp}_\lambda(\Phi_{\sC}) = \oplus_y \Phi_y,
	\end{equation}
	with each $\Phi_y \in \hc{t}(S_{K(\tilde\pi)},\sD_\lambda)_{\m_y}$. Recalling $r_y$ from \eqref{eqn:specialisation}, for each such $y$, we have 
	\begin{align*}
		r_{\lambda}(\Phi_y) = \oplus_\epsilon r_{\lambda}(\Phi_y^\epsilon) \in  \bigoplus_{\epsilon} &\htc(S_{K(\tilde\pi)},\sV_{\lambda}^\vee)_{\m_y}^\epsilon = \htc(S_{K(\tilde\pi)},\sV_{\lambda}^\vee)_{\m_y},
	\end{align*}
	projecting into the decomposition over $\epsilon$ of \S\ref{sec:decomp at infinity}. Now by combining Proposition~\ref{prop:evaluations in families} and Lemma~\ref{lem:compatibility}, we have a commutative diagram
	\[
	\xymatrix@C=20mm{
		\htc(S_{K(\tilde\pi)},\sD_\Omega)^\epsilon \ar[r]^-{\mathrm{Ev}_{\chi,j}^\Omega} \ar[d]_{r_{\lambda} \circ \mathrm{sp}_{\lambda}} &
		\cO_\Omega\ar[d]^{\mathrm{sp}_\lambda}\\
		\htc(S_{K(\tilde\pi)},\sV_{\lambda}^\vee)^\epsilon \ar[r]^-{\cE_{\chi}^{j,\eta_0}} &
		L.
	}
	\]
	Combining this with \eqref{eq:spec sum}, and the fact that $\mathrm{Ev}_{\chi,j}^\Omega(\Phi_{\sC})$ is everywhere non-vanishing, we deduce
	\[
	[\mathrm{Ev}_{\chi,j}^\Omega(\Phi_{\sC})](\lambda)  = \sum_{y \in w^{-1}(\lambda)\cap \sC} \sum_{\epsilon} \cE_{\chi}^{j,\eta_0}\left(r_{\lambda}(\Phi_y^\epsilon) \right) \neq 0.
	\]
	Hence at least one of the terms in the sum is non-zero. By Proposition~\ref{prop:shalika non-vanishing}, we deduce that if this term corresponds to the point $y$, then $\pi_y$ admits an $(\eta_0|\cdot|^{\sw},\psi)$-Shalika model (where $\sw$ is the purity weight of $\lambda$). Thus above each $\lambda \in \Omega_{\mathrm{ncs}}$, there exists at least one classical point $y \in \sC$ corresponding to an automorphic representation $\pi_y$ admitting a Shalika model. In particular, we deduce that $\Omega_{\mathrm{ncs}} \subset w(\sC_{\mathrm{nc}}^{\mathrm{Sha}}).$ Thus $w(\sC_{\mathrm{nc}}^{\mathrm{Sha}})$ is very Zariski-dense in $\Omega$, as required.
\end{proof}

\subsection{Proof of Thm.\ \ref{thm:shalika family}(d--e): \'Etaleness of Shalika families}\label{sec:etaleness S}
At this point we perform a delicate switch in level to prove Theorem \ref{thm:shalika family}(d--e). Fix $\epsilon \in \{\pm 1\}^\Sigma$. 
A key fact about the  level $K_1(\tilde\pi)$  from \eqref{eq:K_1} is the following: 
\begin{proposition}\label{prop:mult one S}
	The vector space $\hc{t}(S_{K_1(\tilde\pi)}, \sV_{\lambda_\pi}^\vee)^\epsilon\locpi$ is 1-dimensional.
\end{proposition}
\begin{proof}
	The $\m_{\tilde\pi}$-torsion in $\pi_f^{K_1(\tilde\pi)}$ is a line; locally, this follows for $v\nmid p$ by \eqref{eq:new vector}, and for $\pri|p$ by \eqref{eq:Up refined line 2}. By Proposition~\ref{prop:non-canonical}, 
	\disp{
	\dim_{\overline{\Q}_p}\ \hc{t}(S_{K_1(\tilde\pi)}, \sV_{\lambda_\pi}^\vee(\overline{\Q}_p))^{\epsilon}\locpi = 1.
}
	We descend to $L$ via \S\ref{sec:periods}.
\end{proof}

Taking $\epsilon$-parts of Proposition~\ref{prop:surjection} gives isomorphisms
\begin{align}\label{eq:double isomorphism}
	\hc{t}(S_{K_1(\tilde\pi)},\sD_{\Omega})\locpi^\epsilon \otimes_\Lambda \Lambda/\m_{\lambda_\pi} &\cong \hc{t}(S_{K_1(\tilde\pi)}, \sD_{\lambda_\pi})\locpi^\epsilon\\ &\cong \hc{t}(S_{K_1(\tilde\pi)},\sV_{\lambda_\pi}^\vee)\locpi^\epsilon\notag
\end{align}
of 1-dimensional vector spaces (where the second isomorphism is non-$Q$-criticality). In particular, there exists a point $x_{\tilde\pi}^{\epsilon}(K_1(\tilde\pi))$ in $\sE_{\Omega,h}^{\epsilon}(K_1(\tilde\pi))$ corresponding to $\tilde\pi$. 
Let 
\[
\bT_{\Omega,\tilde\pi}^{\epsilon}(K_1(\tilde\pi)) = \bT_{\Omega,h}^\epsilon(K_1(\tilde\pi))_{\m_{\tilde\pi}},
\] which acts on $\htc(S_{K_1(\tilde\pi)},\sD_\Omega)\locpi\sshe$ (see Definition~\ref{def:gen eigenspace}).

\begin{proposition}\label{prop:cyclic S}
	There exists a proper ideal $I_{\tilde\pi}^{\epsilon} \subset \Lambda$ such that 
	\[
	\bT_{\Omega,\tilde\pi}^{\epsilon}(K_1(\tilde\pi)) \cong \Lambda/I_{\tilde\pi}^{\epsilon}.
	\]			
\end{proposition}
\begin{proof}
	Since $\htc(S_{K_1(\tilde\pi)},\sD_{\Omega})\ssh$ is a finitely generated $\cO_\Omega$-module, $\htc(S_{K_1(\tilde\pi)},\sD_{\Omega})^\epsilon\locpi$ is finitely generated over $\Lambda$; then Nakayama's lemma applied to \eqref{eq:double isomorphism} implies  that $\htc(S_{K_1(\tilde\pi)},\sD_{\Omega})^\epsilon\locpi$ is non-zero and generated by a single element over $\Lambda$. In particular, it is isomorphic to $\Lambda/I_{\tilde\pi}^{\epsilon}$ for some proper ideal $I_{\tilde\pi}^{\epsilon} \subset \Lambda$. Now, we know that $\T_{\Omega,\tilde\pi}^{\epsilon}(K_1(\tilde\pi))$ is the image of the Hecke algebra in 
	\[
	\End_{\Lambda}\left(\htc(S_{K_1(\tilde\pi)},\sD_\Omega)^\epsilon\locpi\right) \cong \End_{\Lambda}(\Lambda/I_{\tilde\pi}^{\epsilon}) \cong \Lambda/I_{\tilde\pi}^{\epsilon}.
	\]
	But this image contains 1, so $\T_{\Omega,\tilde\pi}^{\epsilon}(K_1(\tilde\pi))$ must be everything, giving the result. 
\end{proof}

To prove (d) and (e), we will combine Proposition~\ref{prop:cyclic S} with Corollary~\ref{cor:T is faithful}(i) (which implies $\bT_{\Omega,\tilde\pi}(K(\tilde\pi))$ is $\Lambda$-torsion free). To switch between levels $K(\tilde\pi)$ and $K_1(\tilde\pi)$, recall the connected component $\sC(K(\tilde\pi))$ from \S\ref{sec:maximal dimension}, and: 
\begin{itemize}
	\item Let $\sC_{\mathrm{nc}}^{\mathrm{lgc}}(K(\tilde\pi))$ be the set of classical cuspidal non-$Q$-critical points $y \in \sC(K(\tilde\pi))$ such that Local-Global Compatibility holds for $\pi_y$ at all $v$. Note (as explained in \S\ref{sec:hypotheses}) that $\sC^{\mathrm{Sha}}_{\mathrm{nc}}(K(\tilde\pi)) \subset \sC_{\mathrm{nc}}^{\mathrm{lgc}}(K(\tilde\pi))$. 
	\item Let $\sC^{\mathrm{lgc}}(K(\tilde\pi))$ be the Zariski-closure of $\sC_{\mathrm{nc}}^{\mathrm{lgc}}(K(\tilde\pi))$, equipped with the induced reduced rigid analytic structure. This contains (the nilreductions of) all Shalika families through $x_{\tilde\pi}(K(\tilde\pi)),$ so by Proposition~\ref{prop:zariski dense shalika no new}, it contains an irreducible component of dimension $\dim(\Omega)$. 
	
\end{itemize}
In the next subsection, we prove:

\begin{proposition}\label{prop:level lowering}
	Let $\tilde\pi$ satisfy the hypotheses of Theorem~\ref{thm:shalika family}(a--c), and suppose $\rho_{\pi}$ is irreducible. Then, up to shrinking $\Omega$, for all $y \in \sC_{\mathrm{nc}}^{\mathrm{lgc}}(K(\tilde\pi))$ and for all $v \in S$, the Whittaker conductors of $\pi_v$ and $\pi_{y,v}$ are equal. In particular, $\pi_{y,f}^{K_1(\tilde\pi)} \neq 0$.
\end{proposition}

\begin{corollary}\label{cor:langlands functoriality}
	For any $\epsilon \in \{\pm1\}^\Sigma$ there exists a closed immersion 
	\[
	\iota : \sC^{\mathrm{lgc}}(K(\tilde\pi)) \hookrightarrow \sE^{\epsilon}_{\Omega, h}(K_1(\tilde\pi))
	\]
	sending $x_{\tilde\pi}(K(\tilde\pi))$ to $x_{\tilde\pi}^{\epsilon}(K_1(\tilde\pi))$.
\end{corollary}
\begin{proof}
	This is a straightforward application of \cite[Thm.~3.2.1]{JoNew}, with the same Hecke algebra $\cH$ and weight space $\Omega$ (by Remark~\ref{rem:weights level}) on both sides, with the identity maps between them. To apply this, it suffices to prove that we have this transfer on a Zariski-dense set of points. The subset of $y \in \sC_{\mathrm{nc}}^{\mathrm{lgc}}(K(\tilde\pi))$ that have non-$Q$-critical slope is Zariski-dense in $\sC^{\mathrm{lgc}}(K(\tilde\pi))$. For such $y$, by Proposition \ref{prop:level lowering} and \eqref{eq:cohomology non-canonical}, we know $\hc{t}(S_{K_1(\tilde\pi)},\sV_\lambda^\vee)^\epsilon_{\m_y} \neq 0$; then by \eqref{eq:double isomorphism} (cf.\ Proposition~\ref{prop:surjection}, Remark \ref{rem:non-Q-critical implies point}) there is a point $y(K_1(\tilde\pi)) \in \sE^{\epsilon}_{\Omega,h}(K_1(\tilde\pi))$ attached to the same Hecke eigensystem as $y$. The transfer is then $y \mapsto y(K_1(\tilde\pi))$ on the (Zariski-dense) subset of non-$Q$-critical slope $y$.
\end{proof}

\begin{corollary}\label{cor:etale bottom level}		Let $\tilde\pi$ satisfy the hypotheses of Theorem~\ref{thm:shalika family}(a--d), and let $\epsilon \in \{\pm1\}^\Sigma$. Then
	\begin{itemize}\setlength{\itemsep}{0pt}
		\item[(i)] 	The weight map $\sE_{\Omega,h}^\epsilon(K_1(\tilde\pi)) \to \Omega$ is \'etale at $x_{\tilde\pi}(K_1(\tilde\pi))$.
		\item[(ii)] The natural map $\sE^\epsilon_{\Omega,h}(K_1(\tilde\pi)) \hookrightarrow \sE_{\Omega, h}(K_1(\tilde\pi))$ is locally an isomorphism at $x_{\tilde\pi}(K_1(\tilde\pi))$.
		\item[(iii)] The weight map $\sE_{\Omega,h}(K_1(\tilde\pi)) \to \Omega$ is \'etale at $x_{\tilde\pi}(K_1(\tilde\pi))$.
	\end{itemize}
\end{corollary}
\begin{proof}
	(i) It suffices to prove that the ideal $I_{\tilde\pi}^{\epsilon}$ from Proposition~\ref{prop:cyclic S} is zero. Suppose it is not; then every irreducible component of $\sE_{\Omega, h}^{\epsilon}(K_1(\tilde\pi))$ through $x_{\tilde\pi}^{\epsilon}(K_1(\tilde\pi))$ has dimension less than $\dim(\Omega)$. But $\sC^{\mathrm{lgc}}(K(\tilde\pi))$ has a component of dimension $\dim(\Omega)$ through $x_{\tilde\pi}^{\epsilon}(K(\tilde\pi))$ by the discussion before Proposition~\ref{prop:level lowering}; under $\iota$ this maps to a component of dimension $\dim(\Omega)$, which is a contradiction.
	
	(ii) Let $\epsilon \neq \epsilon'$, and $\sC^\epsilon, \sC^{\epsilon'}$ be the connected components through $\tilde\pi$ of $\sE_{\Omega,  h}^{\epsilon}(K_1(\tilde\pi))$ and $\sE_{\Omega,  h}^{\epsilon'}(K_1(\tilde\pi))$ respectively. By above, $\sC^\epsilon$ and $\sC^{\epsilon'}$ are \'etale over $\Omega$ and contain Zariski-dense sets $\sC_{\mathrm{nc}}^{\epsilon}, \sC_{\mathrm{nc}}^{\epsilon'}$ of points corresponding to the same set of $Q$-refined RACARs $\{\tilde\pi_y\}_y$. By another application of \cite[Thm.~3.2.1]{JoNew}, there exist closed immersions 
	\[
	\sC^{\epsilon} \hookrightarrow \sC^{\epsilon'} \hookrightarrow \sC^{\epsilon}
	\]
	over $\Omega$ that are the identity on $\{\tilde\pi_y\}_y$; hence $\sC^\epsilon$ and $\sC^{\epsilon'}$ are canonically identified, and $\sC^{\epsilon}$ is independent of $\epsilon$. At $\tilde\pi$, since the Hecke algebra preserves $\epsilon$-parts in cohomology, this means that 
	\[
	\bT_{\Omega,\tilde\pi}(K_1(\tilde\pi)) = \bT_{\Omega,\tilde\pi}^{\epsilon}(K_1(\tilde\pi))
	\]
	as $\Lambda$-modules, and part (ii) follows. Part (iii) is immediate from (i) and (ii).
\end{proof}

Modulo Proposition \ref{prop:level lowering}, this proves Theorem \ref{thm:shalika family}(d1). For (d2), let $\sC^{\mathrm{Sha}}(K(\tilde\pi))^{\mathrm{red}}$ be the nilreduction of $\sC^{\mathrm{Sha}}(K(\tilde\pi))$. By the discussion before Proposition \ref{prop:level lowering}, and Corollary \ref{cor:langlands functoriality}, we have a diagram
\[
\xymatrix@R=5mm{
	\sC^{\mathrm{Sha}}(K(\tilde\pi))^{\mathrm{red}} \sar{r}{\subset}\ar[rd] & \sC^{\mathrm{lgc}}(K(\tilde\pi))\ar[d]\ar@{^{(}->}[r]^{\iota} & \sE^\epsilon_{\Omega,h}(K_1(\tilde\pi))\ar[ld]\\
	& \Omega &
}.
\] 
As $\sC^{\mathrm{Sha}}(K(\tilde\pi))^{\mathrm{red}}$ contains an irreducible component of dimension $\mathrm{dim}(\Omega)$, and $\sE_{\Omega,h}^\epsilon(K_1(\tilde\pi))$ is \'etale over $\Omega$, we deduce $\sC^{\mathrm{Sha}}(K(\tilde\pi))^{\mathrm{red}}$ is \'etale over $\Omega$; hence $\sC^{\mathrm{Sha}}(K(\tilde\pi))$ contains a unique irreducible component, giving (d2). If Local-Global Compatibility holds for all RACARs,  then $\sC^{\mathrm{lgc}}(K(\tilde\pi)) = \sC(K(\tilde\pi))^{\mathrm{red}}$, and the same argument shows this is \'etale over $\Omega$, giving (e).

\subsection{Level-switching: local constancy of conductors}
It remains to prove Proposition~\ref{prop:level lowering}. We use Galois theory. Let $y \in \sC_{\mathrm{nc}}^{\mathrm{lgc}}(K(\tilde\pi))$, with attached $p$-adic Galois representation 
\[
\rho_{\pi_y} : \mathrm{G}_F \to \GL_{2n}(L) \subset  \GL_{2n}(\overline{\Q}_p),
\]
depending on $\iota_p : \C \cong \overline{\Q}_p$. Attached to $\pi_y$ and $v \in S$, we have the Whittaker conductor $m(\pi_{y,v})$ from  \eqref{eq:new vector}, and the \emph{Artin conductor} $a(\rho_{\pi_y}|_{\mathrm{G}_{F_v}})$ of the local restriction, defined by Serre in \cite{SerreConductors}.

\begin{proposition}\label{prop:conductors equal}
	If $y \in \sC_{\mathrm{nc}}^{\mathrm{lgc}}(K(\tilde\pi))$, then for any $v \nmid p$, we have $m(\pi_{y,v}) = a(\rho_{\pi_y}|_{\mathrm{G}_{F_v}})$.
\end{proposition}
\begin{proof}
	Let $m$ and $a$ denote the conductors. Let $\rho_{y,v} = \rho_{\pi_y}|_{\mathrm{G}_{F_v}}$, and $\mathrm{WD}(\rho_{y,v})$ its associated Weil--Deligne representation. By Local-Global Compatibility 	(see \S\ref{sec:hypotheses}) we have 
	\[
	\mathrm{WD}(\rho_{y,v})^{\mathrm{F-ss}} = \iota_p\mathrm{rec}_{F_v}(\pi_{y,v}\otimes|\cdot|^{(1-n)/2}).
	\]
 Fix an unramified non-trivial additive character $\psi_v$ of $F_v$ and let $q_v = \#\cO_v/\varpi_v$. Then:
	\begin{itemize}\setlength{\itemsep}{0pt}
		\item[--] $\rho_{y,v}$ and $\mathrm{WD}(\rho_{y,v})^{\mathrm{F-ss}}$ have the same Artin conductor $a$ (e.g.\ \cite[\S8]{Ulm16});
		\item[--] the map $\mathrm{rec}_{F_v}$ preserves $\varepsilon$-factors \cite{Harris-Taylor}, so $\varepsilon(s,\pi_{y,v}|\cdot|^{(1-n)/2}, \psi_v) = \varepsilon(s,\mathrm{WD}(\rho_{y,v}),\psi_v)$;
		\item[--] by \cite[(3.4.5)]{Tat79}, we have $\varepsilon(s,\mathrm{WD}(\rho_{y,v}),\psi_v) = C\cdot q_v^{-as}$ for $C \in \C^\times$ independent of $s$;
		\item[--] by \cite[(1),Thm.~\S5]{JPSS}, $\varepsilon(s,\pi_{y,v}|\cdot|^{(1-n)/2},\psi_v) = \varepsilon(s+(1-n)/2,\pi_{y,v},\psi_v) = C'\cdot q_v^{-m(1-n)/2} \cdot q_v^{-ms}$, for $C' \in \C^\times$ independent of $s$.
	\end{itemize}
	Hence $C = C'\cdot q_v^{-m(1-n)/2}$ and $a = m$, as required. 
\end{proof}

The study of $m(\pi_{y,v})$ in families is thus reduced to that of $a(\rho_{\pi_y}|_{\mathrm{G}_{F_v}})$, and hence can be studied via Galois theory. 
For simplicity, let $\sC^{\mathrm{lgc}} = \sC^{\mathrm{lgc}}(K(\tilde\pi))$ and $\sC_{\mathrm{nc}}^{\mathrm{lgc}} = \sC_{\mathrm{nc}}^{\mathrm{lgc}}(K(\tilde\pi))$.

\begin{lemma}\label{lem:galois family}
	Suppose $\rho_{\pi}$ is irreducible. Then possibly shrinking $\Omega$, there exists a Galois representation 
	\[
	\rho_{\sC^{\mathrm{lgc}}}: \mathrm{G}_F \to \GL_2(\cO_{\sC^{\mathrm{lgc}}})
	\]
	such that for all $y \in \sC_{\mathrm{nc}}^{\mathrm{lgc}}$, we have $\rho_{\pi_y} = \rho_{\sC^{\mathrm{lgc}}} \newmod{\m_y}$.
\end{lemma}
\begin{proof}
	Let $\nu = (1,0,...,0) \in X_*^+(T_{2n})$. For each $v \notin S\cup\{\pri|p\}$, we have a Hecke operator $T_{\nu,v} \in \cH$ as in \S\ref{sec:unramified H}. If $y \in \sC_{\mathrm{nc}}^{\mathrm{lgc}}$ corresponds to the character $\Psi_y : \cH\otimes L \to L$, then we have 
	\[
	\Psi_y(T_{\nu,v}) = \mathrm{Tr}(\rho_{\pi_y}(\mathrm{Frob}_v))
	\]
	(see e.g.\ \cite[Cor.\ 7.3.4]{Che04}). In particular, property (H) of \cite[\S7.1]{Che04} holds, where we take $a_v$ \emph{ibid}.\  to be the image of $T_{\nu,v}$ in $\cO_{\sC^{\mathrm{lgc}}}$ under the natural map. Then by \cite[Lem.~7.1.1]{Che04}, there exists a $2n$-dimensional Galois pseudo-character 
	\disp{
	t_{\sC^{\mathrm{lgc}}} : \mathrm{G}_F \to \cO_{\sC^{\mathrm{lgc}}}
}
	over $\sC^{\mathrm{lgc}}$ such that for all $v \notin S \cup \{\pri|p\}$ and all $y \in \sC_{\mathrm{nc}}^{\mathrm{lgc}}$, we have 
	\disp{
	\mathrm{sp}_y(t_{\sC^{\mathrm{lgc}}}(\mathrm{Frob}_v)) = \rho_{\pi_y}(\mathrm{Frob}_v).
}

	As $\rho_{\pi}$ is irreducible, by \cite[Lem.~4.3.7]{BC09}, there exists a lift of $t_{\sC^{\mathrm{lgc}}}$ to a Galois representation 
	\disp{
	\rho_{\sC^{\mathrm{lgc}}}: \mathrm{G}_F \to \GL_{2n}(\cO_{\sC^{\mathrm{lgc}}})
}
	with $t_{\sC^{\mathrm{lgc}}} = \mathrm{tr}(\rho_{\sC})$; and $\rho_{\pi_y} = \rho_{\sC^{\mathrm{lgc}}} \newmod{\m_y}$.
\end{proof}

\begin{proposition}\label{prop:artin constant}
	Let $v \in S$ with residue characteristic $\ell \neq p$. After possibly shrinking $\Omega$, the Artin conductor $a(\rho_{\pi_y}|_{\mathrm{G}_{F_v}})$ is constant as $y$ varies in $\sC^{\mathrm{lgc}}$. Hence Proposition \ref{prop:level lowering} holds.
\end{proposition}
\begin{proof}
	Let $(r, N) = \mathrm{WD}(\rho_{\sC^{\mathrm{lgc}}}|_{\mathrm{G}_{F_v}})$ be the family of Weil--Deligne representations associated to $\rho_{\sC^{\mathrm{lgc}}}$ at $v$ \cite[Lem.\ 7.8.14]{BC09}. By construction, the specialisation $(r_y,N_y)$ of $(r,N)$ at $y \in \sC^{\mathrm{lgc}}$ is the Weil--Deligne representation attached to $\rho_{\pi_y}|_{\mathrm{G}_{F_v}}$, and then by definition (see \cite[\S7]{Ulm16}), we have
	\begin{equation}\label{eq:artin conductor formula}
		a(\rho_{\pi_y}|_{\mathrm{G}_{F_v}}) =  a(r_y) + \mathrm{dim} \left(r_y^{\mathrm{I}_v}\right) - \mathrm{dim}[\mathrm{ker}(N_y) \cap r_y^{\mathrm{I}_v}],
	\end{equation}
	with $a(r_y)$ the conductor of $r_y$ (depending only on $r_y|_{\mathrm{I}_v})$. By \cite[Lem.\ 7.8.17]{BC09}, $r|_{\mathrm{I}_v}$ is locally constant over $\sC^{\mathrm{lgc}}$, so we can shrink $\Omega$ so   $a(r_y)$ and $\mathrm{dim}(r_y^{\mathrm{I}_v})$ are constant as $y$ varies in $\sC^{\mathrm{lgc}}$.
	
	Now note that since $\pi$ is essentially self-dual, the specialisation $(r_x, N_x) = \mathrm{WD}(\rho_{\pi}|_{\mathrm{G}_{F_v}})$ is pure. Indeed, it suffices to check this after passing to the base-change $\Pi$ of $\pi$ to a quadratic CM extension $F'/F$ in which $v$ splits as $w\overline{w}$. By \cite[Lem.\ 4.1.4, \S4.3]{CHT08} there exists an algebraic Hecke character $\chi$ over $F'$ such that $\Pi' \defeq \Pi \otimes \chi$ is self-dual, and then \cite[Thm.\ 1.2]{Car12} shows that $\Pi'_w$ is tempered, so has pure Weil--Deligne representation. But purity is preserved by algebraic twist.
	
	Combining \cite[Prop.\ 7.8.19]{BC09} with \cite[Thm.\ 3.1(2)]{Sah17}, purity at $x$ implies that for all $y$ in a neighbourhood of $x$, we have $N_x \sim N_y$ in the sense of \cite[Defs.\ 6.5.1, 7.8.2]{BC09}. This implies that $\mathrm{dim}[\mathrm{ker}(N_y) \cap r_y^{\mathrm{I}_v}]$ -- and hence $a(\rho_{\pi_y}|_{\mathrm{G}_{F_v}})$, by \eqref{eq:artin conductor formula}  -- is constant for $y$ in a neighbourhood of $x$. 
	
	Proposition \ref{prop:level lowering} now follows by combining this with Proposition \ref{prop:conductors equal}.
\end{proof}

\subsection{Remarks on  symplectic components}\label{sec:symplectic components}

The space $\sE_{\Omega,h}(K)$ studied in this section is a local piece of a global parabolic eigenvariety $\sE_{\lambda_0}^{Q}(K)$ varying over $\sW_{\lambda_0}^Q$, constructed in \cite[\S5.2]{BW20}. (Precisely, we take $*=t$ in the notation \emph{op.\ cit}.; that is, this is a `top degree' eigenvariety). Here $\lambda_0$ is any algebraic weight in $\Omega$. We have described its local geometry at certain Shalika points. We now comment on global implications, proving:

\begin{theorem}\label{thm:all classical shalika}
Let $\sI \subset \sE_{\lambda_0}^Q(K)$ be an irreducible component, where $K$ is some parahoric-at-$p$ level. Suppose $\sI$ contains a Shalika point $x_{\tilde\pi}$ attached to a $Q$-refined RASCAR $\tilde\pi$ that is spherical and regular at $p$ and satisfies the hypotheses (a--d) of Theorem \ref{thm:shalika family}, and that $K = K_1(\tilde\pi)$. Then every classical point of $\sI$ with non-$Q$-critical slope and regular weight is a Shalika point.
\end{theorem}

We will prove (in Theorem \ref{thm:all points symplectic}) a stronger result. Let $\cG \defeq \mathrm{Res}_{F/\Q}\mathrm{GSpin}_{2n+1}$ be the split spin group. If $\tilde\pi_y$ is a Shalika point in $\sI$, then $\pi_y$ is the functorial transfer of a RACAR $\Pi_y$ of $\cG(\A)$ (see \S\ref{sec:set-up and previous work}). By \cite[\S3.1]{classical-locus}, there is a refinement $\tilde\Pi_y$ of $\Pi_y$ corresponding to $\tilde\pi_y$. We show there is an irreducible component $\sI^{\cG}$ in a parabolic eigenvariety for $\cG$, and rigid analytic maps between the (nilreductions) of $\sI$ and $\sI^{\cG}$ that interpolate the correspondence $\tilde{\pi}_y \leftrightarrow \tilde{\Pi}_y$ and induce bijections on their sets of points. Thus \emph{every} eigensystem in $\sI$, classical or not, is \emph{symplectic}, a functorial transfer from $\cG$. For non-$Q$-critical slope classical points of regular weight, symplectic is equivalent to Shalika (see Proposition \ref{prop:shalika = symplectic}), so the theorem follows. 

The proof occupies the rest of this section; we sketch it now. One has a natural map from the Hecke algebra for $G$ to that for $\cG$, compatible with Langlands functoriality, induced by a map $\jmath^\vee$ on cocharacters. It also admits a natural section $\iota^\vee$. Using $\iota^\vee$, and properties of Langlands functoriality, one can transfer a Zariski-dense set of Shalika points in $\sI$ from the eigenvariety for $G$ to that for $\cG$. Using an idea of Chenevier, this interpolates to a map $f$ on the nilreduction of $\sI$. Let $\sI^{\cG}$ be the irreducible component containing the (irreducible) image. Applying the same argument in reverse, with $\jmath^\vee$, gives a map $g$ the other way inverse to $f$ on points.

\begin{remark}
	A more detailed study of these phenomena, for all parahoric levels, is the subject of \cite{classical-locus}. As a flavour: in the Iwahori-level eigenvariety, the analogue of Theorem \ref{thm:all classical shalika} (for classical points) should hold, but the stronger analogue (on non-classical points) should not. In the language \emph{op.\ cit}., take a non-critical slope Iwahori refinement of $\pi$ that is optimally $Q$-spin. This varies in a $dn+1$-dimensional component $\sI$ in the Iwahori eigenvariety, but the symplectic locus is a closed $d+1$-dimensional subspace. In \cite{classical-locus} we conjecture that the classical points in $\sI$ lie in the symplectic subspace; but there should exist non-classical non-symplectic points in $\sI$.
\end{remark}

\subsubsection{Hecke algebras for $G$ and $\cG$}
Fix a Borel pair $(\cB,\cT)$ in $\cG$, as in \cite[\S2]{classical-locus}. Attached to $Q \subset G$ is a parabolic $\cQ \subset \cG$, described in \cite[\S2.1]{classical-locus}. Let $\cJ_{\pri}$ be the associated parahoric subgroup. Let $\cU_{\pri}^\circ$ be the associated normalised Hecke operator ($\cU_{\pri,n}^\circ$ in the notation \emph{op.\ cit}.).

Let $S$ be the set of finite primes $v\nmid p$ where $K_v$ is not maximal hyperspecial, and let $\cK = \prod_{v}\cK_v \subset \cG(\A_{F,f})$ be open compact such that $\cK_v$ is maximal hyperspecial for every $v \not\in S\cup\{\pri|p\}$, sufficiently small at $v \in S$, and $\cK_{\pri} = \cJ_{\pri}$. Let $(\cH^{\cG})' = \Qp[\cT_{\nu,v} : \nu \in X_*^+(\cT), v\not\in S\cup\{\pri|p\}]$ be the spherical Hecke algebra for $\cK$, where $X_*^+(\cT)$ is the space of $\cB$-dominant cocharacters of $\cT$ and  $\cT_{\nu,v} = [\cK_v \nu(\varpi_v) \cK_v]$ (as in Definition \ref{def:spherical hecke algebra}). Let $\cH^{\cG} = (\cH^{\cG})'[\cU_{\pri}^\circ : \pri|p]$. 

Henceforth replace the $\Z$-module  $\cH$ (from \S\ref{sec:hecke outside S}) with $\cH\otimes_{\Z}\Qp$. In \cite[\S2]{classical-locus}, a map $\jmath^\vee : X_*(T) \to X_*(\cT)$ is defined. This induces a map 
\[
\jmath^\vee : \cH \to \cH^{\cG}, \qquad U_{\pri}^\circ \mapsto \cU_{\pri}^\circ, \ \ T_{\nu,v} \mapsto \cT_{\jmath^\vee(\nu),v}.
\] 
In the other direction, there is a natural `section'
\[
\iota^\vee : X_*(\cT) \longrightarrow X_*(T) \otimes_{\Z}\Z[1/n]
\]
such that $\iota^\vee \circ \jmath^\vee$ is the identity, given in the notation \emph{op.\ cit}.\ by
\[
\iota^\vee(f_i^*) = e_i^*, \qquad \iota^\vee(f_0^*) = (e_1^* + \cdots + e_n^*)/n.
\]
The denominator means this does not, however, induce a map $\cH^{\cG} \to \cH$. To get around this, for $v \not\in S \cup \{\pri|p\}$ let
\[
Z_v \defeq T_{e_1^*+\cdots+ e_n^*, v} =  [K_v \mathrm{diag}(\varpi_v,...,\varpi_v) K_v], \qquad \cZ_v = \cT_{f_0^*,v}
\]
be the operators attached to $e_1^*+\cdots+e_n^*$ and $f_0^*$ respectively.  Then $\jmath^\vee(Z_v) = \cZ_v^n$.	Any map $\cH^{\cG} \to \cH$ induced by $\iota^\vee$ must send $\cZ_v$ to an $n$th root of $Z_v$. We now make sense of this.

The operators $Z_v$ and $\cZ_v$ act respectively by $\mathrm{diag}(\varpi_v,...,\varpi_v)$ and $f_0^*(\varpi_v)$, elements of the centre of $G(F_v)$ and $\cG(F_v)$ (by \cite[Prop.\ 2.3]{AS06}). Hence they act by the central character evaluated at $\varpi_v$. If $\pi$ is a RASCAR of $G(\A)$ with an $(\eta,\psi)$-Shalika model, then its central character is $\eta^n$, and it is the transfer of a RACAR $\Pi$ of $\cG(\A)$ whose central character is $\eta$ (by \cite[p.178]{AS06}).

This observation allows us to formally define an $n$th root of $Z_v$ over the irreducible component $\sI$. Note (as in Definition \ref{def:shalika family}) $Z_v$ acts on cohomology at any $(\eta_0,\psi)$-Shalika point $y$ by $\big[\eta_0(\varpi_v)|\varpi_v|^{\sw_y}\big]^n$. This varies analytically over any affinoid $\Omega \subset \sW_{\lambda_\pi}^Q$; let
\[
\eta_\Omega(\varpi_v) \defeq \eta_0(\varpi_v) \cdot \sw_\Omega(|\varpi_v|) \in \cO_\Omega^\times,
\]
for $\sw_\Omega$ as in \eqref{eq:sw_Omega}. Note this is well-defined as $v\nmid p$, so $|\varpi_v| \in \Zp^\times$. Then $\eta_\Omega(\varpi_v)^n$ interpolates the action of $Z_v$ on $(\eta_0,\psi)$-Shalika points $\pi_y$ in $\sI$ above $\Omega$. Such points are Zariski-dense in $\sI$ by Theorem \ref{thm:shalika family}, so we deduce $Z_v$ acts via the functions $\eta_\Omega(\varpi_v)^n$ over all of $\sI$.

\begin{definition}
	Let $z_v$ be a formal variable, and let 
	\[
	\tilde{\cH} \defeq \cH\Big[z_v : v\not\in S\cup\{\pri|p\}\Big]/(Z_v - z_v^n).
	\]
\end{definition}

We may summarise much of the above discussion via:

\begin{lemma}
	The map $\jmath^\vee : \cH \to \cH^{\cG}$ extends to a surjective map $\jmath^\vee : \tilde{\cH} \to \cH^{\cG}$. This map has a natural section given by
	\[
	\iota^\vee : \cH^{\cG} \to \tilde{\cH}, \qquad \cU_{\pri}^\circ \mapsto U_\pri^\circ,\ \  \cT_{\nu,v} \mapsto T_{\iota^\vee(\nu),v}.
	\]
\end{lemma}
\begin{proof}
	The extension is defined by $\jmath^\vee(z_v) = \cZ_v$. It is surjective as every generator $\cT_{\nu,v}$ and $\cU_{\pri}^\circ$ is hit. One sees from the definitions that $\iota^\vee$ is a section.
\end{proof} 

\begin{remark}\label{rem:extended action}
	For affinoids $\Omega \subset \cW_{\lambda_\pi}^Q$, and $M$ an $\cH\otimes \cO_\Omega$-module upon which $Z_v$ acts by $\eta_\Omega(\varpi_v)^n$, the action extends to $\tilde{\cH}\otimes \cO_\Omega$, where $z_v$ acts by  $\eta_\Omega(\varpi_v)$. 	From above, this is true for $M = \hc{t}(S_K,\sD_\Omega)\otimes_{\cO_\Omega}\cO_{\sI}$, the specialisation of the cohomology to $\sI$.
\end{remark}

\subsubsection{Eigenvariety data}
By \cite[Cor.\ 3.1.5]{JoNew}, we may recover $\sI$ as the eigenvariety attached to an eigenvariety datum 
\[
\cD = (\sW_{0,\lambda_\pi}^{Q}, \sZ_{\sI}, \sM_{\sI}^t, \cH, \psi)
\]
in the sense of Definition 3.1.1 \emph{op.\ cit}.\ (where the Fredholm hypersurface $\sZ_{\sI}$  and degree $t$ cohomology sheaf $\sM_{\sI}^t = \sM \otimes \cO_{\sI}$ are specialised to isolate $\sI$). We define a modified datum
\[
\tilde{\cD} = (\sW_{0,\lambda_\pi}^{Q}, \sZ_{\sI}, \sM_{\sI}^t, \tilde{\cH}, \tilde{\psi}).
\]
Here $\tilde{\cH}$ acts on $\sM_{\sI}^t$ by Remark \ref{rem:extended action} (giving $\tilde{\psi} : \tilde{\cH} \to \mathrm{End}(\sM_{\sI}^t)$). This gives an eigenvariety $\tilde{\sI}$.

\begin{lemma}\label{lem:tilde equals normal}
	The inclusion $\cH \hookrightarrow \tilde{\cH}$ induces an isomorphism $\tilde{\sI} \isorightarrow \sI$. 
\end{lemma} 
\begin{proof}
	The image of $z_v \otimes 1 \in \tilde{\cH}\otimes\cO_\Omega$ in $\mathrm{End}_{\cO_\Omega}(\hc{t}(S_K,\sD_\Omega)\otimes_{\cO_\Omega}\cO_{\sI}$ is, by definition, equal to $1 \otimes \eta_\Omega(\varpi_v)$, which is also in the image of $\cH \otimes \cO_\Omega$. As this is the only difference between $\tilde{\cH} \otimes \cO_\Omega$ and $\cH \otimes \cO_\Omega$, they have the same image in this endomorphism ring, so the local pieces of $\sI$ and $\tilde{\sI}$ are the same. As the gluing data in \cite[Thm.\ 4.2.2]{Han17} depends only the local pieces, not the abstract Hecke algebra, we conclude.
\end{proof}

Finally, as in \cite[\S5.2.2]{BW20}, at level $\cK$ there is an eigenvariety datum
\[
\cD^{\cG} = (\sW_{\lambda_\pi}^{Q}, \sZ^{\cG}, \sM^{\cG}, \cH^{\cG},\psi^{\cG})
\]
which gives the $\cQ$-parabolic eigenvariety $\sE_{\lambda_\pi}^{\cG,\cQ}(\cK)$ for $\cG$. (Note that $\jmath$, from \cite[\S2]{classical-locus}, identifies the $\cQ$-parabolic weight space for $\cG$ with the $Q$-parabolic weight space for $G$).

\subsubsection{Symplectic points}

\begin{definition}
	Let $y \in \sE_{\lambda_0}^{Q}(K)$ be a point with corresponding eigensystem $\phi_y : \cH \to L$. We say $x$ is \emph{symplectic} if there is a point $y^{\cG} \in \sE_{\lambda_0}^{\cG,\cQ}(\cK)$ for some $\cK$, such that $\phi_y$ factors as 
	\[
		\phi_y : \cH \xrightarrow{\jmath^\vee} \cH^{\cG} \xrightarrow{\phi^{\cG}_y} L,
	\]
	where $\phi_y^{\cG}$ is the eigensystem corresponding to $y$.
\end{definition}

\begin{proposition}\label{prop:shalika = symplectic}
	If $y \in \sE_{\lambda_0}^Q(K)$ is a classical point with non-$Q$-critical slope and regular weight, then $y$ is a symplectic point if and only if it is a Shalika point.
\end{proposition}
\begin{proof}	
	Suppose $y$ is a Shalika point. Let $\tilde\pi_y$ and $\tilde\Pi_y$ be as described after the statement of Theorem \ref{thm:all classical shalika}; then $\phi_y = \phi_{\tilde\pi_y}$. By compatibility of Langlands functoriality (at $v \nmid p$) and \cite[Prop.\ 3.7]{classical-locus} (at $\pri | p$) $\phi_{\tilde\pi_y}$ factors as
	\begin{equation}\label{eq:langlands factor 1}
		\phi_{\tilde\pi_y} : \cH \xrightarrow{\jmath^\vee} \cH^{\cG} \xrightarrow{\phi_{\tilde\Pi_y}} L.
	\end{equation}
	It remains to show $\tilde\Pi_y$ appears in an eigenvariety for $\cG$. Let $\cK \subset \cG(\A_{F,f})$ be open compact as above (maximal hyperspecial at $v \notin S\cup\{\pri|p\}$, parahoric at $\pri|p$) such that $\Pi_y^{\cK} \neq \{0\}$. By \cite[\S3.5]{classical-locus}, the refinement $\tilde\Pi_y$ has non-$\cQ$-critical slope, so by \cite[Prop.\ 5.8]{BW20}, yields a point $y^{\cG} \in \sE_{\lambda_\pi}^{\cG,\cQ}(\cK)$ corresponding to $\phi_{\tilde\Pi_y}$, and $y$ is symplectic.
	
	Conversely, suppose $y$ is symplectic; then by \cite[\S3.5]{classical-locus}, $y^{\cG}$ is non-$\cQ$-critical slope in $\sE_{\lambda_0}^{\cG,\cQ}(\cK)$. Using regular weight, as in the proof of \cite[Prop.\ 5.15]{BW20}, $y^{\cG}$ is classical cuspidal, corresponding to some RACAR $\Pi_y$ of $\cG(\A_F)$. At $v \notin S\cup\{\pri|p\}$, $\Pi_{y,v}$ is unramified; 
 by considering the Satake parameters and using \cite[\S6]{AS06} we see that $\pi_{y,v}$ is the functorial transfer of $\Pi_{y,v}$. By \cite{AS14} this ensures $\pi_y$ is globally the transfer of $\Pi_y$. Thus $\pi_y$ admits a Shalika model, as required.
\end{proof}

\begin{theorem}\label{thm:all points symplectic}
	Let $\sI \subset \sE_{\lambda_0}^Q(K)$ be an irreducible component satisfying the conditions of Theorem \ref{thm:all classical shalika}. Then every point of $\sI$ is a symplectic point.
\end{theorem}
\begin{proof}
We maintain the notation from the proof of Proposition \ref{prop:shalika = symplectic}. 
 By \cite[Prop.\ 5.1]{AS14}, which controls the image of functorial transfer at ramified places, we may choose $\cK \subset \cG(\A_{F,f})$ as above such that $\Pi_y^{\cK} \neq \{0\}$ for all such $y$; we work at this level for $\cG$.

We have the following `inverse' of \eqref{eq:langlands factor 1}; extend $\phi_{\tilde\pi_y}$ to $\tilde{\phi}_{\tilde\pi_y} : \tilde{\cH} \to L$ by sending $z_v \mapsto \eta_0(\varpi_v)|\varpi_v|^{\sw_y}$. As $\iota^\vee$ is a section of $\jmath^\vee$, $\phi_{\tilde\Pi_y}$ factors as
\begin{equation}\label{eq:langlands factor 2}
	\phi_{\tilde\Pi_y} : \cH^{\cG} \xrightarrow{\iota^\vee} \tilde{\cH} \xrightarrow{\tilde{\phi}_{\tilde\pi_y}} L.
\end{equation}
As in the proof of Proposition \ref{prop:zariski dense shalika no new}, a neighbourhood $\sU$ of the given point $\tilde\pi$ contains a Zariski-dense set of non-$Q$-critical slope $(\eta_0,\psi)$-Shalika points $y \in \sU$. By \cite[Lem.\ 2.2.3]{Con99}, this set is also Zariski-dense in $\sI$. By Proposition \ref{prop:shalika = symplectic}, we have associated points $y^{\cG} \in \sE_{\lambda_0}^{\cG,\cQ}(\cK)$.

Let $\tilde{\sI}^\circ$ denote the nilreduction of $\tilde{\sI}$. By \cite[Thm.\ 3.2.1]{JoNew} and \eqref{eq:langlands factor 2}, the map $\iota^\vee$ induces a map $g : \tilde{\sI}^{\circ} \to \sE_{\lambda_\pi}^{\cG,\cQ}(\cK)$ interpolating the association $y \mapsto y^{\cG}$ for the Zariski-dense set of $(\eta_0,\psi)$-Shalika points $y \in \sI$. Conversely let $\sI^{\cG}$ be the irreducible component containing $g(\tilde{\sI}^\circ)$, and $\sI^{\cG,\circ}$ its nilreduction. By the same theorem and \eqref{eq:langlands factor 1}, $\jmath^\vee$ induces a map $f : \sI^{\cG,\circ} \to \sI$. 

As nilreductions do not change closed points, and $\tilde{\sI}$ is isomorphic to $\sI$ by Lemma \ref{lem:tilde equals normal}, the maps $f$ and $g$ induce inverse bijections on the sets of closed points in $\sI$ and $\sI^{\cG}$. By \cite[Thm.\ 3.2.1]{JoNew} again this means every eigensystem in $\sI$ factors through $\jmath^\vee$, and hence is symplectic. 
\end{proof}

Theorem \ref{thm:all classical shalika} follows immediately by combining Theorem \ref{thm:all points symplectic} with Proposition \ref{prop:shalika = symplectic}.

\section{$p$-adic $L$-functions over the eigenvariety}\label{sec:families of p-adic L-functions}\label{ss:local structure}

Finally we construct $p$-adic $L$-functions in families and prove Theorem~\ref{thm:intro 3} of the introduction. 

In \S\ref{sec:shalika families no new}, we proved existence and \'etaleness of Shalika families, but had to consider and compare two separate levels $K(\tilde\pi)$ and $K_1(\tilde\pi)$ to do so. To vary $p$-adic $L$-functions over these families requires more precise control still, since for (2$'$) we must show not only that $\tilde\pi$ varies in a Shalika family, but that specific vectors inside these representations -- the cusp forms $W_f^{\mathrm{FJ}}$ from \S\ref{sec:periods} -- also vary $p$-adic analytically in this family. In this chapter, we prove such variation if $\pi$ satisfies an automorphic hypothesis (Hypothesis \ref{ass:shalika}), which is unconditional in tame level 1.

\subsection{On the choice of local test vectors}\label{sec:hypothesis}

Suppose $\tilde\pi$ satisfies (C1-2) of Conditions~\ref{cond:running assumptions}. Recall from \S\ref{sec:unramified H} that 
\disp{
S = \{v\nmid p\infty : \pi_v\text{ ramified}\}.
}
To vary the cusp form $W_f^{\mathrm{FJ}} =  \otimes_v W_v^{\mathrm{FJ}}$ in a $p$-adic family, we need control on the local vectors $W_v^{\mathrm{FJ}}$. 
		At $\pri|p$ and $v\not\in S$, we have described explicit test vectors. In tame level 1, this crucially this means $K(\tilde\pi) = K_1(\tilde\pi),$ and:
		\begin{proposition}\label{prop:tame level 1}
			Let $\tilde\pi$ be a non-$Q$-critical $Q$-refined RACAR satisfying (C1-2) of Conditions~\ref{cond:running assumptions}. Suppose $\pi$ has tame level 1. Then $\cS_{\psi_f}^{\eta_f}(\pi_f^{K(\tilde\pi)})[\![U_{\pri} - \alpha_{\pri} : \pri|p]\!]$ is a line, generated by
				$W_f^{\mathrm{FJ}} = \otimes_{\pri|p} W_{\pri} \otimes_{v\nmid p \infty} W_v^{\mathrm{FJ}},$
				where each $W_{\pri}$ is as in (C2) and each $W_v^{\mathrm{FJ}}$ is a Friedberg--Jacquet test vector.
		\end{proposition}

\subsection{Shalika new vectors}\label{sec:shalika-new-line}

It is natural to ask if there is a  Shalika analogue of the theory of Whittaker new vectors. We suggest a possible theory.

\subsubsection{Shalika conductors}\label{sec:shalika new vectors}
Let $c \geqslant 1$ be an integer.  Rather than the subgroups $K_{1,v}(c)$ used in the Whittaker theory, we consider the `$Q$-parahoric' analogue
	\begin{align*}
		J_{v}(c)\defeq  \left\{g \in \GL_{2n}(\cO_v) \ \big|\ g \newmod{\varpi_v^c} \in Q(\cO_v/\varpi_v^c)\right\}.
	\end{align*} 
We also set $J_{v}(0)=\GL_{2n}(\cO_v)$. Note that $J_v(1)$ is just the parahoric subgroup $J_v$.

\begin{definition}\label{def:shalika new vectors}
	Suppose $\pi_v$ is an irreducible admissible representation of $\GL_{2n}(F_v)$ that admits an $(\eta_v,\psi_v)$-Shalika model. 
	\begin{itemize}\s
		\item[(1)] The \emph{Shalika conductor} $c(\pi_v)$ of $\pi_v$ is the smallest $c \in \Z_{\geqslant 0}$ (if it exists) such that 
\[		\hspace{-7mm} 
\cS_{\psi_v}^{\eta_v}\big(\pi_v^{J_v(c),\eta_v}\big):=	\left\{ W_v\in \cS_{\psi_v}^{\eta_v}(\pi_v) \, \Big{|} \, 
	W_v(-\cdot k)= \eta_v(\det(k_2)) W_v(-)\ \  \forall k=\smallmatrd{k_1}{*}{ * }{k_2}\in J_{v}(c)\right\}\ne \{0\}.
	\]		
	\item[(2)] 	If  $\cS_{\psi_v}^{\eta_v}\big(\pi_v^{J_v(c(\pi_v)),\eta_v}\big)$ is a line, we call 
	a \emph{Shalika new vector} any  generator of this line.	
	\end{itemize} 
\end{definition}

\begin{lemma}\label{lem:shalika conductor} Suppose $\pi_v$ is an irreducible admissible representation of $\GL_{2n}(F_v)$ that admits an $(\eta_v,\psi_v)$-Shalika model. 
Then the   Shalika conductor $c(\pi_v) \in \Z_{\geqslant 0}$ exists. 

Moreover for any $c \geqslant c(\pi_v)$ one has $\dim(\pi_v^{J_v(c+1),\eta_v})>\dim(\pi_v^{J_v(c),\eta_v})$. 
\end{lemma}

\begin{proof} 
	As $\pi_v$  admits an $(\eta_v,\psi_v)$-Shalika model,  the Friedberg--Jacquet linear functional \cite{FJ93} is a non-zero element of $\Hom_{H(\cO_v)}(\pi_v, \eta_v)$. As $H(\cO_v)$ is compact, one therefore has $\Hom_{H(\cO_v)}(\eta_v, \pi_v)\ne\{0\}$; that is, there exists a non-zero vector $\varphi \in \pi_v^{H(\cO_v),\eta_v}$. 
	
	As $\varphi$ is smooth, there exists some $c \gg 0$ such that $\varphi$ is fixed by $N_Q(\varpi_v^c\cO_v)$ and $N_Q^-(\varpi_v^c\cO_v)$. Let $t_v = \mathrm{diag}(\varpi_v I_n, I_n)$. Since $J_v(2c) = t_v^{-c} N_Q(\varpi_v^c\cO_v) H(\cO_v)N_Q^-(\varpi_v^c\cO_v) t_v^c$, we deduce $t_v^{-c} \cdot \varphi \in \pi_v^{J_v(2c),\eta_v}$. Thus for $c \gg 0$, the space in Definition \ref{def:shalika new vectors}(1) is non-zero, so the conductor exists.

For the proof of the last claim,  to ease notation, we will drop  $\eta_v$ from the exponent. 
	As  $\pi_v^{J_v(c)} \subset \pi_v^{J_v(c +1)}$, it suffices to prove that the inclusion is strict. 
	\begin{itemize}
		\item	If $c = 0$, then $\pi_v$ is spherical, and $\dim(\pi_v^{J_v(1)})=\left(\begin{smallmatrix}2n \\ n\end{smallmatrix}\right) > 1$ (see e.g. \cite[\S3.1]{DJR18}).
	\end{itemize}
	
	 If $c \geqslant 1$, let $\varphi$ be an element of  $\pi_v^{J_v(c)}$ which we can inductively assume not in  
	 $\pi_v^{J_v(c-1)}$. 
	 
	  Suppose that $\pi_v^{J_v(c+1)} = \pi_v^{J_v(c)}$.
	 Note $t_v^{-1}\cdot \varphi \in \pi_v{}^{t_v^{-1}J_v(c)t_v}$, and
	\disp{	J_v(c+1) \subset t_v^{-1} \cdot J_v(c)  \cdot t_v,}  hence 
	\disp{	t_v^{-1} \cdot \varphi \in \pi_v^{J_v(c+1)} = \pi_v^{J_v(c)}.
	}
	We thus deduce
	\disp{
	\varphi \in \pi_v^{t_vJ_v(c) t_v^{-1}},
}
	so $\varphi$ is fixed by both $J_v(c)$ and $t_vJ_v(c)t_v^{-1}$, and hence by the group $J' \subset \GL_{2n}(F_v)$ that they generate. To obtain a contradiction with the assumption that $\varphi \notin \pi_v^{J_v(c-1)}$, it suffices to show 
	\begin{equation}\label{eq:J(c-1)}
		J_v(c-1) \subset J'. 
	\end{equation}
	\begin{itemize}\s
		\item Suppose $c \geqslant 2$. Then $J_v(c-1)$ admits a parahoric decomposition 
			$J_v(c-1) = [J_v(c-1)\cap N^-_Q(\cO_v)] \cdot H(\cO_v)N_Q(\cO_v) = t_v[J_v(c)\cap N^-_Q(\cO_v)]t_v^{-1} \cdot H(\cO_v)N_Q(\cO_v)$. This lies in $t_vJ_v(c)t_v^{-1} \cdot J_v(c) \subset J',$
		as required.
		
		\item  If $c=1$, then observe that 
		\disp{
		H(\cO_v) N_Q(\cO_v) \subset J_v(1) \subset J'}		and 
		\disp{		N^-_Q(\cO_v) = t_vN^-_Q(\varpi_v \cO_v)t_v^{-1} \subset t_vJ_v(1)t_v^{-1} \subset J',}
		hence
					\begin{align}
				J' \supset N^-_Q(\cO_v)\cdot H(\cO_v) \cdot N_Q(\cO_v) =  \smallmatrd{\GL_n(\cO_v)}{\mathrm{M}_n(\cO_v)}{\mathrm{M}_n(\cO_v)}{\mathrm{M}_n(\cO_v)} \cap \GL_{2n}(\cO_v).\label{eq:J''}
			\end{align}
		Let $J''$ be the group generated by this last set; it suffices to prove that $J''$ contains $\GL_{2n}(\cO_v)$. Since by definition $J''$ contains the Borel and Iwahori subgroups in $\GL_{2n}(\cO_v)$, by Bruhat decomposition it suffices to prove that $J''$ contains the Weyl group $S_{2n}$ of $G$ (that is, that it contains the subgroup of all permutation matrices). This subgroup in turn is generated by $S_n \times S_n$ (which is contained in $H(\cO_v)$, hence in $J''$) and the transposition $(n,n+1)$, given by the block-diagonal matrix
		
		\[
		\left(\begin{smallmatrix}I_{n-1} & & &\\
			& 0 & 1 &\\
			& 1 & 0 &\\
			&&& I_{n-1}
		\end{smallmatrix}\right) = 	\left(\begin{smallmatrix}I_{n-1} & & &\\
			& -1 & 1 &\\
			& 1 & 0 &\\
			&&& I_{n-1}
		\end{smallmatrix}\right) 	\left(\begin{smallmatrix}I_{n-1} & & &\\
			& 1 & 0 &\\
			& 1 & 1 &\\
			&&& I_{n-1}
		\end{smallmatrix}\right) 
		\]
		Both elements in the product are in \eqref{eq:J''}, so this element lies in $J''$, and we are done.\qedhere
	\end{itemize}

\end{proof}

For $c = c(\pi_v)$, let 
\disp{
J_{v}^{\eta}(c) \defeq \ker[\eta_v\circ\mathrm{det}_2: J_{v}(c)\to \C^\times]
}
and consider the (diamond) Hecke operators
\[
S_{\alpha_v}=[J_{v}^{\eta}(c)\ \mathrm{diag}( 1, \dots, 1, \alpha_v) \  J_{v}^{\eta}(c)],\ \ \ \ \  \alpha_v \in \cO_v^{\times}.  
\] 
\begin{lemma}\label{lem:equivalent shalika}
	Any $\pi_v$ as in Lemma \ref{lem:shalika conductor} admits a Shalika new vector if and only if 
	\begin{equation}\label{eq:SEL 2}
		\dim_\C\  \pi_v{}^{J_{v}^{\eta}(c(\pi_v))}\big[S_{\alpha_v} - \eta_v(\alpha_v) : \alpha_v \in \cO_v^\times\big] \ = 1.
	\end{equation}
\end{lemma}
\begin{proof}
	For $\alpha_v \in \cO_v^\times$, let $t_{\alpha_v} = \mathrm{diag}(1,...,1,\alpha_v)$. Via $\det_2$, one sees that $\{t_{\alpha_v} : \alpha_v \in \cO_v^\times\}$ contains a complete set of representatives for $J_{v}(c(\pi_v))/J_{v}^{\eta}(c(\pi_v))$. Additionally the Hecke operator $S_{\alpha_v}$ is simply right translation by $t_{\alpha_v}$. Hence \eqref{eq:SEL 2} is a reformulation of Definition~\ref{def:shalika new vectors}(2).
\end{proof}

\subsubsection{Shalika new vectors for parahoric-spherical representations}\label{sec:examples shalika}
If $\pi_v$ is spherical, then it has Shalika conductor $0$, and a spherical vector is a Shalika new vector. If $\pi_v$ is parahoric-spherical, then it has Shalika conductor $\leq 1$ (and conversely, if $\eta_v$ is unramified). We will now see that even in this simple case some
representations admit Shalika new vectors, while others do not.
	We thank David Loeffler and Andrei Jorza for their help in finding these examples.

\begin{lemma}\label{lem:examples} Let $\mathrm{St}_v$ denote the Steinberg representation of $\GL_2(F_v)$.
	\begin{enumerate}[(i)]\setlength{\itemsep}{0pt}
		\item Let $\pi_v$ be the full parabolic induction from $Q(F_v)$ to $\GL_4(F_v)$ of $\mathrm{St}_v \times \mathrm{St}_v$. 
		Then $\pi_v$ is parahoric-spherical and admits a Shalika new vector.  
		
		\item  Let  $P$ denote the $(1,2,1)$ parabolic of $\GL_4$ and let 
		$\pi'_v$ be the full parabolic induction from $P(F_v)$ to $\GL_4(F_v)$ of 
		$\mathbf{1}\times \mathrm{St}_v\times\mathbf{1}$. 
		Then $\pi'_v$ is parahoric spherical but does not admit a Shalika new vector. 
	\end{enumerate}
\end{lemma}
\begin{proof} Let us first observe that both $\pi_v$ and $\pi'_v$ are ramified representations admitting a Shalika model for $\eta_v=\mathbf{1}$. We realise the Weyl groups $W_Q$ and $W_P$ as the subgroups of the 
	Weyl group $S_4$ of $\GL_4$ generated respectively by $\{(12),(34)\}$  and $\{(23)\}$. 
	Let  $\{\beta_1, \beta_2, \beta_3\}$ denote the simple roots of $\GL_4$. Note the parabolic subgroups $Q$ and $P$ correspond to the subsets  $\{\beta_1, \beta_3\}$ and $\{\beta_2\}$ respectively. 
	
	\medskip
	
	(i) It suffices to show that $\dim_\C \pi_v^{J_v}=1$. One can easily check that a set of representatives of the double coset $W_Q\backslash S_4 / W_Q$ is given by $\overline{W}=\{(1), (23), (13)(24)\}$. By 
	\cite[\S1]{DJ-parahoric} the dimension of $\pi_v^{J_v}$ is given by the number of $w\in \overline{W}$ such that $w\cdot \{\beta_1, \beta_3\}\cap \{\beta_1, \beta_3\}=\varnothing$. This is only the case for $w=(23)$.  
	
	\medskip
	
	(ii) It suffices to show that $\dim_\C \pi_v^{J_v}>1$. A set of representatives of the double coset $W_P\backslash S_4 / W_Q$ is given by $\overline{W}'=\{(1), (123), (1243), (243)\}$. In this case, there are two elements 
	$w\in \overline{W}'$, namely $(1)$ and  $(1243)$, for which $w\cdot \{\beta_1, \beta_3\}\cap \{\beta_2\}=\varnothing$.
	The same argument as above shows that the space of $J_v$-invariants in $\pi'_v$ is $2$-dimensional. 
\end{proof}

We refer the interested reader to \cite[\S1]{DJ-parahoric} for a full classification of the parahoric-spherical 
generic representations of $\GL_{2n}$ admitting a Shalika model.

\subsubsection{A hypothesis on Shalika new vectors}\label{sec:shalika hypothesis}
Given the above theory of Shalika new vectors, it seems natural to make the following hypothesis.

\begin{hypothesis}\label{ass:shalika} 
	Let $c \in \Z_{\geqslant 0}$.  For any  $\pi_v$ admitting a Shalika new vector of  conductor $c$, 
a multiple of the latter is also a Friedberg--Jacquet test vector for $\pi_v$ (as in \S\ref{sec:shalika models}).
\end{hypothesis}

	As evidence towards this, we note that Friedberg--Jacquet \cite[Prop.~3.2]{FJ93}   proved that Hypothesis \ref{ass:shalika} holds for $c = 0$  (see also \S\ref{sec:shalika models}). 
	Further, in \cite[\S1]{DJ-parahoric} it is shown that the $\pi_v$ admitting a Shalika new vector of  conductor 
	$c = 1$ are precisely the parahoric-spherical representations which are maximally Steinberg, and further, 
	it is established in \cite[\S2]{DJ-parahoric} that for such $\pi_v$ Hypothesis \ref{ass:shalika} holds provided that 
	$\pi_v$ is regular (i.e., occurs in $\Ind_B^G\theta_v$ with $\theta_v$ regular).

\subsection{Shalika families, refined}\label{sec:shalika families refined}

\subsubsection{Set-up: Shalika Hecke algebra and the eigenvariety $\sE^\dec$}\label{sec:set-up}

Let $\tilde\pi$ be a non-$Q$-critical $Q$-refined RACAR of weight $\lambda_\pi$ satisfying (C1-2) of Conditions~\ref{cond:running assumptions}. Recall $S = \{v\nmid p\infty : \pi_v\text{ ramified}\}$. For the rest of this chapter, we assume that:
\begin{quote} \emph{for all $v \in S$, the local representation $\pi_v$ admits a Shalika new vector of conductor $c(\pi_v)$, and Hypothesis \ref{ass:shalika} holds for $c = c(\pi_v)$.}
\end{quote}

Given these assumptions,  for the rest of the paper we fix a sufficiently large coefficient field $L/\Qp$ as in \S\ref{sec:periods}, and drop it from notation. We also fix a precise choice of level subgroup
\begin{equation}\label{eq:shalika level}
	K(\tilde\pi)=\textstyle\prod_{\pri\mid p} J_\fp \cdot \prod_{v\in S} J_{v}^{\eta}(c(\pi_v))  \prod_{v\notin S\cup\{\pri|p\}}\GL_{2n}(\cO_v)\subset G(\A_f).
\end{equation}

\medskip

Recall the Hecke alegbra $\cH$ from Definition~\ref{def:hecke algebra}. In light of Lemma \ref{lem:equivalent shalika}, it is necessary to modify our Hecke algebra by adding the diamond operators $S_{\alpha_v}$. 

\begin{definition} \label{def:hecke algebra full}
	The \emph{Shalika Hecke algebra of level $K(\tilde\pi)$} is 
	\disp{
	\cH^{\dec} \defeq \cH[S_{\alpha_v} : v\in S, \alpha_v \in \cO_v^{\times}].
} 
This acts on $\pi^{K (\tilde\pi) }$ and $\hc{\bullet}(S_{K (\tilde\pi)},-)$.
\end{definition}

We define some modified objects exactly analogous to their incarnations in \S\ref{sec:shalika families no new}, but with $\cH^\dec$ replacing $\cH$.  Recall $\psi_{\tilde\pi}$ from Definition~\ref{def:gen eigenspace} (noting $E \subset L$).

\begin{definition}\label{def:gen eigenspace shalika}
	Define a character $\psi_{\tilde\pi}^{\dec} : \cH^{\dec}\otimes L \longrightarrow L$ extending $\psi_{\tilde\pi}$ by 
	sending $S_{\alpha_v} \mapsto \eta(\alpha_v)$. Let $\m_{\tilde\pi}^{\dec} \defeq \ker(\psi_{\tilde\pi}^{\dec})$ be the associated maximal ideal. For $\Omega\subset \Wlam$ a neighbourhood of $\lambda_\pi$, as in  \eqref{eq:m pi families} we get an associated maximal ideal, also denoted $\m_{\tilde\pi}^{\dec}$, in $\cH^{\dec} \otimes  \cO_\Omega$. 
\end{definition}

\begin{proposition}\label{thm:mult one} 
	Let $\tilde\pi$ satisfy Conditions~\ref{cond:running assumptions} and assume Hypothesis~\ref{ass:shalika} for  $v \in S$. For any $\epsilon \in \{\pm1\}^{\Sigma}$, 
	\disp{
	\dim_L\htc(S_{K(\tilde\pi)},\sV_{\lambda_\pi}^\vee)^{\epsilon}\locpiS = 1.
}
	If $\tilde\pi$ is non-$Q$-critical, then 
	\disp{
	\dim_L\htc(S_{K(\tilde\pi)},\sD_{\lambda_\pi})^{\epsilon}\locpiS = 1.
}
\end{proposition}

\begin{proof}
		The first equality is identical to Propostion \ref{prop:mult one S}. The second is non-$Q$-criticality.
\end{proof}

Via \S\ref{sec:slope-decomp}, let $\Omega$ be an affinoid neigbourhood of $\lambda_\pi$ in $\Wlam$  such that $\htc(S_{K(\tilde\pi)},\sD_{\Omega})$ admits a slope $\leqslant h$ decomposition with respect to the $U_p$ operator.

\begin{definition}\label{def:local piece}
		Let $\T\USha$ be the image of $\cH^\dec\otimes\cO_\Omega \longrightarrow \End_{\cO_\Omega}\big(\htc(S_{K(\tilde\pi)},\sD_{\Omega})^{\leqslant h}\big).$ Let $\sE\USha$, $\bT\USh$ and $\sE\USh$ be the obvious analogues of Definition \ref{def:local piece S}. For a classical cuspidal point $y \in \cE\USha$, we write $\m_y^{\dec} \defeq \m_{\tilde\pi_y}^{\dec}$.
\end{definition}

\subsubsection{Shalika families, refined: statement}

The following is a more precise version of Theorem \ref{thm:intro shalika family 2}, refining Theorem \ref{thm:shalika family}. We recall that  we have fixed the level subgroup $K(\tilde\pi)$ and coefficient field $L$ (in \S\ref{sec:set-up}), and we drop both from further notation. Let $\alpha_p^\circ=\prod_{\pri\mid p}(\alpha_\pri^\circ)^{e_\pri}$, and fix $h\geqslant v_p(\alpha_p^\circ)$ and $\epsilon \in \{\pm1\}^\Sigma$. By Proposition~\ref{thm:mult one}, $\tilde\pi$ contributes to  $\htc(S_{K(\tilde\pi)},\sD_{\lambda_\pi})\ssh$. 

\begin{theorem}\label{thm:section 7 main theorem}
	Let $\tilde\pi$ be non-$Q$-critical satisfying (C1-2), and suppose that $\pi$ admits a non-zero Deligne-critical $L$-value at $p$ (Definition~\ref{def:non-vanishing}), that $\lambda_\pi$ is $H$-regular \eqref{eq:H-regular}, and that $\tilde\pi$ is strongly non-$Q$-critical (Definition~\ref{def:non-Q-critical}). 
			Suppose that for all $v \in S$, $\pi_v$ admits a Shalika new vector of conductor $c(\pi_v)$, and Hypothesis \ref{ass:shalika} holds for $c = c(\pi_v)$.

	Then there exists a point $x_{\tilde\pi}^{\dec}$ attached to $\tilde\pi$ in $\sE\USha$. 
	Let $\sC$ be the connected component of $\sE\USha$ through $x_{\tilde\pi}^{\dec}$. Then, after possibly shrinking $\Omega$,
	\begin{itemize}\setlength{\itemsep}{0pt}
		\item \textbf{\emph{(\'etaleness)}} the weight map $\sC \to \Omega$ is \'etale,  
		\item \textbf{\emph{(density of Shalika points)}} $\sC$ contains a Zariski-dense set $\sC_{\mathrm{nc}}$ of classical cuspidal points $y$ corresponding to non-$Q$-critical $Q$-refined RACARs $\tilde\pi_y$ satisfying (C1-2) of Conditions~\ref{cond:running assumptions},
		\item \textbf{\emph{(Shalika new vectors in families)}} for each $y \in \sC_{\mathrm{nc}}$ and for all  $v \in S$, $\pi_{y,v}$ admits a Shalika new vector of conductor $c(\pi_{y,v}) = c(\pi_v)$, and
		\item \textbf{\emph{(family of eigenclasses)}} for each $\epsilon \in \{\pm1\}^\Sigma$, there exists a Hecke eigenclass $\Phi_{\sC}^\epsilon \in \htc(S_{K(\tilde\pi)},\sD_{\Omega})^{\epsilon}$ such that for every $y \in \sC_{\mathrm{nc}}$ with $w(y) = \lambda_y$, the specialisation $\mathrm{sp}_{\lambda_y}(\Phi_{\sC}^\epsilon)$ generates $\htc(S_{K(\tilde\pi)},\sD_{\lambda_y})^\epsilon_{\m_y^{\dec}}$.
	\end{itemize}
\end{theorem}

The non-vanishing, $H$-regular and strongly non-$Q$-critical hypotheses hold if $\tilde\pi$ has non-$Q$-critical slope and $\lambda_\pi$ is regular (Theorem~\ref{thm:control}, Lemma~\ref{lem:regular weight implies non-vanishing}), so this implies Theorem~\ref{thm:intro shalika family 2} of the introduction.
The proof of Theorem~\ref{thm:section 7 main theorem} will occupy the rest of \S\ref{sec:shalika families refined}; it is similar to the methods of \S\ref{sec:shalika families no new}, with the addition of some standard arguments, which we highlight.

\subsubsection{Cyclicity results} \label{sec:existence of x}
Let $\T\USpi = (\T_{\Omega, h}^\deceps)\locpiS$. Recall $\Lambda = \cO_{\Omega,\m_{\lambda_\pi}}$.

\begin{proposition}\label{thm:free rank one}
	\begin{itemize}\setlength{\itemsep}{0pt}
		\item[(i)] There exists a proper ideal $I_{\tilde\pi} \subset \Lambda$ such that 
		\disp{
		\T\USpi \cong \Lambda/I_{\tilde\pi}.
	}
		\item[(ii)] The space $\htc(S_{K(\tilde\pi)},\sD_{\Omega})\locpiS^\eps$ is free of rank one over $\T\USpi$. 
	\end{itemize}
\end{proposition}

\begin{proof}
Part (i) is identical to Proposition \ref{prop:cyclic S}. 		The actions of $\Lambda$ and $\T\USpi$ are compatible, so $\htc(S_{K(\tilde\pi)},\sD_{\Omega})^\epsilon\locpiS$ is free of rank one over $\T\USpi$, giving (ii).
\end{proof}

Since $\hc{t}(S_{K(\tilde\pi)}, \sD_\Omega)\locpiS \neq 0$, there exists a point $x_{\tilde\pi}^{\deceps} \in \cE\USh$ attached to $\tilde\pi$.

	To construct the eigenclasses $\Phi_{\sC}^{\epsilon}$ of Theorem~\ref{thm:section 7 main theorem}, we want to delocalise Proposition \ref{thm:free rank one} to a neighbourhood of $x_{\tilde\pi}^{\deceps}$ in $\sE\USh$. A standard procedure using rigid localisations (in the sense of \cite[\S7.3.2]{BGR}; see \cite[Lem.\ 2.10]{BDJ17}) shows that:
	
\begin{proposition}\label{prop:delocalise 1}
Let 
		\disp{
			\sC^\epsilon = \Sp(\bT_{\Omega,\sC}^\deceps) \subset \sE\USh
		}
		be the connected component containing $x_{\tilde\pi}^{\deceps}$. After possibly shrinking $\Omega \subset \Wlam$, there exists an ideal $I_{\sC^\epsilon} \subset \cO_\Omega$ such that  
		\disp{
			\bT_{\Omega,\sC}^{\deceps} \cong \cO_\Omega/I_{\sC^\epsilon}.
		}
\end{proposition}

	From this, another standard argument (cf.\ \cite[Cor.\ 4.8]{BW18} or \cite[\S2.7]{BW-Iwasawa}) yields:

	\begin{proposition}\label{prop:free neighbourhood}
		Let $\sC^\epsilon$ be the connected component from Proposition~\ref{prop:delocalise 1}. Perhaps after further shrinking $\Omega$, we have $\htc(S_{K(\tilde\pi)},\sD_{\Omega})\sshe\otimes_{\T\USh}\bT_{\Omega,\sC}^{\deceps}$ is free of rank $1$ over $\bT_{\Omega,\sC}^{\deceps}$.
	\end{proposition}

%%===================================================
%%===================================================
%%===================================================
%%
%%			ETALE
%%
%%===================================================
%%===================================================
%%===================================================

\subsubsection{\'Etaleness of families}\label{sec:etaleness} We now use Hypothesis~\ref{ass:shalika}.

\begin{proposition}\label{thm:I=0}
	\begin{itemize}\setlength{\itemsep}{0pt}
		\item[(i)] Let 
		\disp{
		\sC^\epsilon = \Sp(\bT_{\Omega,\sC}^{\deceps})\subset \cE\USh
	}
		be as in Proposition~\ref{prop:free neighbourhood}. If $\mathrm{Ev}_{\beta}^{\eta_0}(\Phi_{\tilde{\pi}}^{\epsilon}) \neq 0$ for some $\beta$, then $w : \sC^\epsilon \to \Omega$ is \'etale. 
		\item[(ii)] If $\pi$ admits a non-zero Deligne-critical $L$-value at $p$ with sign $\epsilon$, then $w : \sC^\epsilon \to \Omega$ is \'etale.
	\end{itemize}
\end{proposition}

\begin{proof} 
	Let $\Phi_{\sC}^\epsilon$ be a generator of $\htc(S_{K(\tilde\pi)},\sD_{\Omega})\sshe\otimes_{\T\USh}\bT_{\Omega,\sC}^{\deceps}$ over $\bT_{\Omega,\sC}^{\deceps}$, normalised so that $\mathrm{sp}_{\lambda_\pi}(\Phi_{\sC}^\epsilon) = \Phi_{\tilde\pi}^{\epsilon}$. Combining Propositions~\ref{prop:delocalise 1} and \ref{prop:free neighbourhood}, we have $\mathrm{Ann}_{\cO_\Omega}(\Phi_{\sC}^\epsilon) = I_{\sC^\epsilon}$. Exactly as in Proposition~\ref{prop:ann = zero}, the non-vanishing hypothesis gives 
\disp{
	0 = \mathrm{Ann}_{\cO_\Omega}(\Phi_{\sC}^\epsilon) = I_{\sC^\epsilon},
}
	giving (i). For (ii), we argue exactly as in Corollary~\ref{cor:ann = zero}; the sign condition (Definition~\ref{def:non-vanishing}) is now necessary due to the support of $\cE_{\chi}^{j,\eta_0}$ (see Theorem~\ref{thm:critical value}). In (ii) we have used Hypothesis~\ref{ass:shalika} for $\pi$.
\end{proof}

\begin{proposition}\label{cor:zariski dense cuspidal}
	Suppose $\tilde\pi$ is strongly non-$Q$-critical, $\lambda_{\pi}$ is $H$-regular and $\mathrm{Ev}_\beta^{\eta_0}(\Phi_{\tilde\pi}^{\epsilon_0}) \neq 0$ for some $\epsilon_0, \beta$. Then $\sC^{\epsilon_0}$ contains a Zariski-dense set $\sC_{\mathrm{nc}}^{\epsilon_0}$ of classical cuspidal non-$Q$-critical points.
\end{proposition}
\begin{proof} 
	The conditions ensure $\sC^{\epsilon_0} \to \Omega$ is \'etale, so $\dim(\sC^{\epsilon_0}) = \dim(\Omega)$. We conclude exactly as in Proposition \ref{prop:zariski dense cuspidal} using \cite[Prop.~5.15]{BW20}. 
\end{proof}

\begin{proposition}\label{prop:ind of eps}
	Suppose $\tilde\pi$ is strongly non-$Q$-critical, $\lambda_{\pi}$ is $H$-regular and that $\mathrm{Ev}_\beta^{\eta_0}(\Phi_{\tilde\pi}^{\epsilon_0}) \neq 0$ for some $\epsilon_0$ and some $\beta$. Then $\sC^\epsilon$ is independent of $\epsilon \in \{\pm1\}^\Sigma$, in the sense that for any such $\epsilon$, there is a canonical isomorphism $\sC^{\epsilon_0} \isorightarrow \sC^{\epsilon}$ over $\Omega$.
\end{proposition}
\begin{proof}
	By Propositions~\ref{thm:I=0} and \ref{cor:zariski dense cuspidal}, $\sC^{\epsilon_0}$ is \'etale over $\Omega$ and $\sC_{\mathrm{nc}}^{\epsilon_0} \subset \sC^{\epsilon_0}$ is Zariski-dense. Now let $\epsilon$ be arbitrary. At any $y^{\epsilon_0} \in \sC_{\mathrm{nc}}^{\epsilon_0}$ of weight $\lambda_y$ corresponding to $\tilde\pi_y$, by Proposition~\ref{prop:non-canonical} and non-$Q$-criticality we have
	\disp{
	0 \neq \htc(S_{K(\tilde\pi)},\sV_{\lambda_y}^\vee)^\epsilon_{\m_y^{\dec}} \cong \htc(S_{K(\tilde\pi)},\sD_{\lambda_y})^\epsilon_{\m_y^{\dec}}.
} 
	By Remark~\ref{rem:non-Q-critical implies point}, since $y^{\epsilon_0}$ is cuspidal there exists $y^\epsilon\in\sE\USh$ corresponding to $\tilde\pi_y$. As in Corollary~\ref{cor:etale bottom level}, the map 
	\[
	\sC_{\mathrm{nc}}^{\epsilon_0} \to \sE\USh, \qquad y^{\epsilon_0} \mapsto y^\epsilon
	\]
	interpolates to a closed immersion 
	\[
	\iota^\epsilon : \sC^{\epsilon_0} \hookrightarrow \sE\USh
	\]
	sending $x_{\tilde\pi}^{\dec, \epsilon_0}$ to $x_{\tilde\pi}^{\deceps}$. Thus $\iota^\epsilon(\sC^{\epsilon_0}) \subset \sC^\epsilon$, so $\sC^\epsilon$ contains an irreducible component of dimension $\dim(\Omega)$. As $\cO_{\sC^\epsilon} \cong \cO_\Omega/I_{\sC^\epsilon}$, we deduce $I_{\sC^\epsilon} = 0$, so $\sC^\epsilon \to \Omega$ is \'etale at $\tilde\pi$, and conclude that $\iota^\epsilon$ is an isomorphism as in Corollary~\ref{cor:etale bottom level}.
\end{proof}

Hence $\sE\USha$ and all the $\sE\USh$ are locally isomorphic at $\tilde\pi$, so we drop $\epsilon$ from notation. We deduce existence and \'etaleness of the component $\sC \subset \sE\USha$ in Theorem~\ref{thm:section 7 main theorem}; we take it to be any of the $\sC^\epsilon$. The isomorphisms between the $\sC^\epsilon$ identify all the $x_{\tilde\pi}^{\deceps}$ with a single point $x_{\tilde\pi}^{\dec} \in \sE\USha$. 

\begin{remark}
	Having a family of \emph{cuspidal} automorphic representations is essential here; e.g.\ for $\GL_2$, an Eisenstein series will appear in only one of the $\pm$-eigencurves (see \cite[\S3.2.6]{Bel12}).
\end{remark}

\subsubsection{Eigenclasses for Shalika families}
We now refine Proposition~\ref{cor:zariski dense cuspidal}. Suppose $\tilde\pi$ satisfying (C1-2) of Conditions \ref{cond:running assumptions} is strongly non-$Q$-critical. For $v \in S$, assume $\pi_v$ admits a Shalika new vector of conductor $c(\pi_v)$ and Hypothesis~\ref{ass:shalika} holds for $c = c(\pi_v)$. Suppose $\lambda_\pi$ is $H$-regular and $\pi$ admits a non-zero Deligne-critical $L$-value at $p$. Then by \S\ref{sec:etaleness}, we know:
\begin{itemize}\s
	\item[(1)] that there is a unique irreducible component $\sC$ of $\sE\USha$ through $x_{\tilde\pi}^{\dec}$; 
	\item[(2)] that $\sC = \Sp(\bT_{\Omega,\sC}^{\dec}) = \Sp(\bT_{\Omega,\sC}^{\deceps})$ for all $\epsilon$; and 
	\item[(3)] that $w : \sC\to \Omega$ is \'etale. 
\end{itemize}

As in Proposition~\ref{prop:zariski dense cuspidal}, we deduce that $\sC$ contains a Zariski-dense set $\sC_{\mathrm{nc}}$ of classical cuspidal non-$Q$-critical points. If $y \in \sC_{\mathrm{nc}}$ it corresponds to a $Q$-refined RACAR $\tilde\pi_y$, since by construction $y$ appears in cohomology at (parahoric-at-$p$) level $K(\tilde\pi)$.

\begin{proposition} \label{prop: shalika families}
	For each $\epsilon \in \{\pm1\}^\Sigma$, up to shrinking $\Omega$, there exists a Hecke eigenclass $\Phi_{\sC}^\epsilon\in \htc(S_{K(\tilde\pi)}, \sD_{\Omega})^{\epsilon}$ such that for each $y \in \sC_{\mathrm{nc}}$:
	\begin{itemize}\setlength{\itemsep}{0pt}
		\item[(i)]  $\htc(S_{K(\tilde\pi)},\sD_{\lambda_y})^\epsilon_{\m_y^{\dec}}$ is a line generated by  $\mathrm{sp}_{\lambda_y}(\Phi_{\sC}^\epsilon)$, where $\lambda_y \defeq w(y)$,
		\item[(ii)] the $Q$-refined RACAR $\tilde\pi_y$ satisfies (C1),
		\item[(iii)] for all $v \in S$, $\pi_{y,v}$ admits a Shalika new vector of conductor $c(\pi_{y,v}) = c(\pi_v)$, and
		\item[(iv)] $\tilde\pi_y$ satisfies (C2).
	\end{itemize}
\end{proposition} 
\begin{proof} 
			(i) This is a standard consequence of Propositions \ref{thm:I=0} and \ref{prop:ind of eps}; see e.g.\ \cite[\S4.2]{BDJ17}, \cite[Prop.\ 6.7]{BW18} or \cite[Prop.\ 2.18]{BW-Iwasawa}.
			
			(ii)  We can argue as in Proposition~\ref{prop:zariski dense shalika no new}; indeed here we already have \'etaleness, so this case is easier, and we are terse with details. Fix $(\chi,j)$ with $L^{(p)}(\pi\otimes\chi, j+\tfrac{1}{2}) \neq 0$ (by hypothesis), where $\chi$ has conductor $p^\beta$, with $\beta_{\pri} \geqslant 1$ for all $\pri|p$. Let $\epsilon = (\chi\chi_{\cyc}^j\eta)_\infty$, and define (up to shrinking $\Omega$) an everywhere nonvanishing map $\mathrm{Ev}_{\chi,j}^\Omega : \hc{t}(S_{K(\tilde\pi)},\sD_\Omega) \to \cO_\Omega$. Then for all $y \in \sC_{\mathrm{nc}}$, we have
			\begin{align}\label{eq:ev at y}
				\mathrm{Ev}_{\chi,j}^\Omega(\Phi_{\sC}^\epsilon)(\lambda_y)	&= \cE_\chi^{j,\eta_0}\left(r_{\lambda_y} \circ \mathrm{sp}_{\lambda_y}(\Phi_{\sC}^\epsilon)\right)\neq 0. 
			\end{align}
			As \eqref{eq:ev at y} is non-zero, each $\pi_y$ satisfies (C1) by Proposition~\ref{prop:shalika non-vanishing}, showing (ii).

	We need the following in proving both (iii) and (iv). Combining (i) with non-$Q$-criticality shows
	\disp{
	\dim_L \htc(S_{K(\tilde\pi)},\sV_{\lambda_y}^\vee(L))^\epsilon_{\m_y^{\dec}} = 1.
}
	Base-changing, the same is true with $\overline{\Q}_p$-coefficients. By Proposition~\ref{prop:non-canonical}, we see:
	\[
	(\dagger) \ \text{ the generalised $\cH^{\dec}$-eigenspace in $\pi_{y,f}^{K(\tilde\pi)}$ at $\m_{y}^{\dec}$ is a line over $\C$}.
	\]

	(iii) We now study $\pi_y$ at $v \in S$. Letting $\eta_{y} = \eta_0|\cdot|^{\sw_y}$, by ($\dagger)$ we have
	\[
	\dim_\C\  (\pi_{y,v})^{J_{v}^{\eta_{y}}(c)}\big[S_{\alpha_v} - \eta_{y,v}(\alpha_v) : \alpha_v \in \cO_v^\times \big] \ = 1.
	\]
	Lemma~\ref{lem:equivalent shalika} implies that $\pi_{y,v}$ admits a Shalika new vector of conductor $c = c(\pi_v)$, giving (iii). 
	
	\medskip
	
	(iv) As in (iii), $\pi_{\pri,y}$ is parahoric-spherical as $(\dagger)$ implies
	\begin{equation}\label{eq:Up line family}
		\dim_\C\  \pi_{\pri,y}^{J_{\pri}}\big[U_{\pri}^\circ - \alpha_{\pri,y}^\circ\big] = \dim_\C\  \cS_{\psi_{\pri}}^{\eta_{\pri,y}}(\pi_{\pri,y}^{J_{\pri}})\big[U_{\pri}^\circ - \alpha_{\pri,y}^\circ\big]\ = 1,
	\end{equation}
	where $\alpha_{\pri,y}^\circ$ is the $U_{\pri}^\circ$-eigenvalue of $\tilde\pi_y$. It remains to show the non-vanishing in (C2).
	
	Let $W_{\pri,y}$ be a generator of \eqref{eq:Up line family} for $\pri|p$. Using (C1) at $v \notin S\cup\{\pri|p\}$, and Hypothesis~\ref{ass:shalika} and the equality $c(\pi_{y,v}) = c(\pi_v)$ for $v \in S$, we may take Friedberg--Jacquet test vectors $W_{y,v}^{\mathrm{FJ}}$ for $v\nmid p$ such that
	\[
	W_{y,f}^{\mathrm{FJ}} = \otimes_{v\nmid\pri}W_{y,v}^{\mathrm{FJ}} \otimes_{\pri|p}W_{\pri,y}  \ \in \cS_{\psi_f}^{\eta_{y,f}}\left(\pi_{y,f}^{K(\tilde\pi)}\right)
	\]
	is fixed by $K(\tilde\pi)$. For $\epsilon$ as in \eqref{eq:ev at y}, let 
	\[
	\phi_{y}^\epsilon \defeq \Theta_{i_p}^{K(\tilde\pi),\epsilon}(W_{y,f}^{\mathrm{FJ}}) \in \hc{t}(S_{K(\tilde\pi)}, \sV_{\lambda_y}^\vee(\overline{\Q}_p))_{\m_y^{\dec}}^\epsilon.
	\]
	This line contains $r_{\lambda_y}\circ \mathrm{sp}_{\lambda_y}(\Phi_{\sC}^\epsilon)$, so there is $c_y^\epsilon \in \overline{\Q}_p^\times$ such that $c_y^\epsilon\phi_{y}^\epsilon = r_{\lambda_y}\circ \mathrm{sp}_{\lambda_y}(\Phi_{\sC}^\epsilon)$. Then
	\begin{align*}
		c_y^\epsilon \cdot \cE_\chi^{j,\eta_0}\left(\phi_{y}^\epsilon\right) = \cE_\chi^{j,\eta_0}\left(r_{\lambda_y} \circ \mathrm{sp}_{\lambda_y}(\Phi_{\sC}^\epsilon)\right)  \neq 0,
	\end{align*}
	where non-vanishing is \eqref{eq:ev at y}. As in Theorem~\ref{thm:critical value} (via Lemma \ref{lem:local zeta integrals} and Proposition~\ref{lem:zeta at p}), the left-hand side is of the form $i_p[(*)L^{(p)}(\pi_y\otimes\chi,j+1/2)\prod_{\pri|p} W_{y,\pri}(t_{\pri}^{-\delta_{\pri}})] \neq 0$, for $(*)$ a non-zero scalar. 	As the $L$-function is analytic, we deduce that each $W_{y,\pri}(t_{\pri}^{-\delta_{\pri}}) \neq 0$. We can renormalise $W_{y,\pri}$ (and hence $c_y^\epsilon$) so that $W_{y,\pri}(t_{\pri}^{-\delta_{\pri}}) = 1$, so (C2) holds.
\end{proof}

With this, we have completed the proof of Theorem~\ref{thm:section 7 main theorem}.

\subsection{Families of $p$-adic $L$-functions}\label{sec:definition family Lp}
Let $\tilde\pi$ satisfy (C1-2) of Conditions~\ref{cond:running assumptions}. We also assume that all the hypotheses of Theorem~\ref{thm:section 7 main theorem} are satisfied. 

Let $\sC \subset \sE_{\Omega, h}^{\dec}$ be the unique (Shalika) family through $x_{\tilde\pi}^{\dec}$, and $\sC_{\mathrm{nc}}$ the Zariski-dense subset of classical points, both from Theorem~\ref{thm:section 7 main theorem}. For each $\epsilon \in \{\pm 1\}^\Sigma$ let $\Phi_{\sC}^\epsilon \in \hc{t}(S_{K(\tilde\pi)},\sD_\Omega)^{\epsilon}$ be the resulting Hecke eigenclass. We may renormalise $\Phi_{\sC}^\epsilon $ so that $\mathrm{sp}_\lambda(\Phi_{\sC}^\epsilon) = \Phi_{\tilde\pi}^\epsilon$. The following is an analogue of Definition~\ref{def:non-critical slope Lp} for families.

\begin{definition} \label{def:Lp families} Let $\cL_{p}^{\sC, \epsilon} \defeq \mu^{\eta_0}(\Phi_{\sC}^\epsilon)$. Also let 
	\disp{
	\Phi_{\sC} = \sum_{\epsilon} \Phi_{\sC}^\epsilon \in \hc{t}(S_{K(\tilde\pi)},\sD_\Omega),
}
 which is also a Hecke eigenclass. Define the \emph{$p$-adic $L$-function over $\sC$} to be 
	\[
	\cL_p^{\sC} \defeq \mu^{\eta_0}(\Phi_{\sC}) = \sum_{\epsilon\in\{\pm1\}^\Sigma} \cL_p^{\sC,\epsilon} \in \cD(\Galp,\cO_\Omega).
	\]
	Via the Amice transform, as in Definition~\ref{def:non-critical slope Lp}, after identifying $\sC$ with $\Omega$ via $w$ we may consider $\cL_p^{\sC}$ as a rigid function $\sC \times \sX(\Galp) \to \overline{\Q}_p$.
\end{definition}

The following implies Theorem~\ref{thm:intro 3} of the introduction. The hard/novel part of the proof has already been handled; given Theorem \ref{thm:section 7 main theorem}, the remainder is standard.

\begin{theorem}\label{thm:family p-adic L-functions}
	Suppose $\tilde\pi$ satisfies the hypotheses of Theorem~\ref{thm:section 7 main theorem}. Let $y \in \sC_{\mathrm{nc}}$ be a classical cuspidal point attached to a non-$Q$-critical $Q$-refined RACAR $\tilde\pi_y$ satisfying (C1-2). For each $\epsilon$, there exists a $p$-adic period $c_y^\epsilon \in L^\times$ such that
	\begin{equation}\label{eq:spec in families}
	\cL_p^{\sC,\epsilon}(y, -) = c_y^\epsilon \cdot \cL_p^\epsilon(\tilde\pi_y, -)
	\end{equation}
	as functions $\sX(\Galp) \to \overline{\Q}_p$. In particular, $\cL_p^{\sC}$ satisfies the following interpolation: for any $j \in \mathrm{Crit}(w(y))$, and for any Hecke character $\chi$ of conductor $p^\beta$ with $\beta_{\pri} > 1$ for all $\pri|p$, we have
\[
		i_p^{-1}(\cL_p^{\sC}(y,\chi\chi_{\cyc}^j)) = c_y^\epsilon A \tau(\chi_f)^n \mathrm{N}_{F/\Q}(\fd)^{jn} \prod_{\pri|p}\epy \einfy \tfrac{L^{(p)}(\pi_y\times\chi, j+1/2)}{\Omega_{\pi_y}^\epsilon},
	\]
	where $\epsilon = (\chi\chi_{\cyc}^j\eta)_\infty$ and other notation is as in Theorem~\ref{thm:non-ordinary}. Finally $c_{x_{\tilde\pi}^\dec}^\epsilon = 1$ for all $\epsilon$.
\end{theorem}
\begin{remark}
	The complex periods $\Omega_{\pi_y}^\epsilon$ are only well-defined up to multiplication by $E^\times$, where $E$ is the number field from Definition~\ref{def:gen eigenspace}; the numbers $c_y^\epsilon$ are $p$-adic analogues.
\end{remark}

\begin{proof}
	Let $y$ be as in the theorem and put $\lambda_y= w(y)$. As in \S\ref{sec:level group} (using Hypothesis~\ref{ass:shalika}), fix a Friedberg--Jacquet test vector 
	\[
	W_{y,f}^{\mathrm{FJ}} \in \cS_{\psi_f}^{\eta_{y,f}}(\pi_{y,f})^{K(\tilde\pi)},
	\]
	and for each $\epsilon$ a complex period $\Omega_{\pi_y}^\epsilon$ as in \S\ref{sec:periods}. Since $y \in \sC$ is defined over $L$, as in \S\ref{sec:periods} there exists a class 
	\[
	\phi_{y}^\epsilon \defeq \Theta_{i_p}^{K,\epsilon}(W_{y,f}^{\mathrm{FJ}})\big/i_p(\Omega_{\pi_y}^\epsilon) \in \hc{t}(S_{K(\tilde\pi)}, \sV_{\lambda_y}^\vee(L))_{\m_y^{\dec}}^\epsilon.
	\] 
	Via non-$Q$-criticality, we lift $\phi_{y}^\epsilon$ to a non-zero class $\Phi_{y}^\epsilon \in \hc{t}(S_{K(\tilde\pi)},\sD_{\lambda_y}(L))^\epsilon_{\m_y^{\dec}}$. By Theorem~\ref{thm:section 7 main theorem}, this space is equal to $L\cdot \mathrm{sp}_{\lambda_y}(\Phi_{\sC}^\epsilon)$, so there exists $c_y^\epsilon \in L^\times$ such that 
	\disp{
	\mathrm{sp}_{\lambda_y}(\Phi_{\sC}^\epsilon) = c_y^\epsilon \cdot \Phi_{y}^\epsilon.
}
	By definition, $\cL_p^\epsilon(\tilde\pi_y) = \mu^{\eta_0}(\Phi_{y}^\epsilon).$ As evaluation maps commute with weight specialisation (Proposition~\ref{prop:evaluations in families}), we find 
	\disp{
	\mathrm{sp}_{\lambda_y}(\cL_p^{\sC,\epsilon}) = c_y^\epsilon \cdot \cL_p^\epsilon(\tilde\pi_y),
}
	which is a reformulation of \eqref{eq:spec in families}. The interpolation formula then follows from Remark~\ref{rem:signed interpolation}. Finally, our normalisation of $\Phi_{\sC}^\epsilon$ ensures $c_{x_{\tilde\pi}^{\dec}}^\epsilon = 1$.
\end{proof}

\subsection{Unicity of non-$Q$-critical $p$-adic $L$-functions}\label{sec:uniqueness 2}

Let $\alpha_p^\circ = \prod_{\pri|p}(\alpha_{\pri}^\circ)^{e_{\pri}}$ and $h_p \defeq v_p(\alpha_p^\circ)$. For $\epsilon \in \{\pm1\}^\Sigma$, let $\sX(\Galp)^\epsilon$ be the component of characters $\chi$ with $\epsilon = (\chi\eta)_\infty$. Then 
\disp{
\sX(\Galp) = \bigsqcup_\epsilon \sX(\Galp)^\epsilon
}
(e.g.\ \cite[Rem.~7.3.4]{BH17}). If $\cL :\sC \times \sX(\Galp) \to L$ is a rigid analytic function, then $\cL = \sum_{\epsilon} \cL^\epsilon$, with $\cL^\epsilon$ supported on $\sC\times \sX(\Galp)^\epsilon$.

\begin{proposition}\label{prop:family implies unique}
	Suppose $\tilde\pi$ satisfies the hypotheses of Theorem~\ref{thm:section 7 main theorem}. Suppose Leopoldt's conjecture holds for $F$ at $p$ and that $\cL_p^\epsilon(\tilde\pi) \neq 0$. Let 
	\disp{
	\cL^\epsilon : \sC \times \sX(\Galp)^\epsilon \to L
}
	be any rigid analytic function such that for all $y \in \sC_{\mathrm{nc}}$, the specialisation $\cL^\epsilon(y,-)$ is admissible of growth $h_p$ and there exists $C_y^\epsilon \in L^\times$ such that 
	\begin{equation}\label{eq:spec in families 2}
i_p^{-1}(\cL_p^{\sC}(y,\chi\chi_{\cyc}^j)) = C_y^\epsilon A \tau(\chi_f)^n \mathrm{N}_{F/\Q}(\fd)^{jn} \prod_{\pri|p}\epy \cdot \einfy \tfrac{L^{(p)}(\pi_y\times\chi, j+1/2)}{\Omega_{\pi_y}^\epsilon},
	\end{equation}
	for all finite order $\chi \in \sX(\Galp)$ and $j \in \mathrm{Crit}(\lambda_y)$ such that $(\chi\chi_{\cyc}^j\eta)_\infty = \epsilon$. Then there exists $C \in L$ such that
	\[
	\cL^{\epsilon}\big(x_{\tilde\pi}^{\dec}, -\big) = C\cdot  \cL_p^{\epsilon}(\tilde\pi) \in \cD(\Galp, L).
	\]
\end{proposition}

\begin{proof}
		This is a standard argument (exactly analogous to \cite[Prop.\ 6.15]{BW18}).
\end{proof}

\begin{corollary}
	Suppose $\tilde\pi$ satisfies the hypotheses of Theorem~\ref{thm:section 7 main theorem}. Assume Leopoldt's conjecture for $F$ at $p$. Up to scaling the $p$-adic periods, $\cL_p(\tilde\pi)$ is uniquely determined by interpolation of $L$-values over the unique Shalika family $\sC$ of level $K(\tilde\pi)$ through $\tilde\pi$.
\end{corollary}
In particular, up to these assumptions $\cL_p(\tilde\pi)$ does not depend on our method of construction.

\begin{remarks}
	We expect that $\cL_p^\epsilon(\tilde\pi)$ should always be non-zero. By \eqref{eq:spec in families 2} and Lemma~\ref{lem:regular weight implies non-vanishing}, if $\lambda_\pi$ is regular this is automatic for any $\epsilon$ such that there exists a finite order Hecke character $\chi$ such that $(\chi_{\cyc}^j\chi\eta)_\infty = \epsilon$, where $j$ is any integer strictly above the centre of $\mathrm{Crit}(\lambda_\pi)$.
	
	Without Leopoldt, there is still an analogue for the restriction to 1-dimensional slices of $\Galp$ (cf.\ \cite[Thm.\ 4.7(ii)]{BDJ17} or  \S\ref{sec:non-Q-critical p-adic L-functions}). Thus the restriction of $\cL_p(\tilde\pi)$ to the cyclotomic line is unique. 
\end{remarks}

\begin{appendix}
	\section{Errata for earlier works}
	
	Whilst writing this paper, we found errors in our earlier publications. We clarify them here.
	
	\begin{itemize}
		\item[(1)] In \cite[Rem.\ 4.19]{BW20}, which compared the right actions used in that paper with the left actions used in this, in the final sentence $U_p^*$ should have been $\lambda(\sigma(t)^{-1}t) U_p^\cdot$ (not $\lambda^\vee(\sigma(t)^{-1}t)U_p^\cdot$). This was not used elsewhere \emph{ibid}.; we have used the correct formulation here.
		
		\item[(2)] The power of $q$ in the statement of \cite[Prop.\ 3.4]{DJR18} is incorrect. In the proof, one can reduce the support of the integral in the penultimate displayed equation to the Iwahori subgroup, not to $N_n^-(\cP^\beta)T_n(\cO)N_n(\cP^\beta)$ as stated, so the final volume term is wrong. The proof otherwise holds. A corrected statement is Proposition \ref{lem:zeta at p} of the present paper. (This ensures the final interpolation result is consistent with the Coates--Perrin-Riou conjecture on existence of $p$-adic $L$-functions; see \cite[\S3]{AG94}. Indeed, \cite[Thm.\ B]{DJR18} is not consistent with Coates--Perrin-Riou). The powers of $q$ in Theorem B and Theorem 4.7 of \cite{DJR18} are thus incorrect. The interpolation formulas there should be replaced by that of Theorem \ref{thm:intro non-ord} here.
	\end{itemize}
	
\end{appendix}

\section*{Glossary of key notation/terminology}
\footnotesize \hspace{1pt}
\addcontentsline{toc}{section}{Glossary of notation}
\begin{multicols}{2}
	\noindent $\cA = \cA^Q$ \dotfill Locally analytic function space (\S\ref{sec:parabolic distributions})\\
	$\alpha_{\pri}, \alpha_p$ \dotfill $U_{\pri}, U_p$ eigenvalues (\S\ref{ss:the U_p-refined line}) \\
	$\alpha_{\pri}^\circ, \alpha_p^\circ$ \dotfill $U_{\pri}^\circ, U_p^\circ$ eigenvalues (\S\ref{sec:slope-decomp})\\
	$\beta = (\beta_{\pri})_{\pri|p} \in \Z^{{\pri|p}}$\dotfill Multi-index (\S\ref{sec:auto cycles})\\
	(C1),(C2) \dotfill Assumptions on $\tilde\pi$ (Cond.~\ref{cond:running assumptions})\\
	$\sC$ \dotfill Connected component of $\sE_{\Omega, h}$ (\S\ref{sec:maximal dimension})\\
	$\Cl(I)$\dotfill Narrow ray class gp.\ cond.\ $I$ (\S\ref{sec:notation})\\ 
	$\mathrm{Crit}(\lambda)$\dotfill Deligne-critical $L$-values for $\lambda$ (eqn.\ \eqref{eqn:crit lambda})\\
	$\cD$ \dotfill Locally analaytic distributions (\S\ref{sec:galois groups})\\
	$\cD_\lambda = \cD_\lambda^Q$\dotfill $Q$-parahoric dists.\ of wt. $\lambda$ (\S\ref{sec:parabolic functions})\\
	$\cD_\Omega = \cD_\Omega^Q$ \dotfill $Q$-parahoric dists.\ over $\Omega$ (\S\ref{sec:distributions in families})\\
	$\sD$ \dotfill Local system of distributions (\S\ref{sec:non-arch ls})\\
	$\Delta_p$ \dotfill Monoid in $G(\Qp)$ gen.\ by $J_p, t_{\pri}$ (\S\ref{sec:slope-decomp})\\
	$\delta$ \dotfill Representative of $\pi_0(X_\beta)$ (\S\ref{sec:auto cycles})\\
	$\mathrm{Ev}_{\beta,\delta}^M$ \dotfill Abstract evaluation map (Def.~\ref{def:abstract evaluation})\\
	$\mathrm{Ev}_\beta^{\eta_0}$\dotfill Galois evaluation (eqn.\ \eqref{eq:final evaluation})\\
	$\cE_\chi^{j,\eta_0}$ \dotfill Classical evaluation map at $\chi,j$ (eqn.\ \eqref{eq:classical evaluation})\\
	$\sE\Uha$ \dotfill $Q$-parabolic eigenvariety for $G$ (Def.~\ref{def:local piece S})\\
	$\sE\USh$ \dotfill `modified' $Q$-par.\ eigenvariety (Def.~\ref{def:local piece})\\
	$\ep$ \dotfill C--PR factor at $p$ (Thm.\ \ref{thm:non-ordinary})\\
	$\epp$ \dotfill Def.\ \ref{def:e'_p} \\
	$\einf$\dotfill C--PR factor at $\infty$ (Def.\ \ref{def:e_infty}\\
	$\epsilon$ \dotfill Character $K_\infty/K_\infty^\circ \to \{\pm 1\}^\Sigma$ (\S\ref{sec:decomp at infinity})\\
	$F$ \dotfill Totally real field of degree $d$\\
	$\phi_{\tilde\pi}^\epsilon$ \dotfill Classical class in $\hc{t}$ (Def.~\ref{def:phi})\\
	$\Phi_{\tilde\pi}^\epsilon, \Phi_{\tilde\pi}$ \dotfill Overconvergent classes in $\hc{t}$ (\S\ref{sec:non-Q-critical p-adic L-functions})\\
	$G$ \dotfill $\mathrm{Res}_{\cO_F/\Z}\GL_{2n}$\\
	$G_n$ \dotfill $\mathrm{Res}_{\cO_F/\Z}\GL_n$\\
	$\Galp$\dotfill $\mathrm{Gal}(F^{p\infty}/F)$ (\S\ref{sec:notation})\\
	$\Galp^\cyc$\dotfill $\mathrm{Gal}(\Q^{p\infty}/\Q)$ (\S\ref{sec:notation})\\
	$\Gamma_{\beta,\delta}$ \dotfill Arithmetic group in automorphic cycle (\S\ref{sec:passing to components})\\
	$H$ \dotfill $\mathrm{Res}_{\cO_F/\Z}[\GL_n \times \GL_n$]\\
	$\cH',\cH$ \dotfill Universal Hecke algebras (\S\ref{sec:unramified H}, Def.~\ref{def:hecke algebra})\\
	$\cH^{\dec}$ \dotfill Universal Hecke algebra at all primes (Def.~\ref{def:hecke algebra full})\\
	$\leqslant h$ \dotfill Slope $U_p^\circ \leqslant h$ part (\S\ref{sec:slope-decomp 2})\\
	$\sI$ \dotfill Irreducible component of $\sE_{\Omega, h}$\\
	$i_p$ \dotfill Fixed isomorphism $\C \isorightarrow \overline{\Q}_p$\\
	$\iota$ \dotfill Map $H \hookrightarrow G, (h_1,h_2) \mapsto \mathrm{diag}(h_1,h_2)$\\
	$\iota_\beta$ \dotfill Map in automorphic cycle (eqn. \eqref{eq:iota beta})\\
	$\eta$ \dotfill Shalika character (\S\ref{sec:shalika models})\\
	$J_p$\dotfill Parahoric subgroup of type $Q$ (\S\ref{ss:the U_p-refined line})\\
	$\Jpbeta$ \dotfill Eqn. \eqref{eq:Jp minus}\\
	$K$ \dotfill Open cpct.\ subgp.\ of $G(\A_f)$ (eqn.\ \eqref{eq:general K}) \\
	$K(\tilde\pi)$\dotfill Friedberg--Jacquet level (eqn.\ \eqref{eq:level group},\S\ref{sec:set-up})\\
	$K_1(\tilde\pi)$ \dotfill Whittaker new level (eqn.\ \eqref{eq:K_1})\\
	$\kappaj$ \dotfill Map $V_{\lambda}^\vee \to V^H_{(j,-\sw-j)}$ (\S\ref{sec:choice of basis})\\
	$\kappaj^\circ$ \dotfill Normalised map $V_\lambda^\vee(L) \to L$ (Def.\ \ref{def:kappa_j})\\
	$L$ \dotfill Extension of $\Qp$, coefficient field (\S\ref{sec:periods})\\
	$L_\beta$ \dotfill Open compact in $H(\A_f)$ (Def.~\ref{def:xi})\\
	$L(\pi,s)$\dotfill Standard $L$-fn.\ of $\pi$\\
	$L^{(p)}(\pi,s)$ \dotfill $L()$ with no Euler factors at $p$\\
	$\cL_p(\tilde\pi)$\dotfill $p$-adic $L$-function of $\tilde\pi$ (\S\ref{sec:non-Q-critical p-adic L-functions})\\
	$\cL_p^{\sC}$ \dotfill $p$-adic $L$-function in family over $\sC$ (Def.~\ref{def:Lp families})\\
	$\Lambda$ \dotfill Localisation of $\cO_\Omega$ at $\lambda_\pi$ (\S\ref{sec:existence of x})\\
	$\lambda_\pi$ \dotfill (Pure, dominant, integral) weight of $\pi$ (Def.~\ref{def:Wlam})\\
	$\m_{\tilde\pi}$ \dotfill Max.\ ideal in $\cH$ (Def.~\ref{def:gen eigenspace})\\
	$\m_{\tilde\pi}^{\dec}$ \dotfill Max.\ ideal in $\cH^{\dec}$ (Def.~\ref{def:gen eigenspace shalika})\\
	$\m_\lambda$ \dotfill Max.\ ideal in $\cO_\Omega$\\
	Non-$Q$-critical\dotfill Def.~\ref{def:non-Q-critical}\\
	Non-$Q$-critical slope \dotfill Def.~\ref{def:non-critical slope}\\
	$N_Q$ \dotfill Unipotent radical of $Q$\\
	$\nuj$ \dotfill Element of $V_\lambda(L)$ (\S\ref{sec:choice of basis})\\
	$\Omega$ \dotfill Affinoid in $\Wlam$\\
	$\Omega_\pi^\epsilon$ \dotfill Complex period (\S\ref{sec:periods})\\
	$\cO_\Omega$ \dotfill Ring of rigid functions on $\Omega$\\
	$\OFp$ \dotfill $\cO_F \otimes \Zp$\\
	$p^\beta$\dotfill $\prod \pri^{\beta_{\pri}}$ (Def. \ref{def:xi})\\
	$\mathrm{pr}_\beta$ \dotfill Map $\pi_0(X_\beta) \to \Cl(p^\beta)$ \eqref{eq:pr_beta}\\
	$\pi$\dotfill Auto.\ repn.\ of $G(\A)$ (Conditions~\ref{cond:running assumptions})\\
	$\tilde\pi$ \dotfill $p$-refinement of $\pi$ (\S\ref{ss:the U_p-refined line}, Conditions~\ref{cond:running assumptions})\\
	$\pi_0(X_\beta)$ \dotfill Component group of auto. cycle (before \eqref{eq:pr_beta})\\
	$\varpi_v$ \dotfill Uniformiser of $F_v$\\
	$\psi$ \dotfill Additive character of $F \backslash \A_F$ (\S\ref{sec:shalika models})\\
	$Q$ \dotfill Parabolic subgroup with Levi $H$\\
	$Q$-refined RACAR \dotfill Choice of $Q$-ref't $\tilde\pi_{\pri}$ $\forall \pri|p$ (\S\ref{ss:the U_p-refined line})\\
	$q_{\pri}$ \dotfill $\mathrm{N}_{F/\Q}(\pri)$\\
	RACAR\dotfill regular algebraic cuspidal auto. repn.\\
	RASCAR \dotfill $\eta$-symplectic RACAR (Intro.)\\
	$r_\lambda$ \dotfill Specialisation map $\cD_\lambda \to V_\lambda^\vee$ (eqn.\ \eqref{eqn:specialisation})\\
	Shalika model \dotfill \S\ref{sec:shalika models}\\
	Strongly non-$Q$-critical \dotfill Def.~\ref{def:non-Q-critical} \\
	$S_K$ \dotfill Loc.\ symm.\ space for $G$ of level $K$ (eqn.\ \eqref{eq:loc sym space})\\
	$S$ \dotfill $\{v\nmid p\infty: \pi_v\text{ not spherical}\}$ (\S\ref{sec:unramified H})\\
	$\cS^{\eta}_{\psi}$ \dotfill Shalika model (\S\ref{sec:shalika models}) \\
	$\mathrm{sp}_\lambda$ \dotfill Any map induced from $\newmod{\m_\lambda} : \cO_\Omega \rightarrow L$\\
	$\Sigma$ \dotfill Set of real embeddings of $F$\\
	$\sigma$ \dotfill Real embedding $F \hookrightarrow \overline{\Q}$ \\
	$\sigma(\pri)$ \dotfill Real embedding attached to $\pri|p$ (\S\ref{sec:notation})\\
	$\bT_{\Omega, h}$ \dotfill Hecke algebra using $\cH$ (Def.~\ref{def:local piece S})\\
	$\bT_{\Omega, h}^{\dec}$ \dotfill Hecke algebra using $\cH^{\dec}$ (Def.~\ref{def:local piece})\\
	$t=d(n^2 +n - 1)$ \dotfill Top degree of cusp.\ cohomology  \\
	$t_{\pri}$ \dotfill $\mathrm{diag}(\varpi_{\pri}I_n, I_n) \in \GL_{2n}(F_{\pri})$\\
	$t_p^\beta$ \dotfill $\prod t_{\pri}^{\beta_{\pri}}$ (Def.~\ref{def:xi})\\
	$\tau(\chi_f)$ \dotfill Gauss sum of $\chi_f$ (Thm.~\ref{thm:critical value})\\
	$\tau_\beta^\circ$ \dotfill Twisting map (\S\ref{sec:pulling back to cycles})\\
	$\Theta^{K,\epsilon},\Theta^{K,\epsilon}_{i_p}$ \dotfill Maps $\cS_{\psi_f}^{\eta_f}(\pi_f^K)\to \hc{t}$ (\S\ref{sec:periods})\\
	$U_{\pri} = U_{\pri}^\cdot$ \dotfill Automorphic $U_{\pri}$-operator (\S\ref{ss:the U_p-refined line})\\
	$U_{\pri}^\circ$ \dotfill Integrally normalised $U_{\pri}$ (\S\ref{ss:the U_p-refined line},\S\ref{sec:slope-decomp})\\
	$\sU(I)$\dotfill Integral ideles $\equiv 1 \newmod{I}$ (\S\ref{sec:notation})\\
	$V_\lambda$ \dotfill Alg.\ repn.\ of $G$ of weight $\lambda$ (\S\ref{sec:algebraic weights})\\
	$V_\lambda^H$ \dotfill Alg.\ repn.\ of $H$ of weight $\lambda$\\
	$V_\Omega^H$ \dotfill (\S\ref{sec:distributions in families})\\
	$\cV$ \dotfill Archimedean local system on $S_K$ (\S\ref{sec:arch ls})\\
	$\sV$ \dotfill Non-archimedean local system on $S_K$ (\S\ref{sec:non-arch ls})\\
	$\sv_\lambda, \sv_\Omega$ \dotfill Elements of $V_\lambda^H$ and $V_\Omega^H$ (Not. \ref{not:v_lambda},\ref{not:v_omega})\\
	$\sW_0$ \dotfill Pure weight space (\S\ref{sec:weight spaces})\\
	$\Wlam$\dotfill (Parabolic) Weight space (Def.~\ref{def:Wlam})\\
	$W^{\mathrm{FJ}}$ \dotfill Friedberg--Jacquet test vector (\S\ref{sec:shalika models})\\
	$\sw$ (or $\sw_\lambda,\sw_\Omega)$ \dotfill Purity weight (\S\ref{sec:algebraic weights})\\
	$w_n$ \dotfill Longest Weyl element for $\GL_n$ (Def.~\ref{def:xi})\\
	$X_\beta$ \dotfill Automorphic cycle of level $\beta$ (\S\ref{sec:auto cycles})\\
	$\chi$ \dotfill Finite order Hecke character\\
	$\xi,\xi_{\pri}$\dotfill Twisting operator (Def.~\ref{def:xi})\\
	$x_{\tilde\pi}$\dotfill Point of $\sE\Uha$ corr.\ to $\tilde\pi$ (Thm.~\ref{thm:shalika family})\\
	$x_{\tilde\pi}^{\dec}$\dotfill Point of $\sE\USha$ corr.\ to $\tilde\pi$ (Thm.~\ref{thm:section 7 main theorem})\\
	$\chi_{\cyc}$\dotfill Cyclotomic character of $\Galp$\\
	$Z$ \dotfill Centre of $G$\\
	$\mathbf{x}$ \dotfill Element of $\Cl(p^\beta)$\\
	$*$-action of $\Delta_p$ \dotfill \S\ref{sec:slope-decomp}\\
	$\langle-\rangle_\lambda, \langle-\rangle_\Omega$\dotfill  actions of $H(\Zp)$ (\eqref{eq:langle-rangle},\eqref{eq:action V_Omega})\\
	$-_{\m_{\tilde\pi}}$ \dotfill Localisation at $\m_{\tilde\pi}$ (Def.~\ref{def:gen eigenspace})\\
	$-_{\m_{\tilde\pi}^{\dec}}$ \dotfill Localisation at $\m_{\tilde\pi}^{\dec}$ (Def.~\ref{def:gen eigenspace shalika})
\end{multicols}

\bibliographystyle{hsiam}
\bibliography{master_references}
\end{document}